\documentclass[10pt,twoside,reqno,final]{manuscript}

\usepackage{cite}
\usepackage{natbib}
\usepackage{stmaryrd}
\usepackage{xcolor} 

\usepackage[notcite,notref]{showkeys}

\hypersetup{
  colorlinks=true,
  linkcolor=magenta,
  citecolor=blue,
  filecolor=blue,
  urlcolor=blue
}

\graphicspath{{./figures/}}

\title{Optional Intervals Event and Two n-ary Finitary Operations: An Algebraic Framework for Unifying Parallel-Serial Execution and Axiomatizing Simultaneity from an Epistemological Perspective}

\author[1]{Zhongyuan.Li\thanks{First author, \texttt{li.zhong.yuan@outlook.com}}}
\author[1]{Yanlei.Gong\thanks{\texttt{g.goodian@gmail.com}}}
\author[1]{Lei.Yu\thanks{\texttt{yuleen\_@outlook.com}}}
\author[1]{Yue.Cao\thanks{\texttt{athen\_george@outlook.com}}}
\author[2]{Bo.Yin\thanks{Corresponding author, \texttt{yinbo@bistu.edu.cn}}}
\affil[1]{Independent Researcher}
\affil[2]{Beijing Information Science and Technology University}

\date{Date: \today}

\usepackage{hyperref}
\usepackage{amssymb}
\usepackage{makecell}
\usepackage{mathrsfs}
\usepackage{graphicx}
\usepackage{dcolumn}
\usepackage{bm}
\usepackage{relsize}
\usepackage{algorithm}
\usepackage{algorithmicx}
\usepackage{algpseudocode}

\newcommand{\RNum}[1]{\textbf{\uppercase\expandafter{\romannumeral #1\relax}}}

\newcommand{\EBL}{\mbox{}}

\newcommand{\BF}[1]{\textbf{#1}}                                        
\newcommand{\MXBF}[1]{\textbf{\color{black}#1}}

\newcommand{\CAL}[1]{\mathcal{#1}}                                      
\newcommand{\TSF}[1]{\textsf{#1}}                                        
\newcommand{\RM}[1]{\mathrm{#1}}                                        
\newcommand{\TT}[1]{\textit{#1}}                                        

\newcommand{\Pred}{\mathbb{P}}                                          
\newcommand{\pPre}{\mathfrak{p}}                                        
\newcommand{\fPre}{\mathfrak{f}}                                        

\newcommand{\CARDI}[1]{\left|{#1}\right|}                               

\newcommand{\ISoOieS}{\CAL{IS}}

\newcommand{\bIdxT}{\RM{Idx}\CAL{T}}                                    
\newcommand{\IdxT}{\RM{idx}\CAL{T}}                                     
\newcommand{\idxT}{\RM{idx}\CAL{T}}                                     
\newcommand{\idx}{\RM{idx}}                                             

\newcommand{\Dim}[1]{\dim{\left(#1\right)}}                                             

\newcommand{\TwoTpl}{\RM{2tuple}}                                       

\newcommand{\bTwoTplS}{\RM{2Tuple}\CAL{S}}                                 
\newcommand{\TwoTplS}{\RM{2tuple}\CAL{S}}                                  
\newcommand{\fMinFstOfTwoTplT}[1]{\fPre\RM{Min1of2tuple}\CAL{T}(#1)}    
\newcommand{\fMaxSndOfTwoTplT}[1]{\fPre\RM{Max2of2tuple}\CAL{T}(#1)}    

\newcommand{\bTwoTplSS}{\RM{2Tuple}\CAL{SS}}                               
\newcommand{\TwoTplSS}{\RM{2tuple}\CAL{SS}}                                 

\newcommand{\bTwoTplT}{\RM{2Tuple}\CAL{T}}                                                 
\newcommand{\TwoTplT}{\RM{2tuple}\CAL{T}}                                                        

\newcommand{\bBoundTwoT}{\RM{Bound2Tuple}}                                                 
\newcommand{\fBoundTwoT}[1]{\fPre\RM{Bound2tuple}(#1)}                                      

\newcommand{\bTwoTplTS}{\RM{2Tuple}\mathcal{TS}}                                           
\newcommand{\TwoTplTS}{\RM{2tuple}\mathcal{TS}}                                             

\newcommand{\bDomainFilteredSubTwoTplTS}{\bTwoTplTS_{\BF{Dom}}}                 
\newcommand{\domainFilteredSubTwoTplTS}{\TwoTplTS_{dom}}                    
\newcommand{\fDomainFilteredSubTwoTupleTS}[3]{\fPre\RM{DomFlt}\TwoTplTS(#1, #2, #3)}                                      

\newcommand{\bDomainFilterTwoTpl}{\RM{DomFlt2Tuple}}                                 

\newcommand{\bCompleteAscOrderFilterSubTwoTplTS}{\bTwoTplTS^{ASC}} 
\newcommand{\completeAscOrderFilterSubTwoTplTS}{\TwoTplTS^{asc}}         
\newcommand{\fCompleteAscOrderFilterSubTwoTplTS}[1]{\fPre\RM{AscOrderFlt2tupleTS}(#1)}                                       

\newcommand{\fNatIsoToCP}[2]{\fPre\RM{NatIso2CP}(#1,\ #2)}            

\newcommand{\bE}{\TSF{Event}}                                           
\newcommand{\E}{\TSF{event}}                                             
\newcommand{\bES}{\RM{Event}\CAL{S}}                                   
\newcommand{\ES}{\RM{event}\CAL{S}}                                     
\newcommand{\bEStar}{\RM{Event}^*}                                      
\newcommand{\EStar}{\RM{event}^*}                                       
\newcommand{\bAtomEStar}{\TSF{A}\RM{Event}^*}                              
\newcommand{\bCombEStar}{\TSF{C}\RM{Event}^*}                              

\newcommand{\bOperandS}{\RM{Operand}\CAL{S}}                            
\newcommand{\OperandS}{\RM{operand}\CAL{S}}                             

\newcommand{\bOperationS}{\RM{Operation}\CAL{S}}                            
\newcommand{\OperationS}{\RM{operation}\CAL{S}}                             

\newcommand{\UNSTARTED}{\textit{unstarted}}                                
\newcommand{\RUNNING}{\textit{running}}                                    
\newcommand{\FINISHED}{\textit{finished}}                                  

\newcommand{\bTSs}{\RM{TS}}                                       
\newcommand{\TSs}{\RM{ts}}                                         
\newcommand{\bTSe}{\RM{TE}}                                         
\newcommand{\TSe}{\RM{te}}                                           
\newcommand{\TS}{\RM{t}}

\newcommand{\bAtomE}{\TSF{Atom}\TSF{Event}}                                       
\newcommand{\atomE}{\TSF{atom}\TSF{Event}}                                       
\newcommand{\bCompE}{\TSF{Comp}\TSF{Event}}                                       
\newcommand{\compE}{\TSF{comp}\TSF{Event}}                                         

\newcommand{\C}{\CAL{C}}
\newcommand{\F}{\CAL{F}}
\newcommand{\I}{\CAL{I}}
\newcommand{\A}{\CAL{A}}
\newcommand{\M}{\CAL{M}}

\newcommand{\bOIE}{\RM{OIE}}               
\newcommand{\oie}{\RM{oie}}                
\newcommand{\Oie}{\RM{Oie}}                
\newcommand{\bOIES}{\RM{OIE}\CAL{S}}       
\newcommand{\oieS}{\RM{oie}\CAL{S}}        

\newcommand{\bAtomOIE}{\RM{Atom}\bOIE}                               
\newcommand{\atomOie}{\RM{atom}\Oie}                                  

\newcommand{\bCombOIE}{\RM{Comp}\bOIE}                               
\newcommand{\combOie}{\RM{comp}\Oie}                                  

\newcommand{\bVoidError}{\bOIE_\RM{void}}
\newcommand{\voidError}{\oie_\RM{void}}

\newcommand{\PIsInfeasIntTwoTplT}{\Pred\RM{IsInf2tuple}\CAL{T}}
\newcommand{\bInfeasIntTwoTplTS}{\RM{2Tuple}\CAL{TS}^{Inf}}               
\newcommand{\infeasIntTwoTplTS}{\RM{2tuple}\CAL{TS}^{inf}}                      
\newcommand{\fInfeasIntTwoTplTS}[2]{\fPre\RM{Inf2tupleTS}(#1, #2)}                        

\newcommand{\PIsFeasibleIntvlTwoTplT}{\Pred\RM{IsFeas2tupleT}}
\newcommand{\bFeasibleIntvlTwoTplTS}{\RM{2Tuple}\CAL{TS}^{feas}}                   
\newcommand{\feasibleIntvlTwoTplTS}{\RM{2tuple}\CAL{TS}^{feas}}                     
\newcommand{\fFeasibleIntvlTwoTplTS}[2]{\fPre\RM{Feas2tuple}\CAL{TS}(#1, #2)}     

\newcommand{\OrbitAdd}{\CAL{O}\left(oieS,\ {\oplus|_{\alpha}^{\beta}}\right)}
\newcommand{\OrbitMulti}{\CAL{O}\left(oieS,\ {\otimes}\right)}

\newtheorem{defi}{\BF{Definition}}
\newtheorem{pty}{\BF{Property}}
\newtheorem{cor}{\BF{Corollary}}
\newtheorem{axm}{\BF{Axiom}}

\newif\ifincludeFile
\includeFiletrue             

\begin{document}

\maketitle
\thispagestyle{firststyle}

\begin{abstract}
\ifincludeFile
    
\indent
This paper proposes an algebraic framework for analyzing event execution intervals and sequences,
introducing ``Optional Intervals Event ($\bOIE$)'' as a 4-tuple abstraction $(\C,\ \F,\ \I,\ \A)$
that serves as a pre-execution planning tool for real-world events.
The $\bOIE$ establishes a mapping to real-world events and stores all feasible execution intervals together with dependency relationships among sub-events.
Based on this abstraction,
we define two $n$-ary finitary operations:
(i) ``Complete Sequence Addition ($\oplus|_{\alpha}^{\beta}$)'',
which models concurrent events
with a certain degree of equal opportunity within a shared time domain;
and (ii) ``Complete Sequence Multiplication ($\otimes$)'',
which models strictly ordered sequential events.
We analyze the algebraic properties of these operations,
including closure, non-commutativity, permutational equivalence, and orbit spaces.
We prove that, for any non-degenerate finite $\bOIE$ set,
$\oplus|_{\alpha}^{\beta}$ yields a single-orbit space due to permutational equivalence,
whereas $\otimes$ may yield multiple orbits.
This orbital divergence rigorously captures the fundamental symmetry gap between concurrent and sequential execution.

In computer science,
this framework establishes an axiomatic algebraic system that
formally unifies parallel and serial execution as $n$-ary finitary operations.
It enables constraint-aware pre-execution planning and characterizes concurrent symmetry
via orbit-space analysis and permutational equivalence.
We also discuss applications to probability theory and physics,
including the distinction between process symmetry and outcome symmetry and a novel axiomatization of simultaneity from an epistemological perspective.
\fi

\end{abstract}


\section{Introduction}\label{sec:intro}\EBL
\ifincludeFile
    
This paper opens with two foundational questions that expose critical formal and epistemological gaps
in the analysis of temporal events across physics and computer science:

1. \TT{How can we rigorously formalize the so-called ``simultaneous start of two events''?}

The phrase ``Event X and Event Y started at the same time'' is ubiquitous
in academic and practical discourse, yet it raises an essential question:
to what precision can this simultaneity be verified?
Are the starting timestamps of the two events exactly equal to any number of decimal places?
As of the end of 2024, the world's most advanced atomic clock has a systematic uncertainty of $8 \times 10^{-19}$ seconds~\cite{PhysRevLett.133.023401},
meaning it cannot guarantee the accuracy of the so-called simultaneous starting moments of two observed events to \(10^{-40}\) seconds.
Strictly speaking, therefore, this scenario cannot be deemed truly ``simultaneous'' in a pointwise sense,
necessitating a more rigorous formalization.
Can we define an $n$-ary finitary operation $*$, expressed as
\begin{align}
    *(event_1,\ event_2,\ event_3)
\end{align}
to characterize the ``simultaneous start'' of events with an error below the detection threshold of current technology?

2. \TT{Why do both concurrent head-to-head races and sequential one-by-one races reliably determine a first-place finisher?}

The 100 metres sprint is a head-to-head race typically featuring 8 competing athletes.
In this paper, we use the $n$-ary finitary operation \(\oplus\) to represent a race with all athletes participating concurrently.
For 8 athletes ordered by lane, with the \(i\)-th athlete \(athlete_i\) in the \(i\)-th lane, a 100 metres race can be formalized as
\begin{align}
    \oplus\left( athlete_{1}, athlete_{2}, athlete_{3}, \cdots, athlete_{8} \right).
\end{align}
If we swap the lanes of the athletes in lane 1 and lane 2, the race is expressed as
\begin{align}
    \oplus\left( athlete_{2}, athlete_{1}, athlete_{3}, \cdots, athlete_{8} \right).
\end{align}
Lane swapping does not alter the starting conditions for each athlete,
all of whom start upon hearing the starting gun.
Intuitively, this operation thus exhibits a specific form of transformation invariance with respect to the race time window.

In contrast, downhill skiing is a winter sport where skiers compete sequentially on a single track,
with only one skier permitted on the course at a time.
We use the $n$-ary finitary operator \(\otimes\) to represent a race with all skiers participating in strict sequence.
For a race with 6 skiers, this is formalized as
\begin{align}
    \otimes\left( skier_{1}, skier_{2}, skier_{3}, \cdots, skier_{6} \right)
\end{align}
where the skier indexed on the left of the parameter list starts the race earlier than the skier on the right.
Intuitively, this operation does not exhibit the same transformation invariance as the concurrent race operation with respect to the race time window.

Despite these fundamental differences,
both the 100 metres sprint and downhill skiing are speed races that can definitively rank participants and award medals.
What is the root cause of this equivalence in outcome determinacy?
The essence of a speed race is the sorting of time costs from smallest to largest, a goal achieved by both race formats.
Is there an underlying set of mathematical properties that enables this temporal comparison,
which has yet to be formally characterized in existing literature?

Question 1 is a fundamental problem of great importance to event relationships in both computer science and physics.
In computer science, many theories formalizing parallel and serial execution rely on a global clock.
For example,
Allen's interval algebra~\cite{Allen83} and temporal constraint networks~\cite{DechterMP91} provide powerful inference mechanisms for fixed temporal relations.
They inherently assume intervals are precisely known a priori.
However, limitations of clock timing introduce uncertainty into this global clock,
which limits its ability to fully express physical time.
This gap in the theoretical foundation poses challenges for practical scenarios at small time scales.
The two types of racing competitions in Question 2 precisely mirror parallel and serial models in computer science.
Traditional parallel theories and certain frameworks (e.g., process algebra~\cite{BergstraK84}) abstract away from the uncertainty
arising from the difficulty of using global clocks to precisely formalize physical time,
directly resorting to algebraic symmetry
\begin{align}
    ab \sim ba.
\end{align}
While this approach successfully introduces algebraic methods and enriches both interpretive theory and practical methodology,
it neglects actual physical clock synchronization,
leaving certain aspects unresolved in the formalization of many derived problems.
Both problems essentially point to the need for an algebraic structure in traditional theories
that formalizes physical simultaneity while possessing robust mathematical properties,
an aspect not fully addressed by traditional theories.

To reconcile these two incompatible lines of thought,
this paper introduces sets to encompass the reality of physical clocks and proposes the abstract structure of the \emph{Optional Intervals Event} ($\bOIE$).
From the perspective of planning and prediction,
we then perform $n$-ary finitary operations to realize computations on event ordering at the planning level.
This framework not only covers clock synchronization at the level of physical reality from an epistemological standpoint,
but also possesses strong algebraic properties that enable richer extensions, thereby achieving powerful capabilities in event planning and scheduling.

In computer science, we analyze in detail the algebraic properties of $\bOIE$ and the two $n$-ary finitary operations,
with particular emphasis on permutational equivalence and orbit spaces,
and further demonstrate its tremendous practical potential through an algebraic analysis of sorting parallelization.

Motivated by this, we apply the algebraic system to a foundational question in probability theory:
Why do simultaneous drawing and sequential drawing without replacement yield identical results?
The former relies on combinatorial counting, while the latter depends on conditional probability.
Nevertheless, their outcomes coincide.
We propose the concepts of \emph{process symmetry} and \emph{outcome symmetry}
that are neglected in traditional probability theory.

In physics, from an epistemological perspective,
we use this framework to resolve the dilemma faced by both classical and relativistic mechanics
when axiomatizing simultaneity from the observer’s viewpoint,
and propose an axiomatization of simultaneity that overcomes the limitations and errors of observational technology.

Finally, from a group-theoretic perspective,
we further abstract this algebraic system into a finite commutative semigroup with an absorbing element and reflexive absorption property,
construct its Cayley table, and analyze its widespread manifestations in practical applications including computer science, linguistics and industrial domains.

The remainder of this paper is structured as follows:
\begin{itemize}
    \item \BF{Section~\ref{sec:event}. Analysis of Event Properties} analyzes the fundamental properties of events.
    We formalize the concept of an $\bE$ instance as an entity with three execution states ($\UNSTARTED$, $\RUNNING$, $\FINISHED$),
    define its starting and ending timestamps, introduce the notions of atomic and composite events,
    and establish the core property of persistence of execution status.
    \item \BF{Section~\ref{sec:oie}. Optional Intervals Event and Analysis} proposes the core abstraction of the Optional Intervals Event ($\bOIE$),
    a 4-tuple \((\C, \F, \I, \A)\) that maps to real-world events.
    This 4-tuple stores all feasible execution intervals and dependency relationships among sub-events,
    serving as a formal tool for pre-execution planning.
    We also define the atomic and composite classifications of $\bOIE$ instances.
    \item \BF{Section~\ref{sec:ops}. Two Sequence Operations on the Elements of a Finite Set of $\bOIE$ Instances} introduces two $n$-ary finitary operations on $\bOIE$ instances: Complete Sequence Addition (\(\oplus|_{\alpha}^{\beta}\)),
    which models concurrent events where all participants have a degree of equal opportunity to start and end within a specified time domain;
    and Complete Sequence Multiplication (\(\otimes\)), which models strictly ordered sequential events.
    We detail the operational rules, helper functions, and handling of infeasible combinations.
    \item \BF{Section~\ref{sec:algbrProp}. Algebraic Properties of Sequence Operations}
    analyzes the algebraic properties of the two sequence operations,
    proving core properties including closure, general non-commutativity, and permutational equivalence.
    We define each operation’s orbit space:
    permutation equivalence makes $\oplus|_{\alpha}^{\beta}$ produce a single orbit for finite $\bOIES$ instance,
    whereas ordered constraints make $\otimes$ generate multiple distinct orbits.
    \item \BF{Section~\ref{sec:impl}. The Implementation of Optional Intervals Event}
    formalizes two standardized implementation modes of $\bOIE$ instances,
    establishes a projection mechanism that maps sequence operations to general $n$-ary finitary operations.
    We also prove the uniqueness of the operation set within each equivalence class of the orbit space.
    \item \BF{Section~\ref{sec:app-comp}. Applications and Comparisons in Computer Science, Probability Theory and Physics}
    validates the $\bOIE$ framework and its two proposed sequence operations across three core disciplines,
    and presents a systematic comparative analysis with mainstream conventional methodological frameworks in corresponding fields.
    In computer science, this framework formally unifies parallel and serial execution and enables the modeling of hybrid computing scenarios.
    We perform comparative analyses against mainstream classic concurrency models, namely Process Algebra, Petri nets, and Allen’s interval algebra.
    Based on these comparisons, we further delineate the core strengths of the $\bOIE$ framework relative to conventional formal methods,
    specifically in terms of temporal expression completeness and pre-execution planning capability.
    In probability theory, the framework differentiates between process symmetry and outcome symmetry within classical probability by means of orbit space.
    Focusing on simultaneous and sequential sampling cases,
    it contrasts the paper’s unified algebraic modeling framework against the disjoint analytical approaches of combinatorial counting and conditional probability adopted in conventional classical probability,
    Additionally, this work delivers a novel algebraic interpretation for fundamental probabilistic notions.
    In physics,
    it reconstructs the epistemology of ``simultaneity'' to resolve the dilemma of unprovable pointwise temporal equality under observational limits.
    We conduct a systematic comparison between the equipotential domain based simultaneity perspective proposed in this paper and the absolute simultaneity in Newtonian mechanics as well as the relativity of simultaneity in relativistic mechanics,
    unifies the algebraic nature of concurrent and sequential speed races.
    We also analyze the legitimacy of the repeatability of physical experiments.
    \item \BF{Section~\ref{sec:cayley}. A Finite Commutative Semigroup with Absorbing Element and Reflexive Absorption Property}
    further abstracts the $\bOIE$ algebraic system from a group-theoretic perspective.
    It formalizes a special finite commutative semigroup equipped with an absorbing element and the reflexive absorption property,
    constructs its N-dimensional Full CSA Cayley table via Newton's binomial combinatorial structure,
    and establishes a bijective mapping between this special semigroup and the orbit space of a special class of $\bOIES$ instances under $\oplus|_{\alpha}^{\beta}$.
    It further demonstrates the universal applicability of this algebraic structure by
    revealing its formal correspondence with the referential distinctness constraint of natural language conjunction ``And'' and with composition/inheritance semantics
    in object-oriented programming, thereby unifying linguistic, computational, and physical feasibility constraints under a single algebraic grammar.
    \item \BF{Section~\ref{sec:impli}. Derived Implications}
    outlines promising directions for future work, including:
    the integration of category theory;
    extensions of sequence operations (e.g., ``incomplete'' sequence operations);
    applications to discrete time models;
    generalizations beyond temporal variables;
    the incorporation of probability measures;
    the exploration of additional operation types (subtraction and division);
    an in-depth study of the Full CSA Diagram structure;
    the extension of applicability to non-classical mechanics;
    and analysis of time complexity.
    \item The Appendix provides links to the simulation code for the $\bOIE$ framework(Appendix~\ref{sec:appendix.simul}),
    generation code for the full CSA diagram(Appendix~\ref{sec:appendix.lzyGraph}) and time complexity analysis(Appendix~\ref{sec:appendix.complex}).
\end{itemize}
\fi

\section{Analysis of Event Properties}\label{sec:event}\EBL
\ifincludeFile
Event is a noun with complex meanings in both practical and philosophical aspects.
It is a fundamental concept frequently used
in the fields of physics, computer science, engineering and social sciences~\cite{sep-events}.
Generally, we refer to the event's class as ``$\bE$'' —
the abstract category of events, analogous to a blueprint —
while an $\bE$ instance denotes a specific, concrete instance of this class (a single real-world event built from that blueprint).
In this section, we discuss some basic properties of $\bE$,
which lay the formal foundation for the core $\bOIE$ abstraction and sequence operations proposed in subsequent sections.

\begin{itemize}
\item Subsection~\ref{subsec:event.status}
clarifies the three execution states of an event,
the definitions and core properties of starting and ending timestamps,
and standardizes the representation form of time intervals.

\item Subsection~\ref{subsec:event.atomAndCombEvent}
defines the concepts of atomic events and composite events,
and provides the calculation method for the starting and ending timestamps of composite events.

\item Subsection~\ref{subsec:event.eventStar} proposes the abstract concept of Event with Undetermined Interval ($\bEStar$).
This core intermediate abstraction connects fixed real-world events to the flexible $\bOIE$ framework.
By substituting fixed time intervals with undetermined variables,
it achieves flexible mapping of real-world events and further differentiates atomic events from composite events.
\end{itemize}
\fi
    \subsection{Execution Status, Start and End Timestamps of Events}\label{subsec:event.status}\EBL
    \ifincludeFile
According to the execution situation,
an $\bE$ instance(denoted as $\E$) has three execution statuses:
``$\UNSTARTED$'', ``$\RUNNING$'' and ``$\FINISHED$''.
These statuses transition deterministically based
on the initiation, progress, and completion of the event, as follows:

\begin{itemize}
\item Before $\E$ is executed, it remains in the $\UNSTARTED$ status indefinitely.
\item After $\E$ is executed and completed, it stays in the $\FINISHED$ status permanently.
\item At the moment $\E$ starts execution, its status transitions from $\UNSTARTED$ to $\RUNNING$.
\item At the moment $\E$ ends, its status transitions from $\RUNNING$ to $\FINISHED$.
\end{itemize}

Since the $\bE$ inherently represents dynamic motion rather than static state,
the following core properties of $\bE$ are formalized below.

\begin{pty}[\BF{Persistence of event execution statuses}]\label{prop:event.status.persistence}\EBL

Suppose there exists an $\bE$ instance $\E$.
If $\E$ is in one of the three execution statuses at timestamp $t_0$,
then there must be a real number $\epsilon > 0$,
such that among the left punctured neighborhood set
\begin{align}
    \CAL{S}_{left} = \{ t|t \in (t_0 - \epsilon, t_0) \},
\end{align}
and the right punctured neighborhood set
\begin{align}
    \CAL{S}_{right} = \{ t|t \in (t_0, t_0 + \epsilon) \},
\end{align}
at least one of the two sets contains all real numbers corresponding to timestamp instants
where the event maintains the same status as it had at $t_0$.
\end{pty}

\indent
Property~\ref{prop:event.status.persistence} precludes the following situation:
\begin{quote}
    \TT{$\E$ is in one of the three execution statuses at timestamp $t_0$,
    and there exists a real number $\epsilon > 0$ such that the event is in different status at all time
    instants corresponding to the real numbers in the set $\{ t|t \in (t_0 - \epsilon, t_0) \cup (t_0, t_0 + \epsilon)\}$.}
\end{quote}

To indicate that an $\bE$ instance $\E$ is in a certain status during a time interval,
we use left-closed and right-open intervals for bounded real-number timestamps.
For the special cases of $-\infty$ and $+\infty$, we use open intervals.
For example:
\begin{quote}
    \TT{An $\bE$ instance $\E_i$ is in the $\UNSTARTED$ status in the interval $(-\infty, t_1)$,
    in the $\RUNNING$ status in the interval $[t_1, t_2)$,
    and in the $\FINISHED$ status in the interval $[t_2, +\infty)$.
    This indicates that $\E_i$ undergoes status transitions at timestamps $t_1$ and $t_2$.}
\end{quote}

This paper only focuses on $\bE$ instances with bounded starting and ending timestamps.
That is, for the events we consider,
it is impossible for them to be in the $\RUNNING$ status within the following three types of time intervals:
    \begin{align}
        & (-\infty, t_\alpha),\ [t_\beta, +\infty),\ (-\infty, +\infty),
    \end{align}
where $t_\alpha \neq +\infty \wedge t_\beta \neq -\infty$.

\begin{defi}[\textbf{Starting and ending timestamp of an event}]\label{def:event.startAndEndTS}\EBL

Given an $\bE$ instance named $\E$.
If the status of $\E$ at the timestamp $t_0$ is $\RUNNING$,
and for all real numbers $t$ in the set $\{ t|t < t_0\}$, the status of the event at the corresponding moment is $\UNSTARTED$,
then $t_0$ is defined as the ``starting timestamp'' of $\E$.
We denote it as $\bTSs$.
For an instance of $\bTSs$, we use ``$\TSs$'' to represent it.

If for all real numbers $t$ in the set $\{ t|t \geq t_0^{'}\}$,
the status of $\E$ at the corresponding moment is $\FINISHED$,
and there exists a real number $\epsilon > 0$ such that for all real numbers $t$ in the set $\{t|t \in (t_0^{'} - \epsilon, t_0^{'} )\}$,
the status of $\E$ at the corresponding moment is $\RUNNING$,
then $t_0^{'}$ is defined as the ``ending timestamp'' of $\E$.
We denote it as $\bTSe$.
For an instance of $\bTSe$, we use ``$\TSe$'' to represent it.
\end{defi}

\begin{defi}[\textbf{The interval 2-tuple of an event}]\label{def:event.IntvlTwoTpl}\EBL
\indent Given an $\bE$ instance named $\E$,
$\TSs$ is the starting timestamp of $\E$,
$\TSe$ is the ending timestamp of $\E$.
Then, $[\TSs, \TSe)$ is called ``The interval of $\E$'',
the 2-tuple $(\TSs, \TSe)$ is called ``the interval 2-tuple of $\E$''.
\end{defi}

\indent We adopt the naming convention from Definition~\ref{def:event.startAndEndTS} throughout this paper.
A noun definition specifies the nature of a class of instances.
The question of ``what it is'' operates at the abstract level,
while an instance is a specific object created using the noun definition.
Therefore, instances are usually used in specific practical actions.
For example, consider the following expression:
\begin{quote}
    There is an $\bE$ instance named $\E$, whose $\bTSs$ instance is $\TSs$.
\end{quote}
In this sentence, $\bE$ and $\bTSs$ are both nouns defined earlier, and they are abstract concepts.
In computer science, specifically in object-oriented programming, they are called classes.
By contrast, $\E$ and $\TSs$ are instances, which are concrete.
In object-oriented programming, they are called objects.

\indent The following are several properties of $\bE$ that we consider.
\begin{pty}[\textbf{Uniqueness of starting timestamp and ending timestamp}]\label{prop:event.start.exist}\EBL

An $\bE$ instance $\E$ has and must have exactly one starting timestamp and one ending timestamp.
There cannot be multiple starting or ending timestamps for $\E$.
\end{pty}

\begin{pty}[\textbf{Ordering of starting timestamp and ending timestamp}]\label{prop:event.startEndTs.seq}\EBL

The starting timestamp of an $\bE$ instance must be less than its ending timestamp.
\end{pty}

\begin{pty}[\textbf{Stability of event execution status}]\label{prop:event.status.stability}\EBL

An $\bE$ instance is in the $\RUNNING$ status at every moment within the interval $[\TSs, \TSe)$.
It is in the $\UNSTARTED$ status at every moment within the interval $(-\infty, \TSs)$.
It is in the $\FINISHED$ status at every moment within the interval $[\TSe, +\infty)$.
\end{pty}

\fi
    \subsection{Atomic and Composite Events}\label{subsec:event.atomAndCombEvent}\EBL
    \ifincludeFile
Multiple $\bE$ instances may be combined into one $\bE$ instance.
However, in many cases,
we analyze $\bE$ instances under the assumption that they cannot be split.

Atomic events and composite events are fundamental concepts in the fields
of probability theory, mathematical statistics, and computer science.
Due to the extensive application of probability theory and statistics in physics
and many engineering disciplines, these concepts also appear in many other areas.

The atomicity and compositionality of an event are defined
with respect to the behavior of a chosen quantifiable indicator.
If the selected quantifiable indicator remains unchanged
during the execution of an event,
the event is classified as atomic with respect to that indicator.
Conversely, if the indicator changes,
this event is a composite event.
A composite event can be decomposed according to each stable value of
a certain indicator during its change process,
with each value corresponding to an atomic event.

In this paper,
the type for atomic $\bE$ instances is
denoted as $\bAtomE$, and an instance of it is denoted as $\atomE$.
The type for composite $\bE$ instances is
denoted as $\bCompE$, and an instance of it is denoted as $\compE$.

A $\bCompE$ instance $\compE$ can be expressed as follows:
\[
    \compE = \{ \E_1, \E_2, \E_3, \ldots, \E_n \}.
\]

\begin{defi}[\textbf{Composite event's starting timestamp and ending timestamp}]\label{def:TSsAndTSeOfCombE}\EBL

\indent Given a $\bCompE$ instance $\compE$, with the form
\[
    \compE = \{ \E_1, \E_2, \E_3, \ldots, \E_n \}.
\]
We define the minimum starting timestamp of all its sub-events as the starting timestamp of $\compE$
\begin{align}
    \TSs_{\compE} = \min\{ \TSs_{\E_i} | i \in \{1, 2, \cdots, n\} \},
\end{align}
and the maximum ending timestamp of all its sub-events as the ending timestamp of $\compE$
\begin{align}
    \TSe_{\compE} = \max\{ \TSe_{\E_i} | i \in \{1, 2, \cdots, n\} \}
\end{align}
\end{defi}
\fi
    \subsection{Event with Undetermined Interval}\label{subsec:event.eventStar}\EBL
    \ifincludeFile
In this subsection,
we introduce the first abstraction of real-world events in this paper,
the concept of ``Event with Undetermined Interval'' (denoted as $\bEStar$).
This abstraction is achieved by replacing the fixed time interval of an event with undetermined variables,
and thereby creating a bijective mapping between the abstract entity and the original event.
While preserving other attributes of the event (such as subject and space),
this approach introduces variability into the temporal property,
enabling flexible assignment or combination of different time intervals in subsequent analysis and operations.
Using the example of ``A doctor submitting a paper'',
we illustrate how a real-world event can be abstracted into an $\bEStar$ instance and expressed mathematically,
laying the groundwork for the concepts in later chapters.

\begin{defi}[\textbf{Event with undetermined interval}]\label{def:eventStar}\EBL

An abstract entity derived from an $\bE$ instance by specifying its intervals with undetermined variables,
is defined as an ``Event with Undetermined Interval'', denoted as ``$\bEStar$''.
There exists a bijective mapping between $\bEStar$ and $\bE$, denoted as
\begin{align}
    \bEStar \leftrightarrow \bE.
\end{align}
\end{defi}

Events are typically abstracted to include attributes such as subject and spatial location.
In particular, it should be emphasized that an event has a clear and fixed interval.
Definition~\ref{def:eventStar} precisely focuses on variabilizing the ``interval'' property.
In many cases, we can provide multiple candidate intervals for the event.

For example, suppose there is an $\bE$ instance
\[
    \begin{aligned}
        & \qquad A\ doctor\ submitted\ a\ paper\ to\ a\ journal\ at\ a\ university\ in\ Beijing \\
        & on\ December\ 1,\ 2025,\ between\ 10:00\ a.m.\ and\ 12:00\ p.m\ (UTC+8).
    \end{aligned}
\]
We denote this $\bE$ instance as $\E_{Dr}$,
the interval of $\E_{Dr}$ is ``December 1, 2025, between 10:00 a.m and 12:00 p.m.'',
denoted as
\[
    [\ 1764583200,\ 1764590400\ ).
\]

If we set this interval as ``to be determined'' while keeping all other properties unchanged,
the expression thus becomes as follows:
\[
    \begin{aligned}
        & \qquad A\ doctor\ submitted\ a\ paper\ to\ a\ journal\ at\ a\ university\ in\ Beijing \\
        & on\ December\ 1,\ 2025,\ between\ [\ ts,\ te\ ).
    \end{aligned}
\]

By definition~\ref{def:eventStar},
we can abstract this expression as an $\bEStar$ instance $\EStar_{Dr}$,
and
\[
    \EStar_{Dr} \leftrightarrow \E_{Dr}.
\]

\begin{defi}[\textbf{Atomic and composite $\bEStar$}]\label{def:eventStar.atomic}\EBL

An $\bEStar$ instance is an atomic $\bEStar$ instance if
it is in bijective correspondence with an $\bAtomE$ instance,
denoted as ``$\bAtomEStar$''; while an $\bEStar$ instance is a composite $\bEStar$ instance
if it bijectively maps to a $\bCompE$ instance, denoted as ``$\bCombEStar$''.
\end{defi}

\fi

\section{Optional Intervals Event and analysis}\label{sec:oie}\EBL
\ifincludeFile
Before an event occurs, is it prearranged?
Does the planner fully grasp every moment of its execution?
What pre-execution planning is required for such arranged events,
and can we extract and mathematically formalize key elements from these plans?
This section explores these questions and introduces novel event-related concepts,
with the following subsections organized to elaborate the core framework systematically:

\begin{itemize}
\item Subsection~\ref{subsec:oie.oie}
formally defines Optional Intervals Event ($\bOIE$) as a 4-tuple (\(\C, \F, \I, \A\))
that maps to real-world events.
It stores the feasible execution intervals of the mapped event and the dependency relationships among its sub-events,
and also specifies the properties and classifications ($\bAtomOIE$ and $\bCombOIE$) of $\bOIE$ instances.

\item Subsection~\ref{subsec:oie.oieSet}
introduces $\bOIES$(a set of $\bOIE$ instances) and
focuses on the set of $\I$ derived from its members.

\item Subsection~\ref{subsec:oie.voidOie}
defines the void $\bOIE$ instance ($\bVoidError$),
which cannot map to any real-world event.
\end{itemize}

To standardize the description of complex mathematical types involving 2-tuples
throughout the subsequent formal definition of $\bOIE$,
the following naming conventions are adopted
(with ``Type Name'' referring to the abstract category with the first letter capitalized,
and ``Type Instance'' referring to the specific object with the first letter in lowercase):

\begin{center}
    \begin{tabular}{|l|l|l|} 
        \hline
        Type Name & Type Instance & Core Meaning \\
        \hline
        $\bTwoTplS$ & $\TwoTplS$ & Set of 2-tuples \\
        $\bTwoTplSS$ & $\TwoTplSS$ & Set of sets of 2-tuples \\
        $\bTwoTplT$ & $\TwoTplT$ & Tuple of 2-tuples \\
        $\bTwoTplTS$ & $\TwoTplTS$ & Set of tuples of 2-tuples \\
        \hline
    \end{tabular}
\end{center}
The following are examples for each type
\begin{itemize}
    \item $\bTwoTplS$
    \begin{align}
        \begin{aligned}
            & \TwoTplS_1 = \{(1, 2), (3, 4), (5, 6)\},\\
            & \TwoTplS_2 = \{(7, 8), (9, 10), (11, 12)\},\\
            & \TwoTplS_3 = \emptyset
        \end{aligned}
    \end{align}
    \item $\bTwoTplSS$
    \begin{align}
        \begin{aligned}
            & \TwoTplSS_a = \{ \{(1, 2), (3, 4)\},\ \{(9, 10), (11, 12)\} \}, \\
            & \TwoTplSS_b = \{ \{(1, 2), (3, 4)\},\ \{(9, 10), (11, 12)\},\ \{(5, 6), (7, 8)\} \}, \\
            & \TwoTplSS_c = \emptyset
        \end{aligned}
    \end{align}
    \item $\bTwoTplT$
    \begin{align}
        \begin{aligned}
            & \TwoTplT_x = ( (1, 2), (3, 4), (5, 6), (7, 8) ),\\
            & \TwoTplS_y = ( (9, 10), (11, 12) ),\\
            & \TwoTplS_z = ()
        \end{aligned}
    \end{align}
    \item $\bTwoTplTS$
    \begin{align}
        \begin{aligned}
            & \TwoTplTS_{\phi} = \{ ( (1, 2), (3, 4) ), ( (5, 6), (7, 8) ) \},\\
            & \TwoTplTS_{\chi} = \{ ( (9, 10)), ((11, 12)), ((13, 14)) \},\\
            & \TwoTplTS_{\psi} = \emptyset
        \end{aligned}
    \end{align}
\end{itemize}

Here we emphasize that:

``\BF{\TT{For all 2-tuples presented in this paper,
    the first element denotes the starting timestamp of an event and the second refers to its ending timestamp.
    Accordingly, in this paper the first item in a 2-tuple is always smaller than the second item}}''.

With the above naming conventions clarified,
we now proceed to the formal definition and analysis of $\bOIE$.\fi
    \subsection{Optional Intervals Event}\label{subsec:oie.oie}\EBL
    \ifincludeFile
Building on the concept of $\bE$, we propose an important new abstract concept,
a 4-tuple.
This 4-tuple stores all optional intervals of the mapped $\bE$ instance and guides the starting and ending timestamps.
\fi
        \subsubsection{Definition}\label{subsubsec:oie.oie.def}\EBL
        \ifincludeFile
\begin{defi}[\textbf{Optional Intervals Event}]\label{def:oie}\EBL

An Optional Intervals Event ($\bOIE$) is a 4-tuple that may map to at most one $\bE$ instance.
For an $\bOIE$ instance $\oie$, if there exists an $\bE$ instance $\E$
such that $\oie$ maps to $\E$,
we denote this mapping as:
\begin{align}
    \oie \rightarrow \E.
\end{align}
The formal structure of an $\bOIE$ is defined as the 4-tuple:
\begin{align}
    & \quad \mathlarger{\mathlarger{(\C, \F, \I, \A)}}
\end{align}
where the four constituent elements are specified as following table
\begin{center}
    \begin{tabular}{|l|l|}
        \hline
        $\C$(1st element): & \makecell[l]{A tuple of constituent $\bOIE$ instances that compose the current instance; \\ may be an empty tuple ().} \\ \hline
        $\F$(2nd element): & \makecell[l]{A $\bTwoTplTS$ instance(set of tuples of 2-tuples) storing the details\\ of the feasible interval combinations of the $\bE$ instances mapped by\\ the elements of $\CAL{C}$.} \\ \hline
        $\I$(3rd element): & \makecell[l]{A $\bTwoTplS$ instance(set of 2-tuples) storing the overall intervals\\ which can be derived via $\CAL{F}$.} \\ \hline
        $\A$(4th element): & \makecell[l]{A set of $\bAtomEStar$ instances that map all atomic events of the\\ real-world event corresponding to the current $\bOIE$ instance.} \\ \hline
    \end{tabular}
\end{center}
\end{defi}

As a novel abstraction for pre-execution event planning,
an $\bOIE$ instance can be intuitively analogized to a schedule planner that precomputes
and stores all feasible execution windows for a real world event prior to its occurrence.
Below, we elaborate on the semantic meaning of each 4-tuple element,
with formal examples to illustrate their structure.

\begin{itemize}

\item \BF{1. Component tuple $\C$ - A bill of immediate components}

\indent The first element $\C$ acts as a bill of materials for the $\bOIE$ instance,
enumerating its immediate constituent sub $\bOIE$ instances.
For an $\bOIE$ instance mapping to an atomic event (indivisible in the target context),
$\C$ is an empty tuple (), as no subcomponents exist.
Formal examples include:
\begin{align}
    \begin{aligned}
        & \C_{\oie_A} = (\oie_1, \oie_2), \\
        & \C_{\oie_B} = (\oie_2, \oie_3, \oie_1), \\
        & \C_{\oie_1} = \C_{\oie_2} = \C_{\oie_3} = ().
    \end{aligned}
\end{align}

\item \BF{2. Feasible schedule set $\F$ - A detailed schedule}

\indent The second element $\F$ is a set of feasible execution plans,
where each element is a tuple of 2-tuples specifying the execution interval
of each constituent $\bOIE$ instance in that plan.
The cardinality of $\F$ equals the number of valid planning schemes.
For a plan index $\theta$, the element
\begin{align}
    \left(\ (\TSs_{{\theta}1},\ \TSe_{{\theta}1}),\ (\TSs_{{\theta}2},\ \TSe_{{\theta}2}),\ \cdots,\ (\TSs_{{\theta}n},\ \TSe_{{\theta}n})\ \right) \in \F
\end{align}
denotes that the i-th constituent $\bOIE$ instance executes in the interval $[\TSs_{\theta{i}}, \TSe_{\theta{i}})$ under plan $\theta$,
with $\TSs_{\theta{i}} < \TSe_{\theta{i}}$ for all i.
The formal structure of $\F$ is:
\begin{align}
    \begin{aligned}
        & \CAL{F}_{\oie} = \{ \\
        & \qquad \left(\ (\TSs_{11},\ \TSe_{11}),\ (\TSs_{12},\ \TSe_{12}),\ \cdots,\ (\TSs_{1n},\ \TSe_{1n})\ \right), \\
        & \qquad \left(\ (\TSs_{21},\ \TSe_{21}),\ (\TSs_{22},\ \TSe_{22}),\ \cdots,\ (\TSs_{2n},\ \TSe_{2n})\ \right), \\
        & \qquad \cdots, \\
        & \qquad \left(\ (\TSs_{\lambda{1}},\ \TSe_{\lambda{1}}),\ (\TSs_{\lambda{2}},\ \TSe_{\lambda{2}}),\ \cdots,\ (\TSs_{\lambda{n}},\ \TSe_{\lambda{n}})\ \right) \\
        & \}.
    \end{aligned}
\end{align}

\item \BF{3. An overall schedule - Aggregated Interval Set $\I$}

\indent
The 3rd element $\I$ is the aggregated set of overall execution intervals for the $\bOIE$ instance,
derived by computing the temporal bound of each feasible plan in $\F$.
Each element $(\TSs_i, \TSe_i) \in \I$ corresponds to a valid overall execution window $[\TSs_i, \TSe_i)$
for the mapped $\bE$ instance, with $\TSs_i < \TSe_i$.
The formal structure is:
\begin{align}
    \begin{aligned}
        & \I_{\oie} = \{\ (\TSs_{{1}},\ \TSe_{{1}}),\ (\TSs_{{2}},\ \TSe_{{2}}),\ \cdots,\ (\TSs_{\Omega},\ \TSe_{\Omega})\ \}, \\
        \land\ & \forall (\TSs_{i}, \TSe_{i}) \in \I_{\oie}:\ \TSs_{i} < \TSe_{i}.
    \end{aligned}
\end{align}
Some formal examples of the 3rd item are as follows:
\begin{align}
    \begin{aligned}
        & \CAL{I}_{\oie_1} = \{\ (1759740600,\ 1759741200),\ (1759741200,\ 1759741800)\ \}, \\
        & \CAL{I}_{\oie_2} = \{\ (1759741200,\ 1759741800)\ \}, \\
        & \CAL{I}_{\oie_\Omega} = \emptyset.
    \end{aligned}
\end{align}
1759740600, 1759741200 and 1759741800 are all timestamps,
they are respectively 08:50, 09:00 and 09:10 on November 6, 2025 (Beijing Time).
For $\oie_1$, suppose there is an $\bE$ instance $\E_1$ in the real world that satisfies
\begin{align}
    \oie_1 \rightarrow \E_1.
\end{align}
Then, $\E_1$ has two feasible intervals, which are
\begin{align}
    [1759740600,\ 1759741200),\ [1759741200,\ 1759741800).
\end{align}
Similarly, for $\oie_2$, suppose there is an $\bE$ instance $\E_2$ in the real world that satisfies
\begin{align}
    \oie_2 \rightarrow \E_2,
\end{align}
and $\E_2$ has only one feasible interval, which is
\begin{align}
    [1759741200, 1759741800).
\end{align}
$\oie_\Omega$ is a special case, whose 3rd element $\CAL{I}_{\oie_\Omega}$ is $\emptyset$.
However, it cannot map to any real-world event.

\item 4. \BF{Atomic event set $\A$ - All practical tasks guided by the atomic components}

The 4th element $\A$ is the set of atomic undetermined-interval events ($\bAtomEStar$ instances)
corresponding to the indivisible sub-events of the mapped real-world event.
This element reinforces the mapping between the abstract $\bOIE$ and real-world events,
and enables feasibility validation for composite event planning:
if any atomic event in $\A$ is infeasible,
the entire $\bOIE$ instance is invalid.
Formal examples include:
\begin{align}
    \begin{aligned}
        & \A_{\oie_1} = \{ \EStar_{1} \}, \\
        & \A_{\oie_2} = \{ \EStar_{2} \}, \\
        & \A_{\oie_\Omega} = \emptyset.
    \end{aligned}
\end{align}
\end{itemize}

\BF{Running Example: Three Doctors’ Conference Submissions}

To concretize the $\bOIE$ definition,
we introduce a running example that will be used throughout the remainder of the paper.
Consider a conference submission window spanning 00:00 August 1 to 00:00 August 2, 2025,
denoted as the interval $[0,\ 24)$ (hour-based timestamps for simplicity).
Three doctors ($Dr_A$, $Dr_B$, $Dr_C$) have the following feasible submission windows
(each submission requires exclusive use of the full window):
\begin{itemize}
    \item $Dr_A$: [0, 1), [21, 22),
    \item $Dr_B$: [0, 1), [13, 14), [20, 22),
    \item $Dr_C$: [0, 1), [19, 22).
\end{itemize}

Feasibility constraints:
Only two doctors may submit simultaneously in any window;
three concurrent submissions are infeasible (e.g., all three submitting in $[0,\ 1)$ is invalid, while any two are valid).

We use 3 $\bE$ instances to represent the submission of each doctor
\begin{align}
    \E_{Dr_A},\ \E_{Dr_B},\ \E_{Dr_C}.
\end{align}
And we use 4 $\bOIE$ instances to map the corresponding events
\begin{align}
    \begin{aligned}
        & \oie_A \rightarrow \E_{Dr_A}, \\
        & \oie_B \rightarrow \E_{Dr_B}, \\
        & \oie_C \rightarrow \E_{Dr_C}, \\
        & \oie_{A_\theta} \rightarrow \E_{Dr_A}. \\
    \end{aligned}
\end{align}
$\oie_{A_\theta}$ is a special $\bOIE$ instance that requires $Dr_A$ to submit their papers exclusively within the time slot 21:00–22:00.

In this scenario,
the $\C$(1st element) of each $\bOIE$ instance can be expressed as:
\begin{align}
    \CAL{C}_{\oie_A} = (), \CAL{C}_{\oie_B} = (), \CAL{C}_{\oie_C} = (), \CAL{C}_{\oie_{A_\theta}} = ().
\end{align}
The empty tuple indicates that the submission event is atomic and cannot be decomposed into smaller sub-events.
Subsection~\ref{subsec:oie.oieSet} will elaborate on this concept.

In our analogy, their $\I$ (3rd element) values reflect the final summary of the intervals corresponding to each sub-event.
They can be expressed as follows:
\begin{align}
    \begin{aligned}
        & \CAL{I}_{\oie_A} = \{\ (0,\ 1),\ (21,\ 22)\ \}, \\
        & \CAL{I}_{\oie_B} = \{\ (0,\ 1),\ (13,\ 14),\ (20,\ 22)\ \}, \\
        & \CAL{I}_{\oie_C} = \{\ (0,\ 1),\ (19,\ 22)\ \}, \\
        & \CAL{I}_{\oie_{A_\theta}} = \{\ (21,\ 22)\ \}.
    \end{aligned}
\end{align}

In this example, each doctor's paper submission does not need to be divided into several smaller events.
Therefore, each tuple in $\F$(2nd element) only requires one element.
They can be expressed as follows:
\begin{align}
    \begin{aligned}
        & \CAL{F}_{\oie_A} = \{\ ((0,\ 1)),\ ((21,\ 22))\ \}, \\
        & \CAL{F}_{\oie_B} = \{\ ((0,\ 1)),\ ((13,\ 14)),\ ((20,\ 22))\ \}, \\
        & \CAL{F}_{\oie_C} = \{\ ((0,\ 1)),\ ((19,\ 22))\ \}, \\
        & \CAL{F}_{\oie_{A_\theta}} = \{\ ((21,\ 22))\ \}.
    \end{aligned}
\end{align}

The $\A$ (4th element) of each $\bOIE$ instance reflects
the atomic practical $\bE$ instances with undetermined execution time,
mapped by each atomic $\bOIE$ component.
They can be expressed as follows:
\begin{align}
    \begin{aligned}
        & \CAL{A}_{\oie_A} = \{ \EStar_{Dr_A} \}, \\
        & \CAL{A}_{\oie_B} = \{ \EStar_{Dr_B} \}, \\
        & \CAL{A}_{\oie_C} = \{ \EStar_{Dr_C} \}, \\
        & \CAL{A}_{\oie_{A_\theta}} = \{ \EStar_{Dr_A} \}.
    \end{aligned}
\end{align}

Finally, the forms of these 4 $\bOIE$ instances are as follows:
\begin{align}
    \begin{aligned}
        & \oie_A = \mathlarger{\mathlarger{(}} \\
        & \quad (), \\
        & \quad \{\ (\ (0,\ 1)\ ),\ (\ (21,\ 22)\ )\ \}, \\
        & \quad \{\ (0,\ 1),\ (21,\ 22)\ \}, \\
        & \quad \{\ \EStar_{Dr_A}\ \} \\
        & \mathlarger{\mathlarger{)}}, \\
        & \oie_B = \mathlarger{\mathlarger{(}} \\
        & \quad (), \\
        & \quad \{\ (\ (0,\ 1)\ ),\ (\ (13,\ 14)\ ),\ (\ (20,\ 22)\ )\ \}, \\
        & \quad \{\ (0,\ 1),\ (13,\ 14),\ (20,\ 22)\ \}, \\
        & \quad \{\ \EStar_{Dr_B}\ \} \\
        & \mathlarger{\mathlarger{)}}, \\
        & \oie_C = \mathlarger{\mathlarger{(}} \\
        & \quad (), \\
        & \quad \{\ (\ (0,\ 1)\ ),\ (\ (19,\ 22)\ )\ \}, \\
        & \quad \{\ (0,\ 1),\ (19,\ 22)\ \}, \\
        & \quad \{\ \EStar_{Dr_C}\ \} \\
        & \mathlarger{\mathlarger{)}}, \\
        & \oie_{A_\theta} = \mathlarger{\mathlarger{(}} \\
        & \quad (), \\
        & \quad \{\ (\ (21,\ 22)\ )\ \}, \\
        & \quad \{\ (21,\ 22)\ \}, \\
        & \quad \{\ \EStar_{Dr_A}\ \} \\
        & \mathlarger{\mathlarger{)}}. \\
    \end{aligned}
\end{align}
\fi
        \subsubsection{Properties}\label{subsubsec:oie.oie.prop}\EBL
        \ifincludeFile
In this subsubsection, we analyze some properties of $\bOIE$.
We begin with the deterministic mapping from the feasible schedule set $\F$ (2nd element) to the aggregated interval set $\I$ (3rd element).
First, we define an auxiliary function to describe the temporal bounds of interval tuples.
This function serves as the basis for the above mapping.

\begin{defi}[\BF{The bound 2-tuple of a finite non-empty $\bTwoTplT$ instance \& helper function \& permutation invariance}]\label{def:oie.BoundTwoTuple}\EBL

1) \TSF{Definition:}

Let $\TwoTplT_\theta$ be a finite non-empty $\bTwoTplT$ instance,
presented in the ordered enumeration form
\begin{align}
    \TwoTplT_\theta = (\ (x_1,\ y_1),\ (x_2,\ y_2),\ \dots,\ (x_n,\ y_n)\ ).
\end{align}

Define
\begin{align}
        & L = min\{ x_i | i \in \{1, 2, \cdots, n\} \}, \\
        & R = max\{ y_i | i \in \{1, 2, \cdots, n\} \}.
\end{align}
The 2-tuple $(L, R)$ is called ``The bound 2-tuple of a finite non-empty $\bTwoTplT$ instance'',
denoted as ``$\bBoundTwoT$''.

2) \TSF{Helper function:}

We define the helper function:
\begin{align}
    & \fBoundTwoT{\pPre\TwoTplT} \longmapsto \text{The }\ \bBoundTwoT \text{ instance of } \pPre\TwoTplT,
\end{align}
where
\begin{center}
    \begin{tabular}{|l|l|}
        \hline
        $\pPre\TwoTplT$ & A finite non-empty $\bTwoTplT$ instance \\ \hline
        Return & The bound 2-tuple of $\pPre\TwoTplT$ \\ \hline
    \end{tabular}
\end{center}

3) \TSF{permutation invariance:}

Suppose there are 2 finite non-empty $\bTwoTplT$ instances $\TwoTplT$ and $\TwoTplT^{'}$
with the same length $n \in \mathbb{N}^{+}$,
and they satisfy permutational equivalence via a permutation matrix $\CAL{M}$

\begin{align}
    \TwoTplT \stackrel{\CAL{M}}{\sim} \TwoTplT^{'}.
\end{align}
Then the following holds
\begin{align}
    \fBoundTwoT{\TwoTplT} = \fBoundTwoT{\TwoTplT^{'}}.
\end{align}
\end{defi}

\begin{proof}[\BF{Proof of the permutation invariance in Definition~\ref{def:oie.BoundTwoTuple}}]\EBL

    Let
    \begin{align}
        & \TwoTplT = (\ (x_1,\ y_1),\ (x_2,\ y_2),\ \cdots,\ (x_n,\ y_n)\ ), \\
        & \TwoTplT^{'} = (\ (x_{1}^{'},\ y_{1}^{'}),\ (x_{2}^{'},\ y_{2}^{'}), \cdots,\ (x_{n}^{'},\ y_{n}^{'})\ ))
    \end{align}
    be two finite non-empty $\bTwoTplT$ instances with finite length $n \in \mathbb{N}^{+}$,
    that satisfy permutational equivalence via a permutation matrix $\CAL{M}$
    \begin{align}
        \TwoTplT \stackrel{\CAL{M}}{\sim} \TwoTplT^{'}
    \end{align}

    This means $\TwoTplT$ and $\TwoTplT^{'}$ have same elements,
    i.e.:
    \begin{align}
        \{ (x_1, y_1),\ (x_2, y_2),\ \cdots,\ (x_n, y_n)\} = \{ (x_{1}^{'}, y_{1}^{'}), (x_{2}^{'}, y_{2}^{'}), \cdots, (x_{n}^{'}, y_{n}^{'})\}.
    \end{align}
    Consequently, the set of all first items is preserved:
    \begin{align}
        \{ x_1, x_2, \cdots, x_n \} = \{ x_{1}^{'}, x_{2}^{'}, \cdots, x_{n}^{'}\}.
    \end{align}
    And the set of all second items is preserved:
    \begin{align}
        \{ y_1, y_2, \cdots, y_n \} = \{ y_{1}^{'}, y_{2}^{'}, \cdots, y_{n}^{'}\}.
    \end{align}
    Since the bound 2-tuple $(L,\ R)$  is defined as:
    \begin{align}
        & L = \min\{ x_i | i \in \{1, 2, \cdots, n\} \}, \\
        & R = \max\{ y_i | i \in \{1, 2, \cdots, n\} \},
    \end{align}
    then
    \begin{align}
        & \min\{ x_i^{'} | i \in \{1, 2, \cdots, n\} \} = \min\{ x_i | i \in \{1, 2, \cdots, n\} \} = L, \\
        & \max\{ y_i^{'} | i \in \{1, 2, \cdots, n\} \} = \max\{ y_i | i \in \{1, 2, \cdots, n\} \} = R,
    \end{align}
    i.e.:
    \begin{align}
        & \fBoundTwoT{\TwoTplT} = \fBoundTwoT{\TwoTplT^{'}}.
    \end{align}
    this concludes the proof.
\end{proof}

With this helper function in Definition~\ref{def:oie.BoundTwoTuple},
we formalize the core relationship between the $\F$(2nd element) and $\I$(3rd element) of an $\bOIE$ instance.

\begin{pty}[\BF{The relationship between the $\F$ and $\I$ of $\bOIE$}]\label{prop:oie.2and3}\EBL

    For an $\bOIE$ instance $\oie$ with the form
    \begin{align}
        \oie = \mathlarger{\mathlarger{(}} \CAL{C},\ \CAL{F},\ \CAL{I},\ \CAL{A} \mathlarger{\mathlarger{)}}.
    \end{align}
    If $\CAL{F} \ne \emptyset$:
    \begin{align}
        \CAL{I} = \bigg\{\ \fBoundTwoT{\TwoTplT_i}\ \bigg|\ \TwoTplT_i \in \CAL{F}\ \bigg\},
    \end{align}
    else if $\CAL{F} = \emptyset$:
    \begin{align}
        \CAL{I} = \emptyset.
    \end{align}
\end{pty}

We use the scenario of ``Three doctors' submitting papers'' to give an example that illustrates Property~\ref{prop:oie.2and3}.
We construct the following $\bOIE$ instance.
\begin{align}
    \begin{aligned}
        & \oie_{{Dr_B}\_and\_{Dr_A}} = \mathlarger{\mathlarger{(}} \\
        & \quad (\ \oie_B,\ \oie_A\ ), \\
        & \quad \{ \\
        & \qquad (\ (0,\ 1),\ (0,\ 1)\ ), \\
        & \qquad (\ (0,\ 1),\ (21,\ 22)\ ), \\
        & \qquad (\ (13,\ 14),\ (0,\ 1)\ ), \\
        & \qquad (\ (13,\ 14),\ (21,\ 22)\ ), \\
        & \qquad (\ (20,\ 22),\ (0,\ 1)\ ), \\
        & \qquad (\ (20,\ 22),\ (21,\ 22)\ ), \\
        & \quad \}, \\
        & \quad \{\ (0,\ 1),\ (0,\ 22),\ (0,\ 14),\ (13,\ 22),\ (20,\ 22)\ \}, \\
        & \quad \{\ \EStar_{Dr_B},\ \EStar_{Dr_A}\ \} \\
        & \mathlarger{\mathlarger{)}}. \\
    \end{aligned}
\end{align}


The practical meaning of $\oie_{{Dr_B}\_and\_{Dr_A}}$ is ``$Dr_B$ and $Dr_A$ submit papers''.
This sentence is a textual description of an event, which involves two participants.
Each participant, as a subject, leads an event—specifically,
the respective paper submissions of the two doctors.

The $\F$ of $\oie_{{Dr_B}\_and\_{Dr_A}}$ is:
\begin{align}
    \begin{aligned}
        & \{ \\
        & \quad (\ (0, 1),\ (0, 1)\ ), \\
        & \quad (\ (0, 1),\ (21, 22)\ ), \\
        & \quad (\ (13, 14),\ (0, 1)\ ), \\
        & \quad (\ (13, 14),\ (21, 22)\ ), \\
        & \quad (\ (20, 22),\ (0, 1)\ ), \\
        & \quad (\ (20, 22),\ (21, 22)\ ) \\
        & \}, \\
    \end{aligned}
\end{align}
it indicates
\begin{align}
    \begin{aligned}
        & \E_{Dr_B}\ runs\ 00:00-01:00,\ while\ \E_{Dr_A}\ runs\ 00:00-01:00; \\
        & \E_{Dr_B}\ runs\ 00:00-01:00,\ while\ \E_{Dr_A}\ runs\ 21:00-22:00; \\
        & \E_{Dr_B}\ runs\ 13:00-14:00,\ while\ \E_{Dr_A}\ runs\ 00:00-01:00; \\
        & \E_{Dr_B}\ runs\ 13:00-14:00,\ while\ \E_{Dr_A}\ runs\ 21:00-22:00; \\
        & \E_{Dr_B}\ runs\ 20:00-22:00,\ while\ \E_{Dr_A}\ runs\ 00:00-01:00; \\
        & \E_{Dr_B}\ runs\ 20:00-22:00,\ while\ \E_{Dr_A}\ runs\ 21:00-22:00.
    \end{aligned}
\end{align}
This demonstrates how $F$ serves as the feasible details of plans.
In the real world, one of these plans will be selected to occur.

The $\I$ of $\oie_{{Dr_B}\_and\_{Dr_A}}$ is:
\begin{align}
    \{\ (0,\ 1),\ (0,\ 22),\ (0,\ 14),\ (13,\ 22),\ (20,\ 22)\ \},
\end{align}
it indicates
\begin{align}
    \begin{aligned}
        & The\ entire\ event,\ which\ consists\ of \ \E_{Dr_B}\ and\ \E_{Dr_A},\ runs\ 00:00-01:00; \\
        & The\ entire\ event,\ which\ consists\ of \ \E_{Dr_B}\ and\ \E_{Dr_A},\ runs\ 00:00-22:00; \\
        & The\ entire\ event,\ which\ consists\ of \ \E_{Dr_B}\ and\ \E_{Dr_A},\ runs\ 00:00-14:00; \\
        & The\ entire\ event,\ which\ consists\ of \ \E_{Dr_B}\ and\ \E_{Dr_A},\ runs\ 13:00-22:00; \\
        & The\ entire\ event,\ which\ consists\ of \ \E_{Dr_B}\ and\ \E_{Dr_A},\ runs\ 20:00-22:00.
    \end{aligned}
\end{align}

We further analyze this example from the perspectives of formalization and internal logic.
$\oie_{{Dr_B}\_and\_{Dr_A}}$ maps a real-world $\bE$ instance:
\begin{align}
    \begin{aligned}
        & \oie_{{Dr_B}\_and\_{Dr_A}} \rightarrow \\
        & \qquad ''The\ real-world\ event\ that\ consists\ of\ {Dr_B}'s\ and\ {Dr_A}'s\ submitting\ papers''.
    \end{aligned}
\end{align}

The $\C$(1st element) of this $\bOIE$ instance $(\oie_B, \oie_A)$ expresses
``$Dr_B$ and $Dr_A$'' in the declarative sentence.

By Definition~\ref{def:oie.BoundTwoTuple},
we process each item of the $\F$(2nd element) of $\oie_{{Dr_B}\_and\_{Dr_A}}$
to implement the logic process of Property~\ref{prop:oie.2and3}.
\begin{align}
    \begin{aligned}
        & \fBoundTwoT{\ ((0,\ 1),\ (0,\ 1))\ } \longmapsto (0, 1), \\
        & \fBoundTwoT{\ ((0,\ 1),\ (21,\ 22))\ } \longmapsto (0, 22), \\
        & \fBoundTwoT{\ ((13,\ 14),\ (0,\ 1))\ } \longmapsto (0, 14), \\
        & \fBoundTwoT{\ ((13,\ 14),\ (21,\ 22))\ } \longmapsto (13, 22), \\
        & \fBoundTwoT{\ ((20,\ 22),\ (0,\ 1))\ } \longmapsto (0, 22), \\
        & \fBoundTwoT{\ ((20,\ 22),\ (21,\ 22))\ } \longmapsto (20, 22).
    \end{aligned}
\end{align}
Put all the result 2-tuples on the right into a single set, the result is
\begin{align}
    \{\ (0,\ 1),\ (0,\ 22),\ (0,\ 14),\ (13,\ 22),\ (20,\ 22)\ \}.
\end{align}
This is precisely the $\I$(3rd element) of $\oie_{{Dr_B}\_and\_{Dr_A}}$.

In fact, a careful observation reveals that $\oie_{{Dr_B}\_and\_{Dr_A}}$ is constructed
from $\oie_{Dr_B}$ and $\oie_{Dr_A}$ through certain methods,
analogous to how the real-world event ``Dr. A and Dr. B submit papers'' is composed of
the two individual events ``Dr. A submits a paper'' and ``Dr. B submits a paper''.
Both our mathematical description and people's verbal expression elaborate on this from a formalization perspective.
How $\oie_{{Dr_B}\_and\_{Dr_A}}$ is constructed from $\oie_{Dr_B}$ and $\oie_{Dr_A}$ will be
elaborated in detail in the next chapter, which focuses on operations.
In this section, we concentrate on the $\bOIE$ itself.

Here we present other two $\bOIE$ instances as examples
\begin{align}
    \begin{aligned}
        & \oie_{{Dr_A}\_and\_{Dr_B}} = \mathlarger{\mathlarger{(}} \\
        & \quad (\ \oie_A,\ \oie_B\ ), \\
        & \quad \{ \\
        & \qquad (\ (0,\ 1),\ (0,\ 1)\ ), \\
        & \qquad (\ (0,\ 1),\ (13,\ 14)\ ), \\
        & \qquad (\ (0,\ 1),\ (20,\ 22)\ ), \\
        & \qquad (\ (21,\ 22),\ (0,\ 1)\ ), \\
        & \qquad (\ (21,\ 22),\ (13,\ 14)\ ), \\
        & \qquad (\ (21,\ 22),\ (20,\ 22)\ ), \\
        & \quad \}, \\
        & \quad \{\ (0,\ 1),\ (0,\ 22),\ (0,\ 14),\ (13,\ 22),\ (20,\ 22)\ \}, \\
        & \quad \{\ \EStar_{Dr_A},\ \EStar_{Dr_B}\ \} \\
        & \mathlarger{\mathlarger{)}},
    \end{aligned}
\end{align}
and
\begin{align}
    & \oie_{void} = \mathlarger{\mathlarger{(}}\ (),\ \emptyset,\ \emptyset,\ \emptyset\ \mathlarger{\mathlarger{)}}.
\end{align}
They both satisfy Property~\ref{prop:oie.2and3} when subjected to the same verification method.

Next, we specify the equality condition for two $\bOIE$ instances.
\begin{defi}[\BF{Equality of $\bOIE$ instances}]\label{defi:OPs.oie.Equal}\EBL

    Suppose $\oie_1$ and $\oie_2$ are $\bOIE$ instances, with the forms
    \begin{align}
            & \oie_1 = (\C_1,\ \F_1,\ \I_1,\ \A_1), \\
            & \oie_2 = (\C_2,\ \F_2,\ \I_2,\ \A_2).
    \end{align}
    If
    \begin{align}
        \C_1 = \C_2\ \land\ \F_1 = \F_2\ \land\ \I_1 = \I_2\ \land\ \A_1 = \A_2
    \end{align}
    holds, then $\oie_1$ and $\oie_2$ are equal, denoted as
    \begin{align}
        \oie_1 = \oie_2,
    \end{align}
    else if
    \begin{align}
        \C_1 \ne \C_2\ \lor\ \F_1 \ne \F_2\ \lor\ \I_1 \ne \I_2\ \lor\ \A_1 \ne \A_2
    \end{align}
    holds, they are not equal, denoted as
    \begin{align}
        \oie_1 \ne \oie_2.
    \end{align}
\end{defi}

\fi
        \subsubsection{$\bAtomOIE$ and $\bCombOIE$}\label{subsubsec:oie.oie.atomAndComb}\EBL
        \ifincludeFile
Mirroring the atomic/composite classification of $\bE$ instances in subsection~\ref{subsec:event.atomAndCombEvent},
we categorize $\bOIE$ instances into two disjoint classes based on the type of $\bE$ instance they map to.

\begin{defi}[\textbf{Atomic $\bOIE$}]\label{def:OIE.AtomOIE}\EBL

$\bOIE$ instances that map to $\bAtomE$ instances are called ``atomic $\bOIE$''
and are denoted as ``$\bAtomOIE$''.
An instance is written as ``$\atomOie$''.
\end{defi}

\begin{defi}[\textbf{Composite $\bOIE$}]\label{def:OIE.CombOIE}\EBL

$\bOIE$ instances that map to $\bCompE$ instances are called ``composite $\bOIE$''
and are denoted as ``$\bCombOIE$''.
An instance is written as ``$\combOie$''.
\end{defi}

$\bAtomOIE$ is also a type of $\bOIE$, and its instances cannot be further split due to its corresponding real-world event.
Similarly, $\bCombOIE$ is also a type of $\bOIE$, and their instances can be further split.

In the example ``Submitting papers of 3 doctors'' in subsubsection~\ref{subsubsec:oie.oie.def},
\begin{align}
    \begin{aligned}
        & \oie_A = \mathlarger{\mathlarger{(}} \\
        & \quad (), \\
        & \quad \{\ (\ (0,\ 1)\ ),\ (\ (21,\ 22)\ )\ \}, \\
        & \quad \{\ (0,\ 1),\ (21,\ 22)\ \}, \\
        & \quad \{\ \EStar_{Dr_A}\ \} \\
        & \mathlarger{\mathlarger{)}} \\
        & \oie_B = \mathlarger{\mathlarger{(}} \\
        & \quad (), \\
        & \quad \{\ (\ (0,\ 1)\ ),\ (\ (13,\ 14)\ ),\ (\ (20,\ 22)\ )\ \}, \\
        & \quad \{\ (0,\ 1),\ (13,\ 14),\ (20,\ 22)\ \}, \\
        & \quad \{\ \EStar_{Dr_B}\ \} \\
        & \mathlarger{\mathlarger{)}} \\
        & \oie_C = \mathlarger{\mathlarger{(}} \\
        & \quad (), \\
        & \quad \{\ (\ (0,\ 1)\ ),\ (\ (19,\ 22)\ )\ \}, \\
        & \quad \{\ (0,\ 1),\ (19,\ 22)\ \}, \\
        & \quad \{\ \EStar_{Dr_C}\ \} \\
        & \mathlarger{\mathlarger{)}},
    \end{aligned}
\end{align}
all are $\bAtomOIE$ instance.

The $\bOIE$ instance
\begin{align}
    \begin{aligned}
        & \oie_{{Dr_B}\_and\_{Dr_A}} = \mathlarger{\mathlarger{(}} \\
        & \quad (\ \oie_B,\ \oie_A\ ), \\
        & \quad \{ \\
        & \qquad (\ (0,\ 1),\ (0,\ 1)\ ), \\
        & \qquad (\ (0,\ 1),\ (21,\ 22)\ ), \\
        & \qquad (\ (13,\ 14),\ (0,\ 1)\ ), \\
        & \qquad (\ (13,\ 14),\ (21,\ 22)\ ), \\
        & \qquad (\ (20,\ 22),\ (0,\ 1)\ ), \\
        & \qquad (\ (20,\ 22),\ (21,\ 22)\ ), \\
        & \quad \}, \\
        & \quad \{\ (0,\ 1),\ (0,\ 22),\ (0,\ 14),\ (13,\ 22),\ (20,\ 22)\ \}, \\
        & \quad \{\ \EStar_{Dr_B},\ \EStar_{Dr_A}\ \} \\
        & \mathlarger{\mathlarger{)}}
    \end{aligned}
\end{align}
is a $\bCombOIE$ instance.

$\bAtomOIE$ and $\bCombOIE$ have the following properties,
where ``dim'' denotes the tuple length.

\begin{pty}[\BF{The property of $\bAtomOIE$}]\label{prop:oie.atomOie}\EBL

    Suppose there is an $\bAtomOIE$ instance $\atomOie$,
    it satisfies:
    \begin{align}
        \begin{aligned}
        & \Dim{\C_{\atomOie}} = 0 \\
        \land\ & \forall\ \TwoTplT_i \in \F_{\atomOie}: \Dim{\TwoTplT_i} = 1 \\
        \land\ & \CARDI{\A_{\atomOie}} = 1.
        \end{aligned}
    \end{align}
\end{pty}

\begin{pty}[\BF{The property of $\bCombOIE$}]\label{prop:oie.combOie}\EBL

    Suppose there is a $\bCombOIE$ instance $\combOie$,
    it satisfies:
    \begin{align}
        \begin{aligned}
            & \Dim{\CAL{C}_{\combOie}} > 0 \\
            \land\ & \forall\ \TwoTplT_i \in \CAL{F}_{\combOie}: \Dim{\TwoTplT_i} > 1\ \land\ \Dim{\TwoTplT_i} = \Dim{\CAL{C}_{\combOie}} \\
            \land\ & \CARDI{\A_{\atomOie}} \geq \Dim{\CAL{C}_{\combOie}}.
        \end{aligned}
    \end{align}
\end{pty}
\fi
    \subsection{Set of Optional Intervals Events}\label{subsec:oie.oieSet}\EBL
    \ifincludeFile
In the previous subsection,
we introduced the Optional Intervals Event ($\bOIE$) abstraction.
As analyzed earlier regarding the properties of events,
a single event can be composed of multiple sub-events.
Therefore, this subsection investigates the properties of sets consisting of multiple $\bOIE$ instances.

First, we supplement a fundamental type definition: a set composed of several $\bTwoTplS$ instances
(i.e., sets of 2-tuples) is denoted as $\bTwoTplSS$,
with a single instance denoted as $\TwoTplSS$.
Its ordered enumeration form is as follows:
\begin{align}
    \TwoTplSS = \{\ \TwoTplS_1,\ \TwoTplS_2,\ \cdots,\ \TwoTplS_n\ \}.
\end{align}
For example,
\begin{align}
    \begin{aligned}
        & \{ \\
        & \qquad \{\ (0,\ 1),\ (21,22)\ \}, \\
        & \qquad \{\ (0,\ 1),\ (13,14),\ (20,22)\ \}, \\
        & \qquad \{\ (0,\ 1),\ (19,22)\ \} \\
        & \}.
    \end{aligned}
\end{align}
This set exactly corresponds to the set formed
by the $\I$(third elements) of the three $\bOIE$ instances $\oie_A$, $\oie_B$ and $\oie_C$
in the earlier example of three doctors submitting manuscripts in

Returning to the core topic:
a set consisting of $n$ $\bOIE$ instances with ordered enumeration form
\begin{align}
    \{ \oie_1, \oie_2, \cdots, \oie_n \}
\end{align}
is called $\bOIES$,
which can be $\emptyset$.

Two key conventions need to be clarified here:
\begin{itemize}
    \item This paper only analyzes finite $\bOIES$ instances.
    When the cardinality of an $\bOIES$ instance is equal to $+\infty$,
    we would have to deal with an infinite number of $\bE$ instances.
    That is an extremely special case, and this paper will not discuss such a situation.
    \item Although an $\bOIES$ instance is denoted by an ordered enumeration for convenience,
    it represents an unordered set.
\end{itemize}

For $\bOIES$ instances,
this paper focuses on the set formed by the $\I$(3rd element) of all their elements.
This set will be frequently used in the two sequence operations defined in the next section,
so we first give its formal definition.

\begin{defi}[Set of $\I$ of elements of a finite $\bOIES$ instance]\label{def:oie.oieS.int2tupleSS}\EBL

Given a finite $\bOIES$ instance $\oieS$ with the ordered enumeration form
\begin{align}
    \oieS = \{ \oie_1, \oie_2, \cdots, \oie_n \},
\end{align}
then the set
\begin{align}
    & \{ \I_{\oie_1}, \I_{\oie_2}, \cdots, \I_{\oie_n} \}
\end{align}
is called ``The set of $\I$ of $\oieS$'', abbreviated as $\ISoOieS$ of $\oieS$.
$\ISoOieS$ is a $\bTwoTplSS$ instance.
If $\oieS$ is $\emptyset$, then its $\ISoOieS$ is $\emptyset$.
\end{defi}

Again taking the scenario of three doctors submitting manuscripts as an example.
We place $\oie_{Dr_A}$, $\oie_{Dr_B}$ and $\oie_{Dr_C}$ into a set to construct the $\bOIES$ instance $\oieS_{Drs}$
\begin{align}
    \oieS_{Drs} = \{ \oie_{Dr_A}, \oie_{Dr_B}, \oie_{Dr_C} \}.
\end{align}
Its corresponding set of $\I$ elements is:
\begin{align}
    \begin{aligned}
        & \ISoOieS_{\oieS_{Drs}} \\
        & \quad = \{\ \I_{\oie_{Dr_A}},\ \I_{\oie_{Dr_B}},\ \I_{\oie_{Dr_C}}\ \} \\
        & \quad = \\
        & \qquad \{ \\
        & \quad\qquad \{\ (0,\ 1),\ (21,\ 22)\ \}, \\
        & \quad\qquad \{\ (0,\ 1),\ (13,\ 14),\ (20,\ 22)\ \}, \\
        & \quad\qquad \{\ (0,\ 1),\ (19,\ 22)\ \} \\
        & \qquad \}.
    \end{aligned}
\end{align}
\fi
    \subsection{Void Optional Intervals Event}\label{subsec:oie.voidOie}\EBL
    \ifincludeFile
In this subsection,
we supplement a special class of $\bOIE$ instances that cannot be mapped to any real-world event.

\begin{defi}[Void $\bOIE$]\label{def:pdie.void}\EBL

An $\bOIE$ instance whose $\CAL{C}$(1st element) is $()$,
$\CAL{F}$(2nd element) is $\emptyset$, $\CAL{I}$(3rd element) is $\emptyset$,
and $\CAL{A}$(4th element) is $\emptyset$,
is a void $\bOIE$ instance.
We call its type ``$\bVoidError$''.
$\bVoidError$ is a subclass of $\bOIE$,
and there exists exactly one unique $\bVoidError$ instance globally,
with the canonical form:
\begin{align}
    \voidError = \left(\ (),\ \emptyset,\ \emptyset,\ \emptyset\ \right).
\end{align}
\end{defi}

A void $\bOIE$ instance contains no event subject,
no set of optional execution intervals,
and no other structural information,
hence it cannot establish a mapping to any real-world event.
In the subsequent sections,
we will define sequence operations on $\bOIE$ instances
to plan and guide the execution of real-world events.
If the result of a sequence operation is $\voidError$,
it indicates that the corresponding event planning scheme is infeasible in reality.
\fi

\section{Two Sequence Operations on the Elements of a Finite Set of $\bOIE$ Instances}\label{sec:ops}\EBL
\ifincludeFile
This section details the second core contribution of our work:
a formal system of sequence operations.
We define two $n$-ary finitary operations over the elements of a finite $\bOIES$ instance,
both of which satisfy closure: each operation outputs an $\bOIE$ instance,
and restricts optional interval information of all input $\bOIE$ instances to comply with the inherent algebraic constraints of the operation.
Notably, these operations offer powerful abstraction capabilities
with direct applications across physics, computer science, and probability,
as we contextualize in subsequent discussions.
The primary focus of this section is to formalize the operational rules of these two sequence operations.

\begin{itemize}
    \item \textbf{Subsection~\ref{subsec:ops.natIsoCP}}
    introduces index tuples to represent the operation order of set elements,
    and defines a natural isomorphism function for the Cartesian product of interval sets,
    which flattens the parenthesized structure of the standard Cartesian product
    and lays a formal foundation for subsequent sequence operations.
    \item \textbf{Subsection~\ref{subsec:ops.infea}}
    distinguishes between infeasible and feasible interval combinations,
    proposes the structure-preserving permutational isomorphism axiom
    for infeasible subsets.
    We prove that feasible subsets also satisfy the same isomorphism property.
    We define mutually independent $\bOIES$ as a constraint-free baseline model,
    and proposes a dynamic progressive regression analysis mechanism.
    \item \textbf{Subsection~\ref{subsec:ops.add}}
    defines ``Complete Sequence Addition''($\oplus|_{\alpha}^{\beta}$),
    which models concurrent events with a certain degree of equal opportunity
    within a shared time domain through a four-step operational process,
    and intuitively illustrates its ``fairness'' characteristic using the 100 metres race scenario.
    \item \textbf{Subsection~\ref{subsec:ops.multi}}
    defines ``Complete Sequence Multiplication''($\otimes$),
    which models strictly ordered sequential events by enforcing that
    the ending timestamp of the preceding event is no later than the starting timestamp of the subsequent one,
    and proves that it does not satisfy permutation isomorphism invariance,
    using downhill skiing as an intuitive example.
    \item \textbf{Subsection~\ref{subsec:ops.voidRes}}
    discusses typical cases where the result of sequence operations is the void $\bOIE$ instance ($\voidError$),
    including operations involving identical $\bOIE$ instances (leading to duplicate planning of real-world events)
    and operations involving $\voidError$ itself.
    \item \textbf{Subsection~\ref{subsec:ops.permu}}
    formalizes the permutational equivalence relation for $\bOIE$ instances,
    proves that $\oplus|_{\alpha}^{\beta}$ yields a single-orbit space
    (results of different operand orders are permutationally equivalent),
    while $\otimes$ may yield multiple distinct orbits.
    \item \textbf{Subsection~\ref{subsec:OPs.add.nat}}
    defines ``Natural Complete Sequence Addition'',
    a special case of $\oplus|_{\alpha}^{\beta}$
    that uses the largest common time domain of all participating $\bOIE$ instances as the filtering domain,
    eliminating the subjectivity of artificial domain specification
    and achieving notational unification with $\otimes$.
\end{itemize}

\fi
    \subsection{\textit{Flattened Cartesian Products of Interval Sets with Respect to Index Ordering}}\label{subsec:ops.natIsoCP}\EBL
    \ifincludeFile
We first introduce index tuples to represent the ordering of operands in a set.

\begin{defi}[\BF{An index tuple of a finite index set}]\label{def:ops.idxT}\EBL

Given a finite index set with cardinality $n$ ($n \in \mathbb{N}^+, n > 1$)
\begin{align}
    \{\ 1,\ 2,\ 3,\ \cdots,\ n\ \}.
\end{align}
A tuple $(i_1, i_2, i_3, \cdots, i_n)$ is called ``an index tuple for $n$''(denoted as $\bIdxT$) if and only if
\begin{align}
    \begin{aligned}
        & \forall\ k \in [1,\ n]:\ i_k \in [1,\ n] \\
        \land \ & \forall\ p,q \in [1,\ n] (p \neq q):\ i_p \neq i_q.
    \end{aligned}
\end{align}
\end{defi}

\begin{pty}[\BF{Permutation equivalence of two $\bIdxT$ instances}]\label{def:ops.idxT.Permu}\EBL

For any two distinct $\bIdxT$ instances of length $n(n \in \mathbb{N}^+)$,
\begin{align}
    & \idxT_1 = \{\ i_1,\ i_2,\ i_3,\ \cdots,\ i_n\ \}, \\
    & \idxT_2 = \{\ j_1,\ j_2,\ j_3,\ \cdots,\ j_n\ \}
\end{align}
there always exists a permutation matrix $\CAL{M}$ such that
\begin{align}
    \idxT_1 \cdot \CAL{M} = \idxT_2.
\end{align}
That is, any two $\bIdxT$ instances with the same length are permutationally equivalent, denoted as:
\begin{align}
    \idxT_1 \stackrel{\CAL{M}}{\sim} \idxT_2.
\end{align}
\end{pty}

We emphasize that Definition~\ref{def:ops.idxT} is uniformly applied
throughout this paper to map all sets to their ordered enumeration form.
For any $\bIdxT$ instance,
the order of its elements denotes the traversal order of elements in the set;
consequently, the length of this $\bIdxT$ instance equals the cardinality of the set.

To formalize and refine the sequence operation rules in subsequent sections,
we design an auxiliary function to compute the flattened Cartesian product
of all elements in a finite $\bTwoTplSS$
(set of sets of 2-tuples) instance.

\begin{defi}[\BF{
Natural isomorphism to the Cartesian product of all members of a finite non-empty $\bTwoTplSS$ instance with respect to a given index order
}]\label{def:op.fNItCPo2TplSS}\EBL

\begin{align}
    \begin{aligned}
        & \fNatIsoToCP{\pPre\TwoTplSS}{\pPre\IdxT} \longmapsto \\
        & \quad The\ set\ naturally\ isomorphic\ to\ the\ Cartesian\ product\ \prod_{i = 1}^{n} \pPre\TwoTplS_{\pPre{idx_i}}.
    \end{aligned}
\end{align}

\begin{tabular}{|l|l|}
    \hline
    $\pPre\TwoTplSS$ & \makecell[l]{\indent A finite set of sets of 2-tuples excluding $\emptyset$ with size $n > 0$, in the form of\\ $\pPre\TwoTplSS = \{\ \pPre\TwoTplS_1,\ \pPre\TwoTplS_2,\ \cdots,\ \pPre\TwoTplS_n\ \}$ } \\ \hline
    $\pPre\IdxT$ & \makecell[l]{\indent An index tuple, in the form of $(\ \pPre{idx_1},\ \pPre{idx_2},\ \cdots,\ \pPre{idx_n}\ )$, indicating \\
    which element of $\pPre\TwoTplSS$ is the $i$-th operand in the expression when\\ performing the Cartesian product } \\ \hline
    Return & \makecell[l]{\indent The set naturally isomorphic to the Cartesian product of all elements in\\ $\pPre\TwoTplSS$, with the order of operands constructed using $\pPre\IdxT$ as the index,\\
    and the result is a $\bTwoTplTS$ instance} \\
    \hline
\end{tabular}

We stipulate that, in the special case where $\pPre\TwoTplSS = \emptyset$, the result is $\emptyset$.
\end{defi}

We now explain the rationale for using ``natural isomorphism''.
The standard Cartesian product is left-associative,
which imposes artificial parenthesization on the result.
What we require is not a result constrained by operation-order parentheses,
but rather the flattened tuple form of each element in the resulting set,
because our purpose is to reveal the intrinsic equivalence among different constructions.
This is the essence of natural isomorphism in the present context.

For the Cartesian product operation
\begin{align}
    \pPre\TwoTplS_{\pPre{idx_1}} \cdot \pPre\TwoTplS_{\pPre{idx_2}} \cdot \cdots \cdot \pPre\TwoTplS_{\pPre{idx_n}}.
\end{align}
Indeed, if we directly calculate the Cartesian product,
its rigorous formulation is
\begin{align}
    ( \cdots ((\pPre\TwoTplS_{\pPre{idx_1}} \cdot \pPre\TwoTplS_{\pPre{idx_2}}) \cdot \pPre\TwoTplS_{\pPre{idx_3}}) \cdots ) \cdot \pPre\TwoTplS_{\pPre{idx_n}}.
\end{align}
What we require is not a result constrained by operation-order parentheses,
but rather the flattened form of each element in the resulting set.
This is why we use natural isomorphism.
Thus, it is not a standard Cartesian product but a flattened Cartesian product.
As a $\bTwoTplTS$ instance, it can be expressed as follows:
\begin{align}
    \TwoTplTS^{NatIso2CP} = \{\ \TwoTplT_1,\ \TwoTplT_2,\ \cdots,\ \TwoTplT_k\ \}.
\end{align}

Suppose there is a finite non-empty $\bTwoTplSS$ instance
\begin{align}
    \TwoTplSS = \{\ \TwoTplS_1,\ \TwoTplS_2,\ \cdots,\ \TwoTplS_n\ \},
\end{align}
then according to Definition~\ref{def:op.fNItCPo2TplSS},
we can get the flattened Cartesian product as follows:
\begin{align}
    \TwoTplTS^{NatIso2CP} = \fNatIsoToCP{\TwoTplSS}{\IdxT}.
\end{align}

Its cardinality is
\begin{align}
    \CARDI{\fNatIsoToCP{\TwoTplSS}{\IdxT}} = \prod_{i = 1}^{n} \CARDI{\TwoTplS_{\IdxT_i}}.
\end{align}

We now apply this helper function in Definition~\ref{def:op.fNItCPo2TplSS} to $\bOIES$ instances. \\
For a finite non-empty $\bOIES$ instance $\oieS$ containing no $\voidError$,  with ordered enumeration form
\begin{align}
    \oieS = \{\ \oie_1,\ \oie_2,\ \cdots,\ \oie_n\ \}.
\end{align}
Indexing the flattened Cartesian product of all 3rd-element $\I$ values
for $\oieS$ members under a given $\bIdxT$ instance
yields a $\bTwoTplTS$ instance naturally isomorphic to that product.
By Definition~\ref{def:oie.oieS.int2tupleSS} and Definition~\ref{def:op.fNItCPo2TplSS},
the result can be expressed as
\begin{align}
    \TwoTplTS^{NatIso2CP} = \fNatIsoToCP{\ISoOieS_{\oieS}}{\idxT}.
\end{align}
\fi
    \subsection{Feasible and Infeasible Subsets of Flattened Cartesian Products of $\bOIES$ Interval Sets}\label{subsec:ops.infea}\EBL
    \ifincludeFile
The core purpose of constructing the flattened Cartesian product of all 3rd-element $\I$ values
is to obtain the complete combinatorial space of interval combinations for subsequent sequence operation modeling.
While the flattened Cartesian product enumerates all mathematically possible combinations of execution intervals across constituent $\bOIE$ instances,
$\bOIE$ instances are by design mapped to real-world Event instances,
which are inherently bound by physical laws, logical causality, and domain-specific constraints.
These constraints render a subset of the mathematically valid interval combinations practically infeasible in real-world execution,
necessitating a systematic framework to distinguish feasible and infeasible combinations.

The overarching goal of this subsection is to formalize this feasibility screening process:
we first examine the full combinatorial set generated by the natural isomorphism to the Cartesian product,
then formalize the criteria for identifying infeasible interval combinations,
exclude these invalid elements from the full set,
and ultimately derive the maximum feasible subset that adheres to all real-world constraints.
This subsection is divided into four subsubsections,
whose core research content and contributions are outlined as follows:

\begin{itemize}
\item Subsubsection~\ref{subsubsec:OPs.unfsbleDI} formalizes the definition of infeasible/feasible interval combinations
and their supporting helper functions,
and illustrates the feasibility screening logic with a practical case,
establishing the executable judgment rules for interval combination validity.
\item Subsubsection~\ref{subsubsec:OPs.unfsbleDI.equival} proposes the structure-preserving permutational isomorphism axiom for infeasible subsets,
and proves that the feasible subset also satisfies the same isomorphism property,
providing the core axiomatic support for the algebraic property analysis of subsequent sequence operations.
\item Subsubsection~\ref{subsubsec:OPs.unfsbleDI.independentOIES} defines the mutually independent $\bOIES$ with no cross-event constraints,
gives its formal judgment conditions and applicable scenarios,
and builds a constraint-free baseline model for the sequence operation modeling in subsequent sections.
\item Subsubsection~\ref{subsubsec:OPs.infsbleDI.relativity} develops the epistemological perspective of planning granularity relativity and progressive constraint discovery.
It formalizes the subject's right to switch between isolated and joint planning granularities,
analyzes the hard feedback produced by incomplete models when planners fail to incorporate cross-event constraints,
and proposes a dynamic progressive regression analysis mechanism that iteratively expands $\bOIES$ instances
from local feasibility to global feasibility through structural correction and convergence.
\end{itemize}\fi
        \subsubsection{Infeasible and Feasible Interval Combinations}\label{subsubsec:OPs.unfsbleDI}\EBL
        \ifincludeFile
When an $\bOIE$ instance $\oie$ is used in isolation to plan the execution of its mapped $\bE$ instance,
the feasibility of each candidate interval in its $\I$(3rd element)
is evaluated solely with respect to the internal constraints of that single event.
However, when $\oie$ is brought into a joint planning context with other $\bOIE$ instances,
additional cross-event constraints—stemming from physical laws, logical causality or resource contention
among the underlying $\E$ instances—come into play.
These external constraints are inherited by the $\bOIE$ abstractions,
imposing additional restrictions that are absent when the $\oie$ is planned in isolation.
Consequently, the same interval that is feasible under $\oie$'s individual planning may become infeasible in the joint planning scenario.
In other words, the set $\I_{\oie}$ still records all candidate intervals originally associated with the mapped event,
but the operational feasibility of selecting any particular interval from $\I_{\oie}$ now depends on
whether the chosen interval can coexist with the intervals selected for the other participating $\bOIE$ instances.
Crucially, the set $\I_{\oie}$ itself remains unchanged; rather,
the planning context has shifted from a single-event constraint system to a multi-event constraint system,
rendering some previously viable intervals unrealizable.

To illustrate, consider a screenwriter designing a script about two men.

We abstract their entire lives into two $\bOIE$ instances:
$\oie_{old}$ and $\oie_{young}$, which form an $\bOIES$ instance named $\oieS_{old-young}$.
When only considering their own individual situations,
the screenwriter first makes the following design for their lifespans using $\CAL{I}$:
\begin{align}
    \begin{aligned}
        & \CAL{I}_{\oie_{old}} = \{\ (1830,\ 1900),\ (1910,\ 1990),\ (2050,\ 2140)\ \}, \\
        & \CAL{I}_{\oie_{young}} = \{\ (1860,\ 1930),\ (1930,\ 2010),\ (2077,\ 2277)\ \}.
    \end{aligned}
\end{align}

\indent
Here, for ease of notation, we use years rather than precise timestamps,
with each interval representing the period starting at 00:00 on January 1 of the start year.
If the script only features one member of the old-young pair
and the other is not mentioned at all,
the screenwriter can directly select an interval as lifespan from $\CAL{I}_{\oie_{old}}$ or from $\CAL{I}_{\oie_{young}}$
to arrange the birth/death moments of the protagonist.

The complete sets of life intervals for the two individuals are determined,
and the screenwriter needs to assign their relationship.
These two characters can appear in multiple different scripts,
and any relationship can be assigned to them across these scripts;
they can be strangers, colleagues, master and apprentice, or any combination thereof.
However,
as the screenwriter is currently working on an important script,
he decides to set their relationship in that script as ``\TT{father and son}''.

Then, physical and biological constraints render certain combinations impossible:
\begin{itemize}
    \item \BF{Causal precedence:} The father must be born before the son (with a reasonable biological minimum age difference).
    \item \BF{Biological continuity:} The father's death must occur after conception unless assisted reproduction technology exists in the historical setting.
    \item \BF{Lifespan bounds:} Each lifespan must remain within human biological limits (unless specified as science fiction).
\end{itemize}

According to these 3 restrictions, it is obvious that if the father's lifespan is from 1830 to 1900,
then it is unreasonable for the son's lifespan to be $[1930, 2010)$ or $[2077, 2277)$.

By Definition~\ref{def:oie.oieS.int2tupleSS} and Definition~\ref{def:op.fNItCPo2TplSS},
with the father as the first operand and the son as the second,
we take the flattened Cartesian product of
the $\I$ sets of both $\oie_{father}$ and $\oie_{son}$
\begin{align}
    & \TwoTplTS^{NatIso2CP} = \fNatIsoToCP{\ISoOieS_{\oieS_{father-son}}}{(\idx_{father}, \idx_{son})}.
\end{align}
Then we get the result:
\begin{align}
    \begin{aligned}
        & \TwoTplTS^{NatIso2CP} = \{ \\
        & \qquad (\ (1830,\ 1900),\ (1860,\ 1930)\ ),\ (\ (1830,\ 1900),\ (1930,\ 2010)\ ), \\
        & \qquad (\ (1830,\ 1900),\ (2077,\ 2277)\ ),\ (\ (1910,\ 1990),\ (1860,\ 1930)\ ), \\
        & \qquad (\ (1910,\ 1990),\ (1930,\ 2010)\ ),\ (\ (1910,\ 1990),\ (2077,\ 2277)\ ), \\
        & \qquad (\ (2050,\ 2140),\ (1860,\ 1930)\ ),\ (\ (2050,\ 2140),\ (1930,\ 2010)\ ), \\
        & \qquad (\ (2050,\ 2140),\ (2077,\ 2277)\ ) \\
        & \}.
    \end{aligned}
\end{align}

%

Of the nine mathematically possible combinations,
only three satisfy all real-world constraints:
\begin{itemize}
    \item Modern-era film: Father $[1830,1900)$, Son $[1860,1930)$
    \item Contemporary film: Father $[1910,1990)$, Son $[1930,2010)$
    \item Science-fiction film: Father $[2050,2140)$, Son $[2077,2277)$
\end{itemize}
The remaining six combinations form the infeasible set
\begin{align}
    \begin{aligned}
    & \CAL{S}_{infeasible} = \{ \\
    & \qquad (\ (1830,\ 1900),\ (1930,\ 2010)\ ),\ (\ (1830,\ 1900),\ (2077,\ 2277)\ ), \\
    & \qquad (\ (1910,\ 1990),\ (1860,\ 1930)\ ),\ (\ (1910,\ 1990),\ (2077,\ 2277)\ ), \\
    & \qquad (\ (2050,\ 2140),\ (1860,\ 1930)\ ),\ (\ (2050,\ 2140),\ (1930,\ 2010)\ ) \\
    & \},
    \end{aligned}
\end{align}
which contains biologically impossible scenarios such as $((1830,1900),(1930,2010))$
(a son born in the 1930s, decades after his father's death).

These impossibilities are not introduced by the Cartesian product operation itself,
but by the biological and causal relationship between the father and son entities.
They are determined by $\oie_{father}$ and $\oie_{son}$,
and have nothing to do with the associated Cartesian product operation.
It should be clarified that the computation of the flattened Cartesian product is handled natively by the framework itself.
While the framework defines the combinatorial space,
whether the operational results align with the expected constraints
is determined by the intrinsic properties of the operands themselves.

We now formalize this intuition.

\begin{defi}[\textbf{Infeasible interval $\bTwoTplTS$ instance of the $\I$ of all members in a finite non-empty $\bOIES$ instance excluding $\emptyset$ under an index order \& helper function}]\label{def:OPs.infsbleDI}\EBL

1) \TSF{Definition:}

\indent For a finite non-empty $\bOIES$ instance excluding $\emptyset$, $\oieS$,
with the cardinality of $n \in \mathbb{N}^+$, with the ordered enumeration form
\begin{align}
    \oieS = \{ \oie_1, \oie_2, \oie_3, \cdots, \oie_n \},
\end{align}
and an $\bIdxT$ instance $\idxT$ with the form
\begin{align}
    \idxT = ( \idx_1, \idx_2, \cdots, \idx_n ).
\end{align}
By Definition~\ref{def:oie.oieS.int2tupleSS},
take the flattened Cartesian product of
all the $\I$ of $\oieS$ members under the index order $\idxT$:
\begin{align}
    & \TwoTplTS^{NatIso2CP} = \fNatIsoToCP{\ISoOieS_{\oieS}}{\idxT}.
\end{align}
Define a predicate for any $\bTwoTplT$ instance $\TwoTplT_{\lambda}$
\begin{align}
    \begin{aligned}
        & \PIsInfeasIntTwoTplT(\TwoTplT_{\lambda}): \\
        & \qquad \TwoTplT_{\lambda} \in \TwoTplTS^{NatIso2CP} \\
        & \quad \land The\ allocation\ of\ intervals\ where \\
        & \qquad\qquad (\forall\ i \in \idxT:\ the\ ith\ 2tuple\ in\ \TwoTplT_{\lambda}\ is\ the \\
        & \qquad\qquad \quad interval\ of\ the\ corresponding\ \bE\ instance\ of\ \oie_i\ ) \\
        & \qquad is\ infeasible.
    \end{aligned}
\end{align}
(
\BF{Remark}:
The value of this predicate depends on physical and logical constraints in practical scenarios.
If such constraints keep permutation isomorphism unchanged as stated in Axiom~\ref{axm:OPs.infsble.isom},
relevant algebraic structures can be derived accordingly.
)

We take the maximal set consisting of all $\bTwoTplT$ instances that satisfy the predicate $\PIsInfeasIntTwoTplT$.
The type of this set is called ``Infeasible interval $\bTwoTplTS$ of the $\I$ of all members in $\oieS$ under the index order $\idxT$'',
abbreviated as ``$\bInfeasIntTwoTplTS$''.

2) \TSF{Helper function:}
\begin{align}
    & \fInfeasIntTwoTplTS{\pPre\oieS}{\pPre\idxT} \longmapsto \\
    & \quad The\ \bInfeasIntTwoTplTS\ instance\ of\ \pPre\oieS\ under\ index\ order\ \idxT. \notag
\end{align}

\begin{tabular}{|l|l|}
    \hline
    $\pPre\oieS$ & A finite non-empty $\bOIES$ instance excluding $\emptyset$ \\ \hline
    $\pPre\IdxT$ & An index tuple \\ \hline
    Return & \makecell[l]{The $\bInfeasIntTwoTplTS$ instance of $\pPre\oieS$ under the index order $\pPre\IdxT$} \\ \hline
\end{tabular}
\end{defi}

By Definition~\ref{def:op.fNItCPo2TplSS}, we have already obtained
the flattened Cartesian product for all the $\I$ of an $\bOIES$ instance members,
while Definition~\ref{def:OPs.infsbleDI} gives the set to which the infeasible elements belong.
Accordingly, we can reduce the total set by removing the infeasible combinations.

\begin{defi}[\textbf{Maximum feasible subset of the set naturally isomorphic to the Cartesian product of the $\I$ of all members in a finite non-empty $\bOIES$ instance excluding $\emptyset$ under an index order \& helper function}]\label{def:OPs.feasibleCP}\EBL

1) \TSF{Definition:}

\indent For a finite non-empty $\bOIES$ instance excluding $\emptyset$ with ordered enumeration form
\begin{align}
    \oieS = \{\ \oie_1,\ \oie_2,\ \cdots,\ \oie_n\ \},
\end{align}
and an $\bIdxT$ instance with length $n$
\begin{align}
    \idxT = ( idx_1, idx_2, \cdots, idx_n ).
\end{align}
We take the flattened Cartesian product
of the $\I$ of all members in $\oieS$ under the index order specified by $\idxT$
\begin{align}
    & \TwoTplTS^{NatIso2CP} = \fNatIsoToCP{\ISoOieS_{\oieS}}{\idxT}.
\end{align}
For any $\bTwoTplT$ instance with length $n$ and under an index order $\idxT$
\begin{align}
    \TwoTplT_{\lambda} = (\ \TwoTpl_{idx_1},\ \TwoTpl_{idx_2},\ \cdots,\ \TwoTpl_{idx_n}\ ),
\end{align}
define the predicate:
\begin{align}
    \begin{aligned}
        & \PIsFeasibleIntvlTwoTplT(\TwoTplT_{\lambda}): \\
        & \qquad \ \TwoTplT_{\lambda} \in \TwoTplTS^{NatIso2CP} \\
        & \quad \land \ \TwoTplT_{\lambda} \notin \fInfeasIntTwoTplTS{\oieS}{\idxT}.
    \end{aligned}
\end{align}
\indent Take the maximal set composed of all $\bTwoTplT$ instances that satisfy
the predicate $\PIsFeasibleIntvlTwoTplT$
\begin{align}
    \mathlarger{\{}\ \TwoTplT \mid \PIsFeasibleIntvlTwoTplT(\TwoTplT)\ \mathlarger{\}}.
\end{align}
The type of this set is called
``The maximum feasible subset of the set naturally isomorphic to the Cartesian product of the $\I$ of all members
in $\oieS$ under the index order specified by $\idxT$'',
denoted as ``$\bFeasibleIntvlTwoTplTS$''.

2) \TSF{Helper function:}
\begin{align}
    & \fFeasibleIntvlTwoTplTS{\pPre\oieS}{\pPre\IdxT} \longmapsto \\
    & \quad the\ \bFeasibleIntvlTwoTplTS\ instance\ of\ \pPre\oieS\ under\ the\ index\ order\ \pPre\IdxT. \notag
\end{align}
\begin{center}
    \begin{tabular}{|l|l|}
        \hline
        $\pPre\oieS$ & a finite non-empty $\bOIES$ instance excluding $\emptyset$ \\ \hline
        $\pPre\IdxT$ & An index tuple \\ \hline
        Return & \makecell[l]{The maximum feasible subset of the flattened Cartesian product of the $\I$ of $\pPre\oieS$\\ members under the index order specified by $\pPre\IdxT$} \\ \hline
    \end{tabular}
\end{center}
\end{defi}

Through Definition~\ref{def:OPs.infsbleDI} and Definition~\ref{def:OPs.feasibleCP},
we establish a formal filtering mechanism
that bridges the gap between the pure mathematical flattened Cartesian product and physically feasible execution schemes:
the helper function $\fPre\RM{Inf2tupleTS}$ precisely captures interval combinations that violate physical, logical, or causal constraints.
The helper function $\fPre\RM{Feas2tuple}\CAL{TS}$ extracts the maximum feasible subset satisfying all real-world restrictions.
This screening process bridges abstract algebra and engineering practice within the $\bOIE$ framework.
It also establishes a sound combinatorial foundation for the subsequent definition of sequence operations $\oplus|_{\alpha}^{\beta}$ and $\otimes$.
Invalid temporal configurations are ruled out in advance,
enabling the transition from passive observation to constraint-oriented proactive planning.
\fi
        \subsubsection{The Permutational Isomorphism Relationship}\label{subsubsec:OPs.unfsbleDI.equival}\EBL
        \ifincludeFile
By Definition~\ref{def:OPs.infsbleDI},
in our father-son movie script,
we can obtain formalized infeasible data under index tuple $(1, 2)$ in the following form:
\begin{align}
    \begin{aligned}
    & \fInfeasIntTwoTplTS{\oieS_{father\_son}}{(1, 2)} = \\
    & \quad \{ \\
    & \qquad\quad (\ (1830,\ 1900),\ (1930,\ 2010)\ ),\ (\ (1830,\ 1900),\ (2077,\ 2277)\ ), \\
    & \qquad\quad (\ (1910,\ 1990),\ (1860,\ 1930)\ ),\ (\ (1910,\ 1990),\ (2077,\ 2277)\ ), \\
    & \qquad\quad (\ (2050,\ 2140),\ (1860,\ 1930)\ ),\ (\ (2050,\ 2140),\ (1930,\ 2010)\ ) \\
    & \quad \}.
    \end{aligned}
\end{align}
For the reversed index tuple $(2, 1)$ (son first, father second),
the corresponding infeasible set is:
\begin{align}
    \begin{aligned}
    & \fInfeasIntTwoTplTS{\oieS_{father\_son}}{(2, 1)} = \\
    & \quad \{ \\
    & \qquad\quad ((1930, 2010), (1830, 1900)), ((2077, 2277), (1830, 1900)), \\
    & \qquad\quad ((1860, 1930), (1910, 1990)), ((2077, 2277), (1910, 1990)), \\
    & \qquad\quad ((1860, 1930), (2050, 2140)), ((1930, 2010), (2050, 2140)) \\
    & \quad \}.
    \end{aligned}
\end{align}
For a set of cardinality 2,
the index tuples $(1, 2)$ and $(2, 1)$ are permutationally equivalent via a permutation matrix.
This permutational equivalence extends naturally to the $n$-ary case.
By Definition~\ref{def:ops.idxT}, for any finite index set of size $n$,
all $\bIdxT$ instances are permutationally equivalent via some $n \times n$ permutation matrix.
Formalized examples are as follows:
\begin{align}
    \begin{aligned}
        & \idxT_1 = (1,\ 2),\ \idxT_2 = (2,\ 1), \\
        & \idxT_3 = (1,\ 2,\ 3),\ \idxT_4 = (2,\ 3,\ 1), \\
        & \CAL{M}_1 = \begin{pmatrix} 0 & 1 \\ 1 & 0 \end{pmatrix},\ \CAL{M}_2 = \begin{pmatrix} 0 & 0 & 1 \\ 1 & 0 & 0 \\ 0 & 1 & 0 \end{pmatrix}.
    \end{aligned}
\end{align}
They satisfy the following equations
\begin{align}
    (1,\ 2) \cdot \begin{pmatrix} 0 & 1 \\ 1 & 0 \end{pmatrix} = (2,\ 1),\ \ \ (1,\ 2,\ 3) \cdot \begin{pmatrix} 0 & 0 & 1 \\ 1 & 0 & 0 \\ 0 & 1 & 0 \end{pmatrix} = (2,\ 3,\ 1),
\end{align}
denoted as
\begin{align}
    & \idxT_1 \stackrel{\CAL{M}_1}{\sim} \idxT_2,\ \idxT_3 \stackrel{\CAL{M}_2}{\sim} \idxT_4.
\end{align}

The elements of the two sets
\begin{align}
    & \fInfeasIntTwoTplTS{\oieS_{father\_son}}{(1, 2)}, \\
    & \fInfeasIntTwoTplTS{\oieS_{father\_son}}{(2, 1)}
\end{align}
also exhibit a permutationally equivalent relationship:
\begin{align}
    \begin{aligned}
        & \exists \text{ a permutation matrix }\CAL{M}, \\
        & \quad \forall tuple_{i} \in \fInfeasIntTwoTplTS{\oieS_{father\_son}}{(1, 2)}, \\
        & \qquad \exists tuple_{j} \in \fInfeasIntTwoTplTS{\oieS_{father\_son}}{(2, 1)}: \\
        & \quad\qquad tuple_{i} \stackrel{\CAL{M}}{\sim} tuple_{j}\ (\text{i.e. } tuple_{i} \cdot \CAL{M} = tuple_{j}).
    \end{aligned}
\end{align}
For example,
\begin{align}
    \begin{aligned}
    & ((1830, 1900), (1930, 2010)) \stackrel{\CAL{M}_1}{\sim} ((1930, 2010), (1830, 1900)), \\
    & ((1910, 1990), (1860, 1930)) \stackrel{\CAL{M}_1}{\sim} ((1860, 1930), (1910, 1990)), \\
    & ((2050, 2140), (1860, 1930)) \stackrel{\CAL{M}_1}{\sim} ((1860, 1930), (2050, 2140)). \\
    & \cdots
    \end{aligned}
\end{align}
Intuitively, this phenomenon can be summarized as follows:
\TT{the inherent connections among multiple events will not undergo essential changes due to the order in which these events are stated}.
Based on the above analysis, we can state the following axiom.

\begin{axm}[\textbf{Structure-preserving permutational isomorphism of $\bInfeasIntTwoTplTS$ instances}]\label{axm:OPs.infsble.isom}\EBL

Suppose there is a finite $\bOIES$ instance $\oieS$ with cardinality $n$($n \in \mathbb{N}^{+}, n > 1$),
two $\bIdxT$ instances $\idxT_1$ and $\idxT_2$ for $\oieS$.
Then the following relation must hold
\begin{align}
    \begin{aligned}
    & \exists\ a\ permutation\ matrix\ \CAL{M}: \\
    & \qquad \idxT_1 \cdot \CAL{M} = \idxT_2 \\
    & \quad \land \fInfeasIntTwoTplTS{\oieS}{\idxT_1} \text{ and } \fInfeasIntTwoTplTS{\oieS}{\idxT_2} \\
    & \qquad\quad \text{are structure-preserving permutationally equivalent} \\
    & \qquad\quad \text{with respect to the matrix } \CAL{M}.
    \end{aligned}
\end{align}
which is denoted as
\begin{align}
    \idxT_1 \stackrel{\CAL{M}}{\sim} \idxT_2 \Rightarrow \fInfeasIntTwoTplTS{\oieS}{\idxT_1} \stackrel{\CAL{M}}{\cong} \fInfeasIntTwoTplTS{\oieS}{\idxT_2}.
\end{align}
\end{axm}

The essence of Axiom~\ref{axm:OPs.infsble.isom} is that
the infeasibility constraints among the events mapped by an $\bOIES$ instance constitute intrinsic physical and logical properties
that are independent of the descriptive order used to enumerate the constituent $\bOIE$ instances.
In Subsection~\ref{subsec:ops.natIsoCP},
we have discussed the flattened Cartesian product of
the $\I$ of a finite $\bOIES$ instance.
Indeed, two results with the same $\bTwoTplS$ instances as operands but different operand orders
are structure-preserving permutationally equivalent.
Combining this axiom with the well-established permutational invariance of the Cartesian product,
we derive the corresponding property for feasible interval subsets.

\begin{pty}[\textbf{Structure-preserving permutational isomorphism of $\bFeasibleIntvlTwoTplTS$ instances}] \label{prop:oieOP.fsble.isom}\EBL

Suppose there is an $\bOIES$ instance $\oieS$ with cardinality $n > 1$
\begin{align}
    \oieS = \{ \oie_1, \oie_2, \cdots, \oie_n\},
\end{align}
which satisfies $\forall i \in [1, n], \oie_i \ne \voidError$,
and two $\bIdxT$ instances $\idxT_1$ and $\idxT_2$ for $\oieS$.
Then
\begin{align}
    \begin{aligned}
    & \exists\ a\ permutation\ matrix\ \CAL{M}: \\
    & \qquad \idxT_1 \cdot \CAL{M} = \idxT_2 \\
    & \quad \land \fFeasibleIntvlTwoTplTS{\oieS}{\idxT_1} \ and\ \fFeasibleIntvlTwoTplTS{\oieS}{\idxT_2} \\
    & \qquad\quad are\ structure-preserving\ permutationally\ equivalent\ \\
    & \qquad\quad with\ respect\ to\ matrix\ \CAL{M},
    \end{aligned}
\end{align}
denoted as
\begin{align}
    \idxT_1 \stackrel{\CAL{M}}{\sim} \idxT_2 \Rightarrow \fFeasibleIntvlTwoTplTS{\oieS}{\idxT_1} \stackrel{\CAL{M}}{\cong} \fFeasibleIntvlTwoTplTS{\oieS}{\idxT_2}.
\end{align}
\end{pty}

\begin{proof}[\BF{Proof of Property~\ref{prop:oieOP.fsble.isom}}]\EBL

Let $\oieS$ be a finite $\bOIES$ instance with cardinality $n > 1$,
let $\idxT_1$ and $\idxT_2$ be two index tuples such that $\idxT_1 \stackrel{\CAL{M}}{\sim} \idxT_2$ via permutation matrix $M$.
According to Definition~\ref{def:OPs.feasibleCP},
the maximum feasible subset is constructed by removing the infeasible subset from the flattened Cartesian product:
\begin{align}
    \fFeasibleIntvlTwoTplTS{\oieS}{\idxT} = \fNatIsoToCP{\ISoOieS_{\oieS}}{\idxT} \setminus \fInfeasIntTwoTplTS{\oieS}{\idxT}.
\end{align}

We establish the permutational isomorphism by examining both components:

\BF{Step I}: Permutational isomorphism of the flattened Cartesian product.

By Definition~\ref{def:op.fNItCPo2TplSS},
the product $\fNatIsoToCP{\ISoOieS_{\oieS}}{\idxT}$ flattens the Cartesian product of interval sets from $\oieS$ ordered by $\idxT$.
When the index order changes from $\idxT_1$ to $\idxT_2$ via permutation $M$,
the resulting sets are structure-preserving permutationally equivalent via $M$:
\begin{align}
    \fNatIsoToCP{\ISoOieS_{\oieS}}{\idxT_1} \overset{M}{\cong} \fNatIsoToCP{\ISoOieS_{\oieS}}{\idxT_2}.
\end{align}
This follows because the flattened Cartesian product operation is invariant under operand reordering up to permutation isomorphism,
each tuple in the resulting set merely has its internal elements permuted according to $M$.
This is a trivial conclusion in mathematics.

\BF{Step II}: Permutational isomorphism of the infeasible subset.

By Axiom~\ref{axm:OPs.infsble.isom}, the infeasible subsets satisfy:
\begin{align}
    \fInfeasIntTwoTplTS{\oieS}{\idxT_1} \overset{M}{\cong} \fInfeasIntTwoTplTS{\oieS}{\idxT_2}.
\end{align}
This holds because infeasibility constraints are intrinsic properties determined by the internal relationships among participating $\bOIE$ instances,
independent of the descriptive order used to enumerate the $\bOIE$ instances.

\BF{Step III}: Set difference preserves isomorphism.

Since both the superset and the removed subset
exhibit structure-preserving permutational isomorphism via the same matrix $M$
when transitioning from $\idxT_1$ to $\idxT_2$,
their set difference (the feasible subset) also exhibits structure-preserving permutational isomorphism under $M$.
Specifically, for any $\TwoTplT \in \fFeasibleIntvlTwoTplTS{\oieS}{\idxT_1}$,
its permuted form $\TwoTplT \cdot M$ belongs to $\fFeasibleIntvlTwoTplTS{\oieS}{\idxT_2}$,
and this mapping constitutes a bijection that preserves the internal structure of constraints.
Therefore:
\begin{align}
    \fFeasibleIntvlTwoTplTS{\oieS}{\idxT_1} \overset{M}{\cong} \fFeasibleIntvlTwoTplTS{\oieS}{\idxT_2}.
\end{align}

this concludes the proof.
\end{proof}

It is worth emphasizing that the infeasibility constraints discussed in this subsubsection
are inherent properties determined by the internal relationships among the participating $\bOIE$ instances.
These constraints reflect the physical, logical, or causal limitations between the corresponding real-world events.
Axiom~\ref{axm:OPs.infsble.isom} serves as the formal expression of these constraints.
An innovative aspect of this paper lies in ``quantifying event execution from the perspective of subjective planning''.
During subjective planning,
it is essential to effectively screen out invalid results that violate physical laws, logical constraints, or causal relationships.
Therefore, this subsubsection serves as a crucial bridge connecting the purely formalized theory of this paper with practical applications.
\fi
        \subsubsection{Mutually Independent $\bOIES$}\label{subsubsec:OPs.unfsbleDI.independentOIES}\EBL
        \ifincludeFile
The foregoing subsubsections have established a general formal framework
for distinguishing feasible and infeasible interval combinations within an $\bOIES$ instance,
grounded in the intrinsic physical, logical, and causal constraints between the real-world events
mapped by its constituent $\bOIE$ instances.
In this general framework,
the infeasible subset is non-empty in most practical scenarios,
as inter-event dependencies inevitably invalidate certain interval combinations.
However, there also exists a broad class of canonical scenarios
where the participating events are entirely free of mutual constraints,
with no causal precedence, resource contention, or physical interference between one another.
For such scenarios, we define a baseline, constraint-free $\bOIES$ model
with a trivial infeasible subset.
This special class of $\bOIES$ is formally defined below as Mutually Independent $\bOIES$.

\begin{defi}[Mutually independent $\bOIES$]\label{def:OPs.feasibleCP.independentOIES}\EBL

\indent Suppose there is an $\bOIES$ instance $\oieS$ with cardinality $n \in \mathbb{N}^+$, with the form
\begin{align}
    \oieS = \{ \oie_1, \oie_2, \cdots, \oie_n\},
\end{align}
and every element therein must not be $\voidError$.
Let
\begin{align}
    \idxT_{asc} = (1, 2, 3, \cdots, n).
\end{align}
If
\begin{align}
    \fInfeasIntTwoTplTS{\oieS}{\idxT_{asc}} = \emptyset,
\end{align}
then the type of $\oieS$ is called ``The mutually independent $\bOIES$'',
which is denoted as $\bOIES^{\text{indep}}$.
\end{defi}

Scenarios including 3 doctors submitting their manuscripts via their respective computers,
multiple computing nodes in a supercomputing center executing tasks without mutual interference,
and several athletes competing in their individual lanes without affecting one another,
can all be characterized by Definition~\ref{def:OPs.feasibleCP.independentOIES}.
And many subsequent examples in this paper will also be based on Definition~\ref{def:OPs.feasibleCP.independentOIES}.

This definition of mutually independent $\bOIES$ isolates the pure algebraic behavior from the confounding effects
of inter-event constraints,
making $\bOIES^{\text{indep}}$ the ideal foundational model for the theoretical analysis that follows.
For an $\bOIES^{\text{indep}}$ instance,
the maximum feasible subset of interval combinations
is exactly the full set naturally isomorphic to the Cartesian product
of the interval sets of its constituent $\bOIE$ instances,
with no elements excluded due to cross-event constraints.
This property ensures that the structure-preserving permutational isomorphism
of the feasible subset (Property~\ref{prop:oieOP.fsble.isom}) holds in its most general
and unconstrained form for mutually independent $\bOIES$.

Beyond its core theoretical role as a constraint-free baseline,
the mutually independent $\bOIES$ abstraction also directly maps to
a wide range of real-world application scenarios across the disciplines discussed in this paper.
In concurrent and distributed computing, for example,
it formalizes the execution of non-interfering parallel tasks across distributed nodes,
where the execution interval of one task imposes no causal or resource constraints on another.
In classical probability theory, it aligns with the core assumption of independent and identically distributed trials,
modeling sampling events with no cross-interference between individual draws.
In the physical modeling of competitive speed events,
it captures the idealized scenario of athletes competing in separate lanes with no mutual disruption,
ensuring that each athlete’s feasible execution intervals remain unaffected by the others.
All of these scenarios will be revisited and formally modeled using the sequence operations defined in the following sections,
with the mutually independent $\bOIES$ as their core formal building block.

\fi
        \subsubsection{Planning Granularity Relativity and Progressive Constraint Discovery\label{subsubsec:OPs.infsbleDI.relativity}\EBL}
        \ifincludeFile
The preceding subsections have established that infeasibility constraints
arise from intrinsic physical, logical, or causal relationships among participating events,
and that these constraints are independent of the descriptive order
used to enumerate the constituent $\bOIE$ instances (Axiom~\ref{axm:OPs.infsble.isom}).
However, such constraints are not always transparent to the planner a priori.
This subsection presents two complementary perspectives on isolated and joint planning.
We further derive a dynamic cognitive mechanism,
named ``$\bOIES$ progressive regression analysis''.

\paragraph{Reference frame switching: The subject's right to choose analytical granularity.}
Isolated planning and joint planning are not irreversible stages in a temporal sequence.
Rather, they are alternative configurations of the analytical framework.
Just as one may study the motion of the Earth around the Sun while temporarily ignoring the galactic orbit,
a planner may choose to ``decouple'' a particular $\bOIE$ instance
from a joint context and activate only its internal constraints.
The legitimacy of this switch is guaranteed by the algebraic structure of the $\bOIE$ framework:
the original $\bOIE$ instance is embedded as an immutable component in the $\C$ element of a $\bCombOIE$ instance (Definition~\ref{def:OIE.CombOIE}),
and its $\I$ and $\F$ remain intact as mathematical objects,
subject only to re-filtering at the higher level of composition.

In particular, when a set of $\bOIE$ instances constitutes a \emph{mutually independent $\bOIES$} (Definition~\ref{def:OPs.feasibleCP.independentOIES}),
the cross-event infeasible set is empty:
\begin{align}
    \fInfeasIntTwoTplTS{\oieS^{\RM{indep}}}{\idxT} = \emptyset.
\end{align}
In this case, the maximum feasible subset under joint planning is naturally isomorphic to the Cartesian product of the individually planned interval sets,
and the two granularities are equivalent in terms of feasibility.
This confirms that the choice of planning granularity is fundamentally an \emph{epistemological right} of the subject:
whether a constraint can be observed, although taking its objective existence as a necessary prerequisite,
is not determined by its existence itself,
and its core determinant is whether the planner includes it into the scope of selective activation.

\paragraph{Information asymmetry and planning failure: Hard feedback from incomplete models.}
Conversely, when a joint reality already holds objectively (i.e., cross-event constraints are intrinsic),
yet the planner fails to incorporate them into the $\bOIES$ due to incomplete information,
isolated planning inevitably produces \emph{structural failure}.
For example, the planner holds only
\begin{align}
    \oieS_{planner} = \{\oie_1\},
\end{align}
whereas the objective world contains
\begin{align}
    \oieS_{reality} = \{\oie_1, \oie_2, \ldots, \oie_n\},
\end{align}
where the unmodeled instances $\{\oie_2, \ldots, \oie_n\}$
stand in physical, logical, or causal relations with $\oie_1$.
When subjected to joint filter,
an interval selected from $\I_{\oie_1}$ under isolated planning may
fall into the infeasible subset (Definition~\ref{def:OPs.infsbleDI}).

This failure is not a mathematical error of $\bOIE$,
but a \emph{hard feedback} produced by the gap between the model and reality.
The infeasibility constraints are intrinsic properties of the relationships among $\bOIE$ instances (Subsubsection~\ref{subsubsec:OPs.unfsbleDI.equival});
the planner's mistake lies in the fact that the algebraic structure ($\bOIES$)
subjectively introduced by the planner does not cover the objective constraints.
The $\bOIE$ framework converts the vague predicament of being ``kept in the dark''
into a \emph{diagnosable model incompleteness}:
When the planner finds that the interval of the real-world event
is not contained in its own planning intervals,
the formalism signals that $\C$ and $\A$ should be fixed,
rather than that the operation itself is at fault.

\paragraph{Progressive regression analysis of $\bOIES$: From model failure to structural expansion.}
The two perspectives above jointly motivate a dynamic epistemological mechanism that
we call ``Progressive regression analysis of $\bOIES$''.
The planner may not possess all constraints at once;
rather, new participants and their relationships are discovered in practice,
and the $\bOIES$ instance is expanded iteratively.
Formally, define an $\bOIES$ expansion sequence:
\begin{align}
    \oieS^{(0)} \subset \oieS^{(1)} \subset \oieS^{(2)} \subset \cdots \subset \oieS^{(k)} \approx \oieS_{reality},
\end{align}
where $\oieS^{(0)}$ is the initial isolated-planning set.
If it's not $\oieS_{reality}$, at each step, some newly discovered $\bOIE$ instances are adjoined.
Because the original $\bOIE$ instances are immutable,
each expansion merely activates additional cross-event constraints based on the new $\bOIES$ instance,
while leaving previously verified local structures intact.

Through this process, the planner undergoes repeated \emph{constraint regression}:

\begin{enumerate}
    \item \textbf{Isolated planning stage:}
    Obtain a feasible interval set based on the current $\bOIES$ instance.
    \item \textbf{Joint validation stage:}
    Validate the plan against the expanded $\bOIE$, detect infeasible combinations.
    \item \textbf{Structural correction stage:}
    If it is needed, explicitly model the source of conflict as new $\bOIE$ instances, and recompute the feasible set.
    \item \textbf{Convergence stage:}
    When $\oieS^{(t)}$ stabilizes and the infeasible combinations conform to reality,
    subjective planning and objective constraints have achieved a match.
\end{enumerate}

This mechanism elevates the $\bOIE$ framework from a static algebraic tool
to a \emph{dynamic cognitive methodology}:
Driven by the discovery of constraints,
the subject actively adjusts the analytical granularity,
completing an epistemological transition from local feasibility to global feasibility.
\fi
    \subsection{Complete Sequence Addition}\label{subsec:ops.add}\EBL
    \ifincludeFile
Building on the previous content, in this section,
we propose a class of operations for $\bOIE$ instances.
The result of these operations has a significant impact
on the intervals of the events mapped by the participating $\bOIE$ instances.
Since this impact is related to
the execution order of the $\bE$ instances which the participating $\bOIE$ instances are mapped to,
we introduce the term ``Sequence Operation''.
Before delving into the details of the concept,
we will first prepare several mathematical helper functions.

\begin{defi}[\textbf{Function for finding the minimum 1st item and the maximum 2nd item of a tuple of 2-tuples
\& permutation invariance}]\label{def:OPs.add.min1max2ofTwoTplT}\EBL

1) \TSF{Definition \& Helper functions:}

Suppose there is a $\bTwoTplT$ instance $\TwoTplT_{\lambda}$ with length $n \geq 1$, with the form
\begin{align}
    (\ (x_1,\ y_1),\ (x_2,\ y_2),\ \cdots,\ (x_n,\ y_n)\ ).
\end{align}
We define two extremal functions:
\begin{align}
    & \fMinFstOfTwoTplT{\TwoTplT_{\lambda}} \longmapsto min\{x_i | i \in \{1, 2, \cdots, n\} \}.
\end{align}
\begin{center}
    \begin{tabular}{|l|l|}
        \hline
        $\TwoTplT_{\lambda}$ & A tuple of 2-tuples(a $\bTwoTplT$ instance)\\ \hline
        Return & \makecell[l]{The minimum 1st item among all elements of $\TwoTplT_{\lambda}$} \\ \hline
    \end{tabular}
\end{center}
\begin{align}
    & \fMaxSndOfTwoTplT{\TwoTplT_{\lambda}} \longmapsto \max\{y_i | i \in \{1, 2, \cdots, n\} \}.
\end{align}
\begin{center}
    \begin{tabular}{|l|l|}
        \hline
        $\pPre\TwoTplT$ & A tuple of 2-tuples \\ \hline
        Return & \makecell[l]{The maximum 2nd item among all elements of $\TwoTplT_{\lambda}$} \\ \hline
    \end{tabular}
\end{center}

2) \TSF{permutation invariance:}

Suppose there are two $\bTwoTplT$ instances $\TwoTplT$ and $\TwoTplT^{'}$,
with the same length $n$, and with forms
\begin{align}
    & \TwoTplT = (\ (x_1,\ y_1),\ (x_2,\ y_2),\ \cdots,\ (x_n,\ y_n)\ ), \\
    & \TwoTplT^{'} = (\ (x_{1}^{'},\ y_{1}^{'}),\ (x_{2}^{'},\ y_{2}^{'}), \cdots,\ (x_{n}^{'},\ y_{n}^{'})\ ),
\end{align}
and there exists a permutation matrix $\CAL{M}$ such that
\begin{align}
    \TwoTplT \stackrel{\CAL{M}}{\sim} \TwoTplT^{'}.
\end{align}
Since the two $\bTwoTplT$ instances differ only in the order of their internal elements,
without altering the existence or values of the elements,
we conclude that:
\begin{align}
    \begin{aligned}
        & \TwoTplT \stackrel{\CAL{M}}{\sim} \TwoTplT^{'} \Rightarrow \\
        & \qquad\quad\fMinFstOfTwoTplT{\TwoTplT} = \fMinFstOfTwoTplT{\TwoTplT^{'}} \\
        & \qquad \land\ \fMaxSndOfTwoTplT{\TwoTplT} = \fMaxSndOfTwoTplT{\TwoTplT^{'}}.
    \end{aligned}
\end{align}
\end{defi}

\begin{defi}[\textbf{Domain-filtered subset of a non-empty set of tuples of 2-tuples under a domain-filtering 2-tuple
\& Helper functions \& Permutation isomorphism invariance}]\label{def:OPs.Add.DomainFilteredSubTwoTplTS}\EBL

1) \TSF{Definition:}

\indent Consider a non-empty $\bTwoTplTS$ instance $\TwoTplTS$
\begin{align}
    \TwoTplTS = \{\ \TwoTplT_1,\ \TwoTplT_2,\ \cdots,\ \TwoTplT_{m}\ \},
\end{align}
where every element has the same length $0 < n < +\infty$,
and two real numbers $\alpha$ and $\beta$($\alpha < \beta$),
we build a set
\begin{align}
    \CAL{S} = \{\ \TwoTplT \in \TwoTplTS\ |\ \forall (x, y) \in \TwoTplT:\ x \geq \alpha \land y \leq \beta \}.
\end{align}
Then the set
\begin{align}
    & \TwoTplTS_{dom} = \label{eq:TwoTplTS_dom} \\
    & \begin{cases} \notag
    \CAL{S}, & \text{if } \makecell{(\forall k \in [1, n], \exists \TwoTplT \in \CAL{S}: the\ 1st\ item\ of\ the\ \text{k-th}\ elem\ of\ \TwoTplT = \alpha)\ \land \\ \quad (\forall k \in [1, n], \exists \TwoTplT \in \CAL{S}: the\ 2nd\ item\ of\ the\ \text{k-th}\ elem\ of\ \TwoTplT = \beta)} \notag \\
    \emptyset, & \text{if } \makecell{(\exists k \in [1, n], \nexists \TwoTplT \in \CAL{S}: the\ 1st\ item\ of\ the\ \text{k-th}\ elem\ of\ \TwoTplT = \alpha)\ \lor \\ \quad (\exists k \in [1, n], \nexists \TwoTplT \in \CAL{S}: the\ 2nd\ item\ of\ the\ \text{k-th}\ elem\ of\ \TwoTplT = \beta)} \notag \\
    \end{cases}. \notag
\end{align}
is called ``The domain-filtered subset of $\TwoTplTS$ under the domain-filtering 2-tuple $(\alpha, \beta)$'',
denoted as ``$\bDomainFilteredSubTwoTplTS|_\alpha^\beta$''.
$(\alpha, \beta)$ is a domain-filtering 2-tuple of $\TwoTplTS$, written as ``$\bDomainFilterTwoTpl$''.

2) \TSF{Helper function:}
\begin{align}
    & \fDomainFilteredSubTwoTupleTS{\pPre\TwoTplTS}{\pPre\alpha}{\pPre\beta} \longmapsto A\ \bDomainFilteredSubTwoTplTS|_{\pPre\alpha}^{\pPre\beta}\ instance.
\end{align}
\begin{center}
    \begin{tabular}{|l|l|}
        \hline
        $\pPre\TwoTplTS$ & A non-empty set of tuples of 2-tuples \\ \hline
        $\pPre\alpha$ & \makecell[l]{The left boundary of the domain-filtered subset} \\ \hline
        $\pPre\beta$ & \makecell[l]{The right boundary of the domain-filtered subset} \\ \hline
        Return & \makecell[l]{The domain-filtered subset of $\pPre\TwoTplTS$ under $(\pPre\RM{A}, \pPre\RM{B})$} \\ \hline
    \end{tabular}
\end{center}

3) \TSF{permutation isomorphism invariance:}

Keeping $\pPre\alpha$ and $\pPre\beta$ ($\pPre\alpha < \pPre\beta$) fixed,
if there are two $\bTwoTplTS$ instances $\TwoTplTS_1$ and $\TwoTplTS_2$ with the same cardinality,
and satisfying isomorphic via permutation matrix $\CAL{M}$,
then the two return values of this function are isomorphic via permutation matrix $\CAL{M}$ too.
i.e.
\begin{align}
    & \TwoTplTS_1 \stackrel{\CAL{M}}{\cong} \TwoTplTS_2 \Rightarrow \\
    & \qquad \fDomainFilteredSubTwoTupleTS{\TwoTplTS_1}{\pPre\alpha}{\pPre\beta} \stackrel{\CAL{M}}{\cong} \fDomainFilteredSubTwoTupleTS{\TwoTplTS_2}{\pPre\alpha}{\pPre\beta}. \notag
\end{align}
\end{defi}

\begin{proof}[\BF{Proof of the permutation isomorphism invariance in Definition~\ref{def:OPs.Add.DomainFilteredSubTwoTplTS}}]\EBL

Suppose there are two non-empty $\bTwoTplTS$ instances $\TwoTplTS_1$ and $\TwoTplTS_2$ with the same cardinality
and satisfying isomorphic via permutation matrix $\CAL{M}$
\begin{align}
    \TwoTplTS_1 \stackrel{\CAL{M}}{\cong} \TwoTplTS_2.
\end{align}
Let $\sigma$ be the permutation on $\{1,2,\dots,n\}$ induced by $\M$.
For any
\begin{align}
    \TwoTplT = \big((x_1,y_1),\,(x_2,y_2),\,\dots,\,(x_n,y_n)\big),
\end{align}
define its $\sigma$-rearrangement as
\begin{align}
    & \sigma(\TwoTplT) \;=\; \big((x_{\sigma(1)},y_{\sigma(1)}),\,\dots,\,(x_{\sigma(n)},y_{\sigma(n)})\big), \\
    & \sigma^{-1}(\TwoTplT) \;=\; \big((x_{\sigma^{-1}(1)},y_{\sigma^{-1}(1)}),\,\dots,\,(x_{\sigma^{-1}(n)},y_{\sigma^{-1}(n)})\big),
\end{align}
which can also be denoted as
\begin{align}
    & \sigma(\TwoTplT) = \TwoTplT \cdot \CAL{M}, \\
    & \sigma^{-1}(\TwoTplT^{'}) = \TwoTplT^{'} \cdot \CAL{M}^{-1}.
\end{align}

Then $\TwoTplTS_1$ and $\TwoTplTS_2$ form a bijection
\begin{align}
    \forall \TwoTplT \in \TwoTplTS_1: \sigma(\TwoTplT) \in \TwoTplTS_2, \label{eq:TwoTplT_trans_a} \\
    \forall \TwoTplT^{'} \in \TwoTplTS_2: \sigma^{-1}(\TwoTplT^{'}) \in \TwoTplTS_1. \label{eq:TwoTplT_trans_b}
\end{align}

For fixed $\alpha$ and $\beta$ ($\alpha < \beta$),
we examine the computation processes of \\ $\fDomainFilteredSubTwoTupleTS{\TwoTplTS_1}{\alpha}{\beta}$
and $\fDomainFilteredSubTwoTupleTS{\TwoTplTS_2}{\alpha}{\beta}$ separately.

\BF{Step I.} Construct sets $\TwoTplS_{1}^{dom}$ and $\TwoTplS_{2}^{dom}$. \\
\begin{align}
    & \TwoTplS_{1}^{dom} = \{\ \TwoTplT \in \TwoTplTS_1\ |\ \forall (x, y) \in \TwoTplT:\ x \geq \alpha \land y \leq \beta\ \}, \\
    & \TwoTplS_{2}^{dom} = \{\ \TwoTplT \in \TwoTplTS_2\ |\ \forall (x, y) \in \TwoTplT:\ x \geq \alpha \land y \leq \beta\ \}.
\end{align}
For any
\begin{align}
    \TwoTplT_{\lambda} \in \TwoTplS_{1}^{dom},
\end{align}
let
\begin{align}
    \TwoTplT_{\mu} = \sigma(\TwoTplT_{\lambda}) = \TwoTplT_{\lambda} \cdot \CAL{M}.
\end{align}
Since $\CAL{M}$ only permutes the 2-tuples inside $\TwoTplT_{\lambda}$ structure
without altering the set of these 2-tuples, we have:
\begin{align}
    & \forall (x_{\lambda}, y_{\lambda}) \in \TwoTplT_{\lambda} \Longleftrightarrow \exists (x_{\mu}, y_{\mu}) \in \TwoTplT_{\mu}: x_{\lambda} = x_{\mu} \land y_{\lambda} = y_{\mu}.
\end{align}
Thus, the condition
\begin{align}
    \forall (x_{\lambda}, y_{\lambda}) \in \TwoTplT_{\lambda}: x_{\lambda} \geq \alpha \land y_{\lambda} \leq \beta,
\end{align}
holds if and only if the condition
\begin{align}
    \forall (x_{\mu}, y_{\mu}) \in \TwoTplT_{\mu}: x_{\mu} \geq \alpha \land y_{\mu} \leq \beta
\end{align}
holds.

This indicates that
\begin{align}
    \TwoTplT_{\lambda} \in \TwoTplS_{1}^{dom} \Longleftrightarrow \TwoTplT_{\mu} \in \TwoTplS_{2}^{dom}
\end{align}
Thus, \(\sigma\) restricts to a bijection between $\TwoTplS_{1}^{dom}$ and $\TwoTplS_{2}^{dom}$,
meaning these two filtered subsets are structure-preserving permutationally isomorphic via matrix $\M$.

\BF{Step II.} Permutation invariance of boundary conditions.

Since we have~\eqref{eq:TwoTplT_trans_a} and~\eqref{eq:TwoTplT_trans_b}, so the following equivalence holds:
\begin{align}
    & \forall x \in [1, n], \exists \TwoTplT \in \TwoTplS_{1}^{dom}: the\ 1st\ item\ of\ the\ \text{x-th}\ elem\ of\ \TwoTplT = \alpha \Longleftrightarrow \\
    & \qquad \forall y \in [1, n], \exists \sigma(\TwoTplT) \in \TwoTplS_{2}^{dom}: the\ 1st\ item\ of\ the\ \text{y-th}\ elem\ of\ \sigma(\TwoTplT) = \alpha, \notag \\
    &\forall x \in [1, n], \exists \TwoTplT \in \TwoTplS_{1}^{dom}: the\ 2nd\ item\ of\ the\ \text{x-th}\ elem\ of\ \TwoTplT = \beta \Longleftrightarrow \\
    & \qquad \forall y \in [1, n], \exists \sigma(\TwoTplT) \in \TwoTplS_{2}^{dom}: the\ 2nd\ item\ of\ the\ \text{y-th}\ elem\ of\ \sigma(\TwoTplT) = \beta, \notag \\
    & \exists x \in [1, n], \nexists \TwoTplT \in \TwoTplS_{1}^{dom}: the\ 1st\ item\ of\ the\ \text{x-th}\ elem\ of\ \TwoTplT = \alpha \Longleftrightarrow \\
    & \qquad \exists y \in [1, n], \nexists \sigma(\TwoTplT) \in \TwoTplS_{2}^{dom}: the\ 1st\ item\ of\ the\ \text{y-th}\ elem\ of\ \sigma(\TwoTplT) = \alpha, \notag \\
    & \exists x \in [1, n], \nexists \TwoTplT \in \TwoTplS_{1}^{dom}: the\ 2nd\ item\ of\ the\ \text{x-th}\ elem\ of\ \TwoTplT = \beta \Longleftrightarrow \\
    & \qquad \exists y \in [1, n], \nexists \sigma(\TwoTplT) \in \TwoTplS_{2}^{dom}: the\ 2nd\ item\ of\ the\ \text{y-th}\ elem\ of\ \sigma(\TwoTplT) = \beta. \notag
\end{align}
The construction of the above conditions makes full use of~\eqref{eq:TwoTplTS_dom} in Definition~\ref{def:OPs.Add.DomainFilteredSubTwoTplTS}.

\BF{Step III.} Conclusion

From Step II, it follows that, $\TwoTplS_{1}^{dom}$ satisfies the condition of being non-empty under the boundary rule
if and only if $\TwoTplS_{2}^{dom}$ satisfies the condition of being non-empty under the boundary rule.
According to Definition~\ref{def:OPs.Add.DomainFilteredSubTwoTplTS},
if both $\TwoTplS_{1}^{dom}$ and $\TwoTplS_{2}^{dom}$ satisfy the condition and are non-empty
\begin{align}
    & \fDomainFilteredSubTwoTupleTS{\TwoTplTS_1}{\alpha}{\beta} = \TwoTplS_{1}^{dom}, \\
    & \fDomainFilteredSubTwoTupleTS{\TwoTplTS_2}{\alpha}{\beta} = \TwoTplS_{2}^{dom} = \{\sigma: \TwoTplS_{1}^{dom} \}.
\end{align}
Since $\sigma$ is precisely the permutation mapping defined by $\CAL{M}$,
the two results are isomorphic via permutation matrix $\CAL{M}$.

Otherwise, if neither satisfies the condition, both are $\emptyset$,
$\emptyset$ is isomorphic to itself via any permutation matrix.

To conclude, we have proven the permutation isomorphism invariance in Definition~\ref{def:OPs.Add.DomainFilteredSubTwoTplTS}.

\end{proof}

According to Definition~\ref{def:OPs.Add.DomainFilteredSubTwoTplTS},
the domain-filtered subset of a $\bTwoTplTS$ instance $\TwoTplTS$ with respect to
the domain-filtering 2-tuple $(\alpha, \beta)$ can be expressed as
\begin{align}
    \domainFilteredSubTwoTplTS|_\alpha^\beta = \fDomainFilteredSubTwoTupleTS{\TwoTplTS}{\alpha}{\beta},
\end{align}
which can be $\emptyset$.

The characteristic of the subset $\domainFilteredSubTwoTplTS|_\alpha^\beta$ is that:
Each element of this subset is restricted within the filtering domain.
Moreover, there must be an element in this subset such that for a certain 2-tuple within this element,
its first item is equal to the left boundary $\alpha$.
And there must be an element such that for a certain 2-tuple within this element,
its second item is equal to the right boundary $\beta$.

To utilize this subset, we further define several variable types and related helper functions.

Using the helper function in Definition~\ref{def:oie.BoundTwoTuple},
we can take $\domainFilteredSubTwoTplTS|_\alpha^\beta$ obtained from Definition~\ref{def:OPs.Add.DomainFilteredSubTwoTplTS}
and calculate a set of bound 2-tuples.
These will be used when performing sequence operations on $\bOIES$
to generate the $\I$(3rd element) for the operation result.

Based on the domain-filtering mechanism established above,
now, we propose a sequence operation for $\bOIES$:
``Complete Sequence Addition of members of a finite non-empty $\bOIES$ under a domain-filtering 2-tuple in an index order''.
All members within an $\bOIES$ instance participate in this operation,
and the result of the operation is a new $\bOIE$ instance.

\begin{defi}[\BF{Complete Sequence Addition of members of a finite non-empty $\bOIES$ instance under a domain-filtering 2-tuple in an index order}]\label{def:OPs.Add}\EBL

Suppose there is a finite $\bOIES$ instance $\oieS$ containing $n$ ($n \in \mathbb{N}^{+}, n > 1$) elements in the ordered enumeration form
\begin{align}
    \oieS = \{\oie_1, \oie_2, \oie_3, \cdots, \oie_n\},
\end{align}
an $\bIdxT$ instance with length $n$
\begin{align}
    \idxT = ( \idx_1, \idx_2, \idx_3, \cdots, \idx_n ),
\end{align}
and a $\bDomainFilterTwoTpl$ instance $( \alpha, \beta )$.
For $\oieS$ and $\bIdxT$ instance
\begin{align}
    \idxT_{asc} = (1,\ 2,\ 3,\ \cdots,\ n),
\end{align}
the $\bInfeasIntTwoTplTS$ instance of $\oieS$ in the index order $\idxT_{asc}$ is
\begin{align}
    \infeasIntTwoTplTS_{\oieS} = \fInfeasIntTwoTplTS{\oieS}{\idxT_{asc}}.
\end{align}

Define a finitary operation ``${\oplus|_{\alpha}^{\beta}}$'', whose operational rules are as follows:

\BF{Step 1.}
If
\begin{align}
    & \bigcap_{i=1}^n \CAL{A}_{\oie_{\idx_i}} \neq \emptyset \\
    \lor\ & \exists\ k\ \in\ [1, n]:\ \oie_{\idx_k} = \voidError, \notag
\end{align}
then the result is directly set to the unique $\bVoidError$ instance, with the form
\begin{align}
    \voidError = \mathlarger{\mathlarger{(}}\ (),\ \emptyset,\ \emptyset,\ \emptyset\ \mathlarger{\mathlarger{)}},
\end{align}
and the operation will terminate.
Otherwise, initialize a new $\bOIE$ instance $\oie_{result}$ as the result.
Temporarily set the $\CAL{C}$(1st element) of $\oie_{result}$ to
the result of tuple flattening for all the operands into a tuple in the order of $\idxT$
\begin{align}
    \C_{\oie_{result}} = (\oie_{\idx_1},\ \oie_{\idx_2},\ \oie_{\idx_3},\ \cdots,\ \oie_{\idx_n}).
\end{align}
Then temporarily set the $\CAL{A}$(4th element) of $\oie_{result}$ to the set
\begin{align}
    \A_{\oie_{result}} = \bigcup_{i=1}^n \CAL{A}_{\oie_{\idx_i}}.
\end{align}
Then proceed to Step 2.

\BF{Step 2.}
Using $\infeasIntTwoTplTS_{\oieS}$ and Definition~\ref{def:OPs.feasibleCP},
take the maximal feasible subset of
the flattened Cartesian product of all the $\I$ of members in
$\oieS$ under the index order $\idxT$, denoted as $\feasibleIntvlTwoTplTS$
\begin{align}
    \feasibleIntvlTwoTplTS = \fFeasibleIntvlTwoTplTS{\oieS}{\IdxT}.
\end{align}
If
\begin{align}
    \feasibleIntvlTwoTplTS = \emptyset,
\end{align}
then $\oie_{result}$ is instead set to $\voidError$, with the form
\begin{align}
    \voidError = \mathlarger{\mathlarger{(}}\ (),\ \emptyset,\ \emptyset,\ \emptyset\ \mathlarger{\mathlarger{)}},
\end{align}
and the operation ends without performing the subsequent steps.
Otherwise, proceed to Step 3.

\BF{Step 3.}
Using Definition~\ref{def:OPs.Add.DomainFilteredSubTwoTplTS},
get the corresponding $\bDomainFilteredSubTwoTplTS\big|_{\alpha}^{\beta}$ instance
\begin{align}
    \begin{aligned}
        \domainFilteredSubTwoTplTS\big|_{\alpha}^{\beta} & = \fDomainFilteredSubTwoTupleTS{\feasibleIntvlTwoTplTS}{\alpha}{\beta} \\
        & = \fDomainFilteredSubTwoTupleTS{\fFeasibleIntvlTwoTplTS{\oieS}{\IdxT}}{\alpha}{\beta},
    \end{aligned}
\end{align}
with the form
\begin{align}
    \domainFilteredSubTwoTplTS\big|_{\alpha}^{\beta} = \{ \TwoTplT_1, \TwoTplT_2, \cdots, \TwoTplT_\lambda \}.
\end{align}
If
\begin{align}
    \domainFilteredSubTwoTplTS\big|_{\alpha}^{\beta} \neq \emptyset,
\end{align}
then the $\F$(2nd element) of $\oie_{result}$ is assigned to $\domainFilteredSubTwoTplTS\big|_{\alpha}^{\beta}$:
\begin{align}
    \F_{\oie_{result}} = \domainFilteredSubTwoTplTS\big|_{\alpha}^{\beta},
\end{align}
and then Step 4 is executed.
Otherwise, if
\begin{align}
    \domainFilteredSubTwoTplTS\big|_{\alpha}^{\beta} = \emptyset,
\end{align}
$\oie_{result}$ is instead set to $\voidError$, with the form
\begin{align}
    \voidError = \mathlarger{\mathlarger{(}}\ (),\ \emptyset,\ \emptyset,\ \emptyset\ \mathlarger{\mathlarger{)}}.
\end{align}
And the operation ends without performing the subsequent steps.

\BF{Step 4.}
Using Property~\ref{prop:oie.2and3},
get the $\CAL{I}$(3rd element) of $\oie_{result}$
\begin{align}
    \begin{aligned}
        \I_{\oie_{result}} & = \bigg\{\ \fBoundTwoT{\TwoTplT_\lambda}\ \bigg|\ \TwoTplT_\lambda \in \domainFilteredSubTwoTplTS\big|_{\alpha}^{\beta}\ \bigg\} \\
                           & = \bigg\{\ \fBoundTwoT{\TwoTplT_\lambda}\ \bigg|\ \\
                           & \qquad\quad  \TwoTplT_i \in \fDomainFilteredSubTwoTupleTS{\fFeasibleIntvlTwoTplTS{\oieS}{\IdxT}}{\alpha}{\beta}\ \bigg\}.
    \end{aligned}
\end{align}
At this juncture, the operation is completed.

Finally, the form of $\oie_{result}$ is
\begin{align}
    \oie_{result} = \mathlarger{\mathlarger{(}}\ (),\ \emptyset,\ \emptyset,\ \emptyset\ \mathlarger{\mathlarger{)}},
\end{align}
or
\begin{align}
    \begin{aligned}
        & \oie_{result} = \\
        & \quad \mathlarger{\mathlarger{(}} \\
        & \qquad (\oie_{\idx_1},\ \oie_{\idx_2},\ \oie_{\idx_3},\ \cdots,\ \oie_{\idx_n}), \\
        & \qquad \fDomainFilteredSubTwoTupleTS{\fFeasibleIntvlTwoTplTS{\oieS}{\IdxT}}{\alpha}{\beta}, \\
        & \qquad \bigg\{\ \fBoundTwoT{\TwoTplT_\lambda}\ \bigg|\ \\
        & \qquad\qquad  \TwoTplT_\lambda \in \fDomainFilteredSubTwoTupleTS{\fFeasibleIntvlTwoTplTS{\oieS}{\IdxT}}{\alpha}{\beta}\ \bigg\}, \\
        & \qquad \bigcup_{i=1}^n \CAL{A}_{\oie_{\idx_i}}, \\
        & \quad \mathlarger{\mathlarger{)}}.
    \end{aligned}
\end{align}

The process described above,
including cases where the operation yields the $\bVoidError$ instance, is referred to as
``The Complete Sequence Addition of members of a finite non-empty $\bOIES$ instance
under domain-filtering 2-tuple $(\alpha, \beta)$ in index order $\IdxT$'',
denoted as
\begin{align}
    \overset{n}{\underset{\substack{i=1 \\ {\oplus|_{\alpha}^{\beta}}}}{\sum}} \oie_{\idx_i}
\end{align}
or
\begin{align}
    {\oplus|_{\alpha}^{\beta}}(\oie_{\idx_1},\ \oie_{\idx_2},\ \cdots,\ \oie_{\idx_n})
\end{align}
\end{defi}

The meanings of the four steps of the Complete Sequence Addition are as follows.

\BF{Step 1}: ``
Calculate the 1st and 4th elements of the result,
excluding cases where the same $\bAtomEStar$ instances are involved or $\voidError$ is involved''.

Cases where the same $\bAtomEStar$ instances are involved in operations
lead to duplicate planning of real-world events,
and are invalid under our operation rule.
Additionally, if any operand is $\voidError$,
this means $\emptyset$ will participate in the Cartesian product operation of the sets,
and the result will still be $\emptyset$.
This will directly result in the Complete Sequence Addition being $\voidError$ under our operation rule.
Intuitively, the operation of an entity that cannot plan real-world events
with other plan-capable entities renders the integrated plan invalid.
Therefore, both in terms of operational logic and practical requirements,
we directly stipulate that the results of the $\oplus|_{\alpha}^{\beta}$ in these two cases are $\voidError$.

For other valid conditions, additional constraints will apply in subsequent steps.
If these constraints are triggered,
the operation's result may be set to $\voidError$ too.
Thus, after temporarily assigning the $\CAL{C}$(the 1st element) and $\CAL{A}$(the 4th element)
of $\oie_{result}$, proceed to Step 2.

\BF{Step 2}: ``Obtain all feasible combinations of the flattened Cartesian product''.
This is achieved by performing the flattened Cartesian product,
getting the result, and then selecting the feasible subset.
For detailed procedures, refer to function $\fPre\RM{Feas2tuple}\CAL{TS}$
in Definition~\ref{def:OPs.feasibleCP} and Definition~\ref{def:OPs.feasibleCP}.
If the maximal feasible subset is $\emptyset$,
it indicates that there are no possible intervals for the event to execute in the real world.
Then the result of the operation is directly determined as $\voidError$.

\BF{Step 3}: ``Filter all feasible combinations within the domain''.
The filtering rule is characterized by defining a range and selecting all combinations that fall within this range,
with some combinations necessarily reaching the left and right boundaries of the defined range.
If the subset after filtering is $\emptyset$, it indicates that there is no subset that can fulfill the requirements of the $\oplus|_{\alpha}^{\beta}$ operation.
Then the operation result $\oie_{result}$ is directly set to an $\bVoidError$ instance.
Otherwise, the $\CAL{F}$(the 2nd element) of $\oie_{result}$ is assigned to the subset after filtering.

\BF{Step 4}: ``Summarize the feature set of the subset remaining after the filtering in Step 3''.
This feature set is precisely the $\CAL{I}$(the 3rd element) of $\oie_{result}$.
It's a set of 2-tuples,
storing all possible intervals for the event mapped by the result of $\oplus|_{\alpha}^{\beta}$.
This step enforces the structural relationship defined in Property~\ref{prop:oie.2and3}.

Intuitively, this is a certain level of ``equal opportunity''.
This algebraic fairness can be illustrated by a familiar sports scenario.

Imagine a 100-meter race.
All the athletes have the opportunity to start running after the effective natural reaction time following the gunshot.
But within the scope of classical physics,
no physical phenomenon of all athletes starting at exactly the same moment has ever been observed and recorded to date.
Technically, as of 2026, humans are still unable to perfectly observe a single moment without any error.
Even precision instruments have errors,
and most athletes are physically less capable than machines.
Since the athletes rely only on their ears to hear the sound and with no coordination among themselves,
we find it practically impossible for them to start running at exactly the same moment.
However, they all have this opportunity,
and whether they can achieve it depends on themselves.
This achieves a certain level of opportunity fairness.

\paragraph{Example: Three Doctors' Submissions via $\oplus|_{\alpha}^{\beta}$}\label{ex:csa_three_doctors}

\indent Consider three doctors submitting papers within a 24-hour window $[0,24)$,
with the physical constraint that only two doctors may submit simultaneously.
The atomic $\bOIE$ instances are defined as:
\begin{align}
    \begin{aligned}
        \oie_A &= \big((), \{((0,1)),((21,22))\}, \{(0,1),(21,22)\}, \{\EStar_{Dr_A}\}\big), \\
        \oie_B &= \big((), \{((0,1)),((13,14)),((20,22))\}, \{(0,1),(13,14),(20,22)\}, \{\EStar_{Dr_B}\}\big), \\
        \oie_C &= \big((), \{((0,1)),((19,22))\}, \{(0,1),(19,22)\}, \{\EStar_{Dr_C}\}\big).
    \end{aligned}
\end{align}

We apply Complete Sequence Addition $\oplus|_{0}^{22}$ to $\oieS = \{\oie_A, \oie_B, \oie_C\}$ with $\idxT=(1,2,3)$ and domain $(\alpha,\beta)=(0,22)$.

First, validity checking confirms $\bigcap_{i=1}^{3}\A_{\oie_i} = \emptyset$,
so the operation proceeds with $\C_{result}=(\oie_A, \oie_B, \oie_C)$ and $\A_{result}=\{\EStar_{Dr_A}, \EStar_{Dr_B}, \EStar_{Dr_C}\}$.

Next, computing the flattened Cartesian product yields 12 interval combinations.
The infeasible subset $\CAL{S}_{infeasible} = \{((0,1),(0,1),(0,1))\}$
is excluded (three concurrent submissions violate the resource limit),
leaving 11 feasible tuples in $\fPre\RM{Feas2tupleTS}(\oieS, \idxT)$.

Domain filtering via $\fPre\RM{DomFlt2tupleTS}$ verifies boundary reachability:
some combinations start at $t=0$ and end at $t=22$,
satisfying the domain constraints.
Thus $\F_{result}$ retains all 11 feasible scheduling schemes
\begin{align}
    \begin{aligned}
        & ((0,1),(0,1),(19,22)) \longmapsto (0,22) \\
        & ((0,1),(13,14),(0,1)) \longmapsto (0,14) \\
        & ((0,1),(13,14),(19,22)) \longmapsto (0,22) \\
        & ((0,1),(20,22),(0,1)) \longmapsto (0,22) \\
        & ((0,1),(20,22),(19,22)) \longmapsto (0,22) \\
        & ((21,22),(0,1),(0,1)) \longmapsto (0,22) \\
        & ((21,22),(0,1),(19,22)) \longmapsto (0,22) \\
        & ((21,22),(13,14),(0,1)) \longmapsto (0,22) \\
        & ((21,22),(13,14),(19,22)) \longmapsto (13,22) \\
        & ((21,22),(20,22),(0,1)) \longmapsto (0,22) \\
        & ((21,22),(20,22),(19,22)) \longmapsto (19,22).
    \end{aligned}
\end{align}
Then, applying $\mathfrak{f}\RM{Bound2tuple}$ to each feasible tuple produces the aggregated intervals
$\I_{result} = \{(0,14),(0,22),(13,22),(19,22)\}$
after deduplication.

The resulting $\bOIE$ instance is:
\begin{align}
    \oie_{result} = \big((\oie_A, \oie_B, \oie_C), \F_{result}, \I_{result}, \{\EStar_{Dr_A}, \EStar_{Dr_B}, \EStar_{Dr_C}\}\big).
\end{align}

This demonstrates the key properties of $\oplus|_{\alpha}^{\beta}$:
a certain degree of equal opportunity
(each doctor can access the full domain boundaries),
and constraint awareness (automatic exclusion of invalid three-way concurrency).
\fi
    \subsection{Complete Sequence Multiplication}\label{subsec:ops.multi}\EBL
    \ifincludeFile
The addition operation proposed in subsection~\ref{subsec:ops.add} is characterized by the fact
that after screening out invalid data for each operand participating in the operation,
each participant has an interval affording the opportunity both to start and to end at the same moment.
The emphasis of the Complete Sequence Addition is on ``Fairness''.

In this subsection, we propose another type of sequence operation.
Its concept is contrary to the fairness of the Complete Sequence Addition.
From a certain perspective, it emphasizes ``unfairness'' (i.e., strict sequential ordering).
The operation rules directly constrain the temporal ordering of starting and ending timestamps of all participants.
Under the sequence operation,
each event must finish before the next begins.
We call this ``Complete Sequence Multiplication''.

Before giving the full definition, let's first present some mathematical helper functions.

\begin{defi}[\BF{Ascending ordered filtered subset of the set of tuples of 2-tuples
 \& helper function \& Need not satisfy permutation isomorphism invariance}]\label{def:OPs.Multi.CompleteAscOrderFilteredSubTwoTplTS}\EBL

1) \TSF{Definition:}

Consider a non-empty finite $\bTwoTplTS$ instance $\TwoTplTS$ with length $n > 0$,
with the form
\begin{align}
    \TwoTplTS = \{\ \TwoTplT_1,\ \TwoTplT_2,\ \cdots,\ \TwoTplT_n\ \}.
\end{align}
Then the set
\begin{align}
    \TwoTplTS^{asc} & = \\
    & \left\{ \ \TwoTplT \in \TwoTplTS \Biggr|
    \makecell[l]{ \forall\ \TwoTpl_i,\ \TwoTpl_j \in \TwoTplT\ \land\ i < j: \notag \\
    \ \ \ the\ 2nd\ item\ of\ \TwoTpl_i \leq the\ 1st\ item\ of\ \TwoTpl_j} \right\} \notag
\end{align}
is called ``The ascending ordered filtered subset of $\TwoTplTS$'',
denoted as $\bCompleteAscOrderFilterSubTwoTplTS$.

2) \TSF{Helper function:}
\begin{align}
    & \fCompleteAscOrderFilterSubTwoTplTS{\pPre\TwoTplTS} \longmapsto \\
    & \qquad The\ \bCompleteAscOrderFilterSubTwoTplTS\ instance\ of\ \pPre\TwoTplTS. \notag
\end{align}
\begin{center}
    \begin{tabular}{|l|l|}
        \hline
        $\pPre\TwoTplTS$ & A $\bTwoTplTS$ instance \\ \hline
        Return & \makecell[l]{The $\bCompleteAscOrderFilterSubTwoTplTS$ instance of $\pPre\TwoTplTS$} \\ \hline
    \end{tabular}
\end{center}

3) \TSF{need not satisfy permutation isomorphism invariance:}

If there are two $\bTwoTplTS$ instances $\TwoTplTS_1$ and $\TwoTplTS_2$ with the same length $n$
and satisfying isomorphism via permutation matrix $\CAL{M}$,
then the $\bCompleteAscOrderFilterSubTwoTplTS$ instances of  $\TwoTplTS_1$ and $\TwoTplTS_2$ need not satisfy isomorphism via permutation matrix $\CAL{M}$.
\end{defi}

\begin{proof}[\BF{Proof of need not satisfy permutation isomorphism invariance in Definition~\ref{def:OPs.Multi.CompleteAscOrderFilteredSubTwoTplTS}}]\EBL

Suppose there are two non-empty $\bTwoTplTS$ instances $\TwoTplTS_1$ and $\TwoTplTS_2$ with the same length $n$
and satisfying isomorphism via permutation matrix $\CAL{M}$
\begin{align}
    \TwoTplTS_1 \stackrel{\CAL{M}}{\cong} \TwoTplTS_2.
\end{align}
Let
\begin{align}
    & \TwoTplTS^{asc}_1 = \fCompleteAscOrderFilterSubTwoTplTS{\TwoTplTS_1}, \\
    & \TwoTplTS^{asc}_2 = \fCompleteAscOrderFilterSubTwoTplTS{\TwoTplTS_2}.
\end{align}

\BF{(1) The condition of ``$\TwoTplTS^{asc}_1 = \TwoTplTS^{asc}_2 = \emptyset$''}

In this situation
\begin{align}
    \TwoTplTS^{asc}_1 = \TwoTplTS^{asc}_2 = \emptyset.
\end{align}
So, they are permutation isomorphic via any permutation matrix.

\BF{(2) The condition of ``$\TwoTplTS_1 = \TwoTplTS_2 \land \forall \TwoTplT \in \TwoTplTS_1: \Dim{\TwoTplT}$ = 1''}

For example, all elements of $\TwoTplTS_1$ and $\TwoTplTS_2$
are $\bTwoTplT$ instances with length 1.
In this situation, the permutation matrix $\CAL{M}$ is an identity permutation matrix.
Then we use an ascending ordered filter to get a subset from the same $\bTwoTplTS$ instance.
So, they are permutation isomorphic via an identity permutation matrix.

\BF{(3) The condition of excluding (1) and (2), ``$\TwoTplTS^{asc}_1 \ne \emptyset \lor \TwoTplTS^{asc}_2 \ne \emptyset$'',
    and ``$\TwoTplTS_1 \stackrel{\CAL{M}}{\cong} \TwoTplTS_2$''}

In this situation,
the permutation matrix $\CAL{M}$ isn't an identity permutation matrix.
Thus $\TwoTplTS_1 \stackrel{\CAL{M}}{\cong} \TwoTplTS_2$
means that, for any two $\bTwoTplT$ instances under a bijection
\begin{align}
    \begin{aligned}
        & \TwoTplT_{\lambda} \in \TwoTplTS_1, \TwoTplT_{\mu} \in \TwoTplTS_2 \\
        \land\ & \TwoTplT_{\lambda} \stackrel{\CAL{M}}{\sim} \TwoTplT_{\mu},
    \end{aligned}
\end{align}
there may be two 2-tuples satisfying
\begin{align}
    \begin{aligned}
        & (x_{i}, y_{i}), (x_{j}, y_{j}) \in \TwoTplT_{\lambda} \\
        \land\ & (x_{i}, y_{i}), (x_{j}, y_{j}) \in \TwoTplT_{\mu} \\
        \land\ & x_{i} < y_{i}, x_{j} < y_{j},
    \end{aligned}
\end{align}
their orders in $\TwoTplT_{\lambda}$ and $\TwoTplT_{\mu}$ are swapped with the form
\begin{align}
    & \TwoTplT_{\lambda} = (\ \cdots,\ (x_{i}, y_{i}),\ \cdots,\ (x_{j}, y_{j}),\ \cdots\ ), \\
    & \TwoTplT_{\mu} = (\ \cdots,\ (x_{j}, y_{j}),\ \cdots,\ (x_{i}, y_{i}),\ \cdots\ ).
\end{align}

If
\begin{align}
    x_{i} < y_{i} \leq x_{j} < y_{j} \land \TwoTplTS^{asc}_1 \ne \emptyset
\end{align}
then according to Definition~\ref{def:OPs.Multi.CompleteAscOrderFilteredSubTwoTplTS},
$\TwoTplT_{\lambda}$ satisfies the rule, so it remains in $\TwoTplTS^{asc}_1$.
But $\TwoTplT_{\mu}$ doesn't satisfy the rule, so it doesn't remain in $\TwoTplTS^{asc}_2$.
No instance in which $(x_{j}, y_{j})$ precedes $(x_{i}, y_{i})$
can exist in $\TwoTplTS^{asc}_2$;
thus, no element in that set is permutationally equivalent to $\TwoTplT_{\lambda}$.

Otherwise if
\begin{align}
    x_{j} < y_{j} \leq x_{i} < y_{i} \land \TwoTplTS^{asc}_2 \ne \emptyset,
\end{align}
this is consistent with the analytical approach just employed.

We can conclude that in this situation,
$\TwoTplTS_1$ and $\TwoTplTS_2$ are not isomorphic via permutation matrix $\CAL{M}$.

Combining the conclusions of (1), (2) and (3), we complete the proof.

\end{proof}

Now, we propose another sequence operation for $\bOIES$.
The elements participating in the operation are all the elements within an $\bOIES$ instance,
and the result of the operation is a new $\bOIE$ instance.

\begin{defi}[\textbf{Complete Sequence Multiplication of members of a finite $\bOIES$ instance in an index order}]\label{def:Ops.Multi}\EBL

Suppose there is a finite $\bOIES$ instance $\oieS$ containing $n$ ($n \neq \infty$) elements in the ordered enumeration form
\begin{align}
    \oieS = \{\oie_1, \oie_2, \oie_3, \cdots, \oie_n\},
\end{align}
an $\bIdxT$ instance with length $n$
\begin{align}
    \idxT = ( \idx_1, \idx_2, \idx_3, \cdots, \idx_n ).
\end{align}
For $\oieS$ and the ascend order $\bIdxT$ instance
\begin{align}
    \idxT_{asc} = (1,\ 2,\ 3,\ \cdots,\ n),
\end{align}
the $\bInfeasIntTwoTplTS$ instance of $\oieS$ under the index order $\idxT_{asc}$ is
\begin{align}
    \infeasIntTwoTplTS_{\oieS} = \fInfeasIntTwoTplTS{\oieS}{\idxT_{asc}}.
\end{align}

Define a finitary operation ``$\otimes$'', whose operational rules are as follows:

\BF{Step 1.}
If
\begin{align}
    & \bigcap_{i=1}^n \CAL{A}_{\oie_{\idx_i}} \neq \emptyset \\
    \lor\ & \exists\ k\ \in\ [1, n]:\ \oie_{\idx_k} = \voidError,
\end{align}
then the result is directly set to the unique $\bVoidError$ instance with the form
\begin{align}
    & \voidError = \mathlarger{\mathlarger{(}}\ (),\ \emptyset,\ \emptyset,\ \emptyset\ \mathlarger{\mathlarger{)}},
\end{align}
and the operation will terminate.
Otherwise, initialize a new $\bOIE$ instance $\oie_{result}$ as the result.
Temporarily set the $\CAL{C}$(1st element) of $\oie_{result}$ to tuple
\begin{align}
    \C_{\oie_{result}} = (\oie_{\idx_1},\ \oie_{\idx_2},\ \oie_{\idx_3},\ \cdots,\ \oie_{\idx_n}),
\end{align}
and temporarily set the $\CAL{A}$(4th element) of $\oie_{result}$ to be the set
\begin{align}
    \A_{\oie_{result}} = \bigcup_{i=1}^n \CAL{A}_{\oie_{\idx_i}}.
\end{align}
Then proceed to Step 2.

\BF{Step 2.}
Using $\infeasIntTwoTplTS_{\oieS}$ and Definition~\ref{def:OPs.feasibleCP},
take the maximal feasible subset of
the flattened Cartesian product of all the $\I$ of members in
$\oieS$ under the index order $\idxT$, denoted as $\feasibleIntvlTwoTplTS$
\begin{align}
    & \feasibleIntvlTwoTplTS = \fFeasibleIntvlTwoTplTS{\oieS}{\IdxT}.
\end{align}
If
\begin{align}
    \feasibleIntvlTwoTplTS = \emptyset,
\end{align}
then $\oie_{result}$ is instead set to $\voidError$ with the form
\begin{align}
    & \voidError = \mathlarger{\mathlarger{(}}\ (),\ \emptyset,\ \emptyset,\ \emptyset\ \mathlarger{\mathlarger{)}},
\end{align}
and the operation ends without performing the subsequent steps.
Otherwise, proceed to Step 3.

\BF{Step 3.}
Using Definition~\ref{def:OPs.Multi.CompleteAscOrderFilteredSubTwoTplTS},
get the corresponding $\bCompleteAscOrderFilterSubTwoTplTS$ instance from $\feasibleIntvlTwoTplTS$,
and let $\completeAscOrderFilterSubTwoTplTS$ denote it
\begin{align}
    \begin{aligned}
        \completeAscOrderFilterSubTwoTplTS & = \fCompleteAscOrderFilterSubTwoTplTS{\feasibleIntvlTwoTplTS} \\
        & = \fCompleteAscOrderFilterSubTwoTplTS{\fFeasibleIntvlTwoTplTS{\oieS}{\IdxT}},
    \end{aligned}
\end{align}
whose form is
\begin{align}
    & \completeAscOrderFilterSubTwoTplTS = \{ \TwoTplT_1, \TwoTplT_2, \cdots, \TwoTplT_\lambda \}.
\end{align}

If
\begin{align}
    \completeAscOrderFilterSubTwoTplTS \neq \emptyset,
\end{align}
then the $\F$(2nd element) of $\oie_{result}$ is set to $\completeAscOrderFilterSubTwoTplTS$
\begin{align}
    & \F_{\oie_{result}} = \completeAscOrderFilterSubTwoTplTS,
\end{align}
and then Step 4 is executed.
Otherwise, if
\begin{align}
    \completeAscOrderFilterSubTwoTplTS = \emptyset,
\end{align}
$\oie_{result}$ is instead set to $\voidError$ with the form
\begin{align}
    & \voidError = \mathlarger{\mathlarger{(}}\ (),\ \emptyset,\ \emptyset,\ \emptyset\ \mathlarger{\mathlarger{)}},
\end{align}
and the operation ends without performing the subsequent steps.

\BF{Step 4.}
Using Property~\ref{prop:oie.2and3},
get the $\CAL{I}$(3rd element) of $\oie_{result}$
\begin{align}
    \begin{aligned}
    \I_{\oie_{result}} & = \bigg\{\ \fBoundTwoT{\TwoTplT_i}\ \bigg|\ \TwoTplT_i \in \completeAscOrderFilterSubTwoTplTS\ \bigg\} \\
                       & = \bigg\{\ \fBoundTwoT{\TwoTplT_i}\ \bigg|\ \\
                       & \qquad\quad  \TwoTplT_i \in \fCompleteAscOrderFilterSubTwoTplTS{\fFeasibleIntvlTwoTplTS{\oieS}{\IdxT}}\ \bigg\}.
    \end{aligned}
\end{align}
At this juncture, the operation is completed.

Finally, the form of $\oie_{result}$ is
\begin{align}
    & \oie_{result} = \mathlarger{\mathlarger{(}}\ (),\ \emptyset,\ \emptyset,\ \emptyset\ \mathlarger{\mathlarger{)}},
\end{align}
or
\begin{align}
    \begin{aligned}
        & \oie_{result} = \\
        & \quad \mathlarger{\mathlarger{(}} \\
        & \qquad (\oie_{\idx_1},\ \oie_{\idx_2},\ \oie_{\idx_3},\ \cdots,\ \oie_{\idx_n}), \\
        & \qquad \fCompleteAscOrderFilterSubTwoTplTS{\fFeasibleIntvlTwoTplTS{\oieS}{\IdxT}}, \\
        & \qquad \bigg\{\ \fBoundTwoT{\TwoTplT_i}\ \bigg|\ \\
        & \qquad\qquad  \TwoTplT_i \in \fCompleteAscOrderFilterSubTwoTplTS{\fFeasibleIntvlTwoTplTS{\oieS}{\IdxT}}\ \bigg\}, \\
        & \qquad \bigcup_{i=1}^n \CAL{A}_{\oie_{\idx_i}}, \\
        & \quad \mathlarger{\mathlarger{)}}.
    \end{aligned}
\end{align}

The process described above,
including cases where the operation yields the $\bVoidError$ instance, is referred to as
``Complete Sequence Multiplication of members of a finite $\bOIES$ under the index order $\IdxT$,
denoted as
\begin{align}
    \overset{n}{\underset{\substack{i=1 \\ {\otimes}}}{\prod}} \oie_{\idx_i}
\end{align}
or
\begin{align}
    {\otimes}(\oie_{\idx_1},\ \oie_{\idx_2},\ \cdots,\ \oie_{\idx_n}).
\end{align}
\end{defi}

Similar to the addition operation (Definition~\ref{def:OPs.Add}),
the condition for successfully completing this type of operation
(where the result is not an $\bVoidError$ instance) is that
the specified restrictions must be met in Step 1, Step 2 and Step 3.

The significance of the restriction in Step 1 and Step 2 of $\otimes$
is the same as the Step 1 and Step 2 of ${\oplus|_{\alpha}^{\beta}}$.
The significance of the restriction in Step 3 is that after the screening,
there are valid optional intervals left.
Unlike ${\oplus|_{\alpha}^{\beta}}$,
for each element in the $\F$(2nd element) of $\oie_{result}$,
each interval follows the order of $\IdxT$,
and the starting timestamp of the subsequent interval
is greater than or equal to the ending timestamp of the previous interval segment.
That is to say, the characteristic of multiplication lies in sorting,
rather than emphasizing fairness within a domain like addition.

Intuitively, this is a kind of ``unfairness'' in terms of sequence.
Downhill skiing is an official sport in the Winter Olympics.
It requires skiers to compete alone on the track in a certain order,
and those with shorter times are ranked higher.
This sport is one of the few racing competitions where athletes have exclusive use of the track.
The order of who goes first and who goes later is not random;
it is determined by the preliminary round results.
Skiers with better preliminary round results compete later,
so before the finals, each skier already knows their starting position.

Let $n$ skiers participate in the final,
and $\oie_{skier_i}$ is the $\bOIE$ instance corresponding to the event in which
the $i$-th skier participates in the final.
Suppose the best skiing time of the $i$-th skier is $time^{min}_i(time^{min}_i > 0)$,
and the slowest skiing time is $time^{max}_i(0 < time^{min}_i \leq time^{max}_i)$.
The starting timestamp of the first skier in the final
(that is, the start time of the whole competition) is $\TS_{\alpha}$,
and the latest ending timestamp of the last skier
(that is, the time when the whole competition must end) is $\TS_{\beta}$.

Therefore, before the results of the preliminary round are announced,
when the starting order of the $i$-th skier in the final is uncertain,
the $\I$ of $\oie_{skier_i}$ is
\begin{align}
    \left \{
    (\TSs_i, \TSe_i) \, \middle| \,
    \begin{aligned}
        & \TSe_i - \TSs_i \geq time^{min}_i \\
        \land \ & \TSe_i - \TSs_i \leq time^{max}_i \\
        \land \ & \TSs_i \geq \TS_{\alpha}\\
        \land \ & \TSe_i \leq \TS_{\beta}\\
    \end{aligned}
    \right \}.
\end{align}

When the preliminary round is over,
the order of each contestant in the final is also determined.
Use
\begin{align}
    \idxT = (\idx_1, \idx_2, \idx_3, \cdots, \idx_n)
\end{align}
to represent the starting order.
Then, according to Definition~\ref{def:Ops.Multi}, this final can be expressed as
\begin{align}
    \oie_{final} = \overset{n}{\underset{\substack{i=1 \\ {\otimes}}}{\prod}} \oie_{skier_i}.
\end{align}
or
\begin{align}
    \oie_{final} = \otimes(\oie_{skier_1}, \oie_{skier_2}, \cdots, \oie_{skier_n})
\end{align}
The $\F$ of $\oie_{final}$ contains all possible combinations of the intervals of all skiers.
Once all combinations are mapped onto real events,
the skiers' starting sequence is uniquely fixed and determined by the $\otimes$ expression,
and only one combination ultimately occurs in reality.

\paragraph{Example: Three Doctors' Sequential Processing via $\otimes$.}

\indent
Consider the same three doctors,
but now the conference system mandates strictly sequential processing:
only one submission can be handled at a time,
and the next doctor cannot begin until the previous one finishes.
We apply $\otimes$ to
\begin{align}
    \oieS = \{\oie_{Dr_A}, \oie_{Dr_B}, \oie_{Dr_C}\}
\end{align}
with
\begin{align}
    \idxT=(1, 2, 3).
\end{align}

After validity checking confirms disjoint atomic event sets,
we initialize
\begin{align}
    \C_{result}=(\oie_{Dr_A}, \oie_{Dr_B}, \oie_{Dr_C}).
\end{align}
The Cartesian product of the $\I$ sets again yields 12 theoretical combinations,
but now the critical filtering occurs in Step 3:
the ascending ordered filter $\fPre\RM{AscOrderFlt2tupleTS}$
enforces strict temporal precedence.
For any tuple $((ts_1,te_1),(ts_2,te_2),(ts_3,te_3))$,
the constraint $te_1 \leq ts_2 \land te_2 \leq ts_3$ must hold.

Examining the feasible combinations,
$((0,1),(13,14),(19,22))$ satisfies the chain $1 \leq 13 \leq 14 \leq 19$,
whereas $((0,1),(0,1),(19,22))$ is eliminated due to overlapping starts.
The filtered set $\F_{result}$ contains strictly ordered schedules
such as $((0,1),(13,14),(19,22))$ and $((21,22),\ \cdot,\ \cdot)$.
Though notably, once the first doctor occupies $[21,22)$, no valid intervals remain for subsequent doctors within the 22-hour window, rendering those starting points infeasible for the chain.

Applying $\fPre\RM{Bound2tuple}$ to the valid sequential tuples yields $\I_{result} = \{(0,22), (13,22), (19,22)\}$ depending on the chosen starting intervals.
The final result is:
\begin{align}
    \oie_{result} = \big((\oie_{Dr_A}, \oie_{Dr_B}, \oie_{Dr_C}), \F_{result}, \I_{result}, \{\EStar_{Dr_A}, \EStar_{Dr_B}, \EStar_{Dr_C}\}\big).
\end{align}

This illustrates the defining characteristics of $\otimes$: \emph{strict sequentiality} (no temporal overlap permitted) and \emph{order sensitivity}.\fi
    \subsection{Void Result of Sequence Operations}\label{subsec:ops.voidRes}\EBL
    \ifincludeFile
According to Definition~\ref{def:OPs.Add} and Definition~\ref{def:Ops.Multi},
in ``$\oplus|_{\alpha}^{\beta}$'' and ``$\otimes$'',
if the specified conditions are not met at certain steps during operations,
the result is the only instance of $\bVoidError$, with the form
\begin{align}
    & \voidError = \mathlarger{\mathlarger{(}}\ (),\ \emptyset,\ \emptyset,\ \emptyset\ \mathlarger{\mathlarger{)}}
\end{align}
In this subsection, we will discuss several typical cases
where the result of sequence operations is this special instance.

\begin{pty}[Sequence operations involving identical $\bOIE$ instances result in an $\bVoidError$ instance]\label{pty:OPs.ErrorOIE.sameOIE}
\end{pty}

Property~\ref{pty:OPs.ErrorOIE.sameOIE} serves as a more intuitive scenario for the condition
\begin{align}
    & \bigcap_{i=1}^n \CAL{A}_{\oie_{\idx_i}} \neq \emptyset,
\end{align}
which triggers $\voidError$ result in Step 1 of both $\oplus|_{\alpha}^{\beta}$ and $\otimes$.
From a philosophical perspective,
an $\bE$ instance $\E$ has its own entity;
an entity cannot have an identical copy
existing in the same status within the same time interval.
For example, an athlete cannot simultaneously compete against himself in a sprint race
on two different tracks.
And the sentence ``$Dr_A$ and $Dr_A$ submitted a paper during a certain interval''
is also an obvious ungrammatical statement.
We believe that $\bOIE$ also has the similar property.

To extend Property~\ref{pty:OPs.ErrorOIE.sameOIE}:
Once an event instance becomes a component of two other complex event instances,
it is also irrational to repeatedly plan these two complex events.
The check for whether there is an intersection of the $\CAL{A}$ of $\bOIE$ instances
in Step1 of two sequence operations is a formalized expression of this line of thought.

\begin{pty}[The sequence operations involving $\bVoidError$ instances result $\voidError$]\label{pty:OPs.ErrorOIE.involved}
\end{pty}

When an $\bVoidError$ instance $\voidError$ participates in the two types of sequence operations proposed in this paper,
its $\I$(3rd item) will perform the Cartesian product operation with other $\bOIE$ instances' $\I$.
the result is still $\emptyset$.
Therefore, regardless of the type of sequential operators used,
the $\F$(2nd item) of result will be $\emptyset$,
which in turn causes the $\I$(3rd item) of the result to also be $\emptyset$.
Consequently, the final result will inevitably be an $\bVoidError$ instance.
\fi
    \subsection{Permutation Equivalence Relations for $\bOIE$ Instances and Results of Sequence Operations}\label{subsec:ops.permu}
    \ifincludeFile
To analyze the algebraic symmetry of sequence operations,
we examine the permutational equivalence relations governing both the $\bOIE$ instances themselves
and the results derived from applying $\oplus|_{\alpha}^{\beta}$ and $\otimes$.
This section establishes the formal criteria for structural equivalence under index permutation
and investigates how the choice of sequence operation determines the orbital properties of the result space.
The discussion is organized into three subsubsections:
\begin{itemize}
\item \BF{Subsubsection~\ref{subsubsec:ops.permu.oie}} formalizes the permutational equivalence relation between non-void $\bOIE$ instances,
defining when two instances $oie_1$ and $oie_2$
are structurally identical under a permutation matrix $\CAL{M}$
based on the invariance of their constituent elements $(\C,\F,\I,\A)$.
\item \textbf{Subsubsection~\ref{subsubsec:ops.permu.add}} demonstrates that $\oplus|_{\alpha}^{\beta}$ yields a \textit{single-orbit space}:
regardless of operand ordering, any two results $\oplus|_{\alpha}^{\beta}(\dots)$ obtained from different index tuples are permutationally equivalent,
reflecting the operation's intrinsic fairness and commutativity up to equivalence.
\item \textbf{Subsubsection~\ref{subsubsec:ops.permu.multi}} establishes the contrasting behavior of $\otimes$,
proving that it may generate \textit{multiple distinct orbits}:
different index orders may produce results that are not permutationally equivalent,
as the strict temporal ordering encoded in $\otimes$ is sensitive to the sequence of operands.
\end{itemize}

\fi
        \subsubsection{Permutation Equivalence Relation on $\bOIE$ Instances}\label{subsubsec:ops.permu.oie}
        \ifincludeFile
\begin{defi}[\textbf{Permutation equivalence relation of non-$\voidError$ $\bOIE$ instances}]\label{def:OPs.Permu}\EBL

For two non-$\voidError$ $\bOIE$ instances $\oie_1$ and $\oie_2$.

If there exists a permutation matrix $\CAL{M}$ satisfying the following conditions
\begin{align}
    & \C_{\oie_1} \cdot \CAL{M} = \C_{\oie_2},\
    \F_{\oie_1} \stackrel{\CAL{M}}{\cong} \F_{\oie_2},\
    \I_{\oie_1} = \I_{\oie_2},\
    \text{ and } \A_{\oie_1} = \A_{\oie_2},
\end{align}
then we say that ``$\oie_1$ and $\oie_2$ are permutationally equivalent
with respect to $\CAL{M}$'',
denoted as
\begin{align}
    \oie_1 \stackrel{\CAL{M}}{\sim} \oie_2.
\end{align}
\end{defi}

Here we need to explain the practical significance of Definition~\ref{def:OPs.Permu}.

\begin{align}
    \F_{\oie_1} \cong_{\CAL{M}} \F_{\oie_2}
\end{align}
$\stackrel{\CAL{M}}{\cong}$ means that $\F_{\oie_1}$ and $\F_{\oie_2}$ are isomorphic via the permutation matrix $\CAL{M}$,
i.e.
\begin{align}
    \begin{aligned}
    & \forall\ \TwoTplT_{i} \in \F_{\oie_1},\ \exists\ \TwoTplT_{j} \in \F_{\oie_2}:\ \TwoTplT_{i} \stackrel{\CAL{M}}{\sim} \TwoTplT_{j} \\
    \land \ & \F_{\oie_1} \xrightarrow{\sim} \F_{\oie_2}\ form\ a\ bijection.
        \end{aligned}
\end{align}

It is common in daily life to describe N events as follows.
\[
    \begin{aligned}
        & \qquad The\ first\ event\ is\ in\ a\ specific\ manner,\ the\ second\ event\ is\ in\ a\ specific\ manner... \\
        & The\ Nth\ event\ is\ in\ a\ specific\ manner.
    \end{aligned}
\]
In such a way,
there must be a first one that is presented initially, a second one that appears afterward, and a last one that appears at the end.
But in reality, the order in which these N events are presented
does not affect their inherent nature.

In the corresponding $\bOIES$, this means that the difference between two $\bOIE$ instances
satisfies the permutationally equivalent relation.
The two instances differ in expression form due to
the different presentation orders of their $\C$.
They can be transformed into each other's expression simply by performing a permutation operation.
For the real-world events mapped by each $\bOIE$ instance in the $\C$, the intervals and combinations
of their real-world events are not affected by the permutation operation.

If the $\C$(1st element) of an $\bOIE$ instance has no elements,
this $\bOIE$ instance is $\bAtomOIE$ type.
Consequently, it has a special permutationally equivalent relation,
we can summarize the following property.

\begin{pty}[\textbf{Permutation equivalence of $\bAtomOIE$ instances implies their equality}]\label{def:OPs.Permu.atom}\EBL

    Suppose there are two $\bAtomOIE$ instances $\atomOie_1$ and $\atomOie_2$.
    If they are permutationally equivalent with respect to $\CAL{M_I}$
    \begin{align}
        \atomOie_1 \stackrel{\CAL{M_I}}{\sim} \atomOie_2,
    \end{align}
    then
    \begin{align}
        \begin{aligned}
            & \atomOie_1 = \atomOie_2, \\
            \land\ & \CAL{M_I}\ is\ an\ one-dimensional\ identity\ permutation\ matrix.
        \end{aligned}
    \end{align}
\end{pty}

For $\bVoidError$, it only has one instance.
We argue that it is permutationally equivalent to itself.

\begin{pty}[\textbf{Permutation equivalence relation of $\bVoidError$ instance to itself}]\label{prop:OPs.Permu.void}\EBL

The unique global $\bVoidError$ instance $\voidError=((), \emptyset, \emptyset, \emptyset)$
is permutationally equivalent to itself via any permutation matrix,
and no non-void $\bOIE$ instance holds permutational equivalence with $\voidError$.
\end{pty}

Property~\ref{prop:OPs.Permu.void} clarifies the permutation invariance of the void $\bOIE$ instance,
filling the edge-case gap of the permutational equivalence definition for non-void $\bOIE$ instances.
It guarantees that void results from sequence operations maintain uniform equivalence characteristics across all operand permutations,
providing a consistent logical basis for the subsequent permutational equivalence analysis of the two core sequence operations.\fi
        \subsubsection{Permutation Equivalence Relation on Results of Complete Sequence Addition}\label{subsubsec:ops.permu.add}
        \ifincludeFile
In this subsubsection, we will prove that there exists a permutational equivalence relation
between the results of the same operands arranged in different orders under Complete Sequence Addition.

\begin{pty}[\BF{the permutation equivalent relation between the results of the same
operands arranged in different orders under Complete Sequence Addition.}]\label{pty:OPs.Add.Permu}\EBL

Suppose there is a non-empty finite $\bOIES$ instance
\begin{align}
    \oieS = \{\oie_1,\ \oie_2,\ \cdots,\ \oie_n\},
\end{align}
and two $\bIdxT$ instances with length $n$
\begin{align}
    & \idxT_1 = \{i_1,\ i_2,\ i_3,\ \cdots,\ i_n\}, \\
    & \idxT_2 = \{j_1,\ j_2,\ j_3,\ \cdots,\ j_n\},
\end{align}
and a $\bDomainFilterTwoTpl$ instance $(\alpha, \beta)$.

Perform two Complete Sequence Additions
\begin{align}
    & \oie_i = \oplus|_{\alpha}^{\beta}(\oie_{i_1}, \oie_{i_2}, \cdots, \oie_{i_n}), \\
    & \oie_j = \oplus|_{\alpha}^{\beta}(\oie_{j_1}, \oie_{j_2}, \cdots, \oie_{j_n}).
\end{align}
Then there exists a permutation matrix $\CAL{M}$
such that $\oie_i$ and $\oie_j$ are permutationally equivalent via $\CAL{M}$,
denoted as
    \begin{align}
        \oie_i \stackrel{\CAL{M}}{\sim} \oie_j.
    \end{align}
\end{pty}

\begin{proof}[\BF{Proof of Property~\ref{pty:OPs.Add.Permu}}]\label{proof:OPs.Add.Permu}\EBL

\indent \BF{(1) The situation $\bVoidError$ instance exists in $\oieS$ or the intersection of all $\A$ across all operands isn't $\emptyset$.}

Assume that the $\bOIES$ instance $\oieS$ contains an $\bVoidError$ instance
or the intersection of the $\A$(4th elements) of all operands is non-empty.
Then, by Step 1 of Definition~\ref{def:OPs.Add},
the result of the $\oplus|_{\alpha}^{\beta}$ for any index order is $\voidError$.
By Property~\ref{prop:OPs.Permu.void}, $\voidError$ is permutationally equivalent via any permutation matrix to itself.
Therefore, in this case, the two results of $\oplus|_{\alpha}^{\beta}$ are permutationally equivalent.

this concludes the proof for these cases.

\BF{(2) The situation $\bVoidError$ instance doesn't exist in $\oieS$ and the intersection of all $\A$ across all operands is $\emptyset$.}

Separately for each of the four elements of these two $\bOIE$ instances.

\TT{Step I. Proof for the $\C$(1st element) of $\oie_i$ and $\oie_j$}

In Step 1 of Definition~\ref{def:OPs.Add},
there are two $\C$(1st element) of $\oie_i$ and $\oie_j$ respectively.
According to Property~\ref{def:ops.idxT.Permu},
there must be a permutation matrix $\CAL{M}$ such that
\begin{align}
    \idxT_1 \stackrel{\CAL{M}}{\sim} \idxT_2,
\end{align}
and two $\oplus|_{\alpha}^{\beta}$ operations are exactly
the ones that apply the order of these two $\bIdxT$ instances directly,
according to Definition~\ref{def:OPs.Add},
\begin{align}
    & \C_{\oie_i} = (\oie_{i_1},\ \oie_{i_2},\ \oie_{i_3},\ \cdots,\ \oie_{i_n}), \\
    & \C_{\oie_j} = (\oie_{j_1},\ \oie_{j_2},\ \oie_{j_3},\ \cdots,\ \oie_{j_n}).
\end{align}
Then, the two $\C$(1st element) of $\oie_i$ and $\oie_j$ must be permutationally equivalent based on $\CAL{M}$ too,
i.e.
\begin{align}
    \C_{\oie_i} \cdot \CAL{M} = \C_{\oie_j}.
\end{align}
This equation satisfies the first condition in Definition~\ref{def:OPs.Permu}.
The proof for $\C$ of $\bOIE$ is complete.

\TT{Step II. Proof for the $\F$(2nd element) of $\oie_i$ and $\oie_j$}

In Step 2 of Definition~\ref{def:OPs.Add},
according to Property~\ref{prop:oieOP.fsble.isom},
\begin{align}
    & \fFeasibleIntvlTwoTplTS{\oieS}{\idxT_1} \stackrel{\CAL{M}}{\cong} \fFeasibleIntvlTwoTplTS{\oieS}{\idxT_2}.
\end{align}

If
\begin{align}
    \begin{aligned}
        & \fFeasibleIntvlTwoTplTS{\oieS}{\idxT_1} = \emptyset \\
        \land\ & \fFeasibleIntvlTwoTplTS{\oieS}{\idxT_2} = \emptyset,
    \end{aligned}
\end{align}
then according to Definition~\ref{def:OPs.Add},
\begin{align}
    \oie_i = \oie_j = \voidError.
\end{align}
According to Property~\ref{prop:OPs.Permu.void},
\begin{align}
    \oie_i \stackrel{\CAL{M}}{\sim} \oie_j.
\end{align}
this concludes the proof for this conditional branch.

Otherwise, neither is $\emptyset$.
The process goes to step 3,
and we use a $\bDomainFilterTwoTpl$ instance $(\alpha, \beta)$ to filter
the two isomorphic $\bFeasibleIntvlTwoTplTS$ instances.
Then we assign two results to $\F_{\oie_i}$ and $\F_{\oie_j}$:
\begin{align}
    & \F_{\oie_i} = \fDomainFilteredSubTwoTupleTS{\fFeasibleIntvlTwoTplTS{\oieS}{\idxT_1}}{\alpha}{\beta}, \\
    & \F_{\oie_j} = \fDomainFilteredSubTwoTupleTS{\fFeasibleIntvlTwoTplTS{\oieS}{\idxT_2}}{\alpha}{\beta}.
\end{align}

If
\begin{align}
    \F_{\oie_i} = \F_{\oie_j} = \emptyset,
\end{align}
then according to Definition~\ref{def:OPs.Add},
\begin{align}
    \oie_i = \oie_j = \voidError.
\end{align}
According to Property~\ref{prop:OPs.Permu.void},
\begin{align}
    \oie_i \stackrel{\CAL{M}}{\sim} \oie_j.
\end{align}
this concludes the proof for this conditional branch.

Otherwise, both $\F_{\oie_i}$ and $\F_{\oie_j}$ are not $\emptyset$,
according to the permutation isomorphism invariance in Definition~\ref{def:OPs.Add.DomainFilteredSubTwoTplTS},
it can be concluded that
\begin{align}
    \begin{aligned}
        & \fFeasibleIntvlTwoTplTS{\oieS}{\idxT_1} \stackrel{\CAL{M}}{\cong} \fFeasibleIntvlTwoTplTS{\oieS}{\idxT_2} \Rightarrow \\
        & \qquad \F_{\oie_i} \stackrel{\CAL{M}}{\cong} \F_{\oie_j}.
    \end{aligned}
\end{align}
The proof for $\F$(2nd element) of $\bOIE$ is complete.

\BF{Step III. Proof for the $\I$(3rd element) of $\oie_i$ and $\oie_j$}

The forms of $\I$(3rd element) of $\oie_i$ and $\oie_j$ are as following
\begin{align}
    & \I_{\oie_i} = \bigg\{\ \fBoundTwoT{\TwoTplT}\ \bigg|\ \TwoTplT \in \F_{\oie_i} \ \bigg\}, \\
    & \I_{\oie_j} = \bigg\{\ \fBoundTwoT{\TwoTplT}\ \bigg|\ \TwoTplT \in \F_{\oie_j} \ \bigg\}.
\end{align}

We have already proved that $\F_{\oie_i} $ and $\F_{\oie_j}$ are isomorphic via a permutation matrix $\CAL{M}$.
Thus, there exists a bijection preserving permutational equivalence between their members.
This means that
\begin{align}
    \begin{aligned}
        & \forall \TwoTplT_{\lambda} \in \F_{\oie_i},\ \exists \TwoTplT_{\mu} \in \F_{\oie_j}: \\
        & \qquad \TwoTplT_{\lambda} \stackrel{\CAL{M}}{\sim} \TwoTplT_{\mu},
    \end{aligned}
\end{align}
and this correspondence is bijective.

According to the permutation invariance in Definition~\ref{def:oie.BoundTwoTuple},
for every corresponding pair $\TwoTplT_{\lambda}$ and $\TwoTplT_{\mu}$ of $\F_{\oie_i}$ and $\F_{\oie_j}$,
\begin{align}
    \begin{aligned}
        & \TwoTplT_{\lambda} \stackrel{\CAL{M}}{\sim} \TwoTplT_{\mu} \Rightarrow \\
        & \qquad \fBoundTwoT{\TwoTplT_{\lambda}} = \fBoundTwoT{\TwoTplT_{\mu}}.
    \end{aligned}
\end{align}
Then we can conclude that
\begin{align}
    & \bigg\{\ \fBoundTwoT{\TwoTplT}\ \bigg|\ \TwoTplT \in \F_{\oie_i} \ \bigg\} = \\
    & \qquad \bigg\{\ \fBoundTwoT{\TwoTplT}\ \bigg|\ \TwoTplT \in \F_{\oie_j} \ \bigg\}, \notag
\end{align}
that is $\CAL{I}_{\oie_i} = \CAL{I}_{\oie_j}$.
The proof for $\I$(3rd element) of $\bOIE$ is complete.

\BF{Step IV. Proof for the $\A$(4th element) of $\oie_i$ and $\oie_j$}

In Step 1 of Definition~\ref{def:OPs.Add},
the $\A$(4th element) of the result $\bOIE$ instance is set to the union of all $\A$(4th elements) of the operands:
\begin{align}
    \A_{\oie_{result}} = \bigcup_{k=1}^n \A_{\oie_k}.
\end{align}
For the two Complete Sequence Addition operations with different index orders, we have:
\begin{align}
    \A_{\oie_{i}} = \bigcup_{k=1}^n \A_{\oie_{i_k}},\ \A_{\oie_{j}} = \bigcup_{k=1}^n \A_{\oie_{j_k}}.
\end{align}

For finite $\bOIES$ instance
\begin{align}
    \oieS = \{\ \oie_1,\ \oie_2,\ \cdots,\ \oie_n\ \},
\end{align}
since $\idxT_1 = \{ i_1, i_2, \cdots, i_n \}$ and $\idxT_2 = \{ j_1, j_2, \cdots, j_n \}$ are permutationally equivalent, \\
$\{ \oie_{i_1}, \oie_{i_2}, \cdots, \oie_{i_n} \}$ and $\{ \oie_{j_1}, \oie_{j_2}, \cdots, \oie_{j_n} \}$
are permutationally equivalent, as they both contain exactly the same elements from $\oieS$ (just may be in different orders).
We have:
\begin{align}
    \A_{\oie_{i}} = \bigcup_{k=1}^n \A_{\oie_{i_k}},\ \A_{\oie_{j}} = \bigcup_{k=1}^n \A_{\oie_{j_k}}.
\end{align}

Since set union is a commutative and idempotent operation that is independent of the order of its operands, we have:
\begin{align}
    \bigcup_{k=1}^n \A_{\oie_{i_k}} = \bigcup_{k=1}^n \A_{\oie_{j_k}} = \bigcup_{k=1}^n \A_{\oie_{k}}.
\end{align}
Therefore, we can conclude that
\begin{align}
    \A_{\oie_{i}} = \A_{\oie_{j}}.
\end{align}
The proof for $\A$(4th element) of $\bOIE$ is complete.

Now, the proof for case (2) is complete.

We have completed the proofs for cases (1) and (2) above.
Since (1) and (2) cover all possible cases, we have completed the proof.

\end{proof}
\fi
        \subsubsection{Permutation Equivalence relation on Results of Complete Sequence Multiplication}\label{subsubsec:ops.permu.multi}
        \ifincludeFile
In this subsubsection, we will prove that there need not exist a permutational equivalence relation
between the results of the same operands arranged in different orders under Complete Sequence Multiplication.

\begin{pty}[\BF{the permutation equivalence relation between the results of the same
operands arranged in different orders under Complete Sequence Multiplication.}]\label{pty:OPs.Multi.Permu}\EBL

Suppose there is a non-empty finite $\bOIES$ instance
\begin{align}
    \oieS = \{\oie_1,\ \oie_2,\ \cdots,\ \oie_n\},
\end{align}
and two $\bIdxT$ instances with length $n$
\begin{align}
    & \idxT_1 = (i_1,\ i_2,\ i_3,\ \cdots,\ i_n), \\
    & \idxT_2 = (j_1,\ j_2,\ j_3,\ \cdots,\ j_n).
\end{align}
Perform two $\otimes$ operations
\begin{align}
    & \oie_i = \otimes(\oie_{i_1}, \oie_{i_2}, \cdots, \oie_{i_n}), \\
    & \oie_j = \otimes(\oie_{j_1}, \oie_{j_2}, \cdots, \oie_{j_n}).
\end{align}
Then it is not always true that $\oie_i$ and $\oie_j$ are permutationally equivalent.
\end{pty}

\begin{proof}[\BF{Proof of Property~\ref{pty:OPs.Multi.Permu}}]\label{proof:OPs.Multi.Permu}\EBL

We proceed by considering two cases.

\BF{(1) The situation where $\bVoidError$ instance exists in $\oieS$ or the intersection of all $\A$ across all operands isn't $\emptyset$.}

Assume that the $\bOIES$ instance $\oieS$ contains an $\bVoidError$ instance
or the intersection of the $\A$(4th elements) of all operands is non-empty.
Then, by Step 1 of Definition~\ref{def:Ops.Multi},
the result of the Complete Sequence Multiplication for any index order
is set to $\voidError$.
By Property~\ref{prop:OPs.Permu.void}, $\voidError$ is permutationally equivalent via any permutation matrix to itself.
Therefore, in this case, the two results of Complete Sequence Multiplication are permutationally equivalent.

This proves that the two results of Complete Sequence Multiplication
are permutationally equivalent under these two conditions.

\BF{(2) The situation where $\voidError$ doesn't exist in $\oieS$ and the intersection of all $\A$ across all operands is $\emptyset$.}

In Step 2, according to Property~\ref{prop:oieOP.fsble.isom}
\begin{align}
    & \fFeasibleIntvlTwoTplTS{\oieS}{\idxT_1} \stackrel{\CAL{M}}{\cong} \fFeasibleIntvlTwoTplTS{\oieS}{\idxT_2}.
\end{align}

If
\begin{align}
    \begin{aligned}
        & \fFeasibleIntvlTwoTplTS{\oieS}{\idxT_1} = \emptyset \\
        \land\ & \fFeasibleIntvlTwoTplTS{\oieS}{\idxT_2} = \emptyset,
    \end{aligned}
\end{align}
then according to Definition~\ref{def:Ops.Multi},
\begin{align}
    \oie_i = \oie_j = \voidError,
\end{align}
according to Property~\ref{prop:OPs.Permu.void},
\begin{align}
    \oie_i \stackrel{\CAL{M}}{\sim} \oie_j.
\end{align}
this concludes the proof for this conditional branch.

Otherwise, neither is $\emptyset$.
The process goes to step 3,

We obtain two $\bCompleteAscOrderFilterSubTwoTplTS$ instances of these two $\bFeasibleIntvlTwoTplTS$ instances,
and then assign them to
$\F_{\oie_i}$ and $\F_{\oie_j}$
\begin{align}
    & \F_{\oie_i} = \fCompleteAscOrderFilterSubTwoTplTS{\fFeasibleIntvlTwoTplTS{\oieS}{\idxT_1}}, \\
    & \F_{\oie_j} = \fCompleteAscOrderFilterSubTwoTplTS{\fFeasibleIntvlTwoTplTS{\oieS}{\idxT_2}}.
\end{align}

According to Definition~\ref{def:OPs.Multi.CompleteAscOrderFilteredSubTwoTplTS},
the permutation isomorphism between $\fFeasibleIntvlTwoTplTS{\oieS}{\idxT_1}$ and $\fFeasibleIntvlTwoTplTS{\oieS}{\idxT_2}$
does not guarantee that $\F_{\oie_i}$ and $\F_{\oie_j}$ are permutationally equivalent.
Thus, in situation (2), permutational equivalence need not hold.

Combining the conclusions of (1) and (2), we complete the proof.

\end{proof}
\fi
    \subsection{Natural Complete Sequence Addition}\label{subsec:OPs.add.nat}\EBL
    \ifincludeFile
\indent
Complete Sequence Addition supports modeling concurrent events within a time domain.
In cases where all participating $\bOIE$ instances share the same largest common time domain,
we can simplify the notation of Complete Sequence Addition without altering the operational logic
by using that largest domain.

\begin{defi}[\BF{Natural Complete Sequence Addition of members of a finite non-empty $\bOIES$ instance excluding $\emptyset$ under an index order}]\label{def:OPs.Add.Nat}\EBL

Suppose there is a finite non-empty $\bOIES$ instance excluding $\emptyset$, denoted as $\oieS$,
containing $n(n \in \mathbb{N}^+, n > 1)$ elements in the ordered enumeration form
\begin{align}
    \oieS = \{\oie_1, \oie_2, \cdots, \oie_n\},
\end{align}
and an $\bIdxT$ instance with length $n$
\begin{align}
    \idxT = ( \idx_1, \idx_2, \cdots, \idx_n ).
\end{align}

For any $\bOIE$ instance $\oie_\lambda$, let
\begin{align}
    \pi_1^{\min}(\oie_\lambda) = \min\left\{ \pi_1(2tuple_{\theta}) \mid 2tuple_{\theta} \in \CAL{I}_{\oie_\lambda} \right\},
\end{align}
where $\pi_1(2tuple_{\theta})$ denotes the first projection(get the 1st element) of $2tuple_{\theta}$, and
\begin{align}
    \pi_2^{\max}(\oie_\lambda) = \max\left\{ \pi_2(2tuple_{\theta}) \mid 2tuple_{\theta} \in \CAL{I}_{\oie_\lambda} \right\}.
\end{align}
where $\pi_2(2tuple_{\theta})$ denotes the second projection(get the 2nd element) of $2tuple_{\theta}$.

If
\begin{align}
    \begin{aligned}
        & \exists (A, B), \forall i \in [1, n]: \\
        & \qquad \oie_i \ne \voidError\ \land\ \pi_1^{\min}(\oie_i) = A\ \land\ \pi_2^{\max}(\oie_i) = B,
    \end{aligned}
\end{align}
then by Definition~\ref{def:OPs.Add}, the following Complete Sequence Addition
\begin{align}
    {\oplus|_{A}^{B}}(\oie_{\idx_1},\ \oie_{\idx_2},\ \cdots,\ \oie_{\idx_n})
\end{align}
is called ``Natural Complete Sequence Addition'', denoted as ``$\oplus$'', and the expression can be simplified to
\begin{align}
    \oplus(\oie_{\idx_1},\ \oie_{\idx_2},\ \cdots,\ \oie_{\idx_n}).
\end{align}
\end{defi}

Natural Complete Sequence Addition is a special case of $\oplus|_{\alpha}^{\beta}$
that eliminates subjective intervention:
It directly adopts the shared natural timestamp boundaries of all participating $\bOIE$ instances as the filtering domain,
avoiding the subjectivity introduced by manual domain specification.

The $\oplus$ fully inherits the core semantics of $\oplus|_{\alpha}^{\beta}$ and
is still applied to model concurrent events with equal opportunity within a shared time domain.
It only imposes a naturalized and standardized qualification on the filtering domain
without altering the core rules of the operation.
Meanwhile, it provides a concise form for the standardized modeling of real-world scenarios,
which fits well with concurrent scenarios that inherently share a unified time window, such as the 100 metres race.
For such scenarios, event modeling can be completed directly through the simplified expression
without additional specification of the filtering domain.

From a formal perspective, the removal of the filtering domain in $\oplus$
results in a notational form identical to that of $\times$.
This unification endows the two sequence operations with a high degree of mathematical elegance.
\fi

\section{Algebraic Properties of Sequence Operations}\label{sec:algbrProp}\EBL
\ifincludeFile
Most existing studies investigating algebraic properties of event ordering adopt lattice-theoretic methods
to describe sequentiality of events through the analysis of sequential relation~\cite{birkhoff1940lattice}.
In contrast, this paper establishes its theoretical framework directly based on group theory.
This selection stems from the intrinsic symmetry of algebraic structures built
via $\bOIE$ and sequence operations, such as the permutational equivalence elaborated in Subsection~\ref{subsec:ops.permu},
along with the symmetry inherent in practical application scenarios.

\begin{itemize}
\item \textbf{Section~\ref{subsec:algbrProp.closure}:}
We prove that both Complete Sequence Addition and Complete Sequence Multiplication
satisfy the closure property:
the result of any operation, including edge cases involving $\voidError$,
is always an $\bOIE$ instance.
\item \textbf{Section~\ref{subsec:algbrProp.commutative}:}
We establish the general non-commutativity of both operations,
with a strict distinction between set-theoretic equality (Definition~\ref{defi:OPs.oie.Equal})
and permutational equivalence (Definition~\ref{def:OPs.Permu},
Property~\ref{prop:OPs.Permu.void}) for $\bOIE$ instances.
\item \textbf{Section~\ref{subsec:algbrProp.orbitSpace}:}
We analyze the orbit spaces of the two sequence operations:
Complete Sequence Addition yields a single-orbit space via permutational equivalence,
while Complete Sequence Multiplication generates multiple distinct orbits
under strict temporal ordering constraints.
\end{itemize}
\fi
    \subsection{Closure}\label{subsec:algbrProp.closure}\EBL
    \ifincludeFile
In this subsection, we discuss the closure properties of the two sequence operations.

\begin{pty}[\BF{Complete Sequence Addition and Complete Sequence Multiplication both satisfy closure}]\label{prop:closedness}
\end{pty}

\begin{proof}[\BF{Proof of Property~\ref{prop:closedness}}]\EBL

According to Definition~\ref{def:OPs.Add} and Definition~\ref{def:Ops.Multi},
both Complete Sequence Addition and Complete Sequence Multiplication return an $\bOIE$ instance as their result.
The result is either the unique $\bVoidError$ instance $\voidError$ with the canonical form
\begin{align}
    \left(\ (),\ \emptyset,\ \emptyset,\ \emptyset\ \right),
\end{align}
or a non-$\voidError$ instance with the form
\begin{align}
    \left( \C, \F, \I, \A \right)
\end{align}
in which
\begin{align}
    \Dim{\C} \geq 0 \land \CARDI{\F} > 0 \land \CARDI{\I} > 0 \land \CARDI{\A} > 0.
\end{align}
Since $\voidError$ is also an instance of $\bOIE$,
it follows directly from these two definitions that they satisfy the closure property.
\end{proof}
\fi
    \subsection{Commutativity}\label{subsec:algbrProp.commutative}\EBL
    \ifincludeFile
In this subsection, we discuss the commutativity of the two sequence operations.

\begin{pty}[\BF{Complete Sequence Addition and Complete Sequence Multiplication are non-commutative in general}]\label{prop:OPs.Commutativity}
\end{pty}

\begin{proof}[\BF{Proof of Property~\ref{prop:OPs.Commutativity}}]\EBL
Suppose there exists an $\bOIES$ instance $\oieS$ with cardinality $n > 0$
\begin{align}
    \oieS = \{\ \oie_1,\ \oie_2,\ \oie_3, \cdots, \oie_n\},
\end{align}
and two $\bIdxT$ instances with length $n$
\begin{align}
    & \idxT_i = (\ i_1,\ i_2,\ i_3,\ \cdots,\ i_n\ ), \\
    & \idxT_j = (\ j_1,\ j_2,\ j_3,\ \cdots,\ j_n\ )
\end{align}
satisfying
\begin{align}
    \begin{aligned}
        & \idxT_i \ne \idxT_j \\
        \land\ & \exists\ a\ permutation\ matrix\ \CAL{M}: \idxT_i \stackrel{\CAL{M}}{\sim} \idxT_j,
    \end{aligned}
\end{align}
and a $\bDomainFilterTwoTpl$ instance $(\alpha, \beta)$.

\BF{(1) Complete Sequence Addition}

Let
\begin{align}
    & \oie_i =  {\oplus|_{\alpha}^{\beta}}(\oie_{i_1},\ \oie_{i_2},\ \cdots,\ \oie_{i_n}), \\
    & \oie_j =  {\oplus|_{\alpha}^{\beta}}(\oie_{j_1},\ \oie_{j_2},\ \cdots,\ \oie_{j_n}).
\end{align}
    According to Property~\ref{pty:OPs.Add.Permu}, there exists a permutation matrix $\CAL{M}$
    \begin{align}
        \oie_i \stackrel{\CAL{M}}{\sim} \oie_j.
    \end{align}

    If $\oie_i = \voidError$, then $\oie_j = \voidError$ as well.
    Thus, in this condition
    \begin{align}
        \oie_i = \oie_j = \voidError.
    \end{align}
Otherwise, if $\oie_i \ne \voidError$, then $\oie_j \ne \voidError$ as well.
We have imposed the condition that $\idxT_i \ne \idxT_j$,
from which it follows that $\C_{oie_i} \neq \C_{oie_j}$.
According to Definition~\ref{defi:OPs.oie.Equal}, we can conclude that
    \begin{align}
        \oie_i \ne \oie_j.
    \end{align}
    Thus, we can conclude that Complete Sequence Addition is non-commutative in general.

\BF{(2) Complete Sequence Multiplication}

Let
    \begin{align}
        & \oie_i = {\otimes}(\oie_{i_1},\ \oie_{i_2},\ \cdots,\ \oie_{i_n}), \\
        & \oie_j = {\otimes}(\oie_{j_1},\ \oie_{j_2},\ \cdots,\ \oie_{j_n}).
    \end{align}
According to Property~\ref{pty:OPs.Multi.Permu},
    \begin{align}
        \idxT_i \stackrel{\CAL{M}}{\sim} \idxT_j
    \end{align}
can not guarantee that $\oie_i$ and $\oie_j$ are permutationally equivalent,
it follows that the equality of the two sequence operation results cannot be guaranteed either.
Thus, we can conclude that Complete Sequence Multiplication is non-commutative in general.

Combining the conclusions of (1) and (2), we complete the proof.

\end{proof}
\fi
    \subsection{Orbit Space of Complete Sequence Addition and Complete Sequence Multiplication}\label{subsec:algbrProp.orbitSpace}\EBL
    \ifincludeFile
In this subsection, we will discuss the orbit space of
Complete Sequence Addition and Complete Sequence Multiplication.

\fi
        \subsubsection{Orbit Space of Complete Sequence Addition}\label{subsubsec:algbrProp.orbitSpace.add}\EBL
        \ifincludeFile

\begin{defi}[\BF{Orbit space of Complete Sequence Addition}]\label{defi:OPS.orbitSpace.add}\EBL

Let $\oieS$ be a finite $\bOIES$ instance with cardinality $0 < n < \infty$
\begin{align}
    \oieS = \{\ \oie_1,\ \oie_2,\ \cdots, \oie_n \}.
\end{align}
Let $\idxT\CAL{S}$ be the set of all $\bIdxT$ instances of length $n$
\begin{align}
    \idxT\CAL{S} = \left\{ (\ \idx_1,\ \idx_2,\ \cdots,\ \idx_n\ ) \Big| \{\idx_1,\ \idx_2,\ \cdots,\ \idx_n\} = \{1,\ 2,\ \cdots,\ n\} \right\}.
\end{align}
Let $(\alpha, \beta)$ be a domain-filtering 2-tuple, and let ${\oplus|_{\alpha}^{\beta}}$ denote Complete Sequence Addition.
Let $\stackrel{\CAL{M}}{\sim}$ denote the permutational equivalence relation between two $\bOIE$ instances (Subsubsection~\ref{subsubsec:ops.permu.oie}).

The orbit space of Complete Sequence Addition for $\oieS$, denoted as $\CAL{O}(oieS, {\oplus|_{\alpha}^{\beta}})$,
is the set of equivalence classes of operation results under permutational equivalence, with the form
\begin{align}
    & \OrbitAdd = \left\{ [\oie_{res}]_{\sim} \Big| \exists \idxT \in \idxT\CAL{S}: \oie_{res} = {\oplus|_{\alpha}^{\beta}}(\oie_{\idx_1},\ \oie_{\idx_2},\ \cdots,\ \oie_{\idx_n})\right\}
\end{align}
where $[\oie]_{\sim}$ denotes the equivalence class of $\bOIE$ instances under permutational equivalence.
\end{defi}

\begin{pty}[\BF{Orbit space of Complete Sequence Addition has only one element}]\label{pty:OPS.orbitSpace.add.oneElem}
\end{pty}

\begin{proof}[\BF{Proof of Property~\ref{pty:OPS.orbitSpace.add.oneElem}}]\EBL

Let $\oieS$ be a finite $\bOIES$ instance with n elements $(0 < n < \infty)$
and let $(\alpha, \beta)$ be a $\bDomainFilterTwoTpl$ instance.

Denote by $\idx\CAL{T}\CAL{S}$ the set of all $\bIdxT$ instances of length n
that are permutations of $\{1, 2, 3, \cdots, n\}$
\begin{align}
    \begin{aligned}
        & \forall \idxT_1, \idxT_2(\idxT_1 \ne \idxT_2) \ \in \idx\CAL{T}\CAL{S}, \\
        & \quad \exists\ a\ permutation\ matrix\ \CAL{M}: \\
        & \qquad\qquad \idxT_1 \stackrel{\CAL{M}}{\sim} \idxT_2.
    \end{aligned}
\end{align}
Let $\idxT_1$ and $\idxT_2$ be
\begin{align}
    & \idxT_1 = (i_1, i_2, \cdots, i_n), \\
    & \idxT_2 = (j_1, j_2, \cdots, j_n).
\end{align}
And let
\begin{align}
    & \oie_{res_1} = {\oplus|_{\alpha}^{\beta}}(\oie_{i_1},\ \oie_{i_2},\ \cdots,\ \oie_{i_n}), \\
    & \oie_{res_2} = {\oplus|_{\alpha}^{\beta}}(\oie_{j_1},\ \oie_{j_2},\ \cdots,\ \oie_{j_n}).
\end{align}
Then, Property~\ref{pty:OPs.Add.Permu} guarantees that
\begin{align}
    \oie_{res_1} \stackrel{\CAL{M}}{\sim} \oie_{res_2}
\end{align}
If for some $\bIdxT$ instance, the result is $\voidError$,
then by Property~\ref{pty:OPs.Add.Permu},
the result for every $\bIdxT$ instance is $\voidError$.
If for some $\bIdxT$ instance, the result is a non‑$\voidError$ $\bOIE$ instance,
then by Property~\ref{pty:OPs.Add.Permu} every other result is permutationally equivalent to it,
hence again all results lie in a single equivalence class.

this concludes the proof.
\end{proof}
\fi
        \subsubsection{Orbit Space of Complete Sequence Multiplication}\label{subsubsec:algbrProp.orbitSpace.multi}\EBL
        \ifincludeFile
\begin{defi}[\BF{Orbit space of Complete Sequence Multiplication}]\label{defi:OPS.orbitSpace.mul}\EBL

Let $\oieS$ be a finite $\bOIES$ instance with cardinality $0 < n < \infty$
\begin{align}
    \oieS = \{\ \oie_1,\ \oie_2,\ \cdots, \oie_n\ \}.
\end{align}
Let $\idxT\CAL{S}$ be the set of all $\bIdxT$ instances of length $n$
    \begin{align}
        \idxT\CAL{S} = \left\{ (\ \idx_1,\ \idx_2,\ \cdots,\ \idx_n\ ) \Big| \{\idx_1,\ \idx_2,\ \cdots,\ \idx_n\} = \{1,\ 2,\ \cdots,\ n\} \right\}.
    \end{align}
Let ${\otimes}$ denote the Complete Sequence Multiplication operation (Definition~\ref{def:Ops.Multi}).
Let $\stackrel{\CAL{M}}{\sim}$ denote the permutational equivalence relation between two $\bOIE$ instances (Subsubsection~\ref{subsubsec:ops.permu.oie}).

The orbit space of Complete Sequence Multiplication for $\oieS$, denoted as $\CAL{O}(oieS, {\otimes})$ ,
is the set of equivalence classes of operation results under permutational equivalence, with the form
    \begin{align}
        & \OrbitMulti = \left\{ [\oie_{res}]_{\sim} \Big| \exists \idxT \in \idxT\CAL{S}: \oie_{res} = {\otimes}(\oie_{\idx_1},\ \oie_{\idx_2},\ \cdots,\ \oie_{\idx_n})\right\},
    \end{align}
where $[\oie]_{\sim}$ denotes the equivalence class of $\bOIE$ instances under permutational equivalence.
\end{defi}

\begin{pty}[\BF{Orbit space of Complete Sequence Multiplication has at least one element}]\label{pty:OPS.orbitSpace.multi.multiElem}
\end{pty}

\begin{proof}[\BF{Proof of Property~\ref{pty:OPS.orbitSpace.multi.multiElem}}]\EBL

Let $\oieS$ be a finite $\bOIES $instance with n elements $(0 < n < \infty)$.

Denote by $\idx\CAL{T}\CAL{S}$ the set of all $\bIdxT$ instances of length n
that are permutations of $\{1, 2, \cdots, n\}$
\begin{align}
    \begin{aligned}
        & \forall \idxT_1, \idxT_2(\idxT_1 \ne \idxT_2) \in \idx\CAL{T}\CAL{S}, \\
        & \qquad \exists\ a\ permutation\ matrix\ \CAL{M}: \idxT_1 \stackrel{\CAL{M}}{\sim} \idxT_2.
    \end{aligned}
\end{align}
Let $\idxT_1$ and $\idxT_2$ be
    \begin{align}
        & \idxT_1 = (i_1, i_2, \cdots, i_n), \\
        & \idxT_2 = (j_1, j_2, \cdots, j_n).
    \end{align}
And let
    \begin{align}
        & \oie_{res_1} = {\otimes}(\oie_{i_1},\ \oie_{i_2},\ \cdots,\ \oie_{i_n}), \\
        & \oie_{res_2} = {\otimes}(\oie_{j_1},\ \oie_{j_2},\ \cdots,\ \oie_{j_n}).
    \end{align}
Then, Property~\ref{pty:OPs.Multi.Permu} guarantees that
\begin{align}
    \oie_{res_1}\ and\ \oie_{res_2}\ need\ not\ be\ permutationally\ equivalent.
\end{align}
If for all $\bIdxT$ instances, the result is $\voidError$
, then all results lie in a single equivalence class.
Thus, the orbit space has only one element.
If for more than one $\bIdxT$ instance, their results are non‑$\voidError$ $\bOIE$ instances,
then by Property~\ref{pty:OPs.Multi.Permu}, they are not permutationally equivalent,
hence all results lie in more than one equivalence class.

this concludes the proof.
\end{proof}
\fi

\section{The Implementation of Optional Intervals Event}\label{sec:impl}\EBL
\ifincludeFile
In computer science,
formal models for concurrent system design, real-time task scheduling and temporal logic reasoning have traditionally adopted a global clock.
However, the inherent uncertainty of timing devices fundamentally undermines its reliability.
Even the most advanced atomic clocks cannot eliminate inherent measurement errors,
leading to an unbridgeable discrepancy between observed temporal relationships and objective physical realities.
This section presents a formal definition for the executable implementation of $\bOIE$ instances.
We further construct a projection mechanism that maps the results of sequential operations to general finite n-ary operations,
and prove that the resulting operation set is unique within each equivalence class of the orbit space.

This section is organized as follows:
\begin{itemize}
    \item Section~\ref{subsec:impl.oie}
    formalizes two canonical implementation modes for $\bOIE$ instances:
    Type-1 implementation for atomic events (selecting a single interval from $\I$),
    and Type-2 implementation for composite events
    (selecting a detailed schedule tuple from $\F$ to determine sub-event intervals).
    \item Section~\ref{subsec:impl.res2op}
    establishes the projection mechanism that maps sequence operation results to sets of general $n$-ary finitary operations,
    defining how interval-based $\bOIE$ structures translate into executable operation sequences based on ending timestamp groupings.
    \item Section~\ref{subsec:impl.restriction}
    proves the uniqueness properties of operation sets within orbit space equivalence classes,
    establishing the algebraic constraints that ensure deterministic projection outcomes
    for both $\oplus|_{\alpha}^{\beta}$ (permutation-invariant) and $\otimes$ (strictly ordered).
\end{itemize}
\fi
    \subsection{The Implementation of $\bOIE$ Instances}\label{subsec:impl.oie}\EBL
    \ifincludeFile
In contrast to the concreteness of $\bE$ instances,
a non-$\voidError$ $\bOIE$ instance acts as an abstract carrier
that stores the set of feasible execution intervals
for its corresponding $\bE$ instance.
The primary function of an $\bOIE$ instance is to
constrain the execution interval of its mapped $\bE$ instance.
This subsection formalizes two types of canonical implementations for $\bOIE$ instances,
distinguished by their granularity of temporal planning.

For the first type,
we use the $\I$ of an $\bOIE$ instance to determine the interval of the $\bE$ instance it maps to.

\begin{defi}[\BF{The 1st type of implementation of an $\bOIE$ instance}]\label{def:implementAndProps.implement1}\EBL

Given a non-$\voidError$ $\bOIE$ instance $\oie$
and an $\bE$ instance $\E$, such that
\begin{align}
    \oie \rightarrow \E.
\end{align}
Select an interval 2-tuple from $\I_{\oie}$(3rd element of $\oie$)
as the execution interval of $\E$.
This is called ``The first type of implementation of an $\bOIE$ instance''.
\end{defi}

The 1st type implementation only specifies the global start and end timestamps of the target $\bE$ instance,
without encoding any temporal details of internal sub-processes.
In order to further determine the specific details
within the execution interval of $\E$,
we propose the second type of implementation for the $\bOIE$ instance.

\begin{defi}[\BF{The 2nd type of implementation of an $\bOIE$ instance}]\label{def:implementAndProps.implement2}\EBL

Given a non-empty $\bOIES$ instance excluding $\voidError$ $\oieS$ in the form of
\begin{align}
    \oieS = \{\ \oie_1,\ \oie_2,\ \oie_3,\ \cdots,\ \oie_n\ \},
\end{align}
and an $\bIdxT$ instance
\begin{align}
    \IdxT = \{ i_1, i_2, i_3, \cdots, i_n \},
\end{align}
and a non-$\voidError$ $\bOIE$ instance $\oie_{res}$,
which is derived by a sequence operation ``$\boldsymbol{*}$'':
    \begin{align}
        & \oie_{res} = \boldsymbol{*}(\oie_{i_1}, \oie_{i_2}, \oie_{i_3}, \cdots, \oie_{i_n}),
    \end{align}
and an $\bE$ instance $\E$ mapped by $\oie_{res}$
\begin{align}
    \oie_{res} \rightarrow \E
\end{align}
which consists of $\E_1$, $\E_2$, $\cdots$, $\E_n$ and satisfies
\begin{align}
    \begin{aligned}
        & \oie_1 \rightarrow \E_1, \\
        & \oie_2 \rightarrow \E_2, \\
        & \cdots, \\
        & \oie_n \rightarrow \E_n.
    \end{aligned}
\end{align}
Select a $\bTwoTplT$ instance $\TwoTplT$
\begin{align}
    & \TwoTplT = (\ (x_{i_1}, y_{i_1}),\ (x_{i_2}, y_{i_2}),\ \cdots,\ (x_{i_n}, y_{i_n})\ )
\end{align}
from $\F_{\oie_{res}}$(2nd element of $\oie_{res}$).
Assign the interval $[x_{i_k}, y_{i_k})$ as the execution interval of the sub-event $\E_{i_k}$ mapped to $\oie_{i_k}$ for all $k \in [1, n]$.
This is called ``The second type of implementation of an $\bOIE$ instance''.
\end{defi}

The type-2 implementation schedules both the global time window
of the composite event and the execution intervals of all its sub-events,
enabling fine-grained temporal modeling.

For example, for a composed event,
the overall time is from 1:00 PM to 5:00 PM on a certain day.
It is composed of three sub-events, and their respective intervals are
\begin{align}
    \begin{aligned}
        & [ 1:00 \text{ PM}, 2:00 \text{ PM} ), \\
        & [ 2:00 \text{ PM}, 3:00 \text{ PM} ), \\
        & [ 4:00 \text{ PM}, 5:00 \text{ PM} ).
    \end{aligned}
\end{align}
It can be observed that the second type of implementation (Definition~\ref{def:implementAndProps.implement2})
can express that there is no sub-event execution in the interval $[ 3:00 \text{ PM}, 4:00 \text{ PM} )$.
If the first type of implementation (Definition~\ref{def:implementAndProps.implement1}) is used,
this information cannot be expressed.

Overall, The type-1 implementation is tailored for $\bAtomOIE$ instances (mapping to atomic, indivisible events),
while the type-2 implementation is designed for $\bCombOIE$ instances (mapping to composite events with hierarchical sub-structures).
\fi
    \subsection{Projection from Sequence Operations to General $n$-ary Finitary Operations}\label{subsec:impl.res2op}\EBL
    \ifincludeFile
Building on subsection~\ref{subsec:impl.oie},
this subsection establishes a projection that maps $\bOIE$ interval structures to $n$-ary finitary operations,
translating temporal constraints into executable orderings.

\begin{defi}[\textbf{Projection based on ending timestamps of the $\F$ of an $\bOIE$ instance obtained by a certain sequence operation to a set of $n$-ary finitary operations}]
    \label{def:implementAndProps.endTsProjection}\EBL

We formalize the projection from a sequence operation result ($\bOIE$ instance) to a set of $n$-ary finitary operations, structured as follows:

\noindent \textbf{I. Preconditions.}
Let the following objects be given (all notation consistent with Definitions~\ref{def:oie},~\ref{def:OPs.Add},~\ref{def:Ops.Multi}): 
\begin{enumerate}
    \item A finite set of \emph{operand variables}: $\OperandS = \{ opd_1, opd_2, \cdots, opd_n \}$ with $n \in \mathbb{N}^+, n > 1$;
    \item A finite set of \emph{events}: $\ES = \{ \E_1, \E_2, \cdots, \E_n \}$ with a fixed bijection $\ES \to \OperandS$ (each event maps to a unique operand bijectively);
    \item A finite $\bOIES$ instance: $\oieS = \{ \oie_1, \oie_2, \cdots, \oie_n \}$ with a fixed mapping $\oieS \to \ES$ (each $\bOIE$ instance maps to a unique event);
    \item An \emph{index tuple}: $\idxT = (i_1, i_2, \cdots, i_n)$ where $\{ i_1, i_2, \cdots, i_n \} = \{ 1, 2, \cdots, n \}$;
    \item A \emph{sequence operation}: $* \in \{ \oplus|_\alpha^\beta, \otimes \}$ (Complete Sequence Addition/Multiplication);
    \item A non-$\voidError$ $\bOIE$ result: $\oie_{res} = *(\oie_{i_1}, \oie_{i_2}, \cdots, \oie_{i_n})$ (projection is undefined when the result is $\voidError$);
    \item An $n$-ary finitary operator: $\odot$ for building operations, with an initial empty operation set $\OperationS = \emptyset$.
\end{enumerate}

\noindent \textbf{II. Construction Steps.}
We build $\OperationS$ by processing each $\TwoTplT_{cur} \in \F_{\oie_{res}}$ (The $\F$(2nd element) of ${\oie_{res}}$):
\begin{enumerate}
    \item \textbf{Group by the ending timestamps} \\
    \indent For the current $\bTwoTplT$ instance in $\F_{\oie_{res}}$
    \begin{align}
        \TwoTplT_{cur} = \big( (\TSs_1, \TSe_1), (\TSs_2, \TSe_2), \cdots, (\TSs_n, \TSe_n) \big),
    \end{align}
    let $t_1 < t_2 < \cdots < t_\lambda$($1 \leq \lambda \leq n$) be the \emph{distinct ending timestamps} in $\{ \TSe_1, \TSe_2, \cdots, \TSe_n \}$.
    For each $t_\theta \in \{t_1, t_2, \cdots, t_\lambda\}$, define:
    \begin{align}
        \Gamma_{t_\theta} = \big\{ i_k \mid \TSe_{i_k} = t_{\theta} \big\}, \quad \mathit{Card}_{t_\theta} = \CARDI{\Gamma_{t_\theta}}.
    \end{align}
    Let
    \begin{align}
        \OperandS_{t_\theta} = \{ opd_{i_k} \mid i_k \in \Gamma_{t_\theta} \}
    \end{align}
    be the corresponding operand subset.
    Then, we get two sets
    \begin{align}
        & \Gamma_{group} = \{ \Gamma_{t_1}, \Gamma_{t_2}, \cdots, \Gamma_{t_\lambda} \}, \\
        & \OperandS_{group} = \{ \OperandS_{t_1}, \OperandS_{t_2}, \cdots, \OperandS_{t_\lambda}\}.
    \end{align}

    \item \textbf{Generate intra-group permutations} \\
    \indent For each group ending timestamp $t_\theta \in \{t_1,t_2,\cdots,t_\lambda\}$,
    let $Perm_{t_\theta}$ be the set of all $\mathit{Card}_{t_\theta}!$ permutations of $\OperandS_{t_\theta}$ (all possible orders of operands sharing the same ending timestamp $t_\theta$).

    \item \textbf{Generate global order-respecting permutations} \\
    Valid global permutations are constructed by taking the Cartesian product across all intra-group permutation sets
    \begin{align}
        Perm = Perm_{t_1} \times Perm_{t_2} \times \cdots \times Perm_{t_\lambda}.
    \end{align}
    Each tuple $\pi \in Perm$ corresponds to a \emph{concatenated global permutation}
    \begin{align}
        \pi = \pi_1 \Vert \pi_2 \Vert \cdots \Vert \pi_\lambda\ (\pi_1 \in Perm_{t_1}, \pi_2 \in Perm_{t_2}, \cdots, \pi_\lambda \in Perm_{t_\lambda}),
    \end{align}
    where $\Vert$ denotes permutation concatenation (preserving group order).

    \item \textbf{Build expressions and deduplicate} \\
    \indent For each global permutation $\pi \in Perm$, construct the $n$-ary expression:
    \begin{align}
        \odot\big( \pi(1), \pi(2), \cdots, \pi(n) \big),
    \end{align}
    where $\pi(i)$ means the $i$-th item in $\pi$.
    Insert all such expressions into $\OperationS$, then remove \emph{syntactically duplicate expressions} (equality up to operand order, depending on $\odot$'s properties).
\end{enumerate}

\noindent \textbf{III. Projection Notation.}
The above construction is called ``\emph{The projection based on ending timestamps of $\mathcal{F}$}'', denoted by:
\begin{align}
    *(\oieS, \idxT) \xrightarrow[\ES, \OperandS, \odot]{\text{ascending order of } \bTSe} \OperationS.
\end{align}

\end{defi}

As an example of Definition~\ref{def:implementAndProps.endTsProjection},
consider an $\bOIE$ instance $\oie$ with a $\bTwoTplT$ instance in its $\F$ (2nd element):
\begin{align}
    \TwoTplT = ((0, 5), (6, 10), (8, 10), (11, 15), (12, 15), (13, 15)).
\end{align}
The corresponding operand set is $\OperandS = \{a_1, b_1, b_2, c_1, c_2, c_3\}$
mapped bijectively to the six tuples in the given order
\begin{align}
    \begin{aligned}
        & (0, 5) \leftrightarrow a_1, \\
        & (6, 10) \leftrightarrow b_1, \\
        & (8, 10) \leftrightarrow b_2, \\
        & (11, 15) \leftrightarrow c_1, \\
        & (12, 15) \leftrightarrow c_2, \\
        & (13, 15) \leftrightarrow c_3,
    \end{aligned}
\end{align}
and the $n$-ary finitary operator is $\odot$.
We now need to compute the expressions derived for $\TwoTplT$.

We construct the projection as follows:

\textbf{Step 1: Group by the ending timestamps.}
Extracting the 2nd items (ending timestamps) from each 2-tuple,
we obtain distinct timestamps $t_1 = 5 < t_2 = 10 < t_3 = 15$.
The index groups and corresponding $\bOperandS$ instances are:
\begin{itemize}
    \item $\Gamma_{5} = \{1\}$, $\OperandS_{5} = \{a_1\}$ (cardinality $1$);
    \item $\Gamma_{10} = \{2,3\}$, $\OperandS_{10} = \{b_1, b_2\}$ (cardinality $2$);
    \item $\Gamma_{15} = \{4,5,6\}$, $\OperandS_{15} = \{c_1, c_2, c_3\}$ (cardinality $3$).
\end{itemize}

\textbf{Step 2: Generate intra-group permutations.}
For each group, we enumerate all permutations of its operands:
\begin{itemize}
    \item $\text{Perm}_{5} = \{(a_1)\}$ ($1! = 1$ permutation);
    \item $\text{Perm}_{10} = \{(b_1, b_2), (b_2, b_1)\}$ ($2! = 2$ permutations);
    \item $\text{Perm}_{15} = \{(c_1, c_2, c_3), (c_1, c_3, c_2), (c_2, c_1, c_3), (c_2, c_3, c_1), (c_3, c_1, c_2), (c_3, c_2, c_1)\}$ ($3! = 6$ permutations).
\end{itemize}

\textbf{Step 3: Generate global order-respecting permutations.}
The valid global permutations are formed by the Cartesian product $\text{Perm}_{5} \times \text{Perm}_{10} \times \text{Perm}_{15}$,
concatenated while preserving the ascending order of timestamps $(5, 10, 15)$.
This yields $1 \times 2 \times 6 = 12$ distinct global permutations $\OperandS_{\TwoTplT}$:
\begin{align}
    \begin{aligned}
        & (a_1, b_1, b_2, c_1, c_2, c_3), && (a_1, b_1, b_2, c_1, c_3, c_2), && (a_1, b_1, b_2, c_2, c_1, c_3), \\
        & (a_1, b_1, b_2, c_2, c_3, c_1), && (a_1, b_1, b_2, c_3, c_1, c_2), && (a_1, b_1, b_2, c_3, c_2, c_1), \\
        & (a_1, b_2, b_1, c_1, c_2, c_3), && (a_1, b_2, b_1, c_1, c_3, c_2), && (a_1, b_2, b_1, c_2, c_1, c_3), \\
        & (a_1, b_2, b_1, c_2, c_3, c_1), && (a_1, b_2, b_1, c_3, c_1, c_2), && (a_1, b_2, b_1, c_3, c_2, c_1).
    \end{aligned}
\end{align}

\textbf{Step 4: Build expressions.}
For each global permutation $\pi$,
we construct the $n$-ary expression
\begin{align}
    \odot(\pi(1), \pi(2), \pi(3), \pi(4), \pi(5), \pi(6)).
\end{align}
Thus, the current resulting $\bOperationS$ instance is:
\begin{align}
    \begin{aligned}
        & \OperationS = \big\{ \\
        & \qquad \odot(a_1, b_1, b_2, c_1, c_2, c_3),\; \odot(a_1, b_1, b_2, c_1, c_3, c_2),\; \odot(a_1, b_1, b_2, c_2, c_1, c_3), \\
        & \qquad \odot(a_1, b_1, b_2, c_2, c_3, c_1),\; \odot(a_1, b_1, b_2, c_3, c_1, c_2),\; \odot(a_1, b_1, b_2, c_3, c_2, c_1), \\
        & \qquad \odot(a_1, b_2, b_1, c_1, c_2, c_3),\; \odot(a_1, b_2, b_1, c_1, c_3, c_2),\; \odot(a_1, b_2, b_1, c_2, c_1, c_3), \\
        & \qquad \odot(a_1, b_2, b_1, c_2, c_3, c_1),\; \odot(a_1, b_2, b_1, c_3, c_1, c_2),\; \odot(a_1, b_2, b_1, c_3, c_2, c_1) \\
        & \big\}.
    \end{aligned}
\end{align}

Thus,
\begin{align}
    \CARDI{\OperationS} = 1! \cdot 2! \cdot 3! = 12.
\end{align}
The algebraic structure enforces that $a_1$ (ending at $t=5$) must appear first,
followed by $b_1, b_2$ (ending at $t=10$) in either order,
and finally $c_1, c_2, c_3$ (ending at $t=15$) in any order,
reflecting the temporal precedence constraints encoded in the original $\bTwoTplT$ instance $\TwoTplT$.

To characterize the computational complexity of this projection,
we establish bounds on the cardinality of the resulting operation set.

\begin{pty}[\BF{The number of operation orders determined by projection on Complete Sequence Addition}]\label{prop:addOpNumber}\EBL

Let the following objects be given:

\begin{enumerate}
    \item \BF{A finite set of operands}
    $\OperandS = \{ opd_1, opd_2, \cdots, opd_n \}$ with $n \in \mathbb{N}^+, n > 1$.

    \item \BF{A finite set of events}
    $\ES = \{ \E_1, \E_2, \cdots, \E_n \}$ together with a bijection \\ $\ES \to \OperandS$.

    \item \BF{A set of $\bOIE$ instances}
    $\oieS = \{ \oie_1, \oie_2, \cdots, \oie_n \}$ together with a mapping \\ $\oieS \to \ES$.

    \item \BF{An index tuple}
    $\IdxT = (i_1, i_2, \cdots, i_n)$ where $\{i_1, i_2 \cdots, i_n\} = \{1, 2, \cdots, n\}$.

    \item \BF{Complete sequence operation}
    ${\oplus|_{\alpha}^{\beta}}$.
    Applying ${\oplus|_{\alpha}^{\beta}}$ to $\oieS$ and $\IdxT$ yields
    \begin{align}
        \oie_{res} = {\oplus|_{\alpha}^{\beta}}(\oie_{i_1}, \oie_{i_2}, \cdots, \oie_{i_n}).
    \end{align}
    We require $\oie_{res} \neq \voidError$; otherwise, the projection does not hold.

    \item \BF{An $n$-ary finitary operator}
    $\odot$ used to build expressions.

    \item \BF{A set of $n$-ary finitary operations} $\OperationS$
\end{enumerate}
By Definition~\ref{def:implementAndProps.endTsProjection}, the following projection holds:
\begin{align}
    {\oplus|_{\alpha}^{\beta}}(\oieS,\ \idxT)
    \;\xrightarrow[\ES,\ \OperandS,\ \odot]{\text{ascending order of $\bTSe$}}\;
    \OperationS.
\end{align}

Then the range of the cardinality of $\OperationS$ is
\begin{align}
    [ n,\ n! ]
\end{align}
\end{pty}

\begin{proof}[\BF{Proof of Property~\ref{prop:addOpNumber}}]\EBL

We prove the two bounds separately.

\noindent\textbf{Lower bound ($\geq n$):}

    By Definition~\ref{def:OPs.Add.DomainFilteredSubTwoTplTS}(domain-filtered subset),
    for each index $i \in \{1,\cdots,n\}$,
    there exists a $\TwoTplT_{\lambda} \in \F_{\oie_{res}}$ such that the 2nd element (the ending timestamp) of the $i$-th 2-tuple in $\TwoTplT_{\lambda}$ equals $\beta$.
    Therefore
        \begin{align}
            & \forall opd_i \in \OperandS,\ \exists \text{operation} \in \OperationS: opd_i\ \text{is the last operand in operation}.
        \end{align}
    These $n$ operation orders are distinct because their last operands are different.
    Hence
\begin{align}
    \CARDI{\OperationS} \geq n.
\end{align}

\noindent\textbf{Upper bound ($\leq n!$):}

Each element of $\text{operation}\mathcal{S}$ is an $n$-ary expression $\odot(\pi(1), \cdots, \pi(n))$, where $\pi$ is a permutation of the $n$ operands.
There are at most $n!$ distinct permutations of $n$ distinct operands.
Therefore
\begin{align}
    \CARDI{\OperationS} \leq n!.
\end{align}

\noindent Combining both bounds, we have:
\begin{align}
    n \leq |\text{operation}\mathcal{S}| \leq n!.
\end{align}

this concludes the proof.
\end{proof}

\begin{pty}[\BF{The number of operation orders determined by
projection
on Complete Sequence Multiplication
}]\label{prop:multiOpNumber}\EBL

Let the following objects be given:

\begin{enumerate}
    \item \textbf{A finite set of operands}
    $\OperandS = \{ opd_1, opd_2, \cdots, opd_n \}$ with $n \in \mathbb{N}^+, n > 1$.

    \item \textbf{A set of events}
    $\ES = \{ \E_1, \E_2, \cdots, \E_n \}$ together with a bijection $\ES \to \OperandS$.

    \item \textbf{A set of $\bOIE$ instances}
    $\oieS = \{ \oie_1, \oie_2, \cdots, \oie_n \}$ together with a mapping $\oieS \to \ES$.

    \item \textbf{An index tuple}
    $\IdxT = (i_1, i_2, \cdots, i_n)$ where $\{i_1, i_2 \cdots, i_n\} = \{1, 2, \cdots, n\}$.

    \item \textbf{Complete Sequence Multiplication}
    $\otimes$.
    Applying $\otimes$ to $\oieS$ and $\IdxT$ yields
    \begin{align}
        \oie_{res} = \otimes(\oie_{i_1}, \oie_{i_2}, \cdots, \oie_{i_n}).
    \end{align}
    We require $\oie_{res} \neq \voidError$; otherwise, the projection does not hold.

    \item \textbf{An $n$-ary finitary operator}
    $\odot$ used to build expressions.

    \item \BF{A set of $n$-ary finitary operations} $\OperationS$
\end{enumerate}
By Definition~\ref{def:implementAndProps.endTsProjection}, the following projection holds:
\begin{align}
\otimes(\oieS,\ \idxT)
\;\xrightarrow[\ES,\ \OperandS,\ \odot]{\text{ascending order of $\bTSe$}}\;
\OperationS.
\end{align}

Then the number of operand orders of the $n$-ary finitary operations is 1
\end{pty}

\begin{proof}[\BF{Proof of Property~\ref{prop:multiOpNumber}}]\EBL

According to Definition~\ref{def:Ops.Multi} and Definition~\ref{def:OPs.Multi.CompleteAscOrderFilteredSubTwoTplTS},
\begin{align}
    \begin{aligned}
        \makecell[l]{ \forall\ \TwoTpl_i,\ \TwoTpl_j \in \TwoTplT\ \land\ i < j: \\
        \ \ \ the\ 2nd\ item\ of\ \TwoTpl_i \leq the\ 1st\ item\ of\ \TwoTpl_j}.
    \end{aligned}
\end{align}
Therefore,
the ending timestamps (2nd items) of the intervals in $\TwoTplT$ must be arranged in non-decreasing order,
which determines that the n $\bE$ instances mapped by $\oie_{res}$
must be arranged in the order of the $\bOIE$ instances participating in $\otimes$,
so there is only one permutation, the number is 1.
\end{proof}

\begin{cor}[\BF{Uniqueness of sets of $n$-ary finitary operations within an equivalence class in orbit space}]\label{cor:uniqueOpSet}\EBL

By Definition~\ref{def:implementAndProps.endTsProjection},
given an operand set $\OperandS$ with cardinality $1 < n < \infty$,
an $\bES$ instance $\ES$ and an $\bOIES$ instance $\oieS$.
There exists $\bIdxT$ instances, for any $\bIdxT$ instance $\idxT$ in them,
there is an operation set $\OperationS$ with cardinality $0 < m < \infty$ for $n$-ary finitary operator $\odot$,
they satisfy
\begin{align}
    *(\oieS,\ \idxT)
\;\xrightarrow[\ES,\ \OperandS,\ \odot]{\text{ascending order of $\bTSe$}}\;
\OperationS,
\end{align}
where $*$ denotes either the Complete Sequence Addition $\oplus|_{\alpha}^{\beta}$
or the Complete Sequence Multiplication $\otimes$.

Let $\CAL{O}(\oieS, *)$ be the orbit space of the sequence operation $*$ on $\oieS$
(see Definition~\ref{defi:OPS.orbitSpace.add} and Definition~\ref{defi:OPS.orbitSpace.mul}).
Then, for each equivalence class $[\oie_{res}]_{\sim}$ in $\CAL{O}(\oieS, *)$,
all $\bOIE$ instances in $[\oie_{res}]_{\sim}$ yield the same set of $n$-ary finitary operations via the projection.

That is, for any $\oie_{res_1}, \oie_{res_2} \in [\oie_{res}]_{\sim}$,
the $\bOperationS$ instances are the same.
\end{cor}

\begin{proof}[\BF{Proof of Corollary~\ref{cor:uniqueOpSet}}]\EBL

Assume two non-$\voidError$ $\bOIE$ instances $\oie_{res_{1}}$ and $\oie_{res_{2}}$ belong to the same equivalence class in $\CAL{O}(\oieS,\ast)$.
By Definition~\ref{defi:OPS.orbitSpace.add} and Definition~\ref{defi:OPS.orbitSpace.mul}, $\oie_{res_{1}}$ and $\oie_{res_{2}}$ are permutationally equivalent.
By Definition~\ref{def:OPs.Permu}, it indicates that there exists a permutation matrix $\M$ such that
\begin{align}
    \C_{\oie_{res_{1}}} \cdot \M = \C_{\oie_{res_{2}}}, \quad
    \F_{\oie_{res_{1}}} \stackrel{\M}{\cong} \F_{\oie_{res_{2}}}, \quad
    \I_{\oie_{res_{1}}} = \I_{\oie_{res_{2}}}, \quad
    \A_{\oie_{res_{1}}} = \A_{\oie_{res_{2}}}.
\end{align}
The permutation isomorphism $\F_{\oie_{res_{1}}} \stackrel{\M}{\cong} \mathcal{F}_{\oie_{res_{2}}}$
means that there is a bijection
\begin{align}
    \sigma: \F_{\oie_{res_{1}}} \to \F_{\oie_{res_{2}}}
\end{align}
with $\sigma(\TwoTplT) = \TwoTplT \cdot \M$ for every $\TwoTplT \in \F_{\oie_{res_{1}}}$.

Now consider the projection defined in Definition~\ref{def:implementAndProps.endTsProjection}.
For each $\TwoTplT \in \F_{\oie_{res_{1}}}$,
the projection constructs a set of $n$-ary expressions $\odot(\pi(1),\cdots,\pi(n))$
based on the ending timestamps (the 2nd items) of the $2$-tuples in $\TwoTplT$.
Let us denote this set by $\OperationS_{\TwoTplT}$.
The final operation set $\OperationS_1$ for $\oie_{res_{1}}$ is the union of
all $\OperationS_{\TwoTplT}$ for $\TwoTplT \in \F_{\oie_{res_{1}}}$ (with duplicates removed)
\begin{align}
    \OperationS_1 = \bigcup_{\TwoTplT \in \F_{\oie_{res_{1}}}} (\OperationS_{\TwoTplT}).
\end{align}
Similarly, for $\oie_{res_{2}}$ we obtain $\OperationS_2$ from $\F_{\oie_{res_{2}}}$
\begin{align}
    \OperationS_2 = \bigcup_{\TwoTplT' \in \F_{\oie_{res_{2}}}} (\OperationS_{\TwoTplT'}).
\end{align}

Take an arbitrary $\TwoTplT \in \F_{\oie_{res_{1}}}$
\begin{align}
    \TwoTplT = \big( (\TSs_1,\TSe_1), (\TSs_2,\TSe_2), \cdots, (\TSs_n,\TSe_n) \big).
\end{align}
Its image under $\sigma$ is
\begin{align}
    \TwoTplT{'} = \sigma(\TwoTplT) = \TwoTplT \cdot \CAL{M}
     = \big( (\TSs_{\sigma(1)},\TSe_{\sigma(1)}), (\TSs_{\sigma(2)},\TSe_{\sigma(2)}), \cdots, (\TSs_{\sigma(n)},\TSe_{\sigma(n)}) \big),
\end{align}
where $\sigma$ is the permutation corresponding to $\CAL{M}$.
The ending timestamps(2nd items) in $\TwoTplT$ are the multiset $\{\TSe_1, \TSe_2, \cdots, \TSe_n\}$,
which is the same as in $\TwoTplT{'}$ because $\M$ merely reorders the pairs.
Let
\begin{align}
    t_1 < t_2 < \cdots < t_\lambda,\ 1 \leq \lambda \leq n
\end{align}
be the \emph{distinct ending timestamps} in $\{ \TSe_1, \TSe_2, \cdots, \TSe_n \}$,
Consequently, the grouping of indices by equal ending timestamps is preserved up to the permutation:
if for an ending timestamp $t_\theta \in \{t_1, t_2, \cdots, t_\lambda\}$,
for $\F_{\oie_{res_{1}}}$,
the set of indices with $t_\theta$ in $\TwoTplT$ is
\begin{align}
    \Gamma_{t_\theta},
\end{align}
then for $\F_{\oie_{res_{2}}}$, the set of indices with the same ending timestamp $t_\theta$ in $\TwoTplT{'}$ is
\begin{align}
    \sigma(\Gamma_{t_\theta}) = \{ \sigma(k) \mid k \in \Gamma_{t_\theta} \}.
\end{align}

Recall that the operands are associated with the $\bOIE$ instances via the fixed mappings $\oieS \to \ES \to \OperandS$.
For $\oie_{res_{1}}$, the tuple
\begin{align}
    \C_{\oie_{res_{1}}} = (\oie_{i_1}, \oie_{i_2}, \cdots, \oie_{i_n})
\end{align}
determines that the $k$-th($k \in [1, n]$) component of any $\TwoTplT \in \F_{\oie_{res_{1}}}$ corresponds to operand $opr_{i_k}$.
For $\oie_{res_{2}}$,
because $\C_{\oie_{res_{1}}} \cdot \M = \C_{\oie_{res_{2}}}$,
we have
\begin{align}
    \C_{\oie_{res_{2}}} = (\oie_{i_1^{'}}, \oie_{i_2^{'}}, \cdots,\oie_{i_n^{'}})
\end{align}
with
\begin{align}
    (i_1^{'},i_2^{'},\cdots,i_n^{'}) = (i_{\sigma(1)}, i_{\sigma(2)},\cdots,i_{\sigma(n)}).
\end{align}
Hence, the operand corresponding to the $k$-th($k \in [1, n]$) component of $\TwoTplT'$ is
\begin{align}
    opr_{i_k^{'}} = opr_{i_{M(k)}}.
\end{align}

Now examine the construction of $\OperationS_{\TwoTplT}$ and $\OperationS_{\TwoTplT{'}}$.
In both cases, the operands are partitioned into groups according to the distinct ending timestamps $t_1 < t_2 < \cdots < t_\lambda$.
For $\TwoTplT$,
the group for $t_\theta \in \{t_1, t_2, \cdots, t_\lambda\}$ contains the operands $\{ opr_{i_k} \mid k \in \Gamma_{t_\theta} \}$.
For $\TwoTplT'$, the group for $t_\theta$ contains the operands satisfying
\begin{align}
    \{ opr_{i_k^{'}} \mid k \in \sigma(\Gamma_{t_\theta}) \} = \{ opr_{i_{\sigma(k)}} \mid k \in \sigma(\Gamma_{t_\theta}) \}.
\end{align}
Since $\sigma$ is a permutation, we have
\begin{align}
    \{ opr_{i_{\sigma(k)}} \mid k \in \sigma(\Gamma_{t_\theta}) \} = \{ opr_{i_k} \mid k \in \Gamma_{t_\theta} \}.
\end{align}
Thus, each group consists of exactly the same set of operands.

The projection then generates all permutations of the operands that respect the group order:
the operands of the first group (with the smallest ending timestamp) appear first in any order,
followed by the operands of the second group, etc.
Because the groups are identical as sets, the collection of all such permutations is the same for $\TwoTplT$ and $\TwoTplT'$.
Therefore
\begin{align}
    \OperationS_{\TwoTplT} = \OperationS_{\TwoTplT{'}}.
\end{align}

Since $\sigma$ is a bijection between $\F_{\oie_{res_{1}}}$ and $\F_{\oie_{res_{2}}}$,
By taking the union all sets,
we have
\begin{align}
\bigcup_{\TwoTplT \in \F_{\oie_{res_{1}}}} (\OperationS_{\TwoTplT})
= \bigcup_{\TwoTplT' \in \F_{\oie_{res_{2}}}} (\OperationS_{\TwoTplT{'}}).
\end{align}
Hence
\begin{align}
    \OperationS_1 = \OperationS_2.
\end{align}

This shows that any two permutationally equivalent $\bOIE$ instances produce the same set of $n$-ary finitary operations through the projection.
Therefore, within an equivalence class of the orbit space,
all $\bOIE$ instances yield the same $\bOperationS$ instance.
\end{proof}
\fi
    \subsection{Algebraic Constraints and Operation Set Uniqueness in $\bOIE$ Implementation}\label{subsec:impl.restriction}\EBL
    \ifincludeFile
In this subsection, we will present more general conclusions that can be applied in practice for
Complete Sequence Addition and Complete Sequence Multiplication.

\begin{cor}[\BF{Conditional equivalence of operation sets under the projection of $\oplus|_{\alpha}^{\beta}$}]
    \label{thm:implementAndProps.fit.add}\EBL

Given an $\bOIES$ instance $\oieS$ with cardinality $0<n<+\infty$
\begin{align}
    \oieS = \{ \oie_1, \oie_2, \cdots, \oie_n \},
\end{align}
an $\bES$ instance $\ES$ with cardinality $1<n<+\infty$,
an $\bOperandS$ instance $\OperandS$ satisfying
\begin{align}
    \oieS \rightarrow \ES \rightarrow \OperationS
\end{align}
with the form
\begin{align}
    \OperandS = \{ F_1(\oie_1), F_2(\oie_2), \cdots, F_n(\oie_n) \},
\end{align}
an $\bIdxT$ instance $\idxT$ with the form
\begin{align}
    \idxT = (i_1, i_2, \cdots, i_n),
\end{align}
and an $\bOperationS$ instance $\OperationS$ with n-ary finitary operation $\odot$.
And they form a projection based on ending timestamp of the $\F$ of $\oie_{res}$
obtained by Complete Sequence Addition $\oplus|_{\alpha}^{\beta}$ on $\oieS$ under $\idxT$
(Definition~\ref{def:implementAndProps.endTsProjection})
\begin{align}
        \oplus|_{\alpha}^{\beta}(\oieS,\ \idxT)
    \;\xrightarrow[\ES,\ \OperandS,\ \odot]{\text{ascending order of $\bTSe$}}\;
    \OperationS.
\end{align}
If ``$\odot$'' is an $n$-ary finitary operation that satisfies permutational equivalence
\begin{align}
    \begin{aligned}
        & \forall\ (opr_{i_1}, opr_{i_2}, \cdots, opr_{i_n}) \stackrel{\CAL{M}}{\sim}  (opr_{i_1^{'}}, opr_{i_2^{'}}, \cdots, opr_{i_n^{'}}) : \\
        & \quad \odot(opr_{i_1}, opr_{i_2}, \cdots, opr_{i_n}) = \odot(opr_{i_1^{'}}, opr_{i_2^{'}}, \cdots, opr_{i_n^{'}}),
    \end{aligned}
\end{align}
then
\begin{align}
    \begin{aligned}
        \forall\ \odot(opr_{i_1}, opr_{i_2}, \cdots, opr_{i_n}), \odot(opr_{i_1^{'}}, opr_{i_2^{'}}, \cdots, opr_{i_n^{'}}) \in \OperationS: \\
        \odot(opr_{i_1}, opr_{i_2}, \cdots, opr_{i_n}) = \odot(opr_{i_1^{'}}, opr_{i_2^{'}}, \cdots, opr_{i_n^{'}}).
    \end{aligned}
\end{align}
Otherwise, if ``$\odot$'' is an $n$-ary finitary operation that does not satisfy permutational equivalence,
then the above relationship is not necessarily satisfied.
\end{cor}

\begin{cor}[\textbf{Singleton operation set under the projection of $\otimes$}]\label{thm:implementAndProps.fit.mul}\EBL

Given an $\bOIES$ instance $\oieS$ with cardinality $0<n<+\infty$
\begin{align}
    \oieS = \{ \oie_1, \oie_2, \cdots, \oie_n \},
\end{align}
an $\bES$ instance $\ES$ with cardinality $1<n<+\infty$,
an $\bOperandS$ instance $\OperandS$ satisfying
\begin{align}
    \oieS \rightarrow \ES \rightarrow \OperationS
\end{align}
with the form
\begin{align}
    \OperandS = \{ F_1(\oie_1), F_2(\oie_2), \cdots, F_n(\oie_n) \},
\end{align}
an $\bIdxT$ instance $\idxT$ with the form
\begin{align}
    \idxT = (i_1, i_2, \cdots, i_n),
\end{align}
and an $\bOperationS$ instance $\OperationS$ with n-ary finitary operation $\odot$.
And they form a projection based on ending timestamp of the $\F$ of $\oie_{res}$
obtained by Complete Sequence Multiplication $\otimes$ under $\idxT$
(Definition~\ref{def:implementAndProps.endTsProjection})
    \begin{align}
        \otimes(\oieS,\ \idxT)
        \;\xrightarrow[\ES,\ \OperandS,\ \odot]{\text{ascending order of $\bTSe$}}\;
        \OperationS
    \end{align}
where
\begin{align}
    \idxT = (i_1, i_2, \cdots, i_n),
\end{align}
Then, the $\OperationS$ has only one element operation
\begin{align}
    \odot(F_{i_1}(\oie_{i_1}), F_{i_2}(\oie_{i_2}), \cdots, F_{i_n}(\oie_{i_n})).
\end{align}
\end{cor}

Corollary~\ref{thm:implementAndProps.fit.add} establishes the foundational invariance
that all members of an orbit space equivalence class yield identical sets of $n$-ary finitary operations under projection.
However, this result remains structurally abstract:
it does not distinguish how Complete Sequence Addition($\oplus|_{\alpha}^{\beta}$) permits multiple expressions whose equivalence depends on $\odot$'s permutation properties,
whereas Complete Sequence Multiplication($\otimes$) enforces singleton sets through strict temporal ordering constraints.
This structural divergence critically determines whether a system exhibits sensitivity to operand ordering —
a distinction essential for subsequent applications, from modeling fair concurrent competitions in physics to formalizing parallel versus serial execution in computer science.
The following Corollary bridges this gap by explicitly characterizing the unique structural properties of operation sets within each equivalence class,
thereby rigorously differentiating process symmetry from outcome symmetry at the algebraic level.

\begin{cor}[\BF{Structural uniqueness of the operation sets within orbit space equivalence classes}]
    \label{cor:uniquenessOperationSetOrbit}\EBL

Let $\oieS$ be a finite $\bOIE$ instance and let
\begin{align}
    * \in \{\oplus|_{\alpha(n)}^{\beta}, \otimes\}
\end{align}
be a sequence operation.
Denote by $\CAL{O}(\oieS, *)$ the corresponding orbit space
(see Definitions~\ref{defi:OPS.orbitSpace.add} and Definition~\ref{defi:OPS.orbitSpace.mul}).
For every equivalence class $[\oie_{res}]_{\sim} \in \CAL{O}(\oieS, *)$,
there exists a unique set of $n$-ary finitary operations, denoted
$\OperationS_{[\oie_{res}]_{\sim}}$, such that for each
$\oie_{res} \in [\oie_{res}]_{\sim}$ the projection defined in
Definition~\ref{def:implementAndProps.endTsProjection} yields exactly $\OperationS_{[\oie_{res}]_{\sim}}$.

Moreover, the following structural properties hold:

\begin{enumerate}
    \item \textbf{For Complete Sequence Addition} ($* = \oplus|_{\alpha}^{\beta}$):
    \begin{itemize}
        \item The orbit space $\CAL{O}(\oieS, \oplus|_{\alpha}^{\beta})$
              consists of a single equivalence class (Property~\ref{pty:OPS.orbitSpace.add.oneElem}).
              Consequently, the whole orbit is associated with a unique set $\OperationS$.
        \item According to Corollary~\ref{thm:implementAndProps.fit.add},
              if the $n$-ary finitary operator $\odot$ satisfies permutational equivalence
              (i.e., its result is invariant under any reordering of its operands),
              then every expression in $\OperationS$ denotes the same operation.
              Otherwise, the expressions in $\OperationS$ may denote distinct operations.
    \end{itemize}

    \item \textbf{For Complete Sequence Multiplication} ($* = \otimes$):
    \begin{itemize}
        \item The orbit space $\CAL{O}(\oieS, \otimes)$
              may contain several equivalence classes (Property~\ref{pty:OPS.orbitSpace.multi.multiElem}).
        \item According to Corollary~\ref{thm:implementAndProps.fit.mul}, for each equivalence class $[\oie_{res}]_{\sim}$, the set
              $\OperationS_{[\oie_{res}]_{\sim}}$ contains exactly one expression.
              This expression is uniquely determined by the order of operands specified by
              any index tuple that produces a representative of the class.
    \end{itemize}
\end{enumerate}

\end{cor}

\begin{proof}[\BF{Proof of Corollary~\ref{cor:uniquenessOperationSetOrbit}}]\EBL

We prove the three parts of the corollary in order.

\noindent\textbf{1. Uniqueness within an equivalence class.}
Let $[\oie_{res}]_{\sim}$ be an equivalence class in $\CAL{O}(\oieS, *)$.
Take any two $\bOIE$ instances $\oie_{res_1}, \oie_{res_2} \in [\oie_{res}]_{\sim}$.
By definition, they are permutationally equivalent.
Applying Corollary~\ref{cor:uniqueOpSet}
to $\oie_{res_1}$ and $\oie_{res_2}$ shows that the projection
(Definition~\ref{def:implementAndProps.endTsProjection}) yields the same set of \(n\)-ary finitary operations
for both.
Hence, all members of the equivalence class produce the same set, which we denote by
$\OperationS_{[\oie_{res}]_{\sim}}$.

\noindent\textbf{2. Properties for Complete Sequence Addition.}
Property~\ref{pty:OPS.orbitSpace.add.oneElem} states that the orbit space of Complete Sequence Addition
contains exactly one equivalence class.
Therefore, the whole orbit corresponds to a single set $\OperationS$.
Now assume that the operation $\odot$ satisfies permutational equivalence.
Corollary~\ref{thm:implementAndProps.fit.add} then implies that every
expression in $\OperationS$ (which is precisely the permutations of
operands that respect the grouping induced by the ending timestamps) evaluates to the same result.
If $\odot$ does not satisfy permutational equivalence, the expressions in $\OperationS$
may differ.

\noindent\textbf{3. Properties for Complete Sequence Multiplication.}
By Property~\ref{pty:OPS.orbitSpace.multi.multiElem}, the orbit space of Complete Sequence Multiplication
may have more than one equivalence class.
For an arbitrary class $[\oie_{res}]_{\sim}$,
choose any representative $\oie_{res}$.
Corollary~\ref{thm:implementAndProps.fit.mul}
asserts that the projection for $\oie_{res}$ produces a set $\OperationS_{[\oie_{res}]_{\sim}}$
with exactly one element.
This is because the Complete Sequence Multiplication forces the
ending timestamps to be strictly ordered (each event must finish before the next one starts),
leaving only one admissible ordering of the operands in the projection—namely the order
prescribed by the index tuple that defines the operation.
Hence, $\OperationS_{[\oie_{res}]_{\sim}}$ is a singleton.

The proof is now complete.
\end{proof}
\fi

\section{Applications and Comparisons in Computer Science, Probability Theory and Physics}\label{sec:app-comp}\EBL
\ifincludeFile
\indent With the theoretical foundations of the $\bOIE$ framework and the sequence operations($\oplus|_{\alpha}^{\beta}$ and $\otimes$)
established in the preceding sections,
we now validate their practical efficacy across three foundational disciplines.
This section applies the abstract algebraic machinery to concrete problems in computer science, probability and physics,
demonstrating how $\bOIE$ and sequence operations reconstruct classical notions
of parallelism, simultaneity and probabilistic fairness.
By comparing our approach with well-known methods
such as Process Algebra, Petri nets, classical probability, and classical/relativistic mechanics,
we show that our constraint-aware, planning-based framework can unify different modeling paradigms.

This section is organized as follows:
\begin{itemize}
\item Subsection~\ref{subsec:app.cs} applies the $\bOIE$ framework and the two sequence operations
to a foundational question in competitive events and parallel computation:
Why both concurrent (100-meter sprint, modeled by $\oplus|_{\alpha}^{\beta}$) and strictly sequential (skiing downhill, modeled by $\otimes$)
speed competitions can deterministically identify a champion despite their opposing temporal architectures.
It reveals that both modalities converge to an identical permutation-invariant algebraic structure (the minimum operation),
ensuring outcome determinacy through structural invariance rather than observational simultaneity.
We further extend this approach to determine the time complexity of parallel merge sort,
combining parallel and sequential behavior in one algebraic form.
\item Subsection~\ref{subsec:comp.cs}
benchmarks the $\bOIE$ framework against Process Algebra, Petri nets, and Allen's interval algebra from a computer science theory perspective.
While Process Algebra captures behavioral equivalence and Petri nets model state transitions,
both rely on object-level rules and idealized temporal precision, suffering from state-space explosion or interval rigidity.
The $\bOIE$ framework instead adopts meta-level structural constraints and a planning-ahead paradigm,
using interval sets as basic units to naturally express duration uncertainty and multiple candidate execution schemes.
This unified approach yields significant advantages in real-time scheduling, physical observation constraints, and dynamic planning scenarios.
\item Subsection~\ref{subsec:app.math} employs the $\bOIE$ framework to analyze the classical probability problem of drawing lots.
It distinguishes simultaneous sampling where all participants draw concurrently, from sequential sampling where participants draw one after another.
Classical probability establishes that each participant's marginal probability of drawing a red ball is $m/n$ in both scenarios.
The $\bOIE$ framework reveals that this outcome equivalence arises despite fundamentally distinct algebraic architectures:
Simultaneous sampling corresponds to $\oplus|_{\alpha}^{\beta}$ with its single-orbit permutational equivalence,
whereas sequential sampling corresponds to $\otimes$ with its strict ordering and potentially multiple orbital equivalence classes.
\item Subsection~\ref{subsec:comp.math} examines the traditional classical-probability analysis of simultaneous versus sequential sampling and identifies three essential limitations:
Completely separate mathematical paradigms for the two scenarios, the absence of a strict mathematical definition of ``process symmetry'', and the neglect of temporal constraints' influence on probability results.
Based on the $\bOIE$ algebraic framework, this subsection achieves theoretical innovations in four dimensions: constructing a unified formal description, rigorously defining process symmetry, revealing the impact of temporal constraints on probability outcomes, and distinguishing the descriptive meaning from the ontological meaning of indices in the sampling process.
\item Subsection~\ref{subsec:app.phy} presents an $\bOIE$ reinterpretation of Galileo's legendary free-fall experiment at the Leaning Tower of Pisa.
The conventional narrative implicitly assumes pointwise simultaneous release and pointwise simultaneous landing, both of which confront the unverifiability of pointwise simultaneity.
By modeling the ``simultaneous'' release as $\oplus|_{\alpha}^{\beta}$ within a shared equipotential time domain,
the subsection shows that the experiment verifies not that ``both spheres land at precisely the same instant'',
but rather that, within the equipotential domain, mass difference does not affect the result of the projection operation.
The output remains invariant with respect to the mass parameter.
\item Subsection~\ref{subsec:comp.phy} conducts an in-depth comparison in physics around the formal characterization of simultaneity from the observer's perspective,
using the 100-metre dash as the carrier scenario.
Newtonian mechanics upholds absolute simultaneity, while relativistic mechanics analyzes simultaneity through Lorentz transformations,
yet both still rely on the ideal hypothesis of pointwise simultaneity.
This subsection clarifies that the core contradiction is not the logical self-consistency of pointwise equality at the theoretical level,
but whether humans can strictly verify it through observation.
It proposes the concept of ``Epistemological equipotentiality'' to replace ``Pointwise equality'',
thereby freeing the observer from a fully passive role.
\end{itemize}
\fi
    \subsection{Application in Computer Science: The Essential Reason Why Both the 100-metre Dash and Skiing Downhill Can Determine a Definitive Champion \& Algebraic Analysis of Parallel Sorting Feasibility}\label{subsec:app.cs}\EBL
    \ifincludeFile
This subsection applies the $\bOIE$ framework and the two sequence operations to a foundational question raised in introduction:
why do both concurrent and strictly sequential speed competitions deterministically identify a champion, despite their opposing temporal architectures, and how can this algebraic insight formalize parallel algorithm design?
We addressed this problem using $\bOIE$ and sequence operations.
Through the sorting parallelization case,
we further demonstrated that our proposed methods are remarkably convenient and have tremendous potential for direct application in programming and industrial scenarios.

\begin{itemize}
\item Subsubsection~\ref{subsubsec:app.cs.100m} analyzes the 100-meters dash as an instantiation of $\oplus|_{\alpha}^{\beta}$,
establishing how equipotential time domains and permutational equivalence render lane assignments observationally irrelevant to the final ranking;
\item Subsubsection~\ref{subsubsec:app.cs.downhill} examines downhill skiing through $\otimes$,
demonstrating how strict sequential ordering generates multiple orbital equivalence classes while preserving outcome determinacy via the minimum operation;
\item Subsubsection~\ref{subsubsec:app.cs.unify} synthesizes these analyses to reveal the algebraic unification underlying concurrent and sequential competitions,
proving that champion determinacy arises not from shared physical implementation but from the permutation invariance of the temporal comparison operation.
\item Subsubsection~\ref{subsubsec:app.cs.parallelSorting} extends the theory to practical computer science, employing $\oplus|_{\alpha}^{\beta}$ and $\otimes$ to formally model parallel merge sort.
It identifies commutative matrix-multiplication subexpressions across divide-and-conquer layers, thereby deriving the time complexity
under the constraint of intra-layer parallelism and inter-layer seriality.
\end{itemize}\fi
        \subsubsection{100-meter Dash}\label{subsubsec:app.cs.100m}\EBL
        \ifincludeFile
Suppose the starting gun fires at 10:00:00 Beijing Time on January~7, 2026, corresponding to the Unix timestamp $\mathit{ts}_{\text{start}} = 1736253600$.
While there is no mandatory time limit for finishing a 100‑meter race, for our theoretical model we suppose that the race must be completed within $20$ seconds;
hence the latest allowed finish timestamp is $\mathit{te}_{\max} = 1736253620$.

We define 8 atomic Event instances representing each athlete's performance:
\begin{align}
    \{\, \atomE_{\text{athlete}_1},\,
          \atomE_{\text{athlete}_2},\,
          \dots,\,
          \atomE_{\text{athlete}_8} \,\}.
\end{align}
Abstracting the interval component yields eight Event$^*$ instances:
\begin{align}
    \{\, \text{event}^*_{\text{athlete}_1},\,
          \text{event}^*_{\text{athlete}_2},\,
          \dots,\,
          \text{event}^*_{\text{athlete}_8} \,\}.
\end{align}

Let the index tuple
\begin{align}
    \idxT_{\text{asc}} = (1,2,3,4,5,6,7,8)
\end{align}
represent lane assignments, where $athlete_i$ competes in lane~$i$.
We assume lane assignment does not affect performance, so the events are mutually independent.

Under the rules, each athlete's feasible interval set is
\begin{align}
\I = \{\, (\mathit{ts},\mathit{te}) \mid
                1736253600 \le \mathit{ts} < \mathit{te} \le 1736253620
                \ \land\
                9.4 \le (\mathit{te}-\mathit{ts}) \le 20 \,\}.
\end{align}
Crucially, every athlete has the opportunity to approach the theoretical human limit of 9.4 seconds,
to finish safely, and above all to \TT{share equal opportunity for a fair start}.
Whether an athlete achieves this depends on their performance.

By Definition~\ref{def:OPs.feasibleCP.independentOIES},
we construct a mutually independent $\bOIES$ instance
\begin{align}
    \oieS^{indep} = \{\oie_1, \oie_2, \dots,\oie_8\},
\end{align}
where
\begin{align}
    \oie_i = \Bigl(
       (),
       \{\, ((\mathit{ts},\mathit{te})) \mid (\mathit{ts},\mathit{te})\in\I \,\},
       \I,
       \{\, \E^*_i \,\}
    \Bigr).
\end{align}

We then form the Complete Sequence Addition:
\begin{align}
    \oie_{100\text{dash}}=\oplus|^{1736253620}_{1736253600}(\oie_1, \oie_2, \dots, \oie_8),
\end{align}
by Definition~\ref{def:OPs.Add.Nat}, this simplifies to
\begin{align}
    \oie_{100\text{dash}}=\oplus(\oie_1, \oie_2, \dots, \oie_8).
\end{align}

We use an $n$-ary finitary operation $\odot$ that returns the minimum of its operands to determine the winner.
By Definition~\ref{def:implementAndProps.endTsProjection}, we construct the projection
\begin{align}
\oplus|^{1736253620}_{1736253600}
   (\oieS,\,\idxT_{\text{asc}})
   \xrightarrow[\ES,\ \OperandS,\ \odot]
              {\text{ascending order of TSend}}
   \OperationS.
\end{align}

Ignoring reaction delays, the final result is determined by subtracting the start timestamp from the finish timestamp.
Let $F_i(\oie_i)$ denote the finishing time of $\text{athlete}_i$,
the $n$-ary finitary operation that determines the champion is
\begin{align}
    \odot\bigl(F_1(\oie_1),F_2(\oie_2),\dots,F_8(\oie_8)\bigr),
\end{align}
which returns the minimum of all the operands.

By Corollary~\ref{thm:implementAndProps.fit.add},
since $\odot$ satisfies permutational equivalence,
the projected operation is invariant to the order of finish times.

\paragraph{Equivalence under different lane assignments.}
The expression above corresponds to $\text{athlete}_1$–$\text{athlete}_8$ in lanes 1–8.
If we swap $\text{athlete}_1$ and $\text{athlete}_2$, we obtain
\begin{align}
    \oie_{100\text{dash}_{21}}=\oplus|^{1736253620}_{1736253600}(\oie_2,\oie_1,\oie_3,\dots,\oie_8).
\end{align}
If we reverse the order completely, we obtain
\begin{align}
    \oie_{100\text{dash}_{\text{reverse}}}=\oplus|^{1736253620}_{1736253600}(\oie_8,\oie_7,\dots,\oie_1).
\end{align}

These are merely permutations of operands and are therefore permutationally equivalent:
\begin{align}
    \oie_{100\text{dash}} \stackrel{\CAL{M}}{\sim} \oie_{100\text{dash}_{21}} \stackrel{\CAL{M}^{'}}{\sim} \oie_{100\text{dash}_{\text{reverse}}}.
\end{align}
All lie in an equivalence class of the orbit space
$\CAL{O}(\oieS, \oplus|^{1736253620}_{1736253600})$.
By Corollary~\ref{cor:uniquenessOperationSetOrbit},
permutational equivalence of $\odot$
ensures that the projected n‑ary operation remains identical regardless of lane assignment.

Thus,
lane ordering does not affect the expression that determines the champion or the overall ranking.
\fi
        \subsubsection{Downhill Skiing}\label{subsubsec:app.cs.downhill}\EBL
        \ifincludeFile
Downhill skiing is a winter sport characterized by strict sequential starts:
skiers compete one after another in a predetermined order,
and only one skier is allowed on the track at any time.
The skier who records the shortest interval is declared the champion

Assume the competition starts at 10:00:00 Beijing Time on January 7, 2026,
corresponding to the Unix timestamp $ts_{\text{start}} = 1736253600$.
The competition has a total time limit:
the last skier must finish within 20 minutes,
so the latest allowed finish timestamp is $\TSe_{\max} = \TSs_{\text{start}} + 1200 = 1736254800$.

Each skier's performance is constrained by physiological and technical limits,
defining a range of feasible skiing times.
For simplicity, we assume all skiers have identical feasible time intervals:
a minimum (fastest) time of $T_{\min}^{(i)} = 90$ seconds (1.5 minutes) and a maximum (slowest) time of $T_{\max}^{(i)} = 120$ seconds (2 minutes).
The model can be easily extended to personalized intervals for individual skiers.

We first define 8 atomic Event instances, each representing a skier's competition performance:
\begin{align}
    \{\, \atomE_{\text{skier}_1},\, \atomE_{\text{skier}_2},\, \ldots,\, \atomE_{\text{skier}_8} \,\}.
\end{align}

Abstracting the interval property yields eight Event$^*$ instances:
\begin{align}
    \{\, \EStar_{\text{skier}_1},\, \EStar_{\text{skier}_2},\, \ldots,\, \EStar_{\text{skier}_8} \,\}.
\end{align}

We introduce an index tuple to represent the competition start order (typically determined by preliminary round results):
\begin{align}
    \idxT_{\text{asc}} = (1, 2, 3, 4, 5, 6, 7, 8),
\end{align}
where the $i$-th element indicates that skier $i$ starts in the $i$-th position.
According to the competition rules,
there are no physical or logical conflicts that would render interval combinations infeasible, so we set
\begin{align}
    \fInfeasIntTwoTplTS{\oieS}{\idxT_{\text{asc}}} = \emptyset.
\end{align}

Next, we construct the $\bOIE$ instances for each skier.
The set of feasible intervals ($\I$, the 3rd element of the $\bOIE$) must satisfy two constraints:
1. The interval must lie within the total competition time window $[\TSs_{\text{start}}, \TSe_\text{end}]$.
2. The interval length (skiing time) must fall within the skier's feasible range $[T_{\min}^{(i)}, T_{\max}^{(i)}]$.

Thus, for each $\text{skier}_i$, the feasible interval set is:
\begin{align}
    \I_i = \{\, (\TSs, \TSe) \mid \TSs_{\text{start}} \le \TSs < \TSe \le \TSe_{\text{end}} \ \wedge\ T_{\min}^{(i)} \le (\TSe - \TSs) \le T_{\max}^{(i)} \,\}.
\end{align}

We form the $\bOIES$ instance $\oieS = \{\oie_1, \ldots, \oie_8\}$,
where each $\bOIE$ instance is defined as:
\begin{align}
    \oie_i = \Bigl( (\,),\ \{\, ((\TSs, \TSe)) \mid (\TSs, \TSe) \in \I_i \,\},\ \I_i,\ \{\, \EStar_{\text{skier}_i} \,\} \Bigr).
\end{align}

The core feature of downhill skiing,
sequential starts with no overlapping track usage,
exactly aligns with the semantics of $\otimes$,
which enforces a ``strict order'' constraint between operands.

Thus, we use the $\otimes$ to combine the eight skiers' $\bOIE$ instances
according to the given start order $\idxT_{\text{asc}}$:
\begin{align}
    \oie_{\text{downhill}} = \otimes(\oie_1, \oie_2, \ldots, \oie_8).
\end{align}

Following Definition~\ref{def:implementAndProps.endTsProjection},
we construct the projection from the sequence operation result to a set of $n$‑ary finitary operations.
Let:
\begin{itemize}
    \item $\ES = \{\E_{\text{skier}_1}, \ldots, \E_{\text{skier}_8}\}$ be the set of concrete $\bE$ instances.
    \item $\OperandS = \{F_1(\oie_1), F_2(\oie_2), \ldots, F_8(\oie_8)\}$ be the set of operands, where $F_i(\oie_i)$ denotes the competition time of $\text{skier}_i$.
    \item $\odot$ be the $n$-ary finitary operation to take the minimum of operands.
    \item $\OperationS$ be the set of $n$-ary finitary operations.
\end{itemize}

The projection is written as
\begin{align}
    \otimes(\oieS,\, \idxT_{\text{asc}}) \xrightarrow[\text{ascending order of $\bTSe$}]{\ES,\ \OperandS,\ \odot} \OperationS.
\end{align}

By Property~\ref{prop:multiOpNumber},
$\otimes$ enforces a strict order constraint:
the end timestamp of each skier's interval must be less than or equal to the start timestamp of the next skier's interval.
This constraint ensures that the resulting $\OperationS$ contains exactly one expression:
\begin{align}
    \odot\bigl( F_1(\oie_1), F_2(\oie_2), \ldots, F_8(\oie_8) \bigr).
\end{align}
This is precisely the operation used to determine the downhill skiing champion:
comparing all skiers' competition times and selecting the minimum.

Unlike the 100-meter dash, the result of the $\otimes$ for downhill skiing depends on the start order.
If the start order changes, e.g., swapping athletes 1 and 2:
\begin{align}
    \idxT^{'} = (2, 1, 3, 4, 5, 6, 7, 8),
\end{align}
the resulting $\bOIE$ instance $\oie_{\text{downhill}}^{'}$ of $\otimes$ is {not necessarily}
permutationally equivalent to the original $\oie_{\text{downhill}}$ (by Property~\ref{pty:OPs.Multi.Permu}).
Consequently, the projected operation set may also differ (although it still contains only one expression, the order of the operands may change).
However, because the $\odot$ operation satisfies permutational equivalence (exchanging the order of operands does not affect the resulting value),
the mathematical result for the champion remains unchanged regardless of the start order.
\fi
        \subsubsection{Conclusion: Algebraic Unification of Concurrent and Sequential Competitions}\label{subsubsec:app.cs.unify}\EBL
        \ifincludeFile
Using $\bOIE$ and the sequence operations $\oplus|_{\alpha}^{\beta}$ and $\otimes$,
we provide a rigorous algebraic framework for analyzing how both concurrent and sequential speed competitions deterministically identify a champion.
This subsubsection synthesizes the foregoing analyses of the 100-meter dash and downhill skiing to reveal the deep algebraic symmetry
underlying their seemingly disparate execution structures.

\textbf{Modeling Framework: Fairness versus Sequentiality}

The two racing modalities exhibit complementary temporal architectures.
The 100-meter dash instantiates \textit{concurrent equality}: modeled by $\oplus|_{\alpha}^{\beta}$,
it emphasizes that all participants share an equipotential domain wherein each has an equal opportunity to start
and finish within the same time window.
The interval combinations in the resulting $\F$(2nd element) allow arbitrary permutations of athletes' ending timestamps,
reflecting the physical reality that finish-line crossing order is empirically contingent.

Conversely, downhill skiing instantiates \textit{strict seriality}: modeled by $\otimes$,
it enforces a total order where the $i$-th participant’s interval must conclude before the ($i+1$)-th begins.
Here, the $\F$ (2nd element) contains only those timestamp sequences that respect the non-decreasing order prescribed by the starting index.

Despite this structural opposition—fairness ($\oplus|_{\alpha}^{\beta}$) versus sequentiality ($\otimes$)
—both competitions reduce to an identical algebraic operation for determining the winner: the minimum function $\odot$.
Extended to n arguments, this operation $\odot(x_1, \ldots, x_n)$ satisfies closure, associativity, and \textit{permutational equivalence}.
That is, for any sequence and any of its permutations $(x_{\sigma(1)}, \ldots, x_{\sigma(n)})$, the equality
\begin{align}
    \odot(x_1, \ldots, x_n) = \odot(x_{\sigma(1)}, \ldots, x_{\sigma(n)})
\end{align}
holds invariantly.

\textbf{Determinacy Analysis: Orbit Space and Operational Invariance}

The algebraic properties of $\odot$ guarantee champion determinacy independent of execution order,
but the mechanism of this guarantee differs structurally between the two models.

For the 100-meter dash,
the operation $\min$ satisfies closure, associativity, and \textit{permutational equivalence}.
That is, for any sequence $(x_1, \ldots, x_n)$ and any of its permutations $(x_{\sigma(1)}, \ldots, x_{\sigma(n)})$,
although lane assignments (operand order) can be arbitrarily swapped,
leading to different $\bOIE$ results such as $\oie_{100\text{dash}}$ and $\oie_{100\text{dash}_{21}}$,
Property~\ref{pty:OPS.orbitSpace.add.oneElem} establishes
that these results belong to a single equivalence class in the orbit space $\OrbitAdd$.

By Corollary~\ref{cor:uniquenessOperationSetOrbit},
the projected operation set contains $n$-ary finitary operations corresponding to all possible finish-time permutations,
yet all evaluate to the same $n$-ary finitary operation.
Consequently, the champion is determined by a single order-independent operation:
\begin{align}
    \min(F_1, F_2, \ldots, F_n).
\end{align}

For downhill skiing ($\otimes$):
The starting order is rigidly fixed by the index tuple $\idxT$.
Corollary~\ref{cor:uniquenessOperationSetOrbit} implies that the projection yields a singleton set containing exactly $\min(F_{i_1}, \ldots, F_{i_n})$,
where the operand order matches $\idxT$.
Altering the start order (e.g., swapping the first two athletes) generates a different equivalence class in the orbit space (Property~\ref{pty:OPS.orbitSpace.multi.multiElem}),
and the projected expression changes accordingly.
However, the permutational equivalence of $\odot$ ensures that the computed result (the minimum time)
remains identical regardless of operand ordering.
Thus, while the \textit{algebraic path} depends on the starting sequence,
the \textit{outcome} (the champion) does not.

\BF{Synthesis: From Phenomenological Description to Axiomatic Planning}

Although the 100-meter dash and downhill skiing stand in conceptual opposition regarding ``fairness'' and ``sequentiality'',
their essence as speed competitions — comparing time costs via total ordering —
is algebraically unified.
The critical insight is that the champion-determining operation $\odot$ possesses strong symmetry.
This permutation invariance property ensures that:
\begin{itemize}
    \item \BF{For $\oplus|_{\alpha}^{\beta}$ emphasizing concurrent fairness}: The result is insensitive to the instantaneous arrangement of participants,
    thereby grounding the idealized concept of ``simultaneous start'' in an equipotential domain rather than pointwise temporal equality.
    \item \BF{For $\otimes$ emphasizing strict order}: Although the operation depends on a preset order structure, the permutation invariance of $\odot$ guarantees that the final comparison result remains independent of that specific arrangement.
\end{itemize}

Therefore, both competition forms reliably determine a champion not through shared physical implementations,
but through their convergence to an identical algebraic structure with permutation invariance.
By establishing pre-execution algebraic constraints ($\oplus|_{\alpha}^{\beta}$ or $\otimes$) as the governing framework for event execution,
the $\bOIE$ formalism extends traditional analytical approaches:
The identification of the winner depends on structural invariance rather than relying solely on observational precision.
This perspective offers a unified basis for event scheduling, physical simulation, and computational modeling,
exhibiting robustness against the epistemological constraints inherent in precise time measurement.
\fi
        \subsubsection{Practical Applications: Algebraic Analysis of Parallel Sorting Feasibility}\label{subsubsec:app.cs.parallelSorting}\EBL
        \ifincludeFile
Array sorting is a pivotal research area in computer science.
In this subsection,
we employ $\bOIE$, $\oplus|_{\alpha}^{\beta}$ and $\otimes$
to analyze the implementation of parallel array sorting via an algebraic approach.
Through this example, we demonstrate that $\bOIE$ and sequence operations
can not only reshape the theory of parallel and sequential computation,
but also directly incorporate subjectivity into program design and many industrial applications.

Modern general-purpose servers typically spend 10\% to 30\% of their runtime on sorting,
while servers running specialized data-intensive workloads can reach 80\%.
Suppose there is an array
\begin{align}
    arr = (499, 375, 233, 168, 100, 57, 28, 14)
\end{align}
We sort it in ascending order
\begin{align}
    (499, 375, 233, 168, 100, 57, 28, 14) \to (14, 28, 57, 100, 168, 233, 375, 499)
\end{align}

In fact,
the essence of sorting is to ``\TT{multiply the original sequence by a permutation matrix to yield an ordered sequence}'',
and this example is as follows
\begin{align}
    \begin{aligned}
        \left(499, 375, 233, 168, 100, 57, 28, 14\right) &
        \times
        \color{gray!50}
        \begin{pmatrix}
            0 & 0 & 0 & 0 & 0 & 0 & 0 & \MXBF{1} \\
            0 & 0 & 0 & 0 & 0 & 0 & \MXBF{1} & 0 \\
            0 & 0 & 0 & 0 & 0 & \MXBF{1} & 0 & 0 \\
            0 & 0 & 0 & 0 & \MXBF{1} & 0 & 0 & 0 \\
            0 & 0 & 0 & \MXBF{1} & 0 & 0 & 0 & 0 \\
            0 & 0 & \MXBF{1} & 0 & 0 & 0 & 0 & 0 \\
            0 & \MXBF{1} & 0 & 0 & 0 & 0 & 0 & 0 \\
            \MXBF{1} & 0 & 0 & 0 & 0 & 0 & 0 & 0 \\
        \end{pmatrix} \\
        \color{black}
        & \qquad = \left(14, 28, 57, 100, 168, 233, 375, 499\right)
    \end{aligned}
\end{align}
In the standard merge sort algorithm, merge is the fundamental operation at the core.
The merging rule of the standard merge sort is to merge only adjacent sorted subintervals,
and the merging intervals in each pass are all pairwise combinations of the sorted subintervals obtained
from the previous pass.
Thus, all merging intervals within the same pass are inevitably non-overlapping and disjoint.

Because it is essentially a sequence multiplied by a transposition matrix for the action of merge,
we use matrix
\begin{align}
    \CAL{M}_{([i_l, i_r], [j_l, j_r])}^{(n)}
\end{align}
to express the transposition matrix for merging two sorted subsequences.
In the notation for this matrix variable,
$i_l$, $i_r$ are the left and right boundary indices of one sorted subsequence,
and $j_l$, $j_r$ are those of the other sorted subsequence
(We use 1 as the starting index here).
The ``$n$'' is the length of the array.
For example, the form of $\CAL{M}_{([1, 1], [2, 2])}^{(8)}$ in this case is
\begin{align}
    \CAL{M}_{([1, 1], [2, 2])}^{(8)} =
    \color{gray!50}
    \begin{pmatrix}
        0 & \MXBF{1} & 0 & 0 & 0 & 0 & 0 & 0 \\
        \MXBF{1} & 0 & 0 & 0 & 0 & 0 & 0 & 0 \\
        0 & 0 & \MXBF{1} & 0 & 0 & 0 & 0 & 0 \\
        0 & 0 & 0 & \MXBF{1} & 0 & 0 & 0 & 0 \\
        0 & 0 & 0 & 0 & \MXBF{1} & 0 & 0 & 0 \\
        0 & 0 & 0 & 0 & 0 & \MXBF{1} & 0 & 0 \\
        0 & 0 & 0 & 0 & 0 & 0 & \MXBF{1} & 0 \\
        0 & 0 & 0 & 0 & 0 & 0 & 0 & \MXBF{1} \\
    \end{pmatrix},
\end{align}
then the merge of 1st and 2nd elements can be expressed as
\begin{align}
    \left(499, 375, 233, 168, 100, 57, 28, 14\right) \times \CAL{M}_{([1, 1], [2, 2])}^{(8)} = \left(375, 499, 233, 168, 100, 57, 28, 14\right)
\end{align}

Then, we use transposition matrix multiplication to denote merge sort for the entire sequence.
Its process is
\begin{align}
    \begin{aligned}
        & (499, 375, 233, 168, 100, 57, 28, 14) \\
        & \xrightarrow{\text{merge([1, 1], [2, 2])}} (375, 499, 233, 168, 100, 57, 28, 14), \\
        & \xrightarrow{\text{merge([3, 3], [4, 4])}} (375, 499, 168, 233, 100, 57, 28, 14), \\
        & \xrightarrow{\text{merge([5, 5], [6, 6])}} (375, 499, 168, 233, 57, 100, 28, 14), \\
        & \xrightarrow{\text{merge([7, 7], [8, 8])}} (375, 499, 168, 233, 57, 100, 14, 28), \\
        & \xrightarrow{\text{merge([1, 2], [3, 4])}} (168, 233, 375, 499, 57, 100, 14, 28), \\
        & \xrightarrow{\text{merge([5, 6], [7, 8])}} (168, 233, 375, 499, 14, 28, 57, 100), \\
        & \xrightarrow{\text{merge([1, 4], [5, 8])}} (14, 28, 57, 100, 168, 233, 375, 499).
    \end{aligned}
\end{align}
Translate this process into the form of matrix multiplication
\begin{align}
        & (499, 375, 233, 168, 100, 57, 28, 14) \\
        & \times
        \color{gray!50}
        \begin{pmatrix}
            0 & \textbf{\color{black}1} & 0 & 0 & 0 & 0 & 0 & 0 \notag \\
            \textbf{\color{black}1} & 0 & 0 & 0 & 0 & 0 & 0 & 0 \notag \\
            0 & 0 & \MXBF{1} & 0 & 0 & 0 & 0 & 0 \notag \\
            0 & 0 & 0 & \MXBF{1} & 0 & 0 & 0 & 0 \notag \\
            0 & 0 & 0 & 0 & \MXBF{1} & 0 & 0 & 0 \notag \\
            0 & 0 & 0 & 0 & 0 & \MXBF{1} & 0 & 0 \notag \\
            0 & 0 & 0 & 0 & 0 & 0 & \MXBF{1} & 0 \notag \\
            0 & 0 & 0 & 0 & 0 & 0 & 0 & \MXBF{1} \notag \\
        \end{pmatrix}
        \times
        \begin{pmatrix}
            \MXBF{1} & 0 & 0 & 0 & 0 & 0 & 0 & 0 \notag \\
            0 & \MXBF{1} & 0 & 0 & 0 & 0 & 0 & 0 \notag \\
            0 & 0 & 0 & \MXBF{1} & 0 & 0 & 0 & 0 \notag \\
            0 & 0 & \MXBF{1} & 0 & 0 & 0 & 0 & 0 \notag \\
            0 & 0 & 0 & 0 & \MXBF{1} & 0 & 0 & 0 \notag \\
            0 & 0 & 0 & 0 & 0 & \MXBF{1} & 0 & 0 \notag \\
            0 & 0 & 0 & 0 & 0 & 0 & \MXBF{1} & 0 \notag \\
            0 & 0 & 0 & 0 & 0 & 0 & 0 & \MXBF{1} \notag \\
        \end{pmatrix} \notag \\
        & \times
        \color{gray!50}
        \begin{pmatrix}
            \MXBF{1} & 0 & 0 & 0 & 0 & 0 & 0 & 0 \notag \\
            0 & \MXBF{1} & 0 & 0 & 0 & 0 & 0 & 0 \notag \\
            0 & 0 & \MXBF{1} & 0 & 0 & 0 & 0 & 0 \notag \\
            0 & 0 & 0 & \MXBF{1} & 0 & 0 & 0 & 0 \notag \\
            0 & 0 & 0 & 0 & 0 & \MXBF{1} & 0 & 0 \notag \\
            0 & 0 & 0 & 0 & \MXBF{1} & 0 & 0 & 0 \notag \\
            0 & 0 & 0 & 0 & 0 & 0 & \MXBF{1} & 0 \notag \\
            0 & 0 & 0 & 0 & 0 & 0 & 0 & \MXBF{1} \notag \\
        \end{pmatrix}
        \times
        \begin{pmatrix}
            \MXBF{1} & 0 & 0 & 0 & 0 & 0 & 0 & 0 \notag \\
            0 & \MXBF{1} & 0 & 0 & 0 & 0 & 0 & 0 \notag \\
            0 & 0 & \MXBF{1} & 0 & 0 & 0 & 0 & 0 \notag \\
            0 & 0 & 0 & \MXBF{1} & 0 & 0 & 0 & 0 \notag \\
            0 & 0 & 0 & 0 & \MXBF{1} & 0 & 0 & 0 \notag \\
            0 & 0 & 0 & 0 & 0 & \MXBF{1} & 0 & 0 \notag \\
            0 & 0 & 0 & 0 & 0 & 0 & 0 & \MXBF{1} \notag \\
            0 & 0 & 0 & 0 & 0 & 0 & \MXBF{1} & 0 \notag \\
        \end{pmatrix} \notag \\
        & \times
        \color{gray!50}
        \begin{pmatrix}
            0 & 0 & \MXBF{1} & 0 & 0 & 0 & 0 & 0 \notag \\
            0 & 0 & 0 & \MXBF{1} & 0 & 0 & 0 & 0 \notag \\
            \MXBF{1} & 0 & 0 & 0 & 0 & 0 & 0 & 0 \notag \\
            0 & \MXBF{1} & 0 & 0 & 0 & 0 & 0 & 0 \notag \\
            0 & 0 & 0 & 0 & \MXBF{1} & 0 & 0 & 0 \notag \\
            0 & 0 & 0 & 0 & 0 & \MXBF{1} & 0 & 0 \notag \\
            0 & 0 & 0 & 0 & 0 & 0 & \MXBF{1} & 0 \notag \\
            0 & 0 & 0 & 0 & 0 & 0 & 0 & \MXBF{1} \notag \\
        \end{pmatrix}
        \times
        \begin{pmatrix}
            \MXBF{1} & 0 & 0 & 0 & 0 & 0 & 0 & 0 \notag \\
            0 & \MXBF{1} & 0 & 0 & 0 & 0 & 0 & 0 \notag \\
            0 & 0 & \MXBF{1} & 0 & 0 & 0 & 0 & 0 \notag \\
            0 & 0 & 0 & \MXBF{1} & 0 & 0 & 0 & 0 \notag \\
            0 & 0 & 0 & 0 & 0 & 0 & \MXBF{1} & 0 \notag \\
            0 & 0 & 0 & 0 & 0 & 0 & 0 & \MXBF{1} \notag \\
            0 & 0 & 0 & 0 & \MXBF{1} & 0 & 0 & 0 \notag \\
            0 & 0 & 0 & 0 & 0 & \MXBF{1} & 0 & 0 \notag \\
        \end{pmatrix} \notag \\
        & \times
        \color{gray!50}
        \begin{pmatrix}
            0 & 0 & 0 & 0 & \MXBF{1} & 0 & 0 & 0 \notag \\
            0 & 0 & 0 & 0 & 0 & \MXBF{1} & 0 & 0 \notag \\
            0 & 0 & 0 & 0 & 0 & 0 & \MXBF{1} & 0 \notag \\
            0 & 0 & 0 & 0 & 0 & 0 & 0 & \MXBF{1} \notag \\
            \MXBF{1} & 0 & 0 & 0 & 0 & 0 & 0 & 0 \notag \\
            0 & \MXBF{1} & 0 & 0 & 0 & 0 & 0 & 0 \notag \\
            0 & 0 & \MXBF{1} & 0 & 0 & 0 & 0 & 0 \notag \\
            0 & 0 & 0 & \MXBF{1} & 0 & 0 & 0 & 0 \notag \\
        \end{pmatrix} \notag \\
        & = (14, 28, 57, 100, 168, 233, 375, 499). \notag
    \end{align}
Abbreviated as
\begin{align}
    \begin{aligned}
        & (499, 375, 233, 168, 100, 57, 28, 14) \\
        & \quad \times \CAL{M}_{([1, 1], [2, 2])}^{(8)} \times \CAL{M}_{([3, 3], [4, 4])}^{(8)} \times \CAL{M}_{([5, 5], [6, 6])}^{(8)} \times \CAL{M}_{([7, 7], [8, 8])}^{(8)} \\
        & \quad \times \CAL{M}_{([1, 2], [3, 4])}^{(8)} \times \CAL{M}_{([5, 6], [7, 8])}^{(8)} \\
        & \quad \times \CAL{M}_{([1, 4], [5, 8])}^{(8)} \\
        & = (14, 28, 57, 100, 168, 233, 375, 499).
    \end{aligned}
\end{align}

For a merging pass with the merging unit of length $k$,
where each merge operation combines two adjacent subarrays of length $k$ (total range $2k$)
the total number of merging tasks is $\lfloor \frac{n}{2k} \rfloor$
for a sequence containing $n$ elements.

The elements involved in different merging operations within the same merging pass are disjoint.
Thus, we summarize the rule as follows:
\begin{align}
    \begin{aligned}
        & For\ any\ two\ matrices\ \CAL{M}_{([a_l, a_r], [b_l, b_r])}^{(n)}, \CAL{M}_{([c_l, c_r], [d_l, d_r])}^{(n)} \\
        & \ in\ the\ same\ merging\ pass,\ the\ property \\
        & \quad \{x | x \in [a_l, b_r]\} \cap \{x | x \in [c_l, d_r]\} = \emptyset \\
        & holds.
    \end{aligned}
\end{align}

For example, when $k = 1$, there are 4 merging operations in this pass:
    \begin{align}
        & \text{merging of position 1 and position 2}: \CAL{M}_{([1, 1], [2, 2])}^{(8)}, \\
        & \text{merging of position 3 and position 4}: \CAL{M}_{([3, 3], [4, 4])}^{(8)}, \\
        & \text{merging of position 5 and position 6}: \CAL{M}_{([5, 5], [6, 6])}^{(8)}, \\
        & \text{merging of position 7 and position 8}: \CAL{M}_{([7, 7], [8, 8])}^{(8)}.
    \end{align}

This is represented as
\begin{align}
    \times \CAL{M}_{([1, 1], [2, 2])}^{(8)} \times \CAL{M}_{([3, 3], [4, 4])}^{(8)} \times \CAL{M}_{([5, 5], [6, 6])}^{(8)} \times \CAL{M}_{([7, 7], [8, 8])}^{(8)},
\end{align}
it follows that
\begin{align}
    \{1, 2\} \cap \{3, 4\} \cap \{5, 6\} \cap \{7, 8\} = \emptyset.
\end{align}
Similarly, for
\begin{align}
    & \times \CAL{M}_{([1, 2], [3, 4])}^{(8)} \times \CAL{M}_{([5, 6], [7, 8])}^{(8)},
\end{align}
it follows that
\begin{align}
    \{1, 4\} \cap \{5, 8\} = \emptyset.
\end{align}

Thus, we can conclude that in the matrix multiplication corresponding to each pass of merge sort,
the index ranges modified by different transposition matrices in the same pass are pairwise disjoint.
This indicates that \textit{the disjoint matrix multiplication corresponding to each pass satisfies the commutative law}.
The corresponding practical operation is that
one may choose to merge positions 1 and 2 first and positions 7 and 8 last,
or alternatively merge positions 7 and 8 first and positions 1 and 2 last.

Then, we can leverage the commutative law to add parentheses
to the matrix multiplication expression
\begin{align}
    \begin{aligned}
        & \left( \CAL{M}_{([1, 1], [2, 2])}^{(8)} \times \CAL{M}_{([3, 3], [4, 4])}^{(8)} \times \CAL{M}_{([5, 5], [6, 6])}^{(8)} \times \CAL{M}_{([7, 7], [8, 8])}^{(8)} \right) \\
        & \times \left( \CAL{M}_{([1, 2], [3, 4])}^{(8)} \times \CAL{M}_{([5, 6], [7, 8])}^{(8)} \right) \\
        & \times \CAL{M}_{([1, 4], [5, 8])}^{(8)}.
    \end{aligned}
\end{align}
Disjoint matrix multiplications inside the parentheses satisfy the commutative law,
while matrix operations outside the parentheses do not.

We now construct an $\bOIES$ instance for all matrices in the matrix multiplication of merge sort
\begin{align}
    \oieS_{merge} = \{ \oie_{11\_22},\ \oie_{33\_44},\ \oie_{55\_66},\ \oie_{77\_88},\ \oie_{12\_34},\ \oie_{56\_78},\ \oie_{14\_58}\},
\end{align}
an $\bOperandS$ instance
\begin{align}
    \begin{aligned}
        & \OperandS_{merge} = \{\ \CAL{M}_{([1, 1], [2, 2])}^{(8)},\ \CAL{M}_{([3, 3], [4, 4])}^{(8)},\ \CAL{M}_{([5, 5], [6, 6])}^{(8)}, \\
        & \qquad \CAL{M}_{([7, 7], [8, 8])}^{(8)},\ \CAL{M}_{([1, 2], [3, 4])}^{(8)},\ \CAL{M}_{([5, 6], [7, 8])}^{(8)}, \CAL{M}_{([1, 4], [5, 8])}^{(8)}\ \},
    \end{aligned}
\end{align}
and an $\bES$ instance
\begin{align}
    \ES_{merge} = \{ \E_{11\_22},\ \E_{33\_44},\ \E_{55\_66},\ \E_{77\_88},\ \E_{12\_34},\ \E_{56\_78},\ \E_{14\_58} \}.
\end{align}

According to Definition~\ref{def:implementAndProps.endTsProjection}
and Corollary~\ref{cor:uniquenessOperationSetOrbit},
we can use the following sequence operation to build projection using only $\otimes$
    \begin{align}
        \oie_{merge} = \otimes \left( \otimes(\oie_{11\_22}, \oie_{33\_44}, \oie_{55\_66}, \oie_{77\_88}),\ \otimes(\oie_{12\_34}, \oie_{56\_78}),\ \oie_{14\_58} \right)
    \end{align}
or using $\otimes$ outside the parentheses while $\oplus|_{\alpha}^{\beta}$ inside the parentheses
    \begin{align}
        \oie_{merge} = \otimes \left( \oplus|_{\alpha_1}^{\beta_1}(\oie_{11\_22}, \oie_{33\_44}, \oie_{55\_66}, \oie_{77\_88}),\ \oplus|_{\alpha_2}^{\beta_2}(\oie_{12\_34}, \oie_{56\_78}),\ \oie_{14\_58} \right).
    \end{align}
In most industrial contexts,
hardware and software systems are robust enough to ensure that all subtasks in a parallel job run entirely independently with no cross-interference.
This allows us to use the Definition~\ref{def:OPs.Add.Nat} to convert sequential operations into following expression:
\begin{align}
    \oie_{merge} = \otimes \left( \oplus(\oie_{11\_22}, \oie_{33\_44}, \oie_{55\_66}, \oie_{77\_88}),\ \oplus(\oie_{12\_34}, \oie_{56\_78}),\ \oie_{14\_58} \right).
\end{align}

In cases where both $\oplus|_{\alpha}^{\beta}$ (or $\oplus$) and $\otimes$ are employed,
the part running $\oplus|_{\alpha}^{\beta}$ theoretically underpins the practical applicability of parallel technologies.
Since the time complexity of all merge tasks in a single merge operation is of the same order of magnitude,
for a certain additive subexpression in the arithmetic expression,
we directly set the filter domain to the range corresponding to the time complexity of a single merge task.
In this way, we transform the parallelization of merge sort into the following problem:

\BF{Analyze the time complexity of parallel merge sort for sorting $n$ elements with $k$ processes ($1\leq k\leq n$) under the constraint that only tasks on the same layer of the divide-and-conquer binary tree are executed in parallel, while different layers are processed serially (the next layer starts only after all tasks of the current layer complete).}

Merge sort's execution forms a divide-and-conquer binary tree with $\log_2 n$ layers (indexed by $d$, $0\leq d\leq \log_2 n-1$). The $d$-th layer has $2^d$ subtasks, each processing a subarray of length $n/2^d$, and the total elements processed per layer is always $n$ ($2^d \times n/2^d = n$). For the $d$-th layer, the time complexity $T(d)$ depends on $k$ and $2^d$:
- If $2^d \leq k$ (sufficient processes), all subtasks run in parallel, so $T(d) = \mathcal{O}(n/2^d)$.
- If $2^d > k$ (insufficient processes), subtasks run in batches of $k$, so $T(d) = \lceil 2^d/k \rceil \times \mathcal{O}(n/2^d) = \mathcal{O}(n/k)$ (constant factors are ignored in big-$\mathcal{O}$).

Let $d_0 = \log_2 k$ (critical layer where $2^{d_0}=k$). We split the tree into pre-critical layers ($0\leq d\leq d_0-1$, $2^d <k$) and post-critical layers ($d_0\leq d\leq \log_2 n-1$, $2^d \geq k$). The total time of pre-critical layers is:
\begin{align}
    T_{\text{pre}} = \sum_{d=0}^{\log_2 k -1} \mathcal{O}\left(\frac{n}{2^d}\right) = \mathcal{O}\left(n \times \sum_{d=0}^{\log_2 k -1} \frac{1}{2^d}\right) = \mathcal{O}(n),
\end{align}
since the geometric series sum converges to a constant (2). The total number of post-critical layers is $\log_2 n - \log_2 k = \log_2(n/k)$, so their total time is:
\begin{align}
    T_{\text{post}} = \log_2\left(\frac{n}{k}\right) \times \mathcal{O}\left(\frac{n}{k}\right) = \mathcal{O}\left(\frac{n}{k} \log n\right)
\end{align}
where $\log_2(n/k)$ is simplified to $\log n$ (logarithm base does not affect big-$\mathcal{O}$).

The total time complexity is the sum of $T_{\text{pre}}$ and $T_{\text{post}}$:
\begin{align}
    T_{\text{total}} = \mathcal{O}(n) + \mathcal{O}\left(\frac{n}{k} \log n\right).
\end{align}
For $1\leq k\leq n$ and large $n$, $\mathcal{O}(n/k \log n)$ is the dominant term, so the final time complexity is:
\begin{align}
    T_{\text{total}} = \mathcal{O}\left(\frac{n}{k} \log n\right).
\end{align}

In conclusion, under the constraint of parallel execution on the same layer and serial execution across layers, the time complexity of merge sort for $n$ elements with $k$ processes ($1\leq k\leq n$) is $\mathcal{O}\left(\frac{n}{k} \log n\right)$, where a larger $k$ (within $n$) reduces the complexity, and pre-critical layers contribute only a negligible linear term for large $n$.

This concludes our analysis of parallelized merge sort.
In fact, for all sorting algorithms that fall into the category of selection sort or
those based on the ``Divide-and-Conquer'' method,
once we analyze the sorting steps using array multiplication,
we can identify disjoint matrix multiplication subexpressions that satisfy the commutative law
to varying degrees in the operational formulas.
Thus, with $\bOIE$ and sequence operations,
we can rigorously derive feasible parallelization strategies for all such algorithms.

The key novelty of this method lies in performing the entire analysis using purely algebraic approaches.
Specifically, we first formulate the matrix multiplication expressions,
and then derive the divide-and-conquer strategy based on the algebraic properties of matrix multiplication and two sequential operations.
In many special scenarios, the division step no longer operates on complete and uniform blocks (e.g., various interpolation operations).
Instead, we may be confronted with pre-partitioned segments whose division logic is entirely unknown.
For example, the tasks we inherit may involve a matrix multiplication whose origin or construction process is unknowable.
In such cases, the purely algebraic method presented in this subsubsection demonstrates exceptional generality and adaptability.

From a broader perspective,
matrix multiplication is a method widely applied in engineering practices,
including data processing, object rotation and translation, and many other industrial fields.
Matrix multiplication itself is merely a form of algebraic operation,
and numerous operations across all engineering fields can be abstracted into algebraic expressions of
\textit{n}-ary operations in key scenarios.
It is precisely in these scenarios that the concepts of $\bOIE$ and sequence operations can be effectively introduced.
\fi
    \subsection{Comparison in Computer Science}\label{subsec:comp.cs}\EBL
    \ifincludeFile
This subsection conducts a systematic comparative analysis between the $\bOIE$ framework
and mainstream formal frameworks for concurrent systems in computer science,
including Process Algebra, Petri nets, and Allen's interval algebra,
to elucidate the distinctive advantages and theoretical positioning of $\bOIE$ in expressing parallelism,
handling temporal uncertainty, and unifying parallel-serial execution.
The discussion is organized into two subsections:
\begin{itemize}
    \item \textbf{Subsubsection~\ref{subsubsec:comp.cs.1}} addresses the formal unification of parallelism and serialism.
    It demonstrates how the $\bOIE$ framework overcomes the fundamental defect of traditional theories:
    their inability to formalize permutational symmetry,
    by leveraging interval sets and sequence operations,
    and thereby achieves an algebraic unification of parallel and serial execution within a single axiomatic system.
    \item \textbf{Subsubsection~\ref{subsubsec:comp.cs.2}} examines the axiomatic foundations and epistemological orientations.
    It compares the meta-level constraint axiom (Axiom~\ref{axm:OPs.infsble.isom})
    of the $\bOIE$ framework against the object-level rules of traditional formalisms,
    revealing core differences in levels of abstraction, plan-ahead versus behavioral description, informational efficiency, and temporal modeling.
\end{itemize}\fi
        \subsubsection{Formal Unification of Parallelism and Serialism}\label{subsubsec:comp.cs.1}\EBL
        \ifincludeFile
In this subsection, we discuss the $\bOIE$ framework and sequence operations from a computer science theory perspective.
We compare this framework with several concurrent system specification frameworks
—Process Algebra, Petri nets, and Interval Algebra—
focusing on how each approach expresses event relationships,
particularly in how they formalize parallelism.

Process Algebra stands as a foundational and widely adopted framework for concurrent system modeling,
excelling at characterizing the behavioral equivalence of concurrent processes
with communication and synchronization as its core design primitives~\cite{Milner89,BergstraK84},
and has long served as the backbone of formal verification for concurrent and distributed systems.
While its canonical formalism is built on a point-based temporal model that elegantly captures the discrete state transition behaviors of concurrent processes,
it encounters non-trivial challenges when directly modeling interval-centric temporal properties,
including interval fluctuation, observation errors, and multiple candidate execution windows.
Against this backdrop, our sequence operation framework built on $\bOIE$ fully inherits the core compositionality and compositional reasoning capabilities of Process Algebra,
while redefining its basic modeling unit with interval sets.
This extended formalism enables natural, unified expression of key temporal attributes of events, including duration intervals, uncertain start times, and multiple candidate execution schemes.
It thus demonstrates better alignment with the practical modeling requirements of real-time scheduling, physical observation constraints, and dynamic planning scenarios, with notable advantages in the completeness and flexibility of temporal expression.

Petri net models stand as a foundational and widely adopted framework for discrete event dynamic system modeling,
with intuitive and canonical capabilities to characterize state transitions, resource contention, and concurrent execution behaviors~\cite{Petri62,Murata89}.
Since their inception, they have served as a core formal tool for the modeling, analysis of concurrent, distributed, and industrial automation systems.
While their core formalism elegantly captures the structural and behavioral properties of discrete event systems,
the canonical Petri net framework requires non-trivial additional extensions
to natively support multi-candidate temporal planning and interval constraint reasoning.
Furthermore, its state space is prone to combinatorial explosion and rapid expansion as the number and scale of modeled events increase, presenting non-trivial challenges for large-scale dynamic planning and scheduling scenarios.
Against this backdrop, our sequence operation framework built on $\bOIE$ is natively centered on interval algebraic operations,
enabling a streamlined modeling paradigm for temporal planning problems.
This framework supports native feasibility judgment and constraint screening prior to event execution,
eliminating the need for explicit construction of the global state space.
It thus delivers improved lightweight nature and computational efficiency across planning performance, interval constraint reasoning, and dynamic scheme adjustment, and is particularly well-suited for scheduling and decision-making systems that require pre-validation and pre-determination of feasible execution schemes.

Allen's interval algebra stands as a foundational and widely adopted framework
for interval-based temporal relation reasoning~\cite{Allen83,NebelB95},
forming the backbone of classical temporal constraint solving tasks.
While its core formalism is elegantly tailored for fixed intervals, static pairwise relations, and binary temporal constraints,
it encounters non-trivial challenges when generalized to more complex settings involving observation uncertainty, multiple candidate execution paths, and composite sequence operations.
Against this backdrop, our sequence operation framework built on $\bOIE$ retains full compatibility with the core interval relation expression of Allen's algebra,
while extending its expressive scope to support multiple optional intervals, dynamic constraints, and set-level combinatorial operations.
This extended formalism enables unified treatment of observation errors, temporal fluctuations, and parallel reasoning across multiple candidate schemes,
and demonstrates better alignment with the temporal modeling requirements of real-world physical and computer systems,
with notable advantages in expressive power, computability, and system adaptability.

While these foundational frameworks excel in their respective core application domains,
they present notable constraints when extended to the multi-scenario temporal modeling and dynamic planning problems addressed as follows:
\begin{enumerate}
    \item The current interpretations of parallelism in theoretical computer science suffer from two fundamental problems.
    Some approaches take pointwise simultaneity as the foundation and introduce mathematical tools such as lattice theory for analysis, yet they fail to effectively characterize the permutation symmetry of parallelism.
    The others abandon the use of pointwise simultaneity to interpret parallelism
    and directly base parallelism on abstract symmetric structures like
    \begin{align}
        ab \sim ba
    \end{align}
    Although such approaches perfectly express symmetry, they do not formalize the underlying instant relations, resulting in an inadequate theoretical foundation.
    These two issues form a core contradiction in which their defects are complementary yet mutually irreconcilable.
    \item Parallel and serial execution are treated as mutually exclusive dichotomous concepts, lacking a unified algebraic foundation for modeling.
    For classic parallel scenarios with ``intra-layer parallelism and inter-layer seriality''
    such as circuit routing, MapReduce, and parallel merge sort, formal modeling cannot be completed with a unified mathematical model, and only fragmented phased descriptions can be provided.
    When encountering more abstract processes, steps that could be parallelized are easily overlooked through descriptive analysis alone.
\end{enumerate}

The $\bOIE$ framework and sequence operations proposed in this paper completely transcend the limitations of phenomenological description in traditional theories.
Based on set theory and group theory, they construct a closed, self-consistent, and computable formal system for parallel-serial computing.
Their core advantages are reflected in two dimensions:

First, the $\bOIE$ framework does not avoid the critical issue of pointwise equality in the core scenario of modeling time with real numbers;
instead, it extends the expressive form of single-point equality using interval sets and thoroughly overcomes the shortcoming that traditional theories cannot formalize permutation-symmetric structures.
This framework and sequence operations transform the traditional phenomenological description of parallel/serial execution in the computer field into a rigorous, closed, and engineering-implementable algebraic formal system.
They not only address the inherent flaws of traditional parallel computing theories,
but also provide a brand-new theoretical foundation and analytical tool for core computer fields such as parallel algorithm design, real-time task scheduling, distributed systems, and formal program verification.

Second, it achieves the algebraic unification of parallel and serial execution,
enabling end-to-end formal modeling of complex hybrid computing scenarios.
In traditional theories, parallel and serial execution are mutually exclusive dichotomous concepts.
In contrast, through the nested combination of $\oplus|_{\alpha}^{\beta}$ and $\otimes$, the framework can directly perform unified formal modeling of classic ``hybrid parallel and serial'' scenarios such as parallel merge sort, divide-and-conquer algorithms, and pipeline computing.
Meanwhile, based on the algebraic properties of the operations, it can quantitatively derive the time complexity and upper bound of parallelism of parallel algorithms, replacing the qualitative description of algorithm parallelism in traditional models.

In summary, the $\bOIE$ framework rigorously unifies parallelism and serialism within a single axiomatic algebraic system through the operations $\oplus|_{\alpha}^{\beta}$ and $\otimes$.
By replacing phenomenological descriptions and pointwise simultaneity assumptions with interval sets and equipotential time domains, it overcomes the fundamental defects of traditional theories---notably their inability to formalize permutation symmetry and their reliance on global clock precision.
The framework characterizes the mathematical symmetry of concurrent structures via permutational equivalence and orbit-space analysis, while the nested composition of $\oplus|_{\alpha}^{\beta}$ and $\otimes$ enables full-link formal modeling and quantitative complexity derivation for hybrid computing scenarios (e.g., intra-layer parallelism with inter-layer serialism).
Consequently, it provides a computable, engineering-implementable theoretical foundation for parallel algorithm design, real-time task scheduling, and formal program verification.\fi
        \subsubsection{Axiom 1 and Traditional Formal Frameworks}\label{subsubsec:comp.cs.2}\EBL
        \ifincludeFile
Both the $\bOIE$ framework and classical formalisms
including Process Algebra, Allen's interval algebra, and Petri nets
employ axiomatic methods as their theoretical starting points.
Nevertheless, they differ in the type of axioms they adopt,
the level of abstraction at which these axioms are situated,
and the manner in which they relate to physical reality.
These distinctions are essential for understanding the specific contribution of the $\bOIE$ framework.

\paragraph{(1) Level of abstraction: meta-constraints versus object-level rules.}
Axiom~\ref{axm:OPs.infsble.isom} (structure-preserving permutational isomorphism)
is a \emph{meta-level} statement about constraints:
it postulates that the set of infeasible interval combinations remains structurally invariant
when the participating entities are permuted.
From this single assumption, subsequent algebraic properties can be rigorously derived.
By contrast, traditional frameworks usually stipulate rules directly at the \emph{object level}:
Process Algebra treats the commutativity of parallel composition
\begin{align}
    a \mid b \sim b \mid a
\end{align}
as a primitive syntactic equivalence;
Allen's interval algebra enumerates thirteen basic relations as a predefined taxonomy;
and Petri nets require firing rules that are fixed at the net-structure level.
These classical axioms effectively characterize system behavior,
but they are primarily designed to simplify algebraic manipulation
rather than reflect physical constraint invariance.
Axiom~\ref{axm:OPs.infsble.isom}, in contrast,
attempts to provide a bottom-up derivation path that proceeds
from the intrinsic invariance of physical and logical constraints to the resulting algebraic properties.

\paragraph{(2) Epistemological orientation: planning-ahead versus behavioral description.}
Traditional formalisms are primarily oriented toward the precise description of system execution.
For instance, the commutativity law
\[
    a \mid b \sim b \mid a
\]
in Process Algebra
elegantly captures the observational indistinguishability of parallel compositions,
yet it does so under an idealized assumption of exact synchronization at the discrete-state level.
Similarly, Allen's interval algebra operates on the premise that interval endpoints are known precisely,
enabling static relational inference.
The $\bOIE$ framework introduces a complementary, planning-ahead perspective:
Axiom~\ref{axm:OPs.infsble.isom} does not \emph{prescribe} that parallel composition must be commutative;
rather, it \emph{proves} that if real-world constraints satisfy permutational invariance,
then the corresponding planning-level concurrent operation naturally exhibits commutativity.
In this sense, the properties about commutativity become derived theorems
(Properties~\ref{pty:OPs.Add.Permu} and~\ref{pty:OPS.orbitSpace.add.oneElem}) rather than syntactic postulates.

\paragraph{(3) Derivational power and informational efficiency.}
Axiom~\ref{axm:OPs.infsble.isom} is informationally parsimonious:
together with the real-number continuum,
it suffices to derive the single-orbit space of $\oplus|_{\alpha}^{\beta}$,
the multi-orbit space of $\otimes$,
the uniqueness of projected operation sets (Corollaries~\ref{cor:uniqueOpSet} and~\ref{cor:uniquenessOperationSetOrbit}),
and the rigorous separation of process symmetry from outcome symmetry in probability theory (Subsections~\ref{subsec:app.math} and ~\ref{subsec:comp.math} ).
Traditional frameworks typically introduce multiple and mutually independent axioms
to cover these various aspects:
Process Algebra requires separate axioms for commutativity, associativity, synchronization, and bisimulation equivalence;
Allen's algebra needs thirteen relation definitions plus transitivity tables;
Temporal logics require discrete time steps, atomicity assumptions, and modal operators.
These axioms are usually complementary rather than derivable from one another.
The $\bOIE$ framework attempts to cover a broad range of properties from a compact set of structural assumptions,
reflecting a different formalization strategy rather than a claim of universal superiority.

\paragraph{(4) Time models: continuous interval uncertainty versus ideal precision.}
The force of Axiom~\ref{axm:OPs.infsble.isom} is inseparable from the $\bOIE$ treatment of time as continuous real-number domains (sets)
that incorporates observational uncertainty.
For example,
it acknowledges that even atomic clocks possess finite systematic error,
so that ``simultaneous start'' can only be meaningfully interpreted as equal opportunity within an equipotential domain.
Traditional frameworks generally adopt idealized temporal assumptions:
Process Algebra relies on a point-based discrete-state model with implicit global-clock synchronization;
Allen's interval algebra, although it uses intervals, defines its thirteen relations
as combinations of exact endpoint comparisons
and does not natively accommodate scenarios where an event may execute in $[0,1)$ or $[21,22)$;
temporal logics presuppose a precisely enumerable discrete time axis.
These idealizations are effective within their intended application domains,
whereas the $\bOIE$ framework seeks to leave greater formal room for temporal uncertainty and multiple candidate execution windows.

In summary,
the difference between the $\bOIE$ framework and traditional formalisms
lies not merely in the quantity of axioms,
but in the logical relationship between axioms and physical reality.
Axiom~\ref{axm:OPs.infsble.isom} is not intended as a direct replacement for the foundations of Process Algebra or the relation tables of Allen's algebra;
it offers an alternative path that connects physical and logical constraints to algebraic structures through derivation.
By situating the formalism within a constraint-aware, pre-execution planning paradigm,
the $\bOIE$ framework provides a unified setting in which properties of parallelism, sequentiality, and simultaneity
can be obtained as consequences of structural invariance,
rather than being introduced as independent syntactic postulates.\fi
    \subsection{Application to Probability Theory: Drawing Lots}\label{subsec:app.math}\EBL
    \ifincludeFile
Consider a bag containing $n$ balls,
of which $m$ are red and $n-m$ are black.
Let $k$ participants ($1 < k \leq m < n$) each draw one ball without replacement.
We distinguish two sampling scenarios:
\textit{simultaneous sampling}, where all $k$ participants reach into the bag concurrently,
and \textit{sequential sampling}, where participants draw one after another according to a predetermined order.

Classical probability theory establishes that each participant's marginal probability of drawing a red ball is $m/n$ in both scenarios.
For simultaneous sampling, this follows from combinatorial uniformity;
for sequential sampling, it follows from the chain rule of conditional probability.

The $\bOIE$ framework reveals that this equivalence of outcomes arises despite fundamentally distinct algebraic architectures:
simultaneous sampling corresponds to $\oplus|_{\alpha}^{\beta}$ with its single-orbit permutational equivalence,
while sequential sampling corresponds to $\otimes$ with its strict sequential ordering and potentially multiple orbital equivalence classes.
The key insight is that although the algebraic paths differ, both operations project onto the same probability distribution in the observable outcome space.

We model each participant $i$ as an $\bAtomOIE$ instance $\atomOie_i$ mapping to the atomic event $\atomE_i$ (the action of drawing a ball).
Each $\atomOie_i$ carries a feasible interval set
\begin{align}
    \CAL{I}_{\atomOie_i} = \{(\TSs, \TSe) \mid \alpha \leq \TSs < \TSe \leq \beta \land \TSe - \TSs = \delta\},
\end{align}
where $[\alpha, \beta)$ denotes the shared temporal domain of the sampling process and $\delta > 0$
represents the non-zero duration required to physically execute the draw.

In the simultaneous scenario, all $k$ participants possess equal opportunity to commence their draws within the shared domain $[\alpha, \beta)$,
with no privileged precedence.
This embodies the semantics of $\oplus|_{\alpha}^{\beta}$ (Definition~\ref{def:OPs.Add}).
Executing the operation
\begin{align}
    \oie_{simul} = \oplus|_{\alpha}^{\beta}(\atomOie_1, \atomOie_2, \ldots, \atomOie_k)
\end{align}
yields a result whose orbit space $\CAL{O}(\oieS, \oplus|_{\alpha}^{\beta})$
contains exactly one equivalence class (Property~\ref{pty:OPS.orbitSpace.add.oneElem}).
Consequently, for any two index tuples $\idxT_1$ and $\idxT_2$,
the results are permutationally equivalent:
\begin{align}
    \oplus|_{\alpha}^{\beta}(\oieS, \idxT_1) \stackrel{\CAL{M}}{\sim} \oplus|_{\alpha}^{\beta}(\oieS, \idxT_2).
\end{align}
This algebraic symmetry reflects that the assignment of indices represents merely a descriptive convention,
not an ontological distinction,
since all participants act within the same temporal window without sequential constraint.

Conversely, the sequential scenario enforces a strict order
wherein participant $i$ must complete their draws before participant $i+1$ begins.
This corresponds to $\otimes$ (Definition~\ref{def:Ops.Multi}),
which applies an ascending-ordered filter (Definition~\ref{def:OPs.Multi.CompleteAscOrderFilteredSubTwoTplTS}) requiring $\TSe_i \leq \TSs_{i+1}$ for all consecutive pairs.
The operation
\begin{align}
    \oie_{seq} = \otimes(\atomOie_1, \atomOie_2, \ldots, \atomOie_k)
\end{align}
generates a distinct algebraic structure: by Property~\ref{pty:OPs.Multi.Permu},
changing the index order generally produces a result that is \textit{not} permutationally equivalent to the original.
Thus, the orbit space $\CAL{O}(\oieS, \otimes)$ may contain multiple equivalence classes (Property~\ref{pty:OPS.orbitSpace.multi.multiElem}),
each corresponding to a specific permutation of the drawing order.
The algebraic structure explicitly encodes the temporal ``unfairness'' of precedence.

However, despite these structural differences, both operations yield identical marginal probability distributions when projected onto the outcome space.
Under $\oplus|_{\alpha}^{\beta}$, the single-orbit property ensures that every participant appears in every ordinal position with equal frequency across the feasible set $\CAL{F}$,
directly yielding the uniform marginal probability $m/n$.
Under $\otimes$, the same marginal probability emerges
from the combinatorial properties of conditional probability chains:
the product of conditional probabilities arranged in any order preserves the individual marginal $m/n$ for each position.

This distinction becomes clear when examining the projection to $n$-ary operations (Definition~\ref{def:implementAndProps.endTsProjection}).
For simultaneous sampling, the projection generates a set of $n$-ary operations
respecting permutational equivalence (Corollary~\ref{thm:implementAndProps.fit.add}),
reflecting that the order of drawing in the expression
does not affect the final allocation.
For sequential sampling, the projection yields a single ordered expression (Property~\ref{prop:multiOpNumber}),
yet when the operation $\odot$ for the $\OperationS$ in the projection is defined as probability calculation (specifically, the counting measure over red ball allocations),
both projections yield the identical probability value $m/n$.

The $\bOIE$ framework thus reveals how identical stochastic outcomes may arise from algebraically distinct event structures:
one collapsing all permutations into equipotentiality ($\oplus|_{\alpha}^{\beta}$),
and the other preserving strict causal ordering ($\otimes$).
This provides a rigorous formal distinction between \TT{process symmetry} and \TT{outcome symmetry} in probabilistic sampling.
\fi
    \subsection{Comparison in Probability Theory: Process Symmetry and Outcome Symmetry in Classical Probability}\label{subsec:comp.math}\EBL
    \ifincludeFile
The traditional classical probability theory's analysis of ``simultaneous sampling'' and ``sequential sampling without replacement''
has always revolved around the fairness of results:
it proves through combinatorial counting methods that the marginal probability of each participant drawing the target sample in simultaneous sampling is \( m/n \),
and verifies through conditional probability derivation that the marginal probability of each participant in sequential sampling is also \( m/n \),
ultimately concluding that ``the results of the two sampling methods are equivalent''.
However, the traditional analysis has three essential limitations:

\BF{I},
it adopts completely separate mathematical paradigms for the analysis of the two sampling scenarios:
simultaneous sampling is classified as a combinatorial problem,
while sequential sampling is classified as a conditional probability problem,
without establishing a unified formal description framework.

\BF{II},
it only focuses on the probability symmetry at the result level,
never providing a rigorous mathematical definition and characterization of ``the symmetry of the sampling process'',
and thus cannot explain the core contradiction of ``completely heterogeneous processes but completely consistent results'' from the underlying level.

\BF{III},
it simplifies the temporal difference of events into a difference in combinatorial rules of
``whether to consider the permutation order'',
completely ignoring the underlying influence of the temporal constraints of event execution on the probability results,
and also failing to distinguish between the descriptive meaning and the ontological meaning of indices in the sampling process.

Based on the $\bOIE$ algebraic framework,
this research achieves core theoretical innovations in 4 dimensions for the classical sampling problems in classical probability:

\BF{1}, it constructs a cross-scenario unified algebraic modeling system.
Through the two core operations of $\oplus|_{\alpha}^{\beta}$ and $\otimes$,
the two heterogeneous scenarios of ``simultaneous sampling'' and ``sequential sampling''
are incorporated into the same algebraic formal framework.
Among them, the ``equal opportunity, no priority order'' characteristic of simultaneous sampling is modeled as the $\oplus|_{\alpha}^{\beta}$ operation,
and the ``strict temporal order, causal priority'' characteristic of sequential sampling is modeled as the \( \otimes \) operation.
The process differences between the two scenarios are transformed into derivable and verifiable algebraic property differences.
This solves the problem of inconsistent analysis paradigms for the two sampling scenarios in traditional analysis,
and provides a general formal modeling tool for probabilistic events involving temporal order.

\BF{2}, it rigorously defines the mathematical connotations of process symmetry and result symmetry,
and proves their separability.
Traditional probability theory has never formed a formal definition of ``process symmetry'',
while this research clarifies the algebraic essence of process symmetry
through the theory of orbit space and permutational equivalence:
the result of simultaneous sampling modeled by the $\oplus|_{\alpha}^{\beta}$ operation
is always in a single orbit space,
and any index permutation maintains the permutational equivalence of the algebraic structure,
which is complete algebraic symmetry at the process level;
the result of sequential sampling modeled by the \( \otimes \) operation
may form a multi-orbit space,
and changes in index order will destroy the permutational equivalence of the algebraic structure,
which is algebraic asymmetry at the process level.
At the same time, this research further proves that the two operations with completely opposite process symmetries
can obtain completely consistent marginal probability distributions
when projected onto the probability result space based on the counting measure,
strictly proving the separability of process symmetry and result symmetry
from the underlying algebraic level,
and filling the gap in the formal characterization of process symmetry
in traditional probability theory.

\BF{3}, it reconstructs the underlying logic of the fairness of classical probability,
and clarifies the descriptive and ontological boundaries of indices.
For the classic proposition of ``the fairness of drawing lots'' in classical probability,
traditional analysis attributes fairness to the result characteristic of ``equal probability of each participant being selected'',
while this research reveals its essence through the $\bOIE$ framework:
the fairness of drawing lots does not come from the consistency of the sampling process,
but from the permutation invariance of the result projection operation (counting measure).
At the same time, this research clarifies the dual attributes of indices in the sampling process:
in the concurrent scenario of the \( \oplus \) operation,
the permutation of indices only changes the descriptive form, does not change the algebraic structure, and only has the meaning of a descriptive convention.
In the serial scenario of the \( \otimes \) operation,
the order of indices directly determines the causal temporal order and algebraic structure of events,
and has causal significance at the ontological level.
This distinction provides a novel ontological explanation for the core concept of ``ordered and unordered'' in classical probability.

\BF{4}, it establishes a mapping chain between temporal algebraic structures and probability results,
expanding the formal boundary of probability theory.
The traditional analysis of classical probability strips away the temporal attributes of events,
while this research transforms the temporal constraints of events into operable algebraic structures through the $\bOIE$ framework,
establishing a complete logical chain of ``projection from temporal algebraic operations to probability results''.
This makes probability analysis no longer limited to counting final results.
Instead, it can deeply characterize how temporal structure and constraints of event execution influence probability outcomes.
It also provides a novel theoretical foundation and analysis paradigm for integrating classical probability with formal verification, temporal logic, and concurrent system analysis.

In summary, the analysis of classical sampling problems based on the $\bOIE$ framework in this research not only provides a more rigorous and foundational algebraic explanation for the classic proposition of probability fairness,
but also achieves the formal expansion of classical probability theory through the separation and unification of process and result symmetries, and lays the foundation for the subsequent introduction of probability measures and the expansion of randomized temporal event modeling capabilities by the $\bOIE$ framework.\fi
    \subsection{An $\bOIE$ Reinterpretation of Galileo's Leaning Tower Experiment: The Legitimacy of the Repeatability of Physical Experiments}\label{subsec:app.phy}\EBL
    \ifincludeFile
Galileo's legendary free-fall experiment at the Leaning Tower of Pisa (circa 1590)
is traditionally described as verifying that objects of different masses,
released simultaneously from the same height, will ``hit the ground at the same time.''
However, from the perspective of the $\bOIE$ framework,
this classical formulation confronts the same epistemological dilemma
that plagues the notion of ``simultaneous start'' in speed competitions:
the \emph{unverifiability of pointwise simultaneity}

The conventional narrative implicitly assumes two hypotheses:
both spheres begin their descent at a precise instant (pointwise temporal equality),
and both spheres touch the ground at a precise instant.
As emphasized, we cannot verify the ``pointwise equality'' of the release instants of two balls.

Using the $\bOIE$ framework, we reconstruct the experiment.
Let the heavy ball's fall be $\E_{\text{heavy}}$,
and the light ball's fall be $\E_{\text{light}}$.
Then, construct two $\bOIE$ instances:
\begin{align}
\oie_{\text{heavy}} &= \big( (), \F_{\text{heavy}}, \I_{\text{heavy}}, \{\EStar_{\text{heavy}}\} \big), \\
\oie_{\text{light}} &= \big( (), \F_{\text{light}}, \I_{\text{light}}, \{\EStar_{\text{light}}\} \big),
\end{align}
where $\I_{\text{heavy}}$ and $\I_{\text{light}}$ are feasible interval sets satisfying:
\begin{align}
    \I_{\text{heavy}} = \I_{\text{light}} = \{ (\TSs, \TSe) \mid \TSs \geq \alpha \land \TSe \leq \beta \land \TSe - \TSs = \sqrt{2h/g} \}
\end{align}
with $[\alpha, \beta)$ denoting the experimental time domain (equipotential domain),
$h$ the tower height, and $g$ the gravitational acceleration.

The ``simultaneous'' release of both spheres can be modeled as $\oplus|_{\alpha}^{\beta}$
rather than pointwise simultaneity:
\begin{align}
    \oie_{\text{Galileo}} = \oplus|_{\beta}^{\alpha}(\oie_{\text{heavy}}, \oie_{\text{light}}).
\end{align}

The critical insight lies in Step 3 of Complete Sequence Addition (Definition~\ref{def:OPs.Add}):
The domain-filtered subset (Definition~\ref{def:OPs.Add.DomainFilteredSubTwoTplTS})
requires that all participating events have opportunities to start and end within the shared time domain $[\alpha, \beta)$,
without forcing them to share identical starting timestamps.
By Definition~\ref{def:implementAndProps.endTsProjection} (projection based on ending timestamps),
we project $oie_{\text{Galileo}}$ onto the kinematic outcome space.
Let $\odot$ denote the $n$-ary finitary operation ``Get the minimum landing time'':
\begin{align}
    \oplus|_{\beta}^{\alpha}(\{oie_{\text{heavy}}, \oie_{\text{light}}\}, (1, 2)) \xrightarrow[\ES, \OperandS, \odot]{\text{ascending order of } \bTSe} \OperationS
\end{align}
where $\odot$ satisfies permutational equivalence:
\begin{align}
    \odot(T_{\text{heavy}}, T_{\text{light}}) = \odot(T_{\text{light}}, T_{\text{heavy}}),
\end{align}
where $T_{\text{heavy}}$ and $T_{\text{light}}$ are the landing times of the two spheres.
By Corollary~\ref{cor:uniquenessOperationSetOrbit}, all operations in $\OperationS$ yield the same outcome.

So, the experiment verifies not that ``both spheres land at precisely the same instant $t$'' (pointwise simultaneity),
but rather that ``within the equipotential time domain $[\alpha, \beta)$, mass difference does not affect the projection''.
When the experimenter releases the balls to fall at the so-called ``same time'',
one only needs to match the domain formed by $\alpha$ and $\beta$ to the actual experimental conditions
in accordance with the precision specified in the experimental design,
and all metrics of this experiment are fully captured.
In fact, this experiment can also be modeled using projection of $\otimes$ (one by one),
as it does not even need to satisfy permutational equivalence.
Its constraints even more lenient than those of $\oplus|_{\alpha}^{\beta}$.

This algebraic insight extends beyond the reinterpretation of Galileo's specific experiment to illuminate ``\TT{The foundational epistemology of experimental repeatability in physics}''.
The legitimacy of repeatedly reproducible physical experiments,
whether conducted simultaneously in parallel or sequentially across time,
does not fundamentally rely on the strict attainment of pointwise temporal coincidence (simultaneous start) or exact spatial overlap across trials.
Rather, it depends critically upon the \textbf{preservation of algebraic structure},
specifically the permutational equivalence of the $n$-ary operation in the projection (Definition~\ref{def:implementAndProps.endTsProjection}).

When experimental instances are modeled via $\oplus|_{\alpha}^{\beta}$,
the single-orbit property (Property~\ref{pty:OPS.orbitSpace.add.oneElem}) ensures that all permutations of operand ordering yield permutationally equivalent results.
This algebraic symmetry guarantees that experiments performed within a shared equipotential domain,
regardless of minor temporal displacements or rearrangements,
remain structurally equivalent and thus scientifically legitimate as ``repeated'' instances.
\textbf{The validity of the experimental method derives not from the absolute simultaneity of initiation, but from the invariance of the operational outcome under index permutation.}

Conversely, $\otimes$ does not enforce permutational equivalence (Property~\ref{pty:OPS.orbitSpace.multi.multiElem}),
thereby generating multiple distinct orbits in the orbit space.
This structural characteristic actually confers a different kind of experimental robustness: by mandating strict sequential ordering without requiring permutational invariance,
$\otimes$ accommodates experimental repetition across extended temporal intervals with weaker constraints on temporal overlap.
The sequential structure tolerates, and indeed formalizes, the temporal separation between experimental instances,
rendering cross-temporal replication algebraically tractable without demanding simultaneous occurrence.

\TT{In summary, the legitimacy of repeatable physical experiments is subject to algebraic constraints rather than strict spatiotemporal coincidence.}
The choice between concurrent ($\oplus$) and sequential ($\otimes$) experimental architectures determines the specific symmetry requirements:
$\oplus|_{\alpha}^{\beta}$ demands permutation invariance within an equipotential time window,
while $\otimes$ relaxes simultaneity requirements through strict temporal ordering.
This reveals that physical repeatability is not merely an empirical accident of temporal alignment,
but a \TT{structured algebraic property} grounded in the orbit space characteristics of the underlying sequence operations.
\fi
    \subsection{Comparison in Physics: Unattainable Pointwise Equality and Epistemological Equipotentiality}\label{subsec:comp.phy}
    \ifincludeFile
This subsection continues to take the classic scenario of the 100-meter dash as the carrier,
and conducts an in-depth discussion around the first core question raised in the introduction:
``\TT{The formal description of the simultaneity of the starting moments of object motion from the observer's perspective}''.

Simultaneity is a fundamental core issue in physics.
Newtonian mechanics upholds the concept of absolute simultaneity,
holding that time is a unified and independent physical quantity across the universe.
The simultaneity of two events has absolute meaning,
independent of the observer's state of motion, and remains consistent in any reference frame.
In contrast, relativistic mechanics proposes the relativity of simultaneity,
pointing out that time and space are inseparable,
that simultaneity depends on the observer's inertial reference frame.
Observers in relative motion will have different judgments on the simultaneity of events,
and there is no absolutely unified simultaneity.

These two classical cognitions of simultaneity from the observer's perspective essentially take \TT{the pointwise precise equality of timestamps} as the core of the discussion:
Relativistic mechanics breaks the absolute view of space-time for observers,
but its analysis of simultaneity still relies on the Lorentz transformation for time coordinate conversion between different frames.
Crucially, it does not depart from the ideal hypothesis of pointwise simultaneity for observers.
For the simultaneity of the starting moments of two events,
this paper does not intend to refute the logical self-consistency of pointwise time equality at the theoretical level
(whether in the framework of Newtonian mechanics or relativity).
The core contradiction we focus on is:
Even if there is a theoretically pointwise equal starting moment,
can humans achieve strict verification of it through observation,
or can we ensure that the observed ``pointwise simultaneity'' is completely consistent with the objective physical facts?
As of 2026, this cannot be strictly guaranteed at the physical and epistemological levels.

We hereby clarify the boundaries of the discussion again:
``\TT{From the observer's perspective}'' limits the discussion in this paper to the epistemological level,
and ``\TT{Whether the starting moments of two events (or the moments when two objects start to move) are the same}''
limits our core focus to the equality relationship of the two moments, rather than the specific values of the moments.
We have emphasized in the introduction that the systematic error of the most advanced atomic clock in the world is still $8 \times 10^{-19}$ seconds,
and no observation of any moment can achieve absolute accuracy.
But the event and motion must have an objective starting moment independent of observation.
For example, we can stipulate that the moment when an athlete's foot leaves the ground in a 100-meter dash is the starting moment of his motion.
This moment exists objectively, but we cannot obtain its absolutely accurate value through observation.
In other words, for two objectively existing moment variables \(t_1\) and \(t_2\),
we do not discuss their specific values,
but only focus on the observability and verifiability of the relationship ``Whether the two are strictly equal''.

We believe that the reason why the traditional formal system cannot guarantee that its judgment of simultaneity from the observer's perspective
is consistent with the objective reality lies in the essential limitation of the cognition of the ``Observation'' behavior in the traditional physics research paradigm.
The classical physics research path can be summarized as
\begin{align}
    \begin{aligned}
        & \text{Observation and classification of objective phenomena} \\
        \to & \text{Proposal of hypotheses/laws} \\
        \to & \text{Deductive inference (prediction)} \\
        \to & \text{Experimental verification}.
    \end{aligned}
\end{align}
The first step of this paradigm takes objective phenomena as the only starting point of research,
so that no matter how much subjective analysis the observer introduces in the subsequent steps,
the observer is trapped in passive observation dependence from the very beginning.
The core idea of the $\bOIE$ framework and sequence operations is to add a pre-step of ``Subjective introduction of algebraic system''
before ``Observation and classification of objective phenomena''~\cite{Whitehead1978,kant1998critique,kant2004metaphysical},
so that the research path becomes
\begin{align}
    \begin{aligned}
        & \text{Subjective introduction of algebraic structure} \\
        \to & \text{Observation and classification of objective phenomena} \\
        \to & \text{Proposal of hypotheses/laws} \\
        \to & \text{Deductive inference (prediction)} \\
        \to & \text{Experimental verification}.
    \end{aligned}
\end{align}
Thus, the object of the observation behavior is no longer completely limited to objective phenomena,
but adds the mathematical structure preset by the research subject and matched with the object,
which fundamentally eliminates the constraint of observation accuracy on the theoretical system.

Based on this, we put forward the core viewpoint:
``\TT{As of 2026, the conclusion of pointwise simultaneity obtained from the observer's perspective can only be falsified, not verified}''.
Existing physical theories cannot provide a rigorous formal description of this subtle epistemological relationship,
while the $\bOIE$ framework and sequence operations proposed in this paper can handle the inherent uncertainty
caused by observation errors through strict algebraic formalization means from an epistemological perspective.
If we use the $\oplus|_{\alpha}^{\beta}$ operation defined earlier to model concurrent events,
it naturally introduces an equipotential time domain.
All events (or physical motions) have feasible intervals to start and end within this domain,
and all have opportunities to start at the left boundary of the domain.
This formal system only achieves, to a certain extent, the core physical demand of ``fair start''
that pointwise simultaneity intends to realize.

We need to clarify that $\oplus|_{\alpha}^{\beta}$ is not a compatibility or approximation
of the traditional ``pointwise simultaneity'',
because we have fundamentally abandoned this ideal hypothesis that cannot be realized through observation
(We might get there one day, but we can be more accountable for the present, for the past thousands of years, and for everything before that day).
$\oplus|_{\alpha}^{\beta}$ is only a formal scheme to replace the old cognitive paradigm for appropriate situations.
As mentioned in the last section of this paper,
there are more extended forms of sequence operations,
which are by no means limited to $\oplus|_{\alpha}^{\beta}$ and \(\otimes\).
A variety of alternative formal frameworks may exist that can complement or extend the conventional notion of simultaneity.

In summary, the traditional
``Pointwise simultaneity from the observer's perspective''
is not the same concept as the
``equipotential domain start under sequence operations'' proposed in this paper.
The latter is not designed to be compatible with the former,
because the former is physically unreachable in our epistemological framework.
Our algebraic formalism builds a new research paradigm with planning as the core and active constraints as the basis:
``Simultaneous start'' is reconceptualized as equal opportunity within a time domain,
and the limitations of observation are no longer regarded as approximation errors to the ideal theory,
but as derivative quantities layered on a more fundamental, observer-independent algebraic structure.
``\TT{The world need not conform to the measurement accuracy of our clocks;
rather, clocks are just tools for us to explore the world, whose algebraic structure can be actively introduced and defined by the research subject}''.
And ``\TT{A simultaneous start is not a reward the world grants to us. We inherently deserve our own right to participate and create}''~\cite{chen1991youwu,kern2011flower}.

This epistemological paradigm shift has an applicability far beyond the analysis of simultaneity in the time dimension.
Once we accept the conception that
``The premise of observation consists partly of the mathematical structure subjectively preset by observers,
and partly of the phenomena of the objective world''~\cite{bridgman1936nature, von1984constructivism},
then from the observer's perspective, in addition to the simultaneity of the starting moments of events,
the equality judgment of all physical quantities, such as the coincidence of the starting positions of motions and the equality of the temperature of physical systems, needs to be re-examined.
The traditional formal description of such equality relations takes the unobservable and unverifiable pointwise precise equality as the core ideal hypothesis,
which has exactly the same epistemological limitations as the traditional cognition of time simultaneity.
The ``time variabilization'' extension mentioned in the future work of the last section of this paper is a concrete extension of this idea.
We believe that both the underlying definition of physical quantity equality in theoretical physics and the formal treatment of measurement errors in experimental physics can obtain brand-new analytical ideas and formal tools from the $\bOIE$ framework and research paradigm proposed in this paper.\fi

\section{A Finite Commutative Semigroup with Absorbing Element and Reflexive Absorption Property}\label{sec:cayley}\EBL
\ifincludeFile
Let a binary operation $\oslash$ satisfy closure, associativity, commutativity, the existence of a zero element and the property of reflexive absorption.
The rule of this operation is similar to combining its operands:
If the two operands share a common element,
the result is the zero element; otherwise, the result is a non-zero element.

Consider a special class of $\bOIES$ instances comprising $n$ ($n>2$) distinct $\bOIE$ instances,
such that for any $k$ instances selected therefrom with $2 \le k \le n$,
performing $\oplus|_{\alpha}^{\beta}$ upon them in any order yields a result that is not $\voidError$.
In this section, we investigate the relationship between the special semigroup formed by the $\oslash$ mentioned above
and the algebraic system constituted by such $\bOIES$ instances together with $\oplus|_{\alpha}^{\beta}$.

\begin{itemize}
    \item \textbf{Subsection~\ref{subsec:cayley.semigroup.algebraStructure}}
    formalizes a finite commutative semigroup with an absorbing element and reflexive absorption property,
    constructs its $N$-dimensional Full CSA Cayley table via Newton's binomial combinatorial structure,
    and establishes a bijective mapping between this special semigroup and the orbit space of a special class of $\bOIES$ instances under $\oplus|_{\alpha}^{\beta}$.
    \item \textbf{Subsection~\ref{subsec:cayley.and}}
    demonstrates the universal applicability of this algebraic structure
    by revealing its formal correspondence with the referential distinctness constraint of natural language conjunction ``And''
    and with composition/inheritance semantics in object-oriented programming,
    thereby unifying linguistic, computational, and physical feasibility constraints under a single algebraic grammar.
\end{itemize}
\fi
    \subsection{The Bijection between a Special Semigroup and a Special Algebra Structure of $\bOIES$ on $\oplus|_{\alpha}^{\beta}$}\label{subsec:cayley.semigroup.algebraStructure}\EBL
    \ifincludeFile
In this subsection,
we discuss a special semigroup,
as well as a special algebraic system composed of a particular class of $\bOIES$ instances and $\oplus|_{\alpha}^{\beta}$.
We construct the Cayley diagram for the special semigroup and explore the connections between these two distinct special algebraic systems.

\begin{itemize}
    \item \textbf{Subsubsection~\ref{subsubsec:cayley.semigroup}}:
    This subsubsection axiomatically defines a special finite commutative semigroup with an absorbing element and reflexive absorption property.
    It establishes closure, associativity, commutativity, absorption, and generalized reflexive absorption as intrinsic axioms,
    and presents Algorithm 1 for constructing the $N$-dimensional Full CSA Cayley table based on Newton's binomial combinatorial structure.
    \item \textbf{Subsubsection~\ref{subsubsec:cayley.algebraStructure}}:
    This subsubsection identifies a special class of $\bOIES$ instances whose $\oplus|_{\alpha}^{\beta}$
    yields an algebraic structure satisfying uniqueness and non-degeneracy assumptions.
    It defines the orbital completion set under permutational equivalence and proves that its cardinality equals $2^{n}$.
    \item \textbf{Subsubsection~\ref{subsubsec:cayley.bijection}}:
    This subsubsection constructs an explicit bijection $\Psi$ from the special semigroup to the orbit space of the special $\bOIES$ instances,
    demonstrating that the two algebraic structures are bijective.
\end{itemize}\fi
        \subsubsection{A Special Semigroup}\label{subsubsec:cayley.semigroup}\EBL
        \ifincludeFile
We first formalize the special semigroup.
Suppose there is a set with $n$ distinct elements
\begin{align}
    \CAL{S}^{gen} = \{ v_1,\,v_2,\,\dots,\,v_n \}
\end{align}
as base elements.

Suppose there is a binary operation $\oslash$ on $\CAL{S}^{gen}$,
satisfying the properties of closure, associativity, and commutativity,
the existence of a zero element, and reflexive absorption.

Using $v_1,\ v_2,\ \cdots,\ v_n$ as the generators, let
\begin{align}
    \CAL{P}(\CAL{S}^{gen}) = \{v_{0}\} \cup \{ \oslash(\CAL{S}^{sub}) \mid \CAL{S}^{sub} \subseteq \CAL{S}^{gen}, \CAL{S}^{sub} \neq \emptyset \},
\end{align}
where
\begin{align}
    \CAL{S}^{sub} & = \{ v_{i_1}, v_{i_2}, \cdots, v_{i_k}\}\ (1 \leq k \leq n), \\
    \oslash(\CAL{S}^{sub}) & = ( \cdots ( (v_{i_1} \oslash v_{i_2}) \oslash v_{i_3} ) \cdots v_{i_k})
\end{align}
and $v_{0}$ is the absorbing element.
The cardinality is:
\begin{align}
    \CARDI{\CAL{P}(\CAL{S}^{gen})}=2^n.
\end{align}
The algebraic structure $(\CAL{P}(\CAL{S}^{gen}), \oslash)$ is called a ``\BF{Finite commutative semigroup with absorbing element and reflexive absorption property}''.
Its algebraic properties are as follows:
\begin{enumerate}
    \item \textit{Closure}:
    \begin{align}
        \forall v_a, v_b \in \CAL{P}(\CAL{S}^{gen}): v_a \oslash v_b \in \CAL{P}(\CAL{S}^{gen});
    \end{align}
    \item \textit{Associativity}:
    \begin{align}
        \forall v_a, v_b, v_c \in \CAL{P}(\CAL{S}^{gen}): (v_a \oslash v_b) \oslash v_c = v_a \oslash (v_b \oslash v_c);
    \end{align}
    \item \textit{Commutativity}:
    \begin{align}
        \forall v_a, v_b \in \CAL{P}(\CAL{S}^{gen}): v_a \oslash v_b = v_b \oslash v_a;
    \end{align}
    \item \textit{Absorbing Element Property}:
    \begin{align}
        \exists v_{0}\in \CAL{P}(\CAL{S}^{gen}), \forall v_{\lambda} \in \CAL{P}(\CAL{S}^{gen}): v_{\lambda} \oslash v_{0} = v_{0} \oslash v_{\lambda} = v_{0};
    \end{align}
    \item \textit{Generalized Reflexive Absorption Property}:
    \begin{align}
        \forall v_{\lambda} \in \CAL{P}(\CAL{S}^{gen}): v_{\lambda} \oslash v_{\lambda} = v_{0};
    \end{align}
    \item \textit{Uniqueness}:
    \begin{align}
        \forall \CAL{S}^{sub}_A, \CAL{S}^{sub}_B \subseteq \CAL{S}^{gen}: \CAL{S}^{sub}_A \neq \CAL{S}^{sub}_B \Rightarrow \oslash(\CAL{S}^{sub}_A) \neq \oslash(\CAL{S}^{sub}_B);
    \end{align}
    \item \textit{Non-degeneracy}:
    \begin{align}
        \forall \CAL{S}^{sub}_{\lambda} \subseteq \CAL{P}(\CAL{S}^{gen}): \CARDI{\CAL{S}^{sub}_{\lambda}} > 1 \Rightarrow \oslash(\CAL{S}^{sub}_{\lambda}) \neq v_{0}.
    \end{align}
\end{enumerate}

These properties of the algebraic structure are axiomatic intrinsic constraints, not derived from general semigroups or other more basic algebraic structures.
We thus define it as a distinct class of special semigroups.

%


Now we proceed to construct the Cayley table for this finite semigroup.
The top priority is to determine the arrangement of elements
for the row and column headers.
We use properties of the Newton binomial theorem to build this boundary.

We build the structure of rows and columns as follows
\begin{align}
    \left( v_{0}, v_1, v_2, \cdots, v_n, \oslash(\CAL{S}^{sub}_1), \oslash(\CAL{S}^{sub}_2), \cdots, \oslash(\CAL{S}^{sub}_n) \right)
\end{align}
This sequence is ordered by the cardinality of the subsets, following the combinatorial structure of the nth Newton binomial expansion.
Once the rows and columns have been constructed, we can proceed to directly build the Cayley table.
Below is the algorithm we use for constructing the Cayley table.

\begin{algorithm}[H]
    \caption{Generation algorithm for N-dimensional Cayley table for finite commutative semigroup with absorbing element and reflexive absorption property}
    \label{alg:algo}
    \begin{algorithmic}[1]
        \State $comb\_collection \gets [\ ]$
        \State $binomial\_theorem\_collection \gets [\ ]$\\

        \State Push\ [\ $v_{0}$\ ]\ into\ $binomial\_theorem\_collection$

        \For{$i \gets 0$ to $N - 1$}
            \State $cur\_combos \gets [\ ]$
            \If{$i = 0$}
                \State {$cur\_combos \gets [\ \left[v_1 \right], \left[v_2\right], \cdots \left[v_N\right]\ ]$}
            \Else
                \For {$pre\_combo\_val$\ of\ $comb\_collection[i - 1]$}
                    \State $starting\_idx \gets$ index\ of\ $pre\_combo\_val$'s\ last\ element\ +\ 1
                    \For {$j \gets starting\_idx + 1$ to $N - 1$}
                        \State {$cur\_combo\_val \gets pre\_combo\_val \oslash v_{j}$}
                        \State {Push\ $cur\_combo\_val$\ into\ $cur\_combos$}
                    \EndFor
                \EndFor
            \EndIf
            \State Push\ $cur\_combos$\ into\ $comb\_collection$
            \State Push\ $cur\_combos$\ into\ $binomial\_theorem\_collection$
        \EndFor
        \\
        \State{$all\_combos \gets [\ ]$}
        \For{$combos$ of $binomial\_theorem\_collection$}
            \For{$cur\_combo\_val$ of $combos$}
                \State Push\ $cur\_combo\_val$\ into $all\_combos$
            \EndFor
        \EndFor
        \\
        \State {$SIZE \gets length\ of\ all\_combos$(SIZE's value\ is\ $2^N$)}
        \State {Init\ $asterisk\_cayley\_table[SIZE][SIZE]$}
        \For{$i \gets 0$ to $SIZE - 1$}
            \For{$j \gets 0$ to $SIZE - 1$}
                \State {$asterisk\_cayley\_table[i][j]\ \gets\ all\_combos[i]\ \oslash\ all\_combos[j] $}
                \State {Sort $asterisk\_cayley\_table[i][j]$ with element index with asc order}
            \EndFor
        \EndFor
    \end{algorithmic}
\end{algorithm}

FIG.~\ref{fig:5} shows the 5-dimensional Cayley table for the operation,
which contains 5 generators.

FIG.~\ref{fig:7} shows the 7-dimensional Cayley table for the operation,
which contains 7 generators.

\begin{figure*}[ht]
    \centering
    \includegraphics[width=0.7\textwidth]{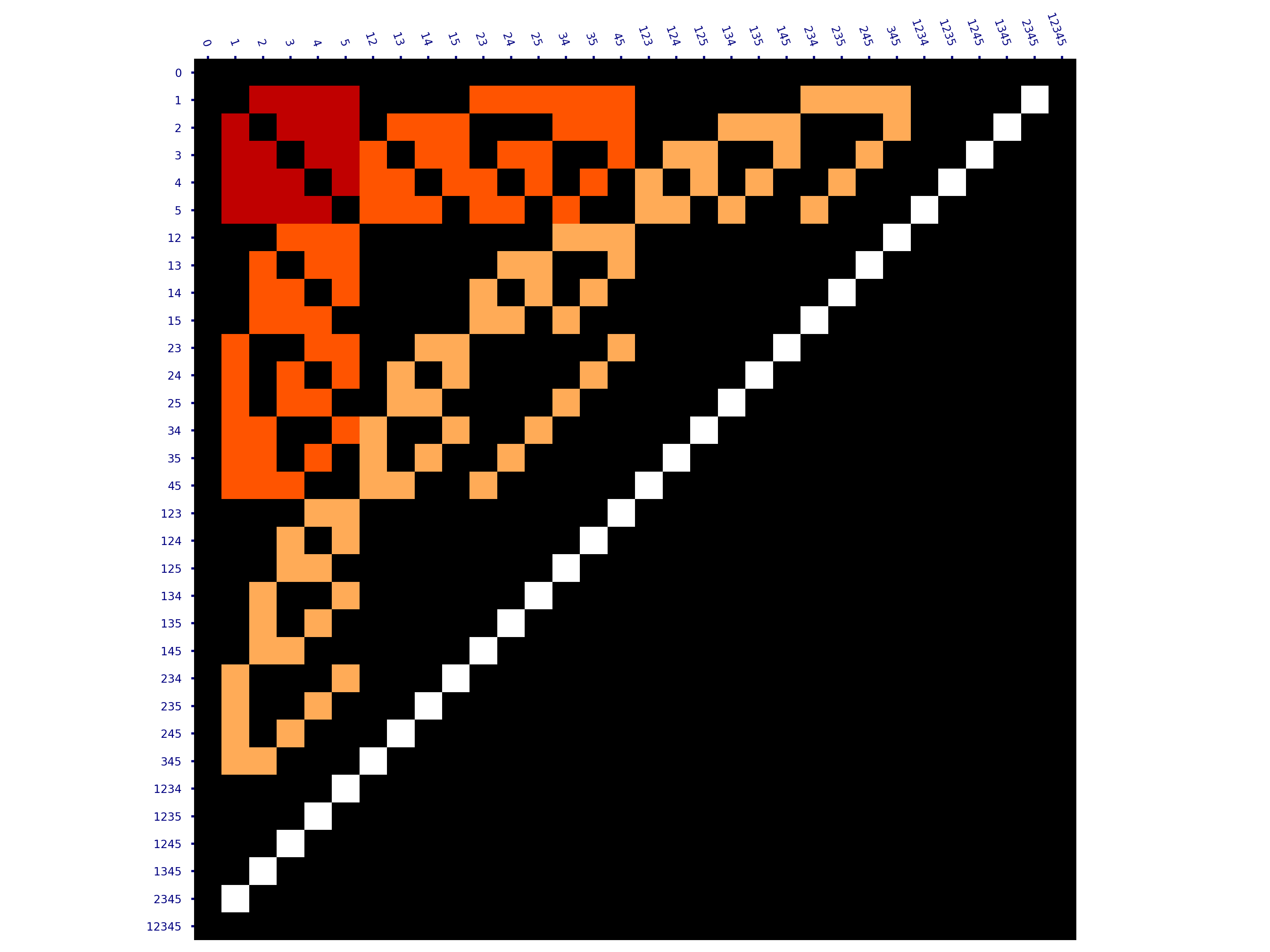}
    \caption{5-dimensional Full CSA Diagram}
    \label{fig:5}
\end{figure*}

\begin{figure*}[ht]
    \centering
    \includegraphics[width=0.7\textwidth]{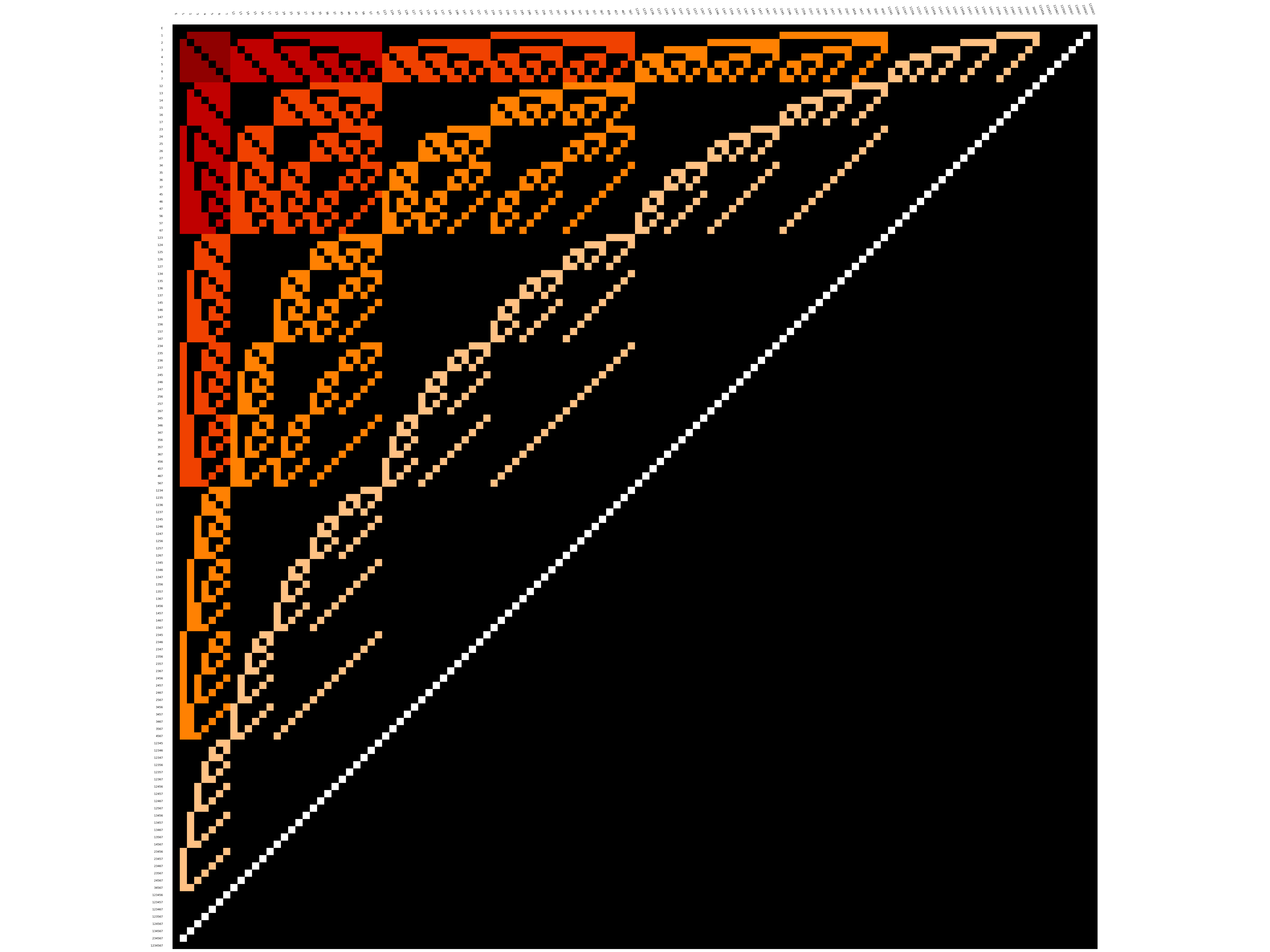}
    \caption{7-dimensional Full CSA Diagram}
    \label{fig:7}
\end{figure*}
\fi
        \subsubsection{A Special Algebra Structure of $\bOIES$ on $\oplus|_{\alpha}^{\beta}$}\label{subsubsec:cayley.algebraStructure}\EBL
        \ifincludeFile
Consider a special class of $\bOIES$ instances comprising $n$ ($n>2$) distinct $\bOIE$ instances, for example
\begin{align}
    \oieS_{special} = \{ \oie_1, \oie_2, \cdots, \oie_n \}\ (\forall k \in [1, n]: \oie_k \ne \voidError)
\end{align}
It satisfies \TT{uniqueness assumption} for $\oplus|_{\alpha}^{\beta}$
\begin{align}
    \begin{aligned}
        & \forall \oieS_1, \oieS_2 \subseteq \oieS_{special},\ \oieS_1 \neq \oieS_2: \\
        & \qquad    \forall \idxT_1 \in \operatorname{Perm}(\oieS_1),\ \forall \idxT_2 \in \operatorname{Perm}(\oieS_2): \\
        & \qquad\qquad    \oplus|_{\alpha}^{\beta}(\oieS_1, \idxT_1) \neq \oplus|_{\alpha}^{\beta}(\oieS_2, \idxT_2),
    \end{aligned}
\end{align}
where $\operatorname{Perm}(\oieS)$ denotes the set of all permutations of the elements in $\oieS$,
i.e., the set of all possible index tuples.

It satisfies the \TT{non-degeneracy assumption}
\begin{align}
    \begin{aligned} \label{eq:oies_non_void}
        & \forall \oieS_1 \subseteq \oieS_{special}: \\
        & \qquad    \forall \idxT_1 \in \operatorname{Perm}(\oieS_1): \\
        & \qquad\qquad    \oplus|_{\alpha}^{\beta}(\oieS_1, \idxT_1) \neq \voidError.
    \end{aligned}
\end{align}

Under these assumptions,
the \emph{orbital completion set} of $\oieS_{special}$ under $\oplus|_{\alpha}^{\beta}$,
denoted $\CAL{P}_{\oplus|_{\alpha}^{\beta}}(\oieS_{special})$,
is defined as the set of equivalence classes of $\bOIE$ instances
under permutational equivalence relation (Definition~\ref{def:OPs.Permu}).
Because $\oplus|_{\alpha}^{\beta}$ yields a single-orbit space (Property~\ref{pty:OPS.orbitSpace.add.oneElem}),
every unordered subset of operands determines a unique equivalence class,
independent of operand order.

Formally,
\begin{equation}
\label{eq:orbital_P}
\begin{aligned}
\CAL{P}_{\oplus|_{\alpha}^{\beta}}(\oieS_{special})
=\;&\big\{[\voidError]_{\sim}\big\} \\
&\cup\big\{[\oie_i]_{\sim}\mid 1\leq i\leq n\big\} \\
&\cup\bigcup_{k=2}^{n}
\left\{
    \left[ \oplus|_{\alpha}^{\beta}(\oie_{i_1},\ldots,\oie_{i_k}) \right]_{\sim}
\;\middle|\;
\{i_{1},\ldots,i_{k}\}\subseteq\{1,\ldots,n\}
\right\},
\end{aligned}
\end{equation}
where $[\,\cdot\,]_{\sim}$ denotes the equivalence class under permutational equivalence (Definition~\ref{def:OPs.Permu}).
By convention, the singleton case $k=1$ is already covered by the second term.

By Property~\ref{pty:OPS.orbitSpace.add.oneElem}, for any fixed unordered subset $\{i_{1},\ldots,i_{k}\}$,
all $k!$ permutations of the operands produce results that are permutationally equivalent;
hence they collapse to a single class in $P_{\oplus|_{\alpha}^{\beta}}$.
Consequently, the number of equivalence classes equals the number of subsets of an $n$-element set:
\begin{align}
    \big|\CAL{P}_{\oplus|_{\alpha}^{\beta}}(\oieS_{special})\big|
    =1+\sum_{k=1}^{n}\binom{n}{k}=2^{n}.
\end{align}

We need to describe these special $\bOIES$ instances.
For instance, an $\bOIES^{indep}$ instance satisfies this condition (Definition~\ref{def:OPs.feasibleCP.independentOIES}).
Such scenarios are extremely common in daily life.
Examples include testing the free-fall time of n balls with distinct masses;
n servers that do not communicate with one another and execute their respective tasks independently;
and all students in a class traveling individually from dormitories to classrooms in the morning.
In all these cases, each event involves a distinct subject, and there is no interaction or influence between them.
All these scenarios can be abstracted into this special type of algebra structure.
\fi
        \subsubsection{Bijection between Two Algebra Structures}\label{subsubsec:cayley.bijection}\EBL
        \ifincludeFile
Let $(\CAL{P}(\CAL{S}^{gen}),\oslash)$ be the finite commutative semigroup with absorbing element and reflexive absorption property constructed in Subsection~\ref{subsubsec:cayley.semigroup},
where $\CAL{S}^{gen}=\{v_{1}, v_2, \cdots,v_{n}\}$.
Define the mapping
\[
    \Psi: \CAL{P}(\CAL{S}^{gen}) \longrightarrow \CAL{P}_{\oplus|_{\alpha}^{\beta}}(\oieS_{special})
\]
by
\begin{equation}
\label{eq:bijection_Psi}
\Psi(\CAL{S}^{sub})=
\begin{cases}
[\voidError]_{\sim}, & \CAL{S}^{sub}=\{v_0\},\\[4pt]
\left[\oplus|_{\alpha}^{\beta}\big(\oie_{i_{1}},\ldots,\oie_{i_{k}}\big)\right]_{\sim}, & \CAL{S}^{sub}=\{v_{i_{1}}, \cdots, v_{i_{k}}\} \neq \emptyset.
\end{cases}
\end{equation}
Because of the uniqueness assumption, $\Psi$ is injective;
Their cardinalities satisfy
\begin{align}
    \CARDI{\CAL{P}_{\oplus|_{\alpha}^{\beta}}(\oieS_{special})} = 2^{n} = \CARDI{\CAL{P}(\CAL{S}_{gen})}
\end{align}

Moreover, $\Psi$ preserves the algebraic structure.
Define the induced binary operation $*$ on $P_{\oplus|_{\alpha}^{\beta}}$ by
\begin{equation}\label{eq:induced_op}
[\oie_A]_{\sim} * [\oie_B]_{\sim} =
\begin{cases}
[\voidError]_{\sim}, & [\oie_A]_{\sim}=[\oie_B]_{\sim} \\
& \quad\lor [\oie_A]_{\sim}=[\voidError]_{\sim} \\
& \quad\lor [\oie_B]_{\sim}=[\voidError]_{\sim},\\[12pt]
\left[\displaystyle\oplus|_{\alpha}^{\beta}\Bigl(constituents(\oie_A)\cup constituents(\oie_B)\Bigr)\right]_{\sim}, & \text{otherwise}.
\end{cases}
\end{equation}
where
\begin{align}
\mathsf{constituents}(\oie) =
\begin{cases}
\{\oie\}, & \oie \text{ is an } \bAtomOIE \text{ instance},\\[6pt]
\{\text{elements of } \C_{\oie}\}, & \oie \text{ is a } \bCombOIE \text{ instance}.
\end{cases}
\end{align}

Then $(P_{\oplus|_{\alpha}^{\beta}}, *)$ is a finite commutative semigroup with absorbing element $[\voidError]_{\sim}$ and reflexive absorption property $[\oie]_{\sim} * [\oie]_{\sim}=[\voidError]_{\sim}$,
and $\Psi$ is a bijection:
\[
\Psi(\CAL{S}^{sub}_A \cup \CAL{S}^{sub}_B) = \Psi(\CAL{S}^{sub}_A)\; * \;\Psi(\CAL{S}^{sub}_B)
\]

Scenarios satisfying uniqueness assumptions and non-degeneracy assumptions are ubiquitous.
For instance, a mutually independent $\bOIES$ instance (Definition~\ref{def:OPs.feasibleCP.independentOIES})
in which all atomic events are distinct and no cross-event constraints force distinct subsets to yield the same feasible schedule set $\mathcal{F}$ will satisfy these conditions.
Examples include $n$ independent servers executing non-communicating tasks, or $n$ students traveling individually from dormitories to classrooms.
\fi
    \subsection{The ``And'' in Natural Language \& ``Composition'' in Programming}\label{subsec:cayley.and}\EBL
    \ifincludeFile
The finite commutative semigroup structure expounded in Section~\ref{subsec:cayley.semigroup.algebraStructure}
exhibits a striking correspondence with the implicit operational semantics
underlying natural language conjunction and program composition.
This subsection elucidates these correspondences,
demonstrating that the reflexive absorption property $a \oslash a = v_{zero}$
encodes a universal constraint across linguistic, computational, and physical domains.

In natural language, the conjunction ``And'' operates under a fundamental presupposition of \textit{referential distinctness}.
When a speaker asserts a proposition of the form ``Event A and Event B'',
there exists an implicit requirement that the referents of A and B be distinguishable within the discourse context.
The utterance ``Dr.A and Dr.A submitted a paper'' is semantically anomalous precisely
because it violates this distinctness condition.
While natural language processing systems often treat ``and'' as an idempotent,
associative operator (interpreting $\{a,\ a,\ b\}$  as $\{a,\ b\}$),
this simplification elides a crucial cognitive constraint:
the mental model construction underlying linguistic comprehension
cannot sustain indistinguishable participants occupying distinct participant roles simultaneously.
This linguistic constraint manifests as \textit{self-referential absorption}:
the composition of an entity with itself fails to produce a meaningful augmentation of the conceptual structure.
Instead, it collapses into a null interpretation--
analogous to the absorption element $v_{zero}$ in our algebraic structure.
The finite commutative semigroup $(\CAL{S}, \oslash)$
thus provides a formal correlate to this linguistic phenomenon,
where the reflexive absorption property encodes the ungrammaticality of conjoining indistinguishable referents.

In computer programming, composition operators (function composition, concurrent process spawning, resource allocation)
enforce the distinctness constraint \textit{syntactically and operationally} rather than merely pragmatically.
Consider the following paradigmatic cases:

\begin{enumerate}
\item \textbf{Resource Allocation Semantics}:
The expression \textit{bind(socket, address)} composed with itself on the same socket descriptor raises a runtime exception (\textit{EADDRINUSE}), not an idempotent no-op.
This reflects the operational reality that a resource cannot be bound twice,
a physical constraint encoded as an absorbing error state.
\item \textbf{Process Composition}: In parallel programming frameworks, constructs like \texttt{par[A, B]} enable concurrent execution.
If $A = B$ and the operation requires distinct thread contexts, the composition yields a scheduling failure or deduplication, not parallelized self-composition.
\item \textbf{Functional Composition with Effects}: While $f \circ f$ is mathematically valid, in effectful computation (state monads, resource-transforming functions), composing an action with itself may violate resource or state consistency constraints.
\end{enumerate}

These examples demonstrate that programming composition,
particularly in the presence of state and resources,
conforms to the \textit{finite commutative semigroup with absorption} structure.
The void element corresponds to the computational bottom type, error monad, or exception state.
This structure proactively excludes invalid compositions before execution,
mirroring our $\bOIE$ framework's shift from passive observation to active planning.

The epistemological shift proposed in our $\bOIE$ framework is thus not merely mathematical but transdisciplinary:
it formalizes the universal constraint that the composition of indistinguishable entities
is not merely redundant but ontologically invalid in any system that schedules, plans, or allocates resources.
The Cayley table of Subsection~\ref{subsec:cayley.semigroup.algebraStructure}
is not merely an algebraic curiosity but a grammar of feasibility applicable to language, computation, and physics alike.

In conclusion,
by embedding Complete Sequence Addition within the structure of the finite commutative semigroup with absorbing element, we provide a single formalism that subsumes:
\begin{itemize}
\item Linguistic well-formedness constraint on ``And''
\item Computational safety constraints on composition
\item Physical feasibility constraints on event scheduling
\end{itemize}
This unification supports our central thesis:
event analysis must transcend passive observation and embrace the algebraic structure of proactive,
constraint-aware planning.
The reflexive absorption property is not an artifact of our formalism but a feature capturing a deep truth across domains:
identity, once used, cannot be reused without consequence.
\fi

\section{Derived Implications}\label{sec:impli}\EBL
\ifincludeFile
This section outlines promising directions for future research
that extend the $\bOIE$ framework and sequence operations
into broader theoretical and practical domains.

\begin{itemize}
    \item \textbf{Subsection~\ref{subsec:impli.category}. Introducing category theory:} Employing category theory to describe $\bOIE$ and sequence operations at a more abstract level, building upon the natural isomorphisms introduced in previous sections.
    \item \textbf{Subsection~\ref{subsec:impli.extensions}. Extensions of sequence operations:} Investigating ``incomplete'' variants of sequence addition and multiplication where only start or end timestamps are constrained, rather than both.
    \item \textbf{Subsection~\ref{subsec:impli.discretization}. Time discretization:} Examining whether the framework's conclusions remain valid under discrete time models as opposed to continuous real-number time.
    \item \textbf{Subsection~\ref{subsec:impli.variableization}. Time variableization:} Generalizing the temporal variable to other physical quantities such as length, temperature, or magnetic field strength, addressing the unobservability of pointwise equality in these dimensions.
    \item \textbf{Subsection~\ref{subsec:impli.probability}. Introducing probability/probability density:} Incorporating probability measures into $\bOIE$ instances to represent the likelihood of different execution intervals, enabling probabilistic event modeling.
    \item \textbf{Subsection~\ref{subsec:impli.others}. Other types of sequence operations:} Exploring the existence and algebraic properties of sequential subtraction and division.
    \item \textbf{Subsection~\ref{subsec:impli.lzyGraph}. Full CSA Diagram:} Investigating the structural properties, high-dimensional geometric patterns, and applications of the Complete Sequence Addition diagram across linguistics, programming, and other fields.
    \item \textbf{Subsection~\ref{subsec:impli.nonClassical}. Non-classical mechanics:} Extending the framework to quantum mechanics and other non-classical systems, leveraging the epistemological shift from observation to planning.
    \item \textbf{Subsection~\ref{subsec:impli.explosion}. Optimization of Cartesian Product:} Addressing computational complexity and practical optimization strategies for implementing sequence operations, particularly regarding the Cartesian product of interval sets.
\end{itemize}\fi
    \subsection{Introducing Category Theory}\label{subsec:impli.category}\EBL
    \ifincludeFile
Since we introduced the natural isomorphism of category theory
in Complete Sequence Addition(Definition~\ref{def:OPs.Add}) and Complete Sequence Multiplication(Definition~\ref{def:Ops.Multi}),
how to use category theory to describe $\bOIE$ and the two sequence operations at a more abstract level
is clearly a promising direction for further research~\cite{BaezStay2011}.\fi
    \subsection{The Extensions of Complete Sequence Addition and Complete Sequence Multiplication}\label{subsec:impli.extensions}\EBL
    \ifincludeFile
This paper only introduces two types of sequence operations,
both of which contain the word ``complete''.
The reason why we named them this way is that both the starting timestamp and the ending timestamp are deeply involved in the operation rules.
Therefore, if they are ``incomplete'', do there exist ``sequence addition based on a filtering 2-tuple and the starting timestamp''
and ``sequence multiplication based on the ending timestamp''?
Judging from the names, in fact, if there is no regulation on when the 100 metres must end,
the 100 metres seems like a ``sequence addition based on the start time'' that we conjecture.

In $\oplus|_{\alpha}^{\beta}$,
our requirement is that each $\bOIE$ instance involved in the operation
has an interval with the opportunity to reach the left boundary of the filter domain;
we do not require that all $\bOIE$ instances have the opportunity to reach it simultaneously.
This also serves as a perspective for constructing other sequence operations.

In subsection~\ref{subsec:comp.phy},
we also emphasize that the classical notion of simultaneous start based on pointwise temporal equality is distinct from Complete Sequence Addition.
From the viewpoint of using sequence operations to rectify the defective epistemology of simultaneity in conventional theories,
many promising research directions can be derived.

In short, sequence operations constitute a rich family of operators under the $\bOIE$ framework.
This work only presents a detailed analysis of two representative instances.
We believe that diverse other sequence operations merit further exploration from both theoretical and practical standpoints.

\fi
    \subsection{Time Discretization}\label{subsec:impli.discretization}\EBL
    \ifincludeFile
In this paper, the understanding of time is to regard time as a real number.

Is time modeled as a real number?
If time is discrete, are the viewpoints in this paper correct?

Since both $\bE$ and $\bOIE$ share the same time model,
we speculate that a discrete time model will not affect the conclusions of this paper.
However, this also requires further in-depth analysis.
\fi
    \subsection{Time Variableization}\label{subsec:impli.variableization}\EBL
    \ifincludeFile
Sequence operations are targeted at time.
As we have discussed in subsection~\ref{subsec:comp.phy},
if we abstract time once again and just regard it as a variable,
then by replacing this variable with other physical elements
(such as length, temperature, magnetic field strength, and so on),
we can generalize the framework to other physical dimensions.
Humans have never observed pointwise-equal simultaneity,
nor have observed pointwise equality in altitude or temperature.
Can the framework of this paper still be applied?

\fi
    \subsection{Introducing Probability / Probability Density}\label{subsec:impli.probability}\EBL
    \ifincludeFile
For the $\I$(3rd element) of an $\bOIE$ instance,
if it is a finite set,
what is the probability of each set element (that is, each interval) being used?
If it is an infinite set, what is the probability density?
This paper discusses the existence of $\bOIE$ and does not involve probability issues.

After introducing probability,
to represent probabilistic data,
the $\bOIE$ proposed in this paper is likely to serve as a fundamental structure,
with additional elements incorporated such that the number of elements within a tuple exceeds four.
This will introduce more properties and methods.

Therefore, this is a great potential for applications.
\fi
    \subsection{Other Types of Sequence Operations}\label{subsec:impli.others}\EBL
    \ifincludeFile
Do sequential subtraction and sequential division exist?
If they do, what are their algebraic properties?
\fi
    \subsection{Full CSA Diagram}\label{subsec:impli.lzyGraph}\EBL
    \ifincludeFile
The N-dimensional Full CSA diagram is not only a kind of Cayley table for $\oplus|_{\alpha}^{\beta}$,
but also a universal representation for any commutative composition operation with reflexive absorption.

There is a class of operations that satisfy the closure property, the commutative law, and the associative law,
but the operation between two identical numbers is not allowed.
Their Cayley tables are all N-dimensional Full CSA diagrams.

For example, the four major oceans of the Earth consist of the Pacific Ocean, the Atlantic Ocean, the Indian Ocean, and the Arctic Ocean, which can be expressed as
\begin{align}
    \begin{aligned}
        & \text{Pacific Ocean} + \text{Atlantic Ocean} \\
        & \quad + \text{Indian Ocean} + \text{Arctic Ocean}
    \end{aligned}
\end{align}
Each ocean can only appear once in the expression.
An expression with repeated occurrences is incorrect.
For example, the Pacific Ocean should not participate in the composition of the world's oceans twice.
If participating twice is wrong, then
\begin{align}
    \text{Pacific Ocean} + \text{Pacific Ocean} = Error
\end{align}
Obviously, the 4-dimensional Full CSA diagram can be used to represent the ways of combining the four major oceans.

Let's take another classic problem in the field of computer science.
``Composition'' and ``Inheritance'' are two important fundamental means for handling code reuse.

Inheritance specifically refers to a way of creating new classes in object-oriented programming.
A subclass can inherit the attributes and methods of its superclass and thus expand the functions of the superclass.
Inheritance represents an ``Is-a'' relationship.

Composition is not limited to object-oriented programming.
It usually means that a class (or structure) contains an object (or structure) of another class as its member variable,
and the required functions are implemented through these member objects.
Composition represents a ``Has-a'' relationship.

Suppose a class $Class_A$ is composed of three classes
\begin{align}
    Class_A = Class_1 + Class_2 + Class_3
\end{align}
Then their objects are
\begin{align}
    Object_A = Object_1 + Object_2 + Object_3
\end{align}
In this way, through the composition of the three objects, $Object_A$ can implement multiple functions. Obviously, the same class cannot participate in the composition twice.

So we believe that the algebraic properties of composition in the field of programming have many similarities with $\oplus|_{\alpha}^{\beta}$, and the N-dimensional Full CSA diagram can be used.
Since composition can be regarded as addition, could inheritance be a certain kind of multiplication? We think this is possible and worthy of further analysis.
From this perspective, there is still a lot of structural work that has not been done in the field of object-oriented programming in computer science alone.

In conclusion, similar scenarios can be found in many fields, so the applicable range of the N-dimensional Full CSA diagram is very wide.

Another aspect worthy of in-depth study is the diagram itself.
As the number of dimensions increases, many complex and beautiful curves and shapes appear.
What can these shapes and curves be used for, and what special properties do high-dimensional diagrams have?
These are also questions worth investigating.
\fi
    \subsection{Non-classical Mechanics}\label{subsec:impli.nonClassical}\EBL
    \ifincludeFile
The contents discussed in this paper are all within the scope of classical Newtonian mechanics.
So, are they still applicable within the scope of non-classical mechanics?

In fact, the philosophical foundation of this paper is
to shift the starting point of observation in physics from pure objectivity to an epistemological perspective.
In section~\ref{sec:impl},
we showed that the n‑ary finitary operation finally observed in reality from the projection on Complete Sequence Addition
is only one of many possible expressions,
which implies that, within the framework of this paper,
real‑world phenomena may correspond to just one case among many possibilities.
In subsection~\ref{subsec:app.math} and subsection~\ref{subsec:comp.math},
we mentioned the difference between process symmetry and outcome symmetry.
In subsection~\ref{subsec:impli.discretization} and subsection~\ref{subsec:impli.variableization},
we mentioned the discretization of physical quantities.
In subsection~\ref{subsec:impli.probability},
we further mentioned that we can extend the arity of $\bOIE$
and introduce probability or probability density.
Furthermore, as noted in subsection~\ref{subsec:impli.category},
this paper has preliminarily applied category theory.
Combining these existing and extensible ideas yields a consistent view:
\begin{align}
    \begin{aligned}
        & \TT{Cognition influences observational outcomes} \\
        & + \TT{Observable outcomes are non‑unique} \\
        & + \TT{Discretization of physical quantities} \\
        & + \TT{Process symmetry and outcome symmetry} \\
        & + \TT{Probability} \\
        & + \TT{Category theory}.
    \end{aligned}
\end{align}
This is the epistemological core of quantum mechanics.

This is advanced as a conjecture,
to be subjected to rigorous theoretical and empirical scrutiny in future work.
\fi
    \subsection{Optimization of Cartesian Product}\label{subsec:impli.explosion}\EBL
    \ifincludeFile
Both $\oplus|_{\alpha}^{\beta}$ and $\otimes$
involve the flattened Cartesian product of sets.
The time complexity of both sequence operations is $\CAL{O}(n \cdot m^n)$ (Appendix~\ref{sec:appendix.complex}).
When applying sequence operations,
we mostly do not perform flattened Cartesian products with heavy computation loads.
Instead, we only analyze algebraic properties to make decisions or explore the derived conclusions.
This is analogous to how infinite sets can be used to analyze business problems in the computer industry,
yet it does not mean you must implement an algorithm with infinite time complexity in order to use infinite sets.

However, if one insists on traversing the results of the Cartesian product
(for instance, in scheduling or deadlock problems with a small number of nodes but high business importance,
where only low-frequency traversal is employed),
it is still necessary to present the time complexity of the two sequence operations.
Current Cartesian product optimizations cover storage indexing, pruning algorithms and distributed execution.
The mainstream idea is to reduce the cardinality of cross products via algebraic simplification first,
then deploy diverse algorithms tailored to business scenarios to cut computation costs and avoid computational explosion.

In most daily and industrial scenarios,
we skip costly high-complexity calculations by determining whether only feasibility checks or full sequence operation execution are required.
Of course, if the values of \(n\) and \(m\) are not large,
directly following the computational process is also acceptable.
In practice, the exponential overhead of the Cartesian product can be effectively controlled using standard database optimization techniques.
Therefore, integrating the $\bOIE$ framework with sequence operations and database optimization remains a viable research direction.
\fi

    \appendix
    \section{The Simulation of $\bOIE$ and 2 Sequence Operations}\label{sec:appendix.simul}\EBL
    \ifincludeFile
(Property~\ref{prop:event.startEndTs.seq}) Ordering of starting timestamp and ending timestamp:

\url{
    https://github.com/yleen/OIE_simulation/blob/master/src/simulation/optional_intervals_event/event.py#L44
}

(Definition~\ref{def:eventStar}) Event with undetermined interval:

\url{
    https://github.com/yleen/OIE_simulation/blob/master/src/simulation/optional_intervals_event/event_star.py#L14
}

(Definition~\ref{def:oie}) Optional Intervals Event:

\url{
    https://github.com/yleen/OIE_simulation/blob/master/src/simulation/optional_intervals_event/optional_intervals_event.py#L17
}

(Definition~\ref{def:oie.BoundTwoTuple}) The bound 2-tuple of a non-empty finite $\bTwoTplT$ instance:

\url{
    https://github.com/yleen/OIE_simulation/blob/master/src/simulation/base/helper.py#L33
}

(Property~\ref{prop:oie.2and3}) The relationship between the $\CAL{F}$(2nd element) and $\CAL{I}$(3rd element) of $\bOIE$:

\url{
    https://github.com/yleen/OIE_simulation/blob/master/src/simulation/base/helper.py#L12
}

(Definition~\ref{defi:OPs.oie.Equal}) Equality of $\bOIE$ instances:

\url{
    https://github.com/yleen/OIE_simulation/blob/master/src/simulation/optional_intervals_event/optional_intervals_event.py#L42
}

(Definition~\ref{def:oie.oieS.int2tupleSS}) Set of $\I$ of finite $\bOIES$ instance:

\url{
    https://github.com/yleen/OIE_simulation/blob/master/src/simulation/optional_intervals_event/optional_intervals_event_set.py#L59
}

(Definition~\ref{def:pdie.void}) Void $\bOIE$:

\url{
    https://github.com/yleen/OIE_simulation/blob/master/src/simulation/optional_intervals_event/optional_intervals_event.py#L102
}

(Definition~\ref{def:op.fNItCPo2TplSS})Function to get the set that is naturally isomorphic to the Cartesian product of all members of a finite $\bTwoTplSS$ instance in an index order:

\url{
    https://github.com/yleen/OIE_simulation/blob/master/src/simulation/optional_intervals_event/naturally_isomorphic_to_cartesian_product.py#L13
}

(Definition~\ref{def:OPs.infsbleDI}) Set of tuples of infeasible interval 2-tuples of all members of a finite $\bOIES$ instance under an index order:

\url{
    https://github.com/yleen/OIE_simulation/blob/master/src/simulation/optional_intervals_event/optional_intervals_event_set.py#L112
}

(Definition~\ref{def:OPs.add.min1max2ofTwoTplT} / function 1) Get the minimum 1st item of a $\bTwoTplT$ instance:

\url{
    https://github.com/yleen/OIE_simulation/blob/master/src/simulation/base/helper.py#L72
}

(Definition~\ref{def:OPs.add.min1max2ofTwoTplT}  / function 2) Get the maximum 2nd item of a $\bTwoTplT$ instance:

\url{
    https://github.com/yleen/OIE_simulation/blob/master/src/simulation/base/helper.py#L89
}

(Definition~\ref{def:OPs.Add.DomainFilteredSubTwoTplTS}) Domain-filtered subset of a non-empty set of tuples of 2-tuples under a domain-filtering 2-tuple:

\url{
    https://github.com/yleen/OIE_simulation/blob/master/src/simulation/sequence_operation/addition/domain_filtered_2tupleTS.py#L34
}

(Definition~\ref{def:OPs.Add}) Complete Sequence Addition:

\url{
    https://github.com/yleen/OIE_simulation/blob/master/src/simulation/sequence_operation/addition/complete_sequence_addition.py#L26
}

(Definition~\ref{def:OPs.Multi.CompleteAscOrderFilteredSubTwoTplTS}) Ascending ordered filtered subset of the set of tuples of 2-tuples:

\url{
    https://github.com/yleen/OIE_simulation/blob/master/src/simulation/sequence_operation/multiplication/complete_asc_order_filtered_2tupleTS.py#L11
}

(Definition~\ref{def:Ops.Multi}) Complete Sequence Multiplication:

\url{
    https://github.com/yleen/OIE_simulation/blob/master/src/simulation/sequence_operation/multiplication/complete_sequence_multiplication.py#L25
}

(Property~\ref{pty:OPs.ErrorOIE.sameOIE}) Sequence operations involving identical $\bOIE$ instances result in an $\bVoidError$ instance:

\url{
    https://github.com/yleen/OIE_simulation/blob/master/src/simulation/sequence_operation/helper.py#L82
}

(Property~\ref{pty:OPs.ErrorOIE.involved}) The sequence operations involving $\bVoidError$ instances result $\voidError$:

\url{
    https://github.com/yleen/OIE_simulation/blob/master/src/simulation/sequence_operation/helper.py#L89
}

(Definition~\ref{def:OPs.Add.Nat}) Natural Complete Sequence Addition:

\url{
    https://github.com/yleen/OIE_simulation/blob/master/src/simulation/sequence_operation/addition/complete_sequence_addition.py#L79
}

(Definition~\ref{def:implementAndProps.implement1}) The 1st type of implementation of an $\bOIE$ instance:

\url{
    https://github.com/yleen/OIE_simulation/blob/master/src/simulation/implement/first_type.py#L13
}

(Definition~\ref{def:implementAndProps.implement2}) The 2nd type of implementation of an $\bOIE$ instance:

\url{
    https://github.com/yleen/OIE_simulation/blob/master/src/simulation/implement/second_type.py#L14
}

For test cases running of Complete Sequence Addition and Complete Sequence Multiplication, visit \\

\url{
    https://github.com/yleen/OIE_simulation/blob/master/src/main.py
}\fi
    \section{The Generation of N-dimensional Full CSA Diagram}\label{sec:appendix.lzyGraph}\EBL
    \ifincludeFile
The simulated code implementation of Algorithm~\ref{alg:algo}:

Generate the binomial theorem collection as edge of Full CSA diagram:\\
\url{
https://github.com/yleen/OIE_simulation/blob/master/src/fullCsaDiagram/diagram_generator.py#L52
}.

Build Full CSA diagram:\\
\url{
https://github.com/yleen/OIE_simulation/blob/master/src/test/full_CSA_diagram/generate.py#L138
}.
\fi
    \section{The Time Complexity}\label{sec:appendix.complex}\EBL
    \ifincludeFile
\indent Let $n = \CARDI{\oieS}$ denote the cardinality of the finite $\bOIES$ instance $\oieS$,
and let $m_i = \CARDI{\I_{\oie_i}}$ denote the cardinality of the $\I$(the 3rd element) of the $i$-th $\bOIE$ instance.

\textbf{Complete Sequence Addition} (Definition~\ref{def:OPs.Add}):
\begin{itemize}
    \item \textit{Step 1} (Validity check): Checking for $\voidError$ and computing the intersection of atomic event sets $\bigcap_{i=1}^{n}\A_{\oie_{\idx_i}}$ requires $\CAL{O}(n)$ time using hash-based set operations.
    \item \textit{Step 2} (Feasible subset construction): Computing the set naturally isomorphic to the Cartesian product $\fNatIsoToCP{\ISoOieS_\oieS}{\idxT}$ generates $\prod_{i=1}^{n} m_i$ tuples.
        So this step requires $\CAL{O}\left(\prod_{i=1}^{n} m_i\right)$ time.
    \item \textit{Step 3} (Domain filtering): Applying $\mathfrak{f}\RM{DomFlt2tupleTS}$ requires checking the domain constraints $(\alpha, \beta)$ for all $n$ components of each tuple, yielding $O\left(n \cdot \prod_{i=1}^{n} m_i\right)$ time.
    \item \textit{Step 4} (Computing $\I$): Computing the bound 2-tuple via $\mathfrak{f}\RM{Bound2tuple}$ for each remaining tuple requires $\CAL{O}(n)$ per tuple, resulting in $\CAL{O}\left(n \cdot \prod_{i=1}^{n} m_i\right)$ time.
\end{itemize}
Thus, the total time complexity is $\CAL{O}\left(n \cdot \prod_{i=1}^{n} m_i\right)$.
If we let $m = \max_{1 \leq i \leq n} m_i$, this simplifies to $\CAL{O}(n \cdot m^n)$.

\textbf{Complete Sequence Multiplication} (Definition~\ref{def:Ops.Multi}):
The complexity is dominated by Step 3,
where $\mathfrak{f}\RM{AscOrderFlt2tupleTS}$ applies the ascending ordered filter (Definition~\ref{def:OPs.Multi.CompleteAscOrderFilteredSubTwoTplTS}) to enforce $\TSe_i \leq \TSe_{i+1}$ for consecutive events.
Verifying this condition requires $\CAL{O}(n)$ per tuple,
yielding the same overall complexity of $\CAL{O}\left(n \cdot \prod_{i=1}^{n} m_i\right) = \CAL{O}(n \cdot m^n)$.

The time complexity derived above represents the worst-case bound,
where full enumeration of the Cartesian product of interval sets is required.
However, we identify a broad class of practically relevant scenarios
(Specifically, mutually independent $\bOIES$ instances (Definition~\ref{def:OPs.feasibleCP.independentOIES}) over continuous time domains),
where the feasibility verification for either $\oplus|_{\alpha}^{\beta}$ or $\oplus$ reduces to a constant $\CAL{O}(1)$.
In such scenarios, the planner need not perform the full four-step enumeration of the flattened Cartesian product;
instead, the applicability of $\oplus|_{\alpha}^{\beta}$ or $\otimes$ can be determined directly
from the physical independence or sequential constraints of the underlying real-world events
(e.g., independent tasks deployed on independent hosts in a web application, and fully isolated hardware circuits).
The complete construction of the resulting $\bOIE$ instance,
particularly the explicit enumeration of its second element $\F$, remains $\CAL{O}(n \cdot m^n)$ in the general discrete case,
but is bypassed in practice through this $\CAL{O}(1)$ structural judgment based on prior physical knowledge.

For a mutually independent $\bOIES$ instance over a continuous domain $[\alpha, \beta)$ (Definition~\ref{def:OPs.feasibleCP.independentOIES}),
\TT{in scenarios where the left and right boundaries are known a priori to satisfy the operational requirements (e.g., the 100-meter dash where all athletes share the same start and finish time window), there is no need to perform Cartesian product enumeration or complex set intersection operations at all}.
Both \TT{feasibility verifications} (Notice: not computational complexity) reduce directly to $\CAL{O}(1)$ structural judgments:
the implementer merely needs to confirm that all participants can access the left boundary $\alpha$ and the right boundary $\beta$, while $\otimes$ only needs to confirm that the total time window can accommodate the sum of minimum sequential durations, thereby completely eliminating the computational burden of full enumeration.
\fi

\bibliographystyle{unsrt}

\phantomsection
\addcontentsline{toc}{section}{\refname}
\bibliography{arxiv}

@PREAMBLE{
 "\providecommand{\noopsort}[1]{}"
 # "\providecommand{\singleletter}[1]{#1}%"
}

@article{PhysRevLett.133.023401,
   title = {Clock with $8\ifmmode\times\else\texttimes\fi{}{10}^{\ensuremath{-}19}$ Systematic Uncertainty},
   author = {Aeppli, Alexander and Kim, Kyungtae and Warfield, William and Safronova, Marianna S. and Ye, Jun},
   journal = {Phys. Rev. Lett.},
   volume = {133},
   issue = {2},
   pages = {023401},
   numpages = {7},
   year = {2024},
   month = {Jul},
   publisher = {American Physical Society},
   doi = {10.1103/PhysRevLett.133.023401},
   url = {https://link.aps.org/doi/10.1103/PhysRevLett.133.023401}
}

@article{Allen83,
    author  = {James F. Allen},
    title   = {Maintaining Knowledge about Temporal Intervals},
    journal = {Communications of the ACM},
    year    = {1983},
    volume  = {26},
    number  = {11},
    pages   = {832--843},
    month   = nov,
    doi     = {10.1145/182.358434},
}

@article{DechterMP91,
  author  = {Rina Dechter and Itay Meiri and Judea Pearl},
  title   = {Temporal Constraint Networks},
  journal = {Artificial Intelligence},
  year    = {1991},
  volume  = {49},
  number  = {1-3},
  pages   = {61--95},
  doi     = {10.1016/0004-3702(91)90007-3},
}

@article{BergstraK84,
  author  = {J. A. Bergstra and J. W. Klop},
  title   = {Process Algebra for Synchronous Communication},
  journal = {Information and Control},
  year    = {1984},
  volume  = {60},
  number  = {1-3},
  pages   = {109--137},
  doi     = {10.1016/S0019-9958(84)80025-X},
}

@InCollection{sep-events,
   author       = {Casati, Roberto and Varzi, Achille},
   title        = {{Events}},
   booktitle    = {The {Stanford} Encyclopedia of Philosophy},
   editor       = {Edward N. Zalta and Uri Nodelman},
   howpublished = {\url{https://plato.stanford.edu/archives/fall2023/entries/events/}},
   year         = {2023},
   edition      = {{F}all 2023},
   publisher    = {Metaphysics Research Lab, Stanford University}
}

@book{birkhoff1940lattice,
   title={Lattice Theory},
   author={Birkhoff, G.},
   isbn={9780821810255},
   lccn={66023707},
   series={American Mathematical Society colloquium publications},
   url={https://books.google.com.hk/books?id=ePqVAwAAQBAJ},
   year={1940},
   publisher={American Mathematical Society}
}

@inproceedings{BaezStay2011,
  title = {Physics, Topology, Logic and Computation: A {Rosetta} Stone},
  author = {Baez, John C. and Stay, Mike},
  editor = {Coecke, Bob},
  booktitle = {New Structures for Physics},
  series = {Lecture Notes in Physics},
  volume = {813},
  pages = {95--172},
  year = {2011},
  publisher = {Springer}
}

@book{Milner89,
  author    = {Robin Milner},
  title     = {Communication and Concurrency},
  series    = {Prentice-Hall International Series in Computer Science},
  publisher = {Prentice-Hall},
  year      = {1989},
  isbn      = {0131150073},
}

@phdthesis{Petri62,
  author      = {Carl Adam Petri},
  title       = {{Kommunikation mit Automaten}},
  school      = {Universit{\"a}t Bonn},
  year        = {1962},
  address     = {Bonn, Germany},
  note        = {Schriften des Institutes f{\"u}r Instrumentelle Mathematik},
}

@article{Murata89,
  author  = {Tadao Murata},
  title   = {{Petri Nets: Properties, Analysis and Applications}},
  journal = {Proceedings of the IEEE},
  year    = {1989},
  volume  = {77},
  number  = {4},
  pages   = {541--580},
  month   = apr,
  doi     = {10.1109/5.24169},
}

@article{NebelB95,
  author  = {Bernhard Nebel and Hans-Jürgen Bürckert},
  title   = {Reasoning about Temporal Relations: A Maximal Tractable Subclass of Allen's Interval Algebra},
  journal = {Journal of the ACM},
  year    = {1995},
  volume  = {42},
  number  = {1},
  pages   = {43--66},
  doi     = {10.1145/200836.200848},
}

@book{Whitehead1978,
  author    = {Alfred North Whitehead},
  editor    = {David Ray Griffin and Donald W. Sherburne},
  title     = {Process and Reality: An Essay in Cosmology},
  edition   = {Corrected},
  publisher = {Free Press},
  address   = {New York},
  year      = {1978},
  isbn      = {0-02-934570-7},
}

@book{kant1998critique,
  author    = {Kant, Immanuel},
  title     = {Critique of Pure Reason},
  publisher = {Cambridge University Press},
  year      = {1998},
  edition   = {The Cambridge Edition of the Works of Immanuel Kant},
  note      = {Original work published 1781. Key sections: Transcendental Aesthetic (Space and Time as forms of intuition) and Transcendental Deduction (Categories as conditions of possibility for experience)}
}

@book{kant2004metaphysical,
  author    = {Kant, Immanuel},
  title     = {Metaphysical Foundations of Natural Science},
  publisher = {Cambridge University Press},
  year      = {2004},
  note      = {Original work published 1786. Discusses how mathematical construction precedes empirical physics}
}

@article{bridgman1936nature,
  author    = {Bridgman, Percy W.},
  title     = {The Nature of Some of Our Physical Concepts},
  journal   = {The British Journal for the Philosophy of Science},
  volume    = {1},
  number    = {4},
  pages     = {257--272},
  year      = {1936},
  note      = {Discusses the operational definition of simultaneity and time}
}

@book{von1984constructivism,
  author    = {von Glasersfeld, Ernst},
  title     = {An Introduction to Radical Constructivism},
  booktitle = {The Invented Reality: How Do We Know What We Believe We Know?},
  publisher = {Norton},
  year      = {1984},
  note      = {Knowledge is not passively received but actively built up by the cognizing subject}
}

@book{chen1991youwu,
  author    = {Chen, Lai},
  title     = {The Realm of Being and Non-Being: The Spirit of {W}ang {Y}angming's Philosophy},
  publisher = {Renmin Chubanshe (People's Publishing House)},
  year      = {1991},
  note      = {Chapter 3 analyzes the `flower' episode from the \textit{Chuanxilu} (Instructions for Practical Living), arguing that Wang's thesis of ``no things outside the mind'' refers to the intentional constitution of meaning rather than ontological idealism. Revised edition published by Peking University Press in 2013},
  address   = {Beijing}
}

@article{kern2011flower,
  author    = {Kern, Iso},
  title     = {The `Flower' Episode in {W}ang {Y}angming's Philosophy: A Phenomenological Analysis},
  journal   = {Dao: A Journal of Comparative Philosophy},
  volume    = {10},
  number    = {2},
  pages     = {187--202},
  year      = {2011},
  note      = {Uses Husserl's noesis-noema correlation to analyze the `flower' episode, demonstrating that Wang's concept of ``the flower and the mind both returning to silence'' reflects the phenomenological thesis that the sense (Sinn) of objects is constituted through intentional consciousness}
}

\end{document}